\documentclass[final,onefignum,onetabnum]{siamart250211}
\usepackage{graphicx}
\usepackage{amsmath,multirow}
\usepackage{amssymb}
\usepackage{amscd}
\usepackage{amsfonts}
\usepackage{enumerate}
\usepackage{mathrsfs}
\usepackage{xypic}
\usepackage{stmaryrd}
\usepackage{graphicx}
\usepackage{fancybox}
\usepackage{color}
\usepackage{exscale}
\usepackage{enumitem,array}%
\usepackage{amscd}
\usepackage{amsfonts}
\usepackage{xy}
\usepackage{mathtools}
\usepackage{stmaryrd}
\usepackage{fancybox}
\usepackage{color}
\usepackage{booktabs}
\usepackage{subcaption}
\usepackage{extarrows}
\usepackage{xypic}
\usepackage{tikz}
\usetikzlibrary{arrows.meta,positioning,calc}
\usepackage{algorithm}
\usepackage{algorithmic}
\usepackage[a4paper,left=2.5cm,right=3cm,top=2.5cm,bottom=2.5cm]{geometry}
\usepackage{manyfoot}
\usepackage{verbatim}
\usepackage{amsopn}
\usetikzlibrary{arrows,chains,matrix,positioning,scopes}
\makeatletter
\tikzset{join/.code=\tikzset{after node path={%
\ifx\tikzchainprevious\pgfutil@empty\else(\tikzchainprevious)%
edge[every join]#1(\tikzchaincurrent)\fi}}}
\makeatother
\tikzset{>=stealth',every on chain/.append style={join},
         every join/.style={->}}
\tikzstyle{labeled}=[execute at begin node=$\scriptstyle,
   execute at end node=$]

\newsiamremark{remark}{Remark}
\newsiamremark{hypothesis}{Hypothesis}
\crefname{hypothesis}{Hypothesis}{Hypotheses}
\newsiamthm{claim}{Claim}

\usepackage{amssymb,amsmath,geometry,color}
\usepackage{epstopdf}
\usepackage{epsfig,multirow}
\usepackage{color,cases}

\hypersetup{
    colorlinks=true,
    linkcolor=black, 
    citecolor=black  
}

\newcommand{\fe}{\mathrm{e}}

\newcommand{\eps}{\varepsilon}

\newcommand{\bT}{{\mathbb T}}

\newcommand{\abs}[1]{\left\vert#1\right\vert}

\newcommand{\bx}{\mathbf{x}}

\newcommand{\mD}{\mathcal{D}}
\newcommand{\mF}{\mathcal{F}}

\newcommand{\mW}{\mathcal{W}}
\newcommand{\hbZ}{\widehat{\mathbf{Z}}}
\newcommand{\norm}[1]{\left\Vert#1\right\Vert}
\def\ii{\textmd{i}}

\headers{Uniformly accurate schemes for coupled dispersive models}{L. Wang, J. Li, B. Wang}

\title{High-order uniformly accurate fully discrete schemes with long time near-conservations for two nonlinear coupled dispersive models in the nonrelativistic regime
\thanks{Submitted to the editors DATE.
\funding{This work was supported partially by the National Natural Science Foundation of China (12371403), the Youth Project (Category A) of Shaanxi Provincial Natural Science Basic Research Program (2026JC-JCQN-013), Shaanxi Fundamental Science Research Project for Mathematics and Physics (25JSY046) and  Fundamental  Research  Funds  for  the Central Universities (xzy022026045).}}}

\author{Lina Wang\thanks{School of Mathematics and Statistics, Xi'an Jiaotong University, 710049 Xi'an, China
  (\email{wanglina@stu.xjtu.edu.cn}).}
  \and Jiyong Li\thanks{School of Mathematics Science, Hebei Key Laboratory of Computatioanal Mathematics  and Applications, Hebei International Joint Research Center for Mathematics and Interdisciplinary Science, Hebei Normal University, Shijiazhuang 050024,  China
  (\email{ljyong406@163.com}).}
  \and Bin Wang\thanks{Corresponding author.
  School of Mathematics and Statistics, Xi'an Jiaotong University, 710049 Xi'an, China
  (\email{wangbinmaths@xjtu.edu.cn}).}
 }

\begin{document}

\maketitle

\begin{abstract}
Nonlinear coupled dispersive models arise in numerous physical settings, including quantum electrodynamics, nonlinear optics, and plasma
dynamics. In the nonrelativistic limit, such systems exhibit rapid temporal oscillations of wavelength $\mathcal{O}(\varepsilon^2)$, controlled by a dimensionless parameter $0<\varepsilon\ll1$. Moreover, owing to the Yukawa-coupled
interaction terms, numerical schemes that simultaneously achieve high-order uniform
accuracy and approximately preserve the mass and the energy are essential for reliable simulations. In this work, we propose a class of fully discrete schemes for nonlinear coupled dispersive models in the nonrelativistic regime that achieve
higher-order uniform accuracy and approximately preserve both invariants over long times. Spatial discretization is
based on the Fourier pseudospectral method, and time advancement employs two-scale exponential Runge--Kutta methods. Numerical experiments on the highly oscillatory
Klein--Gordon--Dirac and Klein--Gordon--Schr\"odinger equations in one, two, and three dimensions confirm that the resulting schemes attain the theoretically predicted uniform accuracy and the long-time near-conservation of both invariants.  
\end{abstract}



\section{Introduction}
Nonlinear coupled dispersive models play a fundamental role in both mathematics
and physics. Representative examples include the Klein--Gordon--Schr\"{o}dinger
system, the Klein--Gordon--Dirac system, and the Klein--Gordon--Zakharov system,
whose variants have found wide applications in the nonlinear optics and optical communications \cite{Shen}, nonlinear quantum electrodynamics
\cite{Gre,Hol,Ohl,SK}, and Langmuir turbulence in plasma dynamics
\cite{MRZ,Ni}. Theoretical and numerical studies of such nonlinear coupled
systems have long attracted extensive attention.

 \subsection{Nonlinear coupled dispersive models}
In this work, we develop a class of fully discrete schemes that achieve uniform
high accuracy and approximately preserve two conserved quantities of two prominent
nonlinear coupled dispersive models in the nonrelativistic regime: the
Klein--Gordon--Dirac (KGD) equation and the Klein--Gordon--Schr\"{o}dinger (KGS)
equation. As fundamental models in quantum electrodynamics and meson--nucleon
interaction theory, respectively, the KGD and KGS systems and their variants have
been widely employed, e.g., in the Yukawa description of the nuclear force between
nucleons and in the study of Hawking radiation.

In $d$ dimensions ($d=1,2,3$), the dimensionless KGD system with external
time-independent electromagnetic potentials reads \cite{CY1,CY2,CY3}
\begin{equation}\label{KGD}
\begin{aligned}
&\eps^{2}\partial_{tt}\phi(t,\bx)-\Delta\phi(t,\bx)+\frac{1}{\eps^2}\phi(t,\bx)
=g\,\Psi^{*}(t,\bx)\sigma_{3}\Psi(t,\bx), \quad t>0, \quad \bx\in\mathbb{R}^{d}, \\
&\ii\partial_{t}\Psi(t,\bx)+\bigg(\frac{\ii}{\eps}\sum\limits_{j=1}^{d}\alpha_{j}\partial_{j}
-\frac{\omega}{\eps^2}\beta-W(\bx)\bigg)\Psi(t,\bx)=g\,\phi(t,\bx)\beta\Psi(t,\bx), \\
&\phi(0,\bx)=\phi_{0}(\bx), \quad \partial_{t}\phi(0,\bx)=\frac{1}{\eps^2}\phi_{1}(\bx),
\quad \Psi(0,\bx)=\Psi_{0}(\bx), \quad \bx\in\mathbb{R}^{d},
\end{aligned}
\end{equation}
which describes the interaction between a complex-valued Dirac spinor field
$\Psi(t,\bx)\in\mathbb{C}^{4}$ and a neutral real-valued scalar meson field
$\phi\coloneqq \phi(t,\bx)\in\mathbb{R}$ through the Yukawa coupling. Here,
$\ii=\sqrt{-1}$ is the imaginary unit, $t$ is the time variable, and
$\bx=(x_{1},\ldots,x_{d})^{\intercal}\in\mathbb{R}^{d}$ denotes the Cartesian
spatial coordinate with $\partial_{j}=\frac{\partial}{\partial x_{j}}$
($j=1,\ldots,d$); $\Delta=\nabla^{2}$ is the Laplace operator and
$\Psi^{*}=\overline{\Psi}^{\intercal}$ denotes the conjugate transpose. The
external electromagnetic potential is given by
$W\coloneqq W(\bx)=\omega\big(V(\bx)I_{4}-\sum_{j=1}^{d}A_{j}(\bx)\alpha_{j}\big)$,
where $V(\bx)$ and $A_{j}(\bx)$ ($j=1,\ldots,d$) represent the real-valued
electric and magnetic potentials, respectively. $g>0$ is the Yukawa coupling
constant, $0<\omega=\mathcal{O}(1)\leq 1$ denotes the ratio between the mass of
the electron and that of the meson, and $\alpha_{j}$ and $\beta$ are $4\times 4$
matrices given by
\begin{equation}\label{equ-9-19-1}
\alpha_{j}=\begin{pmatrix} \mathbf{0} & \sigma_{j} \\ \sigma_{j} & \mathbf{0} \end{pmatrix},
\quad j=1,2,3, \quad
\beta=\begin{pmatrix} I_{2} & \mathbf{0} \\ \mathbf{0} & -I_{2}  \end{pmatrix},
\end{equation}
where $I_{n}$ is the $n\times n$ identity matrix for $n\in\mathbb{N}$, and
$\sigma_{j}$ ($j=1,2,3$) are the Pauli matrices
\begin{equation}\label{paulimatr}
\sigma_{1}=\begin{pmatrix} 0 & 1 \\ 1 & 0 \end{pmatrix}, \quad
\sigma_{2}=\begin{pmatrix} 0 & -\ii \\ \ii & 0 \end{pmatrix}, \quad
\sigma_{3}=\begin{pmatrix} 1 & 0 \\ 0 & -1 \end{pmatrix}.
\end{equation}
Throughout this work, $0<\eps\leq 1$ denotes the dimensionless parameter inversely
proportional to the speed of light. The KGD system \eqref{KGD} is dispersive and
conserves two invariants. More precisely, the $L^{2}$-mass of the Dirac spinor
field $\Psi(t,\bx)$,
\begin{equation}\label{mass}
\norm{\Psi(t,\cdot)}_{L^{2}}^{2}\coloneqq \int_{\mathbb{R}^{d}}\abs{\Psi(t,\bx)}^{2}
\,\mathrm{d}\bx\equiv \norm{\Psi_{0}}_{L^{2}}^{2}, \quad t\geq 0,
\end{equation}
and the total energy
\begin{equation}\label{energy}
\begin{aligned}
E(t)\coloneqq &\ \frac{\eps^2}{2}\int_{\mathbb{R}^{d}}\abs{\partial_{t}\phi}^{2}\mathrm{d}\bx
+\frac{1}{2}\int_{\mathbb{R}^{d}}\abs{\nabla\phi}^{2}\mathrm{d}\bx
+\frac{1}{2\eps^2}\int_{\mathbb{R}^{d}}\abs{\phi}^{2}\mathrm{d}\bx \\
&+\int_{\mathbb{R}^{d}}\bigg[\frac{\ii}{\eps}\Psi^{*}\sum\limits_{j=1}^{d}\alpha_{j}\partial_{j}\Psi
-\frac{\omega}{\eps^2}\Psi^{*}\beta\Psi
-g\phi\,\Psi^{*}\beta\Psi\bigg]\mathrm{d}\bx \\
&-\int_{\mathbb{R}^{d}}\omega\bigg[V(\bx)\abs{\Psi}^{2}
-\sum\limits_{j=1}^{d}A_{j}(\bx)\Psi^{*}\alpha_{j}\Psi\bigg]\mathrm{d}\bx
\equiv E(0), \quad t\geq 0
\end{aligned}
\end{equation}
are invariants of the motion.

Furthermore, in lower dimensions ($d=1,2$), it is conventional to employ
two-component Dirac spinor fields \cite{BCJT,BCJY}, in which case the KGD system
\eqref{KGD} reduces to the following form ($d=1,2$):
\begin{equation}\label{KGD-two-com}
\begin{aligned}
&\eps^{2}\partial_{tt}\phi(t,\bx)-\Delta\phi(t,\bx)+\frac{1}{\eps^2}\phi(t,\bx)
=g\,\Psi^{*}(t,\bx)\sigma_{3}\Psi(t,\bx), \quad t>0, \quad \bx\in\mathbb{R}^{d}, \\
&\ii\partial_{t}\Psi(t,\bx)+\bigg(\frac{\ii}{\eps}\sum\limits_{j=1}^{d}\sigma_{j}\partial_{j}
-\frac{\omega}{\eps^2}\sigma_{3}
-\omega\Big(V(\bx)I_{2}-\sum\limits_{j=1}^{d}A_{j}(\bx)\sigma_{j}\Big)\bigg)\Psi(t,\bx)
=g\,\phi(t,\bx)\sigma_{3}\Psi(t,\bx),
\end{aligned}
\end{equation}
where $\phi\in\mathbb{R}$ and $\Psi\in\mathbb{C}^{2}$. Owing to its simplicity,
this formulation has been widely adopted in one and two dimensions. When
$0<\eps\ll 1$, the KGD system is characterized by rapid oscillations in time with
a wavelength of order $\mathcal{O}(\eps^2)$.

Another class of coupled nonlinear dispersive models that has attracted
considerable interest consists of the dimensionless KGS equations in $d$ dimensions
($d=3,2,1$) \cite{AA,FT,BZ}:
\begin{equation}\label{KGS}
\begin{aligned}
&\ii\partial_{t}\psi(t,\bx)+\Delta\psi(t,\bx)+\lambda\phi(t,\bx)\psi(t,\bx)=0,
\quad t>0, \quad \bx\in\mathbb{R}^{d}, \\
&\eps^{2}\partial_{tt}\phi(t,\bx)-\Delta\phi(t,\bx)+\frac{\mu^{2}}{\eps^{2}}\phi(t,\bx)
-\lambda\abs{\psi(t,\bx)}^{2}=0, \\
&\psi(0,\bx)=\psi_{0}(\bx), \quad \phi(0,\bx)=\phi_{0}(\bx), \quad
\partial_{t}\phi(0,\bx)=\frac{1}{\eps^2}\phi_{1}(\bx),
\end{aligned}
\end{equation}
which serves as a classical model for describing the dynamics of a complex-valued
scalar nucleon field $\psi(t,\bx)$ interacting with a neutral real-valued scalar
meson field $\phi(t,\bx)$ through the Yukawa coupling. Here $t$ is the time
variable, $\bx\in\mathbb{R}^{d}$ is the spatial coordinate, $\mu>0$ denotes the
ratio of the meson mass to the nucleon mass, and $\lambda\in\mathbb{R}$ is the
coupling constant; both $\mu$ and $\lambda$ are dimensionless and independent of
$\eps$. The KGS system \eqref{KGS} is dispersive and time-reversible, and it
conserves the mass of the nucleon field
\begin{equation}\label{KGS-mass}
\norm{\psi(t,\cdot)}^{2}_{L^{2}}\coloneqq \int_{\mathbb{R}^{d}}\abs{\psi(t,\bx)}^{2}
\,\mathrm{d}\bx\equiv \norm{\psi_{0}}_{L^{2}}^{2}, \quad t\geq 0,
\end{equation}
and the Hamiltonian, i.e., the total energy,
\begin{equation}\label{KGS-energy}
\begin{aligned}
E(t)\coloneqq\int_{\mathbb{R}^{d}}\bigg[\frac{1}{2}\Big(\eps^{2}\abs{\partial_{t}\phi}^{2}
+\abs{\nabla\phi}^{2}+\frac{\mu^{2}}{\eps^{2}}\abs{\phi}^{2}\Big)+\abs{\nabla\psi}^{2}
-\lambda\abs{\psi}^{2}\phi \bigg]\mathrm{d}\bx\equiv E(0), \quad t\geq 0.
\end{aligned}
\end{equation}

\subsection{State of the art}
\label{sec:state}

The KGD and KGS systems have been studied extensively, both theoretically and
numerically. For well-posedness and regularity results, we refer to
\cite{Bou1,Bou2,CZ,Bil,GM,WZ} and the references therein. On the numerical
side, a variety of efficient and accurate methods have been proposed and
analyzed, including time-splitting methods \cite{BY,YRS}, exponential wave
integrators \cite{LWCD,Li}, and finite difference methods \cite{BS,CY2,CY4,WZJ}.

In parallel, various structure-preserving methods have been developed for the
KGD and KGS systems to address their conservation properties. For instance, a
fourth-order compact finite difference scheme for the KGD system was
constructed in \cite{LGW}, with a proof of discrete conservation of both the
total mass and the total energy. By combining a symplectic Runge--Kutta method
with the quadratic auxiliary vector (QAV) approach, mass- and
energy-preserving schemes of arbitrarily high order were constructed in
\cite{FGWC} for the relativistic KGS system. A high-order energy-preserving
algorithm for the relativistic KGS system was proposed in \cite{LSZ} by
introducing Lagrange multipliers, and its feasibility and effectiveness in the
nonrelativistic regime were also verified; see also \cite{LJ,ZS,MLJ}.

In the nonrelativistic limit $0<\varepsilon\ll1$, the KGD and KGS systems
exhibit rapid temporal oscillations with a wavelength of order
$\mathcal{O}(\varepsilon^{2})$. For classical numerical methods, an accurate
resolution of the solution phase then entails a step-size restriction
governed by $\varepsilon$, which leads to a considerable computational burden.
Moreover, the Yukawa coupling term introduces additional difficulties in the
design and analysis of high-order uniformly accurate (UA) schemes.

Substantial progress has been made in the development of UA schemes for highly
oscillatory evolution equations \cite{CCLM,CCLMZ,BCZ,CY3,WZ,WWL}. A principal
advantage of UA methods is that their time-step and mesh strategies are
independent of $\varepsilon$. For nonlinear (coupled) dispersive equations,
three main approaches have been pursued: two-scale formulation (TSF) methods
\cite{CCLM,WWL}, multiscale time integrator Fourier pseudospectral (MTI-FP)
methods \cite{BZ,CY1}, and nested Picard iterative (NPI) integrators
\cite{BCZ,CY3}; see also \cite{BFS,FS}. For the KGD system in the
nonrelativistic regime, a uniformly first-order accurate MTI method was
introduced in \cite{CY1}, and an effective decoupling of the Yukawa
interaction subsequently enabled the construction of a UA NPI integrator of
order up to two in \cite{CY3}, with extensions to arbitrary order possible in
principle. For the KGS system, an MTI-FP scheme that is uniformly first-order
accurate for $\varepsilon\in(0,1]$ was developed in \cite{BZ}: the method
attains the optimal second-order rate $\mathcal{O}(h^{2})$ for
$0<h\lesssim\varepsilon^{2}$, with $h$ the time step, while its error is
$\mathcal{O}(\varepsilon^{2})$, independent of $h$, for
$0<\varepsilon\lesssim h^{1/2}$. Asymptotically consistent exponential-type
integrators of order up to two were proposed in \cite{BKS,Cal}, and UA NPI
integrators of order up to three were developed in \cite{CZ}, with
arbitrary-order generalizations feasible in principle. Based on the approach
of \cite{BZ}, a UA MTI-FP method with simplified transmission conditions was
further proposed in \cite{FL}, achieving second-order accuracy under
$h\lesssim\varepsilon^{2}$.

\subsection{Contributions and outline}
\label{sec:contrib}

The UA approaches reviewed above are each tailored to a particular class of
problems or accuracy requirements. The MTI-FP approach relies on
problem-dependent multiscale asymptotic expansions together with frequency
analysis, so that its higher-order extensions typically require
model-specific analysis. NPI methods admit a recursive extension to arbitrary
order in principle, while the associated algebraic constructions become
increasingly involved as the order grows. The TSF method was originally
introduced for highly oscillatory Schr\"odinger and Klein--Gordon equations
\cite{CCLM} and was subsequently extended to the Vlasov--Poisson equation
under strong magnetic fields \cite{CCLMZ}; more recently, within the two-scale
framework, a class of fourth-order UA integrators with approximate
structure-preserving properties was devised in \cite{WWL} for the Dirac
equation in the nonrelativistic regime. Since the two-scale framework
separates the fast and slow dynamics at the vector-field level, it provides a
natural setting for constructing high-order integrators with approximate
structure preservation. To date, however, it has not been applied to
nonlinear coupled dispersive models with Yukawa coupling, and no existing
scheme for the KGD and KGS systems simultaneously offers high-order uniform
accuracy and approximate conservation of the mass and the energy. Motivated by these observations, we aim to fill this gap in the present work.
We construct uniformly accurate fully discrete schemes with the following
features.
\begin{itemize}
    \item The two-scale approach is employed for the first time for highly
    oscillatory nonlinear coupled dispersive models, providing a general and
    systematic framework for constructing UA schemes of arbitrarily high
    order.
    \item Two fourth-order uniformly accurate fully discrete schemes are
    derived, one explicit and one symmetric: the explicit scheme is
    computationally efficient, while the symmetric scheme exhibits excellent
    long-time near-conservation of the mass and the energy for the coupled
    systems.
    \item The proposed schemes are validated numerically for the coupled
    systems in up to three spatial dimensions ($d=3$), which is reported here
    for the first time.
\end{itemize}

The remainder of this paper is organized as follows. In
Section~\ref{sec-2}, we present the two-scale framework for the nonlinear
coupled dispersive models. In Section~\ref{sec-3}, we introduce a class of
fully discrete schemes, in which the Fourier pseudospectral method is used
for the spatial discretization and explicit and symmetric exponential
Runge--Kutta integrators are adopted for the time integration within the
two-scale framework. In Section~\ref{sec-4}, we verify the effectiveness of
the fully discrete schemes for the two coupled dispersive systems in one,
two, and three dimensions.

\section{Two-scale framework}\label{sec-2}
In this section, the KGD and KGS systems are first reformulated by means of the
operator splitting technique, which facilitates the construction of numerical
schemes. The two-scale formulation and the corresponding initial conditions are
then introduced. For simplicity, we restrict ourselves to the one-dimensional
case $d=1$; the extension to higher dimensions is straightforward, and all the
results remain valid.

\subsection{Equivalent reformulation}\label{sec-2.1}

Throughout, for an integer $m>0$ and $\Omega=[a,b]$, let $H^{m}(\Omega)$ denote
the standard Sobolev space. In what follows, we confine ourselves to the case
where $\omega>0$ and $\mu>0$ are rational numbers.

We introduce two operators that are uniformly bounded with respect to $\eps$,
namely $\Lambda_{\mu}: H^{m}(\Omega)\to H^{m-1}(\Omega)$ with
$\Lambda_{\mu}=\sqrt{\mu^{2}-\eps^{2}\Delta}$, and
$\mathcal{S}^{\eps}: (H^{m}(\Omega))^{2}\to (H^{m-1}(\Omega))^{2}$ with
$\mathcal{S}^{\eps}=-\ii \eps\sigma_{1}\partial_{x}+\omega\sigma_{3}$.
In the phase space (Fourier domain), $\mathcal{S}^{\eps}$ is diagonalizable and
admits the decomposition (\cite{BCJT,BMP})
\begin{equation}\label{equ-9-3-1}
\mathcal{S}^{\eps}=\sqrt{\omega^{2}I_{2}-\eps^{2}\Delta}\,\Pi_{+}^{\eps}
-\sqrt{\omega^{2}I_{2}-\eps^{2}\Delta}\,\Pi_{-}^{\eps},
\end{equation}
where $\Pi_{\pm}^{\eps}: (H^{m}(\Omega))^{2}\to (H^{m}(\Omega))^{2}$ are
projector operators given by
\begin{equation}\label{equ-9-3-2}
\Pi_{+}^{\eps}=\frac{1}{2}\Big[I_{2}+(\omega^{2}I_{2}-\eps^2 \Delta)^{-1/2}\mathcal{S}^{\eps}\Big], \quad
\Pi_{-}^{\eps}=\frac{1}{2}\Big[I_{2}-(\omega^{2}I_{2}-\eps^2 \Delta)^{-1/2}\mathcal{S}^{\eps}\Big],
\end{equation}
which are uniformly bounded with respect to $\eps$. It is readily verified that
$\Pi_{+}^{\eps}+\Pi_{-}^{\eps}=I_{2}$,
$\Pi_{+}^{\eps}\Pi_{-}^{\eps}=\Pi_{-}^{\eps}\Pi_{+}^{\eps}=\mathbf{0}$, and
$(\Pi_{\pm}^{\eps})^{2}=\Pi_{\pm}^{\eps}$.
Define the operators
\begin{align}
&\mD_{\mu}\coloneqq \frac{1}{\eps^{2}}\big(\sqrt{\mu^{2}-\eps^{2}\Delta}-\mu\big)
=-(\sqrt{\mu^{2}-\eps^{2}\Delta}+\mu)^{-1}\Delta, \label{equ-9-3-4a} \\
&\mD_{\omega}\coloneqq \frac{1}{\eps^{2}}\big(\omega^{2}I_{2}-\eps^{2}\Delta\big)^{1/2}-\frac{\omega}{\eps^{2}}I_{2}
=-(\sqrt{\omega^{2}I_{2}-\eps^{2}\Delta}+\omega I_{2})^{-1}\Delta, \label{equ-9-3-4b}
\end{align}
which are uniformly bounded (with respect to $\eps$) from $H^{m}(\Omega)$ to
$H^{m-2}(\Omega)$ for $m\geq 2$. Moreover, they satisfy
$0\leq \mD_{\mu}\leq -\frac{1}{2\mu}\Delta$ and
$0\leq \mD_{\omega}\leq -\frac{1}{2\omega}\Delta I_{2}$ for all
$\eps\in(0,1]$.

Consequently, the evolution operators $\fe^{\pm\ii\Lambda_{\mu} t/\eps^{2}}$ and
$\fe^{\pm\ii\mathcal{S}^{\eps}t/\eps^{2}}$ can be decomposed as
\begin{equation}\label{operator-decom}
\fe^{\pm\frac{\ii\Lambda_{\mu} t}{\eps^2}}=\fe^{\pm\frac{\ii \mu t}{\eps^2}}\fe^{\pm\ii \mD_{\mu}t}, \qquad
\fe^{\pm\frac{\ii\mathcal{S}^{\eps}t}{\eps^2}}=\fe^{\pm\frac{\ii\omega t}{\eps^2}}\fe^{\pm\ii\mD_{\omega}t}\Pi_{+}^{\eps}
+\fe^{\mp\frac{\ii\omega t}{\eps^2}}\fe^{\mp\ii\mD_{\omega}t}\Pi_{-}^{\eps}.
\end{equation}

\subsubsection{KGD equation}\label{sec-2.1.1}
We first introduce the transformation
\begin{equation}\label{equ-9-3-3}
u(t,x)=\phi(t,x)-\ii\eps^{2}\Lambda_{1}^{-1}\partial_{t}\phi(t,x), \quad t>0, \quad x\in\Omega,
\end{equation}
where $u(t,x)\in\mathbb{C}$ satisfies periodic boundary conditions. It follows
immediately that $\phi(t,x)=\tfrac{1}{2}(u(t,x)+\overline{u}(t,x))
=\operatorname{Re}(u(t,x))$, with $\operatorname{Re}(\cdot)$ denoting the real
part, and $\partial_{t}\phi(t,x)=\frac{\ii\Lambda_{1}}{2\eps^2}
(u(t,x)-\overline{u}(t,x))$. Then the KGD system \eqref{KGD-two-com} can be
reformulated as the following first-order system in time:
\begin{equation}\label{first-order}
\begin{aligned}
&\partial_{t}u(t,x)-\frac{\ii}{\eps^{2}}\Lambda_{1} u(t,x)
=-\ii g\Lambda_{1}^{-1}(\Psi^{*}(t,x)\sigma_{3}\Psi(t,x)),
\quad t>0, \quad x\in\Omega, \\
&\partial_{t}\Psi(t,x)+\frac{\ii}{\eps^{2}}\mathcal{S}^{\eps}\Psi(t,x)
=-\ii\big(W(x)\Psi(t,x)+g\operatorname{Re}(u(t,x))\sigma_{3}\Psi(t,x)\big),
\quad t>0, \quad x\in\Omega, \\
&u(0,x)=\phi_{0}(x)-\ii\Lambda_{1}^{-1}\phi_{1}(x), \quad \Psi(0,x)=\Psi_{0}(x).
\end{aligned}
\end{equation}
To filter out the dominant oscillations in the above system, we introduce a
family of filtering transformations
\begin{equation}\label{filter}
\zeta(t,x)\coloneqq \fe^{-\frac{\ii\Lambda_{1} t}{\eps^{2}}}u(t,x), \quad
\Xi(t,x)=(\xi_{1}(t,x),\xi_{2}(t,x))\coloneqq \fe^{\frac{\ii\mathcal{S}^{\eps}t}{\eps^{2}}}\Psi(t,x),
\quad x\in\Omega.
\end{equation}
Then the system \eqref{first-order} is transformed into
\begin{equation}\label{fil-system}
\begin{aligned}
&\partial_{t}\zeta(t,x)=\fe^{-\frac{\ii\Lambda_{1} t}{\eps^{2}}}\left[-\ii g\Lambda_{1}^{-1}
\left(\fe^{-\frac{\ii\mathcal{S}^{\eps}t}{\eps^{2}}}\Xi(t,x)\right)^{*}\sigma_{3}
\left(\fe^{-\frac{\ii\mathcal{S}^{\eps}t}{\eps^{2}}}\Xi(t,x)\right)\right], \\
&\partial_{t}\Xi(t,x)=\fe^{\frac{\ii\mathcal{S}^{\eps}t}{\eps^{2}}}\left[-\ii
\left(W(x)+g\operatorname{Re}\!\left(\fe^{\frac{\ii\Lambda_{1} t}{\eps^{2}}}\zeta(t,x)\right)\sigma_{3}\right)
\left(\fe^{-\frac{\ii\mathcal{S}^{\eps}t}{\eps^{2}}}\Xi(t,x)\right)\right], \\
&\zeta(0,x)=\phi_{0}(x)-\ii\Lambda_{1}^{-1}\phi_{1}(x), \quad \Xi(0,x)=\Psi_{0}(x).
\end{aligned}
\end{equation}
In view of the decompositions \eqref{operator-decom}, the system \eqref{fil-system}
is equivalent to
\begin{equation}\label{fil-system-2}
\begin{aligned}
&\partial_{t}\zeta(t,x)=\Big(\fe^{-\frac{\ii t}{\eps^2}}\fe^{-\ii \mD_{1}t}\Big)\,
\varrho\bigg(\Big(\fe^{-\frac{\ii\omega t}{\eps^2}}\fe^{-\ii\mD_{\omega}t}\Pi_{+}^{\eps}
+\fe^{\frac{\ii\omega t}{\eps^2}}\fe^{\ii\mD_{\omega}t}\Pi_{-}^{\eps}\Big)\Xi(t,x)\bigg), \\
&\partial_{t}\Xi(t,x)=\Big(\fe^{\frac{\ii\omega t}{\eps^2}}\fe^{\ii\mD_{\omega}t}\Pi_{+}^{\eps}
+\fe^{-\frac{\ii\omega t}{\eps^2}}\fe^{-\ii\mD_{\omega}t}\Pi_{-}^{\eps}\Big)\\
&\qquad\qquad\times\mathcal{G}\bigg(\Big(\fe^{\frac{\ii t}{\eps^2}}\fe^{\ii \mD_{1}t}\Big)\zeta(t,x),\,
\Big(\fe^{-\frac{\ii\omega t}{\eps^2}}\fe^{-\ii\mD_{\omega}t}\Pi_{+}^{\eps}
+\fe^{\frac{\ii\omega t}{\eps^2}}\fe^{\ii\mD_{\omega}t}\Pi_{-}^{\eps}\Big)\Xi(t,x)\bigg), \\
&\zeta(0,x)=\phi_{0}(x)-\ii\Lambda_{1}^{-1}\phi_{1}(x), \quad \Xi(0,x)=\Psi_{0}(x),
\end{aligned}
\end{equation}
where $\varrho(\Psi)=-\ii g\Lambda_{1}^{-1}(\Psi^{*}\sigma_{3}\Psi)$ and
$\mathcal{G}(u,\Psi)=-\ii\big(W(x)+g\operatorname{Re}(u)\sigma_{3}\big)\Psi$.

\subsubsection{KGS equation}\label{sec-2.1.2}
The analogous transformation
$\mathfrak{u}=\phi-\ii\eps^{2}\Lambda_{\mu}^{-1}\partial_{t}\phi$ allows the KGS
system \eqref{KGS} in one dimension to be rewritten as
\begin{equation}\label{equ-9-20-1}
\begin{aligned}
&\partial_{t}\psi(t,x)=\ii\Delta\psi(t,x)+\ii\lambda\operatorname{Re}(\mathfrak{u}(t,x))\,\psi(t,x),
\quad \psi(0,x)=\psi_{0}(x), \\
&\partial_{t}\mathfrak{u}(t,x)=\frac{\ii\Lambda_{\mu}}{\eps^2}\mathfrak{u}(t,x)
-\ii\lambda\Lambda_{\mu}^{-1}\abs{\psi(t,x)}^{2}, \quad
\mathfrak{u}(0,x)=\phi_{0}(x)-\ii\Lambda_{\mu}^{-1}\phi_{1}(x).
\end{aligned}
\end{equation}
Introducing the filtered variables
\begin{equation}\label{equ-9-20-2}
\mathfrak{q}(t,x)=\fe^{-\ii\Delta t}\psi(t,x), \quad \mathfrak{p}(t,x)=\fe^{-\frac{\ii\Lambda_{\mu}t}{\eps^2}}\mathfrak{u}(t,x), \quad x\in\Omega,
\end{equation}
and invoking the operator \eqref{equ-9-3-4a}, we obtain
\begin{equation}\label{equ-9-20-3}
\begin{aligned}
&\partial_{t}\mathfrak{q}(t,x)=\ii\lambda\,\fe^{-\ii t\Delta}\left[
\operatorname{Re}\!\left(\fe^{\ii t\mD_{\mu}}\fe^{\frac{\ii\mu t}{\eps^{2}}}\mathfrak{p}(t,x)\right)
\left(\fe^{\ii t\Delta}\mathfrak{q}(t,x)\right)\right], \quad
\mathfrak{q}(0,x)=\psi_{0}(x), \\
&\partial_{t}\mathfrak{p}(t,x)=-\ii\lambda\,\Big(\fe^{-\ii\mD_{\mu}t}\fe^{-\frac{\ii\mu t}{\eps^2}}\Big)
\left(\Lambda_{\mu}^{-1}\abs{\fe^{\ii t\Delta}\mathfrak{q}(t,x)}^{2}\right), \quad
\mathfrak{p}(0,x)=\phi_{0}(x)-\ii\Lambda_{\mu}^{-1}\phi_{1}(x).
\end{aligned}
\end{equation}

Consequently, both the coupled KGD system \eqref{fil-system-2} and the KGS system
\eqref{equ-9-20-3} admit the unified representation
\begin{equation}\label{ori-coup}
\partial_{t}\mathfrak{b}(t,x)=\mF\big(t,t/\eps^{2},\mathfrak{b}(t,x)\big), \quad
\mathfrak{b}(0,x)=\mathfrak{b}_{0}(x), \quad t\in[0,T], \quad \eps\in(0,1],
\end{equation}
where $\mathfrak{b}=[\mathfrak{b}_{1},\ldots,\mathfrak{b}_{\tilde{k}}]^{\intercal}$
with $\tilde{k}\in\mathbb{N}_{+}$ (e.g., $\tilde{k}=3$ for the KGD system and
$\tilde{k}=2$ for the KGS system), the solution $t\mapsto \mathfrak{b}(t)$ is a
smooth map into the Sobolev space $H^{s}(\Omega)$ for $s>d/2+8$, and the vector
field $(t,\tau,\mathfrak{b})\mapsto \mF(t,\tau,\mathfrak{b})\in H^{s}$ is smooth
and $P$-periodic in $\tau\in\Theta$ with $\Theta\coloneqq\mathbb{R}/P\mathbb{Z}$
denoting the corresponding one-dimensional torus. Here the period $P$ is well
defined owing to the rationality assumption on $\omega$ and $\mu$ made at the
beginning of this subsection.

\subsection{Two-scale formulation and initial data}\label{sec-2.2}
We observe that system \eqref{ori-coup} involves two distinct time scales,
namely the slow scale $t$ and the fast scale $t/\eps^2$. To distinguish them
explicitly, we introduce an augmented two-scale formulation of the original
system, in which $U(t,\tau,x)$ denotes the augmented solution and is
$P$-periodic with respect to $\tau$. The resulting two-scale formulation reads
\begin{equation}\label{two-scale}
\partial_{t}U(t,\tau,x)+\frac{1}{\eps^2}\partial_{\tau}U(t,\tau,x)
=\mF\big(t,\tau,U(t,\tau,x)\big), \quad \tau\in\mathbb{T}_{\tau},
\end{equation}
where $\mathbb{T}_{\tau}\coloneqq\mathbb{R}/P\mathbb{Z}$. We recall that the
period $P$ is well defined owing to the rationality assumption on $\omega$ and
$\mu$ made in Subsection \ref{sec-2.1}. Since the operators
$\mD_{\mu}: H^{m}(\Omega)\to H^{m-2}(\Omega)$,
$\mD_{\omega}: (H^{m}(\Omega))^{2}\to (H^{m-2}(\Omega))^{2}$, and
$\Pi_{\pm}^{\eps}: (H^{m}(\Omega))^{2}\to (H^{m}(\Omega))^{2}$ are uniformly
bounded with respect to $\eps$, each differentiation of the vector field
$\mF(t,\tau,U)$ with respect to $t$ costs two derivatives in the space variable.
Accordingly, we impose the following assumption on $\mF$.

\vspace{6pt}
\noindent \textbf{Assumption (A)} The functional
\[
\partial_{t}^{\alpha}\partial_{\tau}^{\beta}\partial_{u}^{\gamma}\mF(t,\tau,U): \mathbb{R}_{+}\times \mathbb{T}_{\tau}\times H^{s}\mapsto \mathcal{L}(\underbrace{H^{\sigma}\times\cdots \times H^{\sigma}}_{\gamma\times\beta \ \text{times}},H^{\sigma-2\alpha})
\]
is continuous and locally bounded, for all $\alpha\in\{0\cdots 4\}$, $\beta\in \{0,1\}$, and $\gamma\in \{0\cdots 4\}$, for all $s,\sigma$ such that $s\geq \sigma>2\alpha+d/2$.
\vspace{6pt}

It suffices for the initial data to satisfy $U(0,0,x)=\mathfrak{b}_{0}(x)$, whereas $U(0,\tau,x)$ may be freely designed as follows.
For a periodic function $\kappa(\cdot)$ on $\bT$, define the averaging, differentiation and its inverse operators by
\[
\Pi \kappa:=\frac{1}{2\pi}\int_0^{2\pi}\kappa(\tau)d\tau, \quad
 L:=\partial_\tau, \quad
 L^{-1}\kappa:=(I-\Pi)\int_0^\tau \kappa(s)ds,\quad
  A:=L ^{-1} (I-\Pi).
\]
Following the Chapman–Enskog expansion \cite{CCLM,Zhao}, the solution $U(t,\tau,x)$ of \eqref{two-scale} can be decomposed as
 \begin{equation}\label{equ-9-4-1}
U(t,\tau,x)=\underline{U}(t,x)+h(t,\tau,x),
\end{equation}
where $\underline{U}(t,x):=\Pi U(t,\tau,x)$ is the averaged part and $h(t,\tau,x)$ is the correction part satisfying
$\Pi h(t,\tau,x)=0.$ Applying the operator $\Pi $ to \eqref{two-scale} gives
\begin{align}\label{equ-9-4-2}
\partial_{t} \underline{U}(t, x)=\Pi \mF(t,\tau,\underline{U}(t,x)+h(t,\tau,x)).
\end{align}
Subtracting \eqref{equ-9-4-2} from \eqref{two-scale} and applying the inverse operator $L$ yields
\begin{align}\label{equ-9-4-4}
h(t,\tau,x)=\eps^{2}\big( A \mF(t,\tau,\underline{U}(t,x)+h(t,\tau,x))- L^{-1}\left(\partial_{t} h(t,\tau,x)\right)\big).
\end{align}
Through asymptotic expansion in $\underline{U}(t,x)$, we obtain
 \begin{equation}\label{equ-9-4-5}
h(t,\tau,x)=\eps^{2} h_{1}(t,\tau,\underline{U}(t,x))+\eps^{4} h_{2}(t,\tau,\underline{U}(t,x))+\eps^{6} h_{3} (t,\tau,\underline{U}(t,x))+\mathcal{O}(\eps^{8}),
\end{equation}
with
\begin{equation}\label{equ-9-4-9}
\begin{aligned}
h_{1}(t,\tau,\underline{U}):=&A\mF(t,\tau,\underline{U}), \\
h_{2}(t,\tau,\underline{U}):=&A\partial_{u}\mF(t,\tau,\underline{U}) A\mF(t,\tau,\underline{U})-A^{2}\partial_{t}\mF(t,\tau,\underline{U})-A^{2}\partial_{u}\mF(t,\tau,\underline{U})
\Pi \mF(t,\tau,\underline{U}), \\
h_{3}(t,\tau,\underline{U}):=&A\partial_{u}\mF(t,\tau,\underline{U}) h_{2}+\frac{1}{2}A\partial_{u}^{2}\mF(t,\tau,\underline{U})(h_{1},h_{1})-A^{2}\partial_{u}\partial_{t}\mF(t,\tau,\underline{U})
A\mF(t,\tau,\underline{U}) \\
&-A^{2}\partial_{u}^{2}\mF(t,\tau,\underline{U})(\Pi \mF(t,\tau,\underline{U}),A\mF(t,\tau,\underline{U}))-A^{2}\partial_{u}\mF(t,\tau,\underline{U}) A\partial_{t}\mF(t,\tau,\underline{U}) \\
&-A^{2}\partial_{u}\mF(t,\tau,\underline{U}) A\partial_{u}\mF(t,\tau,\underline{U})\Pi \mF(t,\tau,\underline{U}) \\
&-A^{2}\partial_{u}\mF(t,\tau,\underline{U})(\Pi\partial_{u}\mF(t,\tau,\underline{U}),A \mF(t,\tau,\underline{U}))+A^{3}\partial_{t}^{2}\mF(t,\tau,\underline{U}) \\
&+2A^{3}\partial_{t}\partial_{u}\mF(t,\tau,\underline{U})\Pi \mF(t,\tau,\underline{U})+A^{3}\partial_{u}^{2}\mF(t,\tau,\underline{U})(\Pi \mF(t,\tau,\underline{U}), \Pi \mF(t,\tau,\underline{U})) \\
&+A^{3}\partial_{u}\mF(t,\tau,\underline{U})\Pi\partial_{t}\mF(t,\tau,\underline{U})+A^{3}\partial_{u}\mF(t,\tau,\underline{U})
\Pi\partial_{u}\mF(t,\tau,\underline{U})\Pi \mF(t,\tau,\underline{U}).
\end{aligned}
\end{equation}
Since $\Pi$  and $ A$ are bounded on $C^0(\mathbb{T}_{\tau};H^{\sigma})$ for any $\sigma$, the functions $h_{1}, h_{2}, h_{3}$ are of order $\mathcal{O}(1)$ in the appropriate Sobolev norms.

The complete initial value $U(0,\tau,x)$ is now to be determined. Evaluating the Chapman–Enskog expansions \eqref{equ-9-4-1} and \eqref{equ-9-4-5} at the initial time gives
\[
U(0,\tau,x)=\underline{U}(0)+\eps^{2} h_{1}(0,\tau,\underline{U}(0))+\eps^{4} h_{2}(0,\tau,\underline{U}(0))+\eps^{6} h_{3}(0,\tau,\underline{U}(0))+\mathcal{O}(\eps^{8}),
\]
In view of the constraint
$\mathfrak{b}_{0}=U(0,0,x),$
the proper $\underline{U}(0)$ can be determined through the iterative construction:
\begin{equation}\label{equ-9-4-10}
\underline{U}^{[0]}:=\mathfrak{b}_{0},\ \underline{U}^{[r]}=\mathfrak{b}_{0}-\sum_{l=1}^r\eps^{2(r-l+1)} h_{l}(0,0,\underline{U}^{[r-l]}),\ r=1,2,3,\  \ \mbox{where}\
\underline{U}^{[r]}-\underline{U}(0)=\mathcal{O}(\eps^{2(r+1)}).
\end{equation}
Consequently, the complete initial data for the two-scale equation \eqref{two-scale} is given by
\begin{equation}\label{init-data}
U(0,\tau,x)=\underline{U}^{[3]}+\eps^{2} h_{1}(0,\tau,\underline{U}^{[3]})+\eps^{4} h_{2}(0,\tau,\underline{U}^{[3]})+\eps^{6} h_{3}(0,\tau,\underline{U}^{[3]})=:U_0(\tau,x),\quad
\tau\in\bT.
\end{equation}

\section{Fully discrete scheme}\label{sec-3}
In this section, a UA fully discrete scheme is developed for the resulting
two-scale equation \eqref{two-scale}, and the implementation of the proposed
scheme is described in detail.

\subsection{Spatial semidiscretization}\label{sec-3.1}
The Fourier spectral method \cite{STW} is adopted to discretize the system
\eqref{ori-coup} in space. Denote by $\triangle x=(b-a)/N_{x}$ the uniform
spatial mesh size, where $N_{x}$ is an even positive integer. The corresponding
grid points are $x_{j}\coloneqq a+j\triangle x$ for $j=0,1,\ldots,N_{x}-1$. The
first- and second-order Fourier differentiation matrices are then written as
\begin{equation}\label{equ-2-19-2}
\begin{aligned}
&{ B^{(1)}:=b_{kj}^{(1)}}=\left\{
\begin{array}{ll}
\frac{(-1)^{k+j}}{2}\cot\big(\frac{(k-j)\pi}{N_{x}}\big),  & k\neq j, \\
0,  & k=j,
\end{array}
\right. \ \
{ B^{(2)}:=b_{kj}^{(2)}}=\left\{
\begin{array}{ll}
\frac{(-1)^{k+j}}{-2}\sin^{-2}\big(\frac{(k-j)\pi}{N_{x}}\big),   & k\neq j, \\
\frac{N_{x}^{2}}{-12}-\frac{1}{6},   & k=j,
\end{array}
\right.
\end{aligned}
\end{equation}
with $k, j=0,\ldots, N_{x}-1$. In the Fourier domain, $\Lambda_{\mu}$ takes the
form $\Lambda_{\mu, f}=(\mu^{2}I-\eps^{2}B^{(2)})^{1/2}$, and the operators
$\mD_{\mu}$, $\mD_{\omega}$, and $\Pi_{\pm}$ accordingly admit the discrete
representations
\begin{align*}
&\mD_{\mu,f}=\frac{1}{\eps^2}((\mu^{2}I-\eps^{2}B^{(2)})^{1/2}-\mu I), \
\mD_{\omega,f}=\frac{1}{\eps^{2}}\operatorname{diag}
\Big((\omega^{2}I-\eps^{2}B^{(2)})^{1/2}-\omega I,(\omega^{2} I-\eps^{2}B^{(2)})^{1/2}-\omega I\Big),\\
&\Pi_{\pm,f}^{\eps}=\frac{1}{2}\begin{pmatrix} I\pm \omega(\omega^{2}I-\eps^{2}B^{(2)})^{-1/2} & \mp \ii\eps B^{(1)}(\omega^{2}I-\eps^{2}B^{(2)})^{-1/2} \\
\mp \ii\eps B^{(1)}(\omega^{2}I-\eps^{2}B^{(2)})^{-1/2}  & I\mp \omega(\omega^{2}I-\eps^{2}B^{(2)})^{-1/2} \end{pmatrix},
\end{align*}
where $I=I_{N_{x}\times N_{x}}$ is the $N_{x}\times N_{x}$ identity matrix.
Both $\mathcal{D}_{\omega,f}$ and $\Pi_{\pm,f}^{\eps}$ are $2\times 2$ block
matrices consisting of blocks of size $N_{x}\times N_{x}$, and
$(I-\eps^{2}B^{(2)})^{1/2}$ denotes the element-wise square root.

The coupled system \eqref{ori-coup} is discretized into the following system of
ODEs:
\begin{equation}\label{equ-2-20-1}
\frac{\mathrm{d}}{\mathrm{d}t}\mW(t)=\mF(t,t/\eps^2,\mW(t)),
\end{equation}
where
$\mW(t)\coloneqq [\mW_{1},\ldots,\mW_{\tilde{k}}]^{\intercal}
=[\mathfrak{b}_{1}(x_{0},t),\ldots,\mathfrak{b}_{1}(x_{N_{x}-1},t),\ldots,
\mathfrak{b}_{\tilde{k}}(x_{0},t),\ldots,\mathfrak{b}_{\tilde{k}}(x_{N_{x}-1},t)]^{\intercal}$.
The discrete forms of the two coupled systems are detailed as follows:
\begin{itemize}
  \item \textbf{KGD equation.} We have $\tilde{k}=3$,
  $\mW_{1}(t)=[\zeta(x_{0},t),\ldots,\zeta(x_{N_{x}-1},t)]^{\intercal}$,
  $\vec{\mW}(t) \allowbreak\coloneqq \allowbreak [\mW_{2}(t);\mW_{3}(t)]
  \allowbreak=[\xi_{1}(x_{0},t),\ldots,\xi_{1}(x_{N_{x}-1},t),\xi_{2}(x_{0},t),\ldots,\xi_{2}(x_{N_{x}-1},t)]^{\intercal}$,
      and
  \begin{equation*}
    \mF(t,t/\eps^2,\mW)=\begin{bmatrix}
   (\fe^{-\frac{\ii t}{\eps^2}}\fe^{-\ii t\mD_{1,f}})\varrho_{f}\left((\fe^{-\frac{\ii\omega t}{\eps^2}}\fe^{-\ii t\mD_{\omega,f}}\Pi_{+,f}^{\eps}+\fe^{\frac{\ii\omega t}{\eps^2}}\fe^{\ii t\mD_{\omega,f}}\Pi_{-,f}^{\eps})\vec{\mW}(t)\right) \\
   \begin{aligned}
   &(\fe^{\frac{\ii\omega t}{\eps^2}}\fe^{\ii t \mD_{\omega,f}}\Pi_{+,f}^{\eps}+\fe^{-\frac{\ii\omega t}{\eps^2}}\fe^{-\ii t \mD_{\omega,f}}\Pi_{-,f}^{\eps})\mathcal{G}_{f}\big((\fe^{\frac{\ii t}{\eps^2}}\fe^{\ii t\mD_{1,f}})\mW_{1}(t) \\ &(\fe^{-\frac{\ii\omega t}{\eps^2}}\fe^{-\ii t\mD_{\omega,f}}\Pi_{+,f}^{\eps}+\fe^{\frac{\ii\omega t}{\eps^2}}\fe^{\ii t\mD_{\omega,f}}\Pi_{-,f}^{\eps})\vec{\mW}(t)\big)
   \end{aligned}
   \end{bmatrix},
  \end{equation*}
  where $\varrho_{f}(\Psi_{f})=-\ii g\Lambda_{f}^{-1}\Psi_{f}(t)\sigma_{3}\Psi_{f}(t)$,
  $\mathcal{G}_{f}(u_{f},\Psi_{f})=-\ii(\chi_{f}+g\operatorname{Re}(u_{f}(t)))\Psi_{f}(t)$, and
  $\chi_{f}\coloneqq\left(\begin{array}{cc} V_{f} & -A_{1,f} \\ -A_{1,f} & V_{f} \end{array}\right)$;
  here $V_f$ and $A_{1,f}$ correspond to the discrete potential values on the
  spatial grid, namely, ${V(x_j)}_{j=0,1,\ldots,N_{x}-1}$ and
  ${A_{1}(x_j)}_{j=0,1,\ldots,N_{x}-1}$.
  \item \textbf{KGS equation.} We have $\tilde{k}=2$ with
  $\mW_{1}(t)=[\mathfrak{q}(x_{0},t),\ldots,\mathfrak{q}(x_{N_{x}-1},t)]^{\intercal}$
  and
  $\mW_{2}(t)=[\mathfrak{p}(x_{0},t),\allowbreak \ldots,\mathfrak{p}(x_{N_{x}-1},t)]^{\intercal}$,
  and
  \begin{equation*}
  \mF(t,t/\eps^2,\mW)=\begin{bmatrix}
                        \ii\lambda\fe^{-\ii tB^{(2)}}\left(\frac{1}{2}\operatorname{Re}(\fe^{\ii t\mD_{\mu,f}}\fe^{\frac{\ii\mu t}{\eps^{2}}}\mW_{2}(t))(\fe^{\ii tB^{(2)}}\mW_{1}(t))\right) \\
                        -\ii\lambda (\fe^{-\ii\mD_{\mu,f}t}\fe^{-\frac{\ii\mu t}{\eps^2}})\left(\Lambda_{\mu,f}^{-1}\abs{\fe^{\ii tB^{(2)}}\mW_{1}(t)}^{2}\right)
                      \end{bmatrix}.
  \end{equation*}
\end{itemize}

Within the framework of Subsection \ref{sec-2.2}, we introduce $Z(t,\tau)$,
whereby system \eqref{equ-2-20-1} can be recast in the two-scale form
\begin{align}\label{equ-2-21-1}
\partial_{t}Z(t,\tau)+\frac{1}{\eps^{2}}\partial_{\tau}Z(t,\tau)
=\Upsilon(t,\tau,Z(t,\tau)).
\end{align}
For $\tau=t/\eps^{2}$, it holds that $Z(t,\tau)=\mathcal{W}(t)$. Let $N_{\tau}$
be an even positive integer, and introduce the uniform grid
$\tau_{k}=\frac{2\pi}{N_{\tau}}k$ with $\triangle\tau=\frac{2\pi}{N_{\tau}}$,
$k=0,1,\ldots, N_{\tau}-1$. For any periodic function $v(\tau)$ on $[0,2\pi)$,
denote by $P_{\mathcal{M}} : L^{2}([0,2\pi))\rightarrow Y_{\mathcal{M}}$ the
standard projection operator and by $I_{\mathcal{M}} :
C([0,2\pi))\rightarrow Y_{\mathcal{M}}$ the trigonometric interpolation
operator, defined as
\[
(P_{\mathcal{M}}v)(\tau)=\sum_{l\in\mathcal{M}}\widehat{v_{l}}\fe^{\ii l\tau},
\qquad
(I_{\mathcal{M}}v)(\tau)=\sum_{l\in\mathcal{M}}\widetilde{v_{l}}\fe^{\ii l\tau}
\]
with $l\in \mathcal{M}\coloneqq\left\{ -\frac{N_{\tau}}{2},-\frac{N_{\tau}}{2}+1,
\ldots, \frac{N_{\tau}}{2}-1 \right\}$. The Fourier coefficients
$\widehat{v_{l}}$ and the discrete Fourier coefficients $\widetilde{v_{l}}$ are
given by
$\widehat{v}_{l}=\frac{1}{2\pi}\int_{0}^{2\pi}v(\tau)\fe^{-\ii l\tau}\,\mathrm{d}\tau$
and
$\widetilde{v}_{l}=\frac{1}{N_{\tau}}\sum_{k=0}^{N_{\tau}-1}v(\tau_{k})\fe^{-\ii l\tau_{k}}$.
Applying the Fourier spectral method in the $\tau$-direction, we determine
trigonometric polynomials
\[
Z^{\mathcal{M}}(t,\tau)=P_{\mathcal{M}} Z(t,\tau)
=\left(Z_{1,j}^{\mathcal{M}}(t,\tau),Z_{2,j}^{\mathcal{M}}(t,\tau),\ldots,
Z_{\tilde{k},j}^{\mathcal{M}}(t,\tau)\right)_{j=1,2,\ldots,N_{x}}
\]
satisfying the equation
\begin{align}\label{equ-2-21-4}
\partial_{t}Z^{\mathcal{M}}(t,\tau)+\frac{1}{\eps^{2}}\partial_{\tau}Z^{\mathcal{M}}(t,\tau)
=\Upsilon(t,\tau,Z^{\mathcal{M}}(t,\tau)).
\end{align}
Let $\widehat{\mathbf{Z}}=[\widehat{\mathbf{Z}}_{1}; \widehat{\mathbf{Z}}_{2};
\ldots;\widehat{\mathbf{Z}}_{\tilde{k}}]$ with
$\widehat{\mathbf{Z}}_{\tilde{k}}\coloneqq(\widehat{Z}^l_{\tilde{k},j})_{j=1,2,\ldots,N_{x},\,l\in\mathcal{M}}$,
where $(\widehat{Z}^l_{\tilde{k},j})_{l\in\mathcal{M}}$ denote the Fourier
coefficients of $Z^{\mathcal{M}}_{\tilde{k},j}$. Then one arrives at
\begin{equation}\label{equ-2-21-2}
\frac{\mathrm{d}}{\mathrm{d}t}\widehat{\mathbf{Z}}(t)=\ii\Theta\widehat{\mathbf{Z}}(t)
+\Sigma(t,\widehat{\mathbf{Z}}(t)),
\end{equation}
in which $\widehat{\mathbf{Z}}$ is a $\tilde{k} N_{x}N_{\tau}$-dimensional
vector, and
$\Theta\coloneqq\operatorname{diag}(\Theta_{1},\Theta_{2},\ldots,\Theta_{\tilde{k} N_{x}})$
with
$\Theta_{1}=\Theta_{2}=\ldots=\Theta_{\tilde{k} N_{x}}\coloneqq
\frac{1}{\eps^{2}}\operatorname{diag}\left(\frac{N_{\tau}}{2},\frac{N_{\tau}}{2}-1,
\ldots,-\frac{N_{\tau}}{2}+1 \right)$. The term $\Sigma$ for the KGD system
\eqref{fil-system-2} is expressed as
\[
\Sigma(t,\widehat{\mathbf{Z}}(t))\coloneqq\Big[
\mathfrak{F}\big(\mathcal{C}_{1,-}(\varrho_{f}(\mathcal{C}_{\omega,-}
\mathfrak{F}^{-1}\widehat{\vec{\mathbf{Z}}}(t)))\big);
\mathfrak{F}\big(\mathcal{C}_{\omega,+}(\mathcal{G}_{f}
(\mathcal{C}_{1,+}\mathfrak{F}^{-1}\widehat{\mathbf{Z}}_{1},
\mathcal{C}_{\omega,-}\mathfrak{F}^{-1}\widehat{\vec{\mathbf{Z}}}(t)) )\big)\Big],
\]
where $\mathfrak{F}$ stands for the discrete fast Fourier transform (FFT), and
the operators are defined by
\begin{align*}
\mathcal{C}_{1,\pm}\coloneqq \left(\fe^{\pm\frac{\ii t}{\eps^2}}\fe^{\pm\ii t\mD_{1,f}}\right)_{k=0,1,\ldots,N_{\tau}-1}, \quad
\mathcal{C}_{\omega,\pm}\coloneqq \operatorname{diag}\left(\fe^{\mp\ii\tau_{k}}\fe^{\mp\ii t \mathcal{D}_{\omega,f}}\Pi_{+,f}^{\eps}+\fe^{\pm\ii\tau_{k}}\fe^{\pm\ii t \mathcal{D}_{\omega,f}}\Pi_{-,f}^{\eps} \right)_{k=0,1,\ldots,N_{\tau}-1}.
\end{align*}
The KGS case is similar, with only minor differences in the specific terms.

\subsection{Time discretization}\label{sec-3.2}
We next consider the time discretization. Applying Duhamel's formula to
\eqref{equ-2-21-2}, one obtains
\begin{align}\label{equ-6-23-1}
\hbZ(t_{n}+\eta)=\fe^{\frac{\ii\Theta\eta}{\eps^{2}}}\hbZ(t_{n})
+\int_{0}^{\eta}\fe^{(\eta-\theta)\frac{\ii\Theta}{\eps^{2}}}
\Sigma(t_{n}+\theta,\hbZ(t_{n}+\theta))\,\mathrm{d}\theta,
\ 0\leq\eta\leq\Delta t, \ n\geq 0.
\end{align}
We adopt the framework of $s$-stage exponential Runge--Kutta methods
\cite{HO06,HO10}, given by
\begin{equation}\label{equ-6-23-2}
\begin{aligned}
&\hbZ^{n,j}=\fe^{c_{j}\ii\Theta\triangle t}\hbZ^{n}+\triangle t
\sum_{k=1}^{s} a_{jk} (\ii\Theta\triangle t)\Sigma(t_{n}+c_{k}\triangle t,\hbZ^{n,k}),\quad j=1,2,\ldots,s,\\
&\hbZ^{n+1}=\fe^{\ii\Theta\triangle t}\hbZ^{n}+\triangle t
\sum_{j=1}^{s} b_{j} (\ii\Theta\triangle t)\Sigma(t_{n}+c_{j}\triangle t,\hbZ^{n,j}),\quad n\geq0,
\end{aligned}
\end{equation}
where $t_{n}=n\triangle t$ with the time step $\triangle t>0$, $c_j\in[0,1]$,
and $a_{jk}(z), b_j(z)$ are bounded coefficient functions of $z\in \mathbb{C}$
for $j,k=1,2,\ldots,s$ with $s\in\mathbb{N}_+$. If the coefficients satisfy
\begin{equation}\label{equ-2-10-5}
c_{\rho}=1-c_{s+1-\rho},\quad b_{\rho}(z)=\fe^{z}b_{s+1-\rho}(-z),\quad
a_{j\rho}(z)=\fe^{c_{j}z}b_{s+1-\rho}(-z)-a_{s+1-j,s+1-\rho}(-z)
\end{equation}
for $\rho,j=1,2,\ldots,s$, then the integrator \eqref{equ-6-23-2} is time
symmetric.

Under the symmetry and stiff order conditions of exponential Runge--Kutta methods
\cite{HO06,HO10}, we formulate two numerical schemes.

\medskip
\noindent$\bullet$ \textbf{Symmetric scheme}. In the case $s=3$, solving the fourth-order
order conditions of exponential Runge--Kutta methods with
symmetric nodes  leads to the coefficients (\cite{LWZ,WWL})
\begin{equation}\label{sori4 coe}
\begin{aligned}
&c_{1}=1, \ c_{2}=\frac{1}{2}, \ c_{3}=0, \  a_{31}=a_{32}=a_{33}=0,\ \
a_{21}=-\frac{1}{4} \varphi_{2,2} + \frac{1}{2}\varphi_{3,2},\ \ a_{22}=\varphi_{2,2} -\varphi_{3,2}, \\ &a_{23}=\frac{1}{2}\varphi_{1,2}-\frac{3}{4}\varphi_{2,2} +\frac{1}{2}\varphi_{3,2}, \ \
a_{11}=b_{1}= 4\varphi_3  - \varphi_2 , \ \
 a_{12}=b_{2}=4\varphi_2 -8\varphi_3, \\
& a_{13}=b_{3}=\varphi_1-3\varphi_2 +4\varphi_3,
\end{aligned}
\end{equation}
where
$\varphi_{i,j}=\varphi_{i,j}(\ii\Theta\triangle t)=\varphi_{i}(c_{j}\ii\Theta\triangle t)$
with
$\varphi_{\rho}(z)\coloneqq\int_0^1
\theta^{\rho-1}\frac{\fe^{(1-\theta)z}}{(\rho-1)!}\,\mathrm{d}\theta$.
The resulting scheme is symmetric and is denoted by SUA-EI4.

\medskip
\noindent$\bullet$ \textbf{Explicit scheme}. We now turn to the case $s=5$, for
which an explicit scheme is obtained by taking the coefficients as
\cite{HO06}
\begin{equation} \label{eori4 coe}
\begin{aligned}
&c_{1}=0,  \ c_{2}=c_{3}=c_{5}=\frac{1}{2}, \  c_{4}=1, \ a_{21}=\frac{1}{2}\varphi_{1,2}, \  a_{31}=\frac{1}{2}\varphi_{1,3}-\varphi_{2,3}, \   a_{32}=\varphi_{2,3}, \\
&a_{41}=\varphi_{1,4}-2\varphi_{2,4}, \ \qquad  a_{42}=a_{43}=\varphi_{2,4}, \ \ \qquad  a_{51}=\frac{1}{2}\varphi_{1,5}-2a_{52}-a_{54}, \\
&a_{52}=\frac{1}{2}\varphi_{2,5}-\varphi_{3,4}+\frac{1}{2}\varphi_{2,4}-\frac{1}{2}\varphi_{3,5}, \ \ a_{53}=a_{52}, \ \ a_{54}=\frac{1}{4}\varphi_{2,5}-\varphi_{5,2}, \\
&b_{1}=\varphi_{1}-3\varphi_{2}+4\varphi_{3}, \ \ \ b_{2}=b_{3}=0, \ \ \ b_{4}=-\varphi_{2}+4\varphi_{3}, \ \ \ b_{5}=4\varphi_{2}-8\varphi_{3}.
\end{aligned}
\end{equation}
This scheme is not symmetric but is explicit. It is referred to as EUA-EI4.

Consequently, the fully discrete scheme is given by \eqref{equ-6-23-2} with
coefficients \eqref{sori4 coe} or \eqref{eori4 coe}. In practical
implementations, the integrals required for computing the Fourier coefficients
are not amenable to direct evaluation; the scheme \eqref{equ-6-23-2} is
therefore typically implemented via interpolation. Given
$\mathbf{Z}^{0}=\big(Z^{0}(\tau_k)\big)_{k=0,1,\ldots,N_{\tau}-1}$, the Fourier
pseudospectral (FP) discretization takes the form
\begin{align}\label{equ-3-14-1}
Z^{n+1}(\tau)=\sum\limits_{l\in \mathcal{M}}(\widetilde{\mathbf{Z}}^{n+1})_{l}\fe^{\ii l\tau},
\end{align}
where
\begin{equation}\label{full-discre}
\begin{aligned}
&\widetilde{\mathbf{Z}}^{n,j}=\fe^{c_{j}\ii\Theta\triangle t}\widetilde{\mathbf{Z}}^{n}+\triangle t
\sum_{k=1}^{s} a_{jk} (\ii\Theta\triangle t)\Sigma(t_{n}+c_{k}\triangle t,\widetilde{\mathbf{Z}}^{n,k}),\quad j=1,2,\ldots,s,\\
&\widetilde{\mathbf{Z}}^{n+1}=\fe^{\ii\Theta\triangle t}\widetilde{\mathbf{Z}}^{n}+\triangle t
\sum_{j=1}^{s} b_{j} (\ii\Theta\triangle t)\Sigma(t_{n}+c_{j}\triangle t,\widetilde{\mathbf{Z}}^{n,j}),\quad n\geq0,
\end{aligned}
\end{equation}
with the coefficients \eqref{sori4 coe} or \eqref{eori4 coe}.

We conclude this section by presenting the fully discrete schemes for the
original coupled system; Algorithm \ref{algo-1} for the KGD equation \eqref{KGD}
is derived therefrom, and the KGS case \eqref{KGS} is analogous.

\begin{algorithm}[t]
\caption{UA fully discrete scheme for the KGD system}
\label{algo-1}
\begin{algorithmic}[1]

\REQUIRE A time step $\Delta t>0$, positive even integers $N_x$ and $N_\tau$, the final time $T$, and the initial data $[\phi_0;\Psi_0]$.

\STATE \textbf{Initialize:} Evaluate \eqref{init-data} with $\mathfrak{b}_0$ to obtain $Z^0(\tau):=Z(0,\tau)$, and set
\[
\widehat{\mathbf{Z}}^{0}:=\operatorname{FFT}\Big(\big(Z^0(\tau_k)\big)_{k=0}^{N_\tau-1}\Big).
\]

\FOR{$n=0,1,\ldots,T/\Delta t-1$}

    \STATE Compute $\widehat{\mathbf{Z}}^{n+1}$ from the UA-EI4s scheme \eqref{equ-6-23-2}, with coefficients specified by \eqref{sori4 coe} or \eqref{eori4 coe}.

    \STATE Reconstruct $Z^{n+1}$ according to
    \[
    Z^{n+1}:=\sum_{l\in\mathcal{M}}\big(\widehat{\mathbf{Z}}^{n+1}\big)_l\fe^{\ii l t_{n+1}/\eps^2}, \qquad t_{n+1}:=(n+1)\Delta t.
    \]

    \STATE Recover the KGD solution by
    \[
    \phi^{n+1}:=\operatorname{Re}\!\left(\fe^{\ii\Lambda t_{n+1}/\eps^2}Z_1^{n+1}\right), \ 
    \Psi^{n+1}:=\fe^{-\ii\mathcal{S}^{\eps}t_{n+1}/\eps^2}\vec{Z}^{n+1}.
    \]

\ENDFOR
\end{algorithmic}
\end{algorithm}

\subsection{Properties of the algorithm}\label{sec-3.3}
In this subsection, we establish both the uniform accuracy of the fully discrete
scheme \eqref{equ-6-23-2} and its long-time near-conservation of mass and energy;
these results constitute the main theorems of this paper. In what follows, we
write $A\lesssim B$ to indicate that there exists a generic constant $C>0$
independent of $\eps$, $\triangle t$, $\triangle x$, and $\triangle\tau$ such
that $A\leq C B$.

We begin by assuming that the initial data and the solution of the coupled
system \eqref{ori-coup} satisfy the following condition.
\begin{flalign*}
&\text{(B)} \qquad \norm{\mathfrak{b}_{0}}_{H^{\nu+m_{0}+4}}\lesssim 1, \quad \norm{\mathfrak{b}}_{L^{\infty}([0,T];H^{\nu+m_{0}+4})}\lesssim 1. &
\end{flalign*}
Specifically, for the KGD equation \eqref{KGD}, the initial data are assumed to
satisfy
\begin{flalign*}
&\text{(C1)} \qquad \norm{\phi_{0}}_{H^{\nu+m_{0}+4}}, \norm{\phi_{1}}_{H^{\nu+m_{0}+4}}, \ \mbox{and}  \ \norm{\Psi_{0}}_{(H^{\nu+m_{0}+4})^2} \ \mbox{are uniformly bounded for all} \  \eps\in(0,1], &
\end{flalign*}
and the electromagnetic potentials and the exact solution are assumed to satisfy
\begin{flalign*}
&\text{(C2)} \qquad \norm{V}_{W_{p}^{\nu+m_{0}+4,\infty}}+\norm{A_{1}}_{W_{p}^{\nu+m_{0}+4,\infty}}\lesssim 1, \ \norm{\Phi}_{L^{\infty}([0,T];H_{p}^{\nu+m_{0}+4})}+\norm{\Psi}_{L^{\infty}([0,T];(H_{p}^{\nu+m_{0}+4})^{2})}\lesssim 1, &
\end{flalign*}
where $\nu>d/2+8$, $m_{0}\geq 1$, and
$W_{p}^{m,\infty}=\{u| u\in W^{m,\infty}(\Omega),\partial_{x}^{l}u(a)=\partial_{x}^{l}u(b), l=0, \ldots, m-1\}$,
$H_{p}^{m}(\Omega)=\{u|u\in H^{m}(\Omega),\partial_{x}^{l}u(a)=\partial_{x}^{l}u(b), l=0, \ldots, m-1\}$ for $m\in\mathbb{N}$.
The assumptions on the initial data and the solution for the KGS equation are
analogous.

First, we derive error bounds for the fully discrete schemes proposed above.
\begin{theorem}\label{theo-1} (\textbf{Uniform accuracy})
Under assumptions \text{(A)} and \text{(B)}, there exist $t_{0}>0$, $h_{0}>0$ and $\tau_{0}>0$ independent of $\eps$, such that for any $0< \eps \leq 1$, whenever $0<\triangle t\leq t_{0}$, $0<\triangle x\leq h_{0}$ and $0<\triangle \tau\leq \tau_{0}$, the global errors of the fully discrete schemes \eqref{equ-6-23-2} defined by
$e_{\mathfrak{b}}^{n}:=\mathfrak{b}(t_{n},x)-\mathfrak{b}^n(x)$ satisfy the bounds
\begin{equation}\label{equ-2-10-6}
\norm{\fe_{\mathfrak{b}}^{n}}_{H^{1}}  \lesssim \triangle t^4 +\triangle x^{m_{0}}+ \triangle\tau^{\nu-1} ,\qquad \   n=0,1,\ldots, T/\triangle t,
\end{equation}
where $\nu>d/2+8$ and $m_{0}\geq 1$.
\end{theorem}
\begin{proof}The error estimates established in Theorem \ref{theo-1} carry over to the scheme
\eqref{equ-3-14-1}--\eqref{full-discre}. The proof is analogous to those given in
\cite{FXY,WZ} and  therefore only the main steps are presented here.  

\emph{Step 1: spatial and micro-step errors.}
The Fourier pseudospectral error for functions in
$H_{p}^{\nu+m_{0}}(\Omega)$ yields 
$\norm{P_{N_{x}}\mathfrak{b}-\mathfrak{b}}_{H^{\nu}}
\lesssim \triangle x^{m_{0}}$,
while the resolution of the auxiliary micro-scale variables on the grid
of size $\triangle\tau$ contributes
$\mathcal{O}(\triangle\tau^{\nu-1})$ in the $H_{1}$ norm,
owing to the smoothness of the profile functions
guaranteed by \text{(B)} and $\nu>d/2+8$.

\emph{Step 2: local truncation error.}
By combining the preparation of the initial layer and the multiscale
expansions of the exact solution with the local error analysis of the
exponential Runge--Kutta method,
the local truncation errors of the interior variables are bounded by  
\begin{equation*}
  \norm{\mathcal{T}^{n+1}_{\mathfrak{b}}}_{H^{1}}
  \lesssim 
  \norm{\triangle t^{4} \Psi_{4}\big(\ii\Theta\triangle t\big)  \frac{\mathrm{d}^{3}}{\mathrm{d} t^{3}}\Sigma(t_{n},\hbZ(t_{n}))}_{H^{1}}+\triangle t^{5}
\end{equation*}
where $\Psi_{\rho}(z):=\varphi_{\rho}(z)-\sum_{j=1}^{s}b_j(z)c_j^{\rho-1}/(\rho-1)!$. The coefficients satisfy $\Psi_\rho(i\Theta\Delta t)=0$ for $\rho=1,2,3$
and $\Psi_4(0)=0$. The corresponding stage defects are treated using the order conditions of each scheme.

\textit{Step 3: stability and induction.}
Applying the convergence analysis of exponential Runge--Kutta method with the above defect expansion, the corresponding stiff order conditions, and the uniform stability bounds gives
$\norm{\hbZ(t_n)-\hbZ^{n}}_{H^{1}}\lesssim \Delta t^{4}.
$
Combining this estimate with \emph{Step 1} yields
\begin{align*}
\norm{e_{\mathfrak{b}}^{n}}_{H^{1}}\lesssim
\Delta t^{4}+\Delta x^{m_{0}}+\Delta\tau^{\nu-1}.
\end{align*}
This completes the proof.
\end{proof}

Next, we present results on the long-time near-conservation of mass and energy.
Following the modulated Fourier expansion (MFE) procedure in \cite{WZ23, LWZ},
we introduce the notation
\begin{equation*}\displaystyle
\begin{array}[c]{ll}
&
\mathbf{k}=\left(k_{-\frac{N_{\tau}}{2},1},\ldots,
k_{\frac{N_{\tau}}{2}-1,1},k_{-\frac{N_{\tau}}{2},2},\ldots,
 k_{\frac{N_{\tau}}{2}-1,2},\ldots,k_{-\frac{N_{\tau}}{2},2N_{x}},\ldots,
k_{\frac{N_{\tau}}{2}-1,2N_{x}}\right),\\
&\displaystyle |\mathbf{k}|=\sum_{l=1}^{2N_{x}}\sum_{j=-N_{\tau}/2}^{N_{\tau}/2-1}
|k_{j,l}|,\ \
 \boldsymbol{\omega}=(\textmd{diagonal elements of}\  \Theta),\ \ \mathbf{k}\cdot \boldsymbol{\omega}=\sum_{l=1}^{2N_{x}}\sum_{j=-N_{\tau}/2}^{N_{\tau}/2-1}
k_{j,l}\omega_{j,l}.
\end{array}
\end{equation*}
For
$\mathcal{Q}= \{\mathbf{k}\in \mathbb{Z}^{D}:  \textmd{there exists an}\ l\in \{1,\ldots,2N_{x}\}\
 \textmd{such that}\
|\mathbf{k}_{:,l}|=|\mathbf{k}|\}$, we define the resonance module
$\mathcal{M}= \{\mathbf{k}\in {\mathcal{Q}}:\ \mathbf{k}\cdot
\boldsymbol{\omega}=0\}$. Let $\mathcal{K}$ be a set of representatives of the
equivalence classes in ${\mathcal{Q}}/ \mathcal{M}$, chosen so that for each
$\mathbf{k}\in\mathcal{K}$, $|\mathbf{k}|$ is minimal in its class
$[\mathbf{k}] = \mathbf{k} +\mathcal{M}$, and such that
 $-\mathbf{k}\in\mathcal{K}.$
For any integer $N>0$, we introduce the notation
\begin{equation}\begin{array}{ll}\mathcal{N}_N={\{\mathbf{k}\in\mathcal{K}:  |\mathbf{k}| \leq N \}},\ \ \
\mathcal{N}_N^*=\mathcal{N} \bigcup \{\langle 0\rangle_l\}_{l=1,2,\ldots,2N_{x}},
\end{array} \label{NN}\end{equation}
where $\langle j\rangle_l$ stands for the unit coordinate vector in
$ \mathbb{R}^{D}$, written as $(0, \ldots , 0, 1, 0, \ldots,0)^{\intercal}$, with
its sole non-zero component $1$ at the $(j,l)$-th position.

\begin{theorem}(\textbf{Long-time near-conservation of mass and energy})\label{Long-time thm}
For the initial value \eqref{init-data} of the two-scale equation \eqref{two-scale}, assume
$0<\delta_0:=\norm{\mathfrak{b}_{0}}_{L_{\tau}^{\infty}(H^{\nu+m_{0}+4})}<1$.
We further require the non-resonance condition
$|\sin\big(\frac{1}{2}\triangle t \omega_{j,1}\big)| \geq c_1 \sqrt{\eps}
\   \textmd{for}\     j=-\frac{N_{\tau}}{2},-\frac{N_{\tau}}{2}+1,\ldots,\frac{N_{\tau}}{2}-1,$ and the time-step constraint $\triangle t/\sqrt{\eps} \geq c_2 > 0.$
For the symmetric numerical approximation \eqref{equ-6-23-2} produced by SUA-EI4, the following long-time near-conservation properties hold:
\begin{equation}\label{er}
\begin{aligned}
  &\frac{\eps^2}{\delta_0^2} |E(\mathfrak{b}^{n})- E(\mathfrak{b}^{0})|= {\mathcal{O}(\eps^{2} \delta_0^4)}+ \mathcal{O}(\delta_{\mathcal{F}}),\  \frac{1}{\delta_0^2}|M(\mathfrak{b}^{n})- M(\mathfrak{b}^{0})| = {\mathcal{O}(\eps^{2}\delta_0^4)}+ \mathcal{O}(\delta_{\mathcal{F}}),
\end{aligned}
\end{equation}
where $   0\leq n\triangle t\leq  \eps^{-2}\delta_0^{-N+5}$, $N$ is an arbitrarily large integer as in \eqref{NN}, $\mathcal{O}$ denotes a term bounded in the Euclidean norm, and $\delta_{\mathcal{F}}$ represents the error introduced by the Fourier pseudospectral discretizations in $x$ and $\tau$.
The constants symbolized by $\mathcal{O}$ may depend on $c_1, c_2,   N, N_{\tau}, N_{x}$ but not on $n, \Delta t, \eps, \delta_0$.
\end{theorem}

\begin{proof}
The proof is carried out within the modulated Fourier expansion (MFE)
framework of \cite{WZ23,LWZ}, adapted to the present two-scale setting,
and consists of five steps. Full details are given in \cite{LWZ,WWL}.

\emph{Step 1: modulated expansion of the numerical solution.}
Over the whole time interval, the symmetric numerical solution
$\mathfrak{b}^{n}$ is expanded into a modulated Fourier series with respect to
the discrete frequency vector $\boldsymbol{\omega}$, truncated at
order $N$:
\begin{equation*}
  \mathfrak{b}^{n} \approx \sum_{\mathbf{k}\in\mathcal{N}_{N}}
  z_{\mathbf{k}}(\eps t)\,\mathrm{e}^{\,\mathrm{i}(\mathbf{k}\cdot\boldsymbol{\omega})t},
\end{equation*}
where the modulation functions $z_{\mathbf{k}}$ vary on the slow time
scale $\eps t$. The symmetric structure of SUA-EI4 ensures that the
defect of this expansion satisfies the backward error estimate
$\mathcal{O}(\eps^{2}\delta_0^{N+1})$ on the whole interval.

\emph{Step 2: modulation equations.}
Substituting the expansion into the two-scale system \eqref{two-scale}
and comparing the coefficients of
$\mathrm{e}^{\,\mathrm{i}(\mathbf{k}\cdot\boldsymbol{\omega})t}$ yield
a system of modulation equations for the modulation functions
$z_{\mathbf{k}}$, $\mathbf{k}\in\mathcal{N}_{N}$, posed on the slow
time scale. The non-resonance condition
$|\sin(\tfrac{1}{2}\triangle t\,\omega_{j,1})|\geq c_1\sqrt{\eps}$
together with the step-size constraint
$\triangle t/\sqrt{\eps}\geq c_2>0$ guarantees that all small divisors
$|\mathbf{k}\cdot\boldsymbol{\omega}|$ arising in this construction are
bounded from below, uniformly in $\eps$ and $\triangle t$. Moreover,
the Yukawa coupling enters the modulation system only through slowly
varying coefficients, so that no additional oscillatory components are
generated.

\emph{Step 3: almost-invariants of the modulation system.}
The modulation system admits almost-invariants $\widetilde{E}$ and
$\widetilde{M}$ that approximate the mass and the energy to leading
order:
\begin{equation*}
  E(\mathfrak{b}^{n})=\widetilde{E}( t_{n})
  +\mathcal{O}(\delta_0^{6}),\qquad
  M(\mathfrak{b}^{n})=\widetilde{M}( t_{n})
  +\mathcal{O}(\delta_0^{6}).
\end{equation*}
Along the solutions of the modulation equations, the almost-invariants
satisfy the estimates
\begin{equation*}
  \widetilde{E}( t_{n})=\widetilde{E}(0)
  +\mathcal{O}( t_{n}\eps^{2}\delta_0^{N+1}),\qquad
  \widetilde{M}( t_{n})=\widetilde{M}(0)
  +\mathcal{O}( t_{n}\eps^{2}\delta_0^{N+1}).
\end{equation*}

\emph{Step 4: near-conservation on the long time interval.}
Combining \emph{Step 3} with the error estimate of
\emph{Step 1} and the time interval
$0\leq n\triangle t\leq \eps^{-2}\delta_0^{-N+5}$, we obtain
\begin{equation*}
  \frac{\eps^{2}}{\delta_0^{2}}\bigl|E(\mathfrak{b}^{n})-E(\mathfrak{b}^{0})\bigr|
  =\mathcal{O}(\eps^{2}\delta_0^{4})
  +\mathcal{O}( n\triangle t\,\eps^{4}\delta_0^{N-1}),
\end{equation*}
and analogously for the mass. Since $N$ is arbitrary, the term
$\mathcal{O}( n\triangle t\,\eps^{4}\delta_0^{N-1})$ can be absorbed
into $\mathcal{O}(\eps^{2}\delta_0^{4})$ for
$0\leq n\triangle t\leq \eps^{-2}\delta_0^{-N+5}$.

\emph{Step 5: contribution of the spectral discretizations.}
Finally, the errors introduced by the Fourier pseudospectral
discretizations in $x$ and $\tau$ are bounded by the standard
interpolation estimates under the regularity assumed in \text{(B)},
which yields the remainder $\mathcal{O}(\delta_{\mathcal{F}})$ in
\eqref{er} and completes the proof.  
\end{proof}

\section{Numerical verification}\label{sec-4}
In this section, a series of numerical experiments are carried out for the KGD
and KGS equations to demonstrate the properties of the proposed numerical scheme
\eqref{equ-6-23-2}. The uniform accuracy of the scheme and the long-time
near-conservation of mass and energy are first examined in one dimension.
Subsequently, numerical simulations of the nonlinear dynamics of the KGD and KGS
equations are presented in two and three dimensions (2D and 3D). We take a
bounded domain $\Omega=(a,b)^{k}$, $k=1,2,3$, and impose periodic boundary
conditions. All computations are performed using MATLAB R2024a on a PC running
Windows 11, equipped with an Intel Core i7-13700 CPU @ 2.10 GHz and 64 GB of
RAM.

\subsection{KGD equation}\label{sec-4.1}
We first assess the performance of the two proposed numerical schemes for the
KGD equation \eqref{KGD}, for which the relative errors in accuracy
$err_{\phi}$ and $err_{\Psi}$, mass $err_{M}$, and energy $err_{E}$ are employed
as metrics.

\subsubsection{One-dimensional KGD equation}
We first solve the 1D KGD equation numerically with $g=1$ and $\omega=1$ on the
bounded interval $\Omega=(-32,32)$. The initial data are taken from \cite{CY3}:
\begin{align*}
\Psi_{0}(x)=(e^{-x^{2}/2},e^{-(x-1)^{2}/2})^{\intercal}, \quad \phi_{0}(x)=e^{-x^{2}/2}, \quad \phi_{1}(x)=\frac{3}{2}e^{-x^{2}/2}, \ x\in\mathbb{R}, \quad t\geq 0,
\end{align*}
and the electromagnetic potentials are chosen as
$V(x)=(1-x)/(1+x^2)$ and $A_{1}(x)=(x+1)^2/(1+x^2)$.

The spatial discretization error of the proposed SUA-EI4 scheme is first
investigated, with the corresponding results displayed in Figure
\ref{fig-5-1-1}. Here, the reference solution is computed by SUA-EI4 with
$h=0.01$, $N_{x}=2^{11}$, and $N_{\tau}=2^{6}$. Since EUA-EI4 yields comparable
results, they are omitted. These results indicate that the spatial
discretization in both the $x$- and $\tau$-directions achieves spectral accuracy
for all $\eps\in(0,1]$. Accordingly, $N_{x}=2^{10}$ and $N_{\tau}=2^{5}$ are
fixed in the subsequent experiments, which suffice to attain machine accuracy.

We then examine the temporal errors, where the reference solution is computed by
the second-order TSFP method \cite{YRS} with an ultra-fine time step of $10^{-7}$.
Figures \ref{fig-5-1-2} and \ref{fig-5-1-3} display the temporal errors at $T=1$
for varying $\eps$ and $h$. As the numerical results indicate, both SUA-EI4 and
EUA-EI4 exhibit fourth-order convergence in time, in agreement with the
theoretical bounds.

We proceed to evaluate the long-time performance of the two schemes through
their numerical energy $E(t_{n})$ and mass $M(t_{n})$. The corresponding
conservation errors are displayed in Figure \ref{fig-5-1-4}. The results clearly
reveal that the symmetric SUA-EI4 scheme retains excellent long-time
conservation, whereas the nonsymmetric EUA-EI4 scheme drifts appreciably in both
energy and mass.

\begin{figure}[t!]
    \centering
        \includegraphics[width=0.38\linewidth,height=3.6cm]{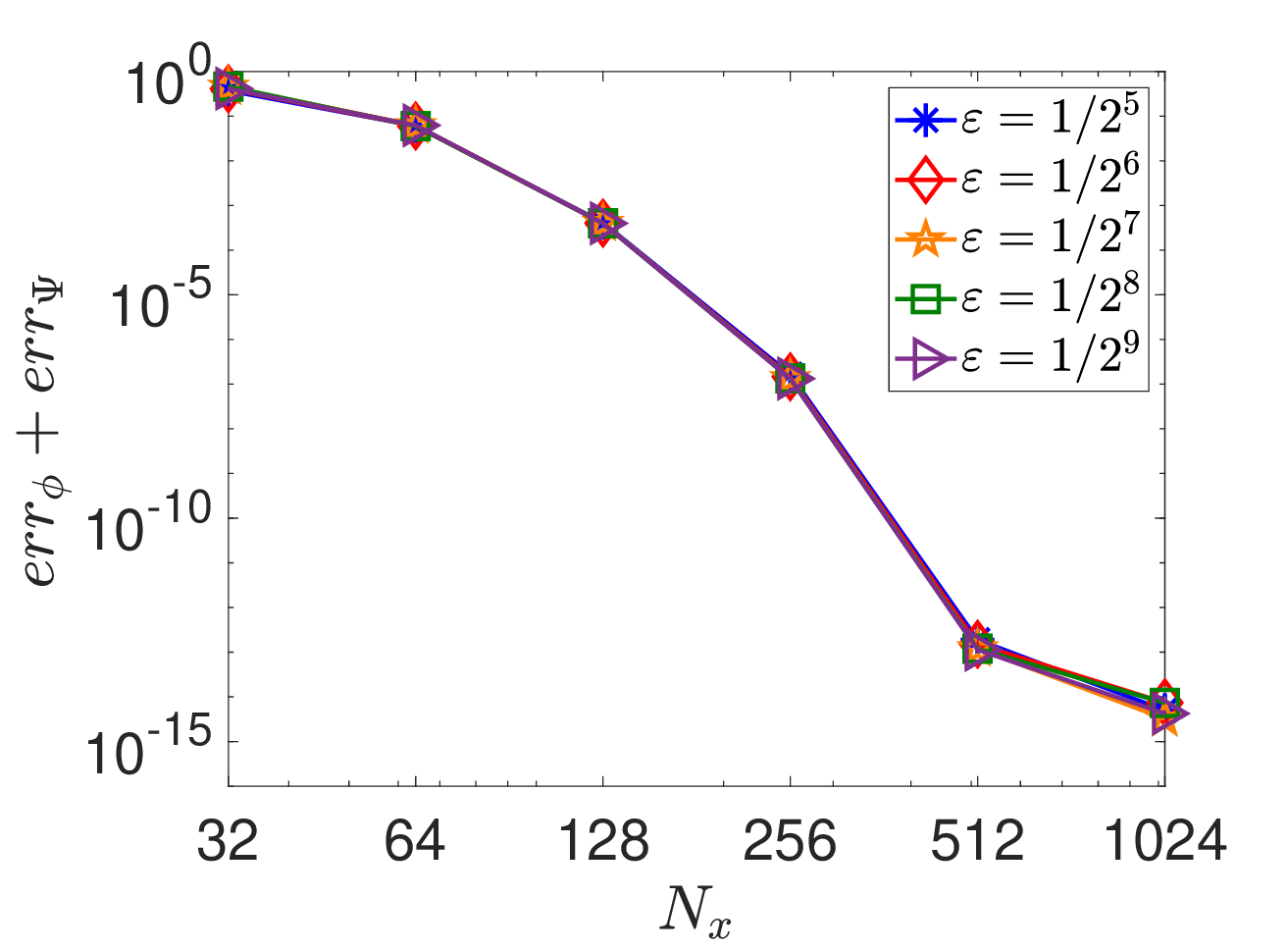}
        \includegraphics[width=0.38\linewidth,height=3.6cm]{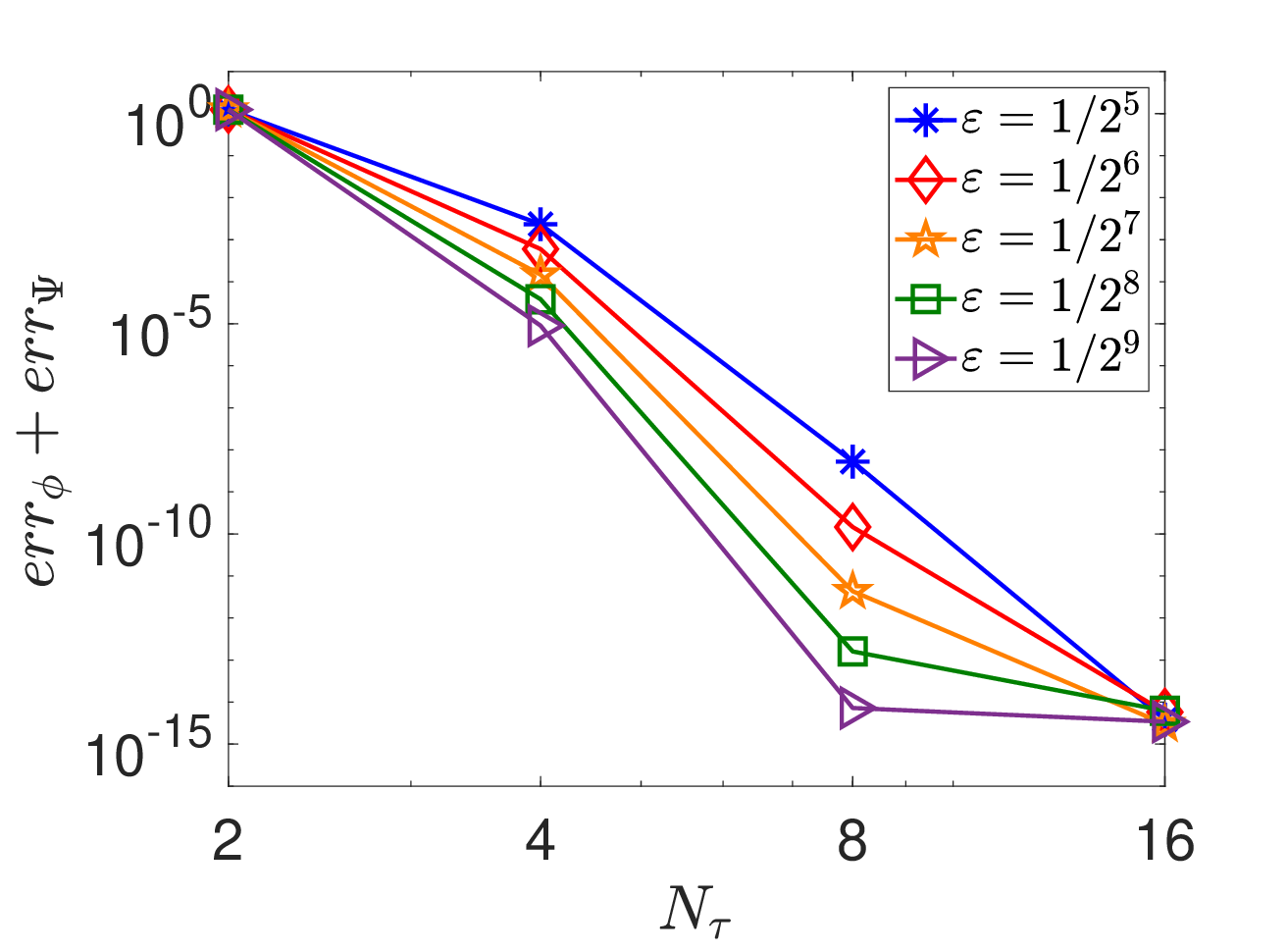}
    \caption{Spatial errors of KGD in 1D at $T=1$ under different $\eps$ in the $x$-direction (left) and $\tau$-direction (right).}
    \label{fig-5-1-1}
\end{figure}

\begin{figure}[t!]
    \centering
        \includegraphics[width=0.38\linewidth,height=3.6cm]{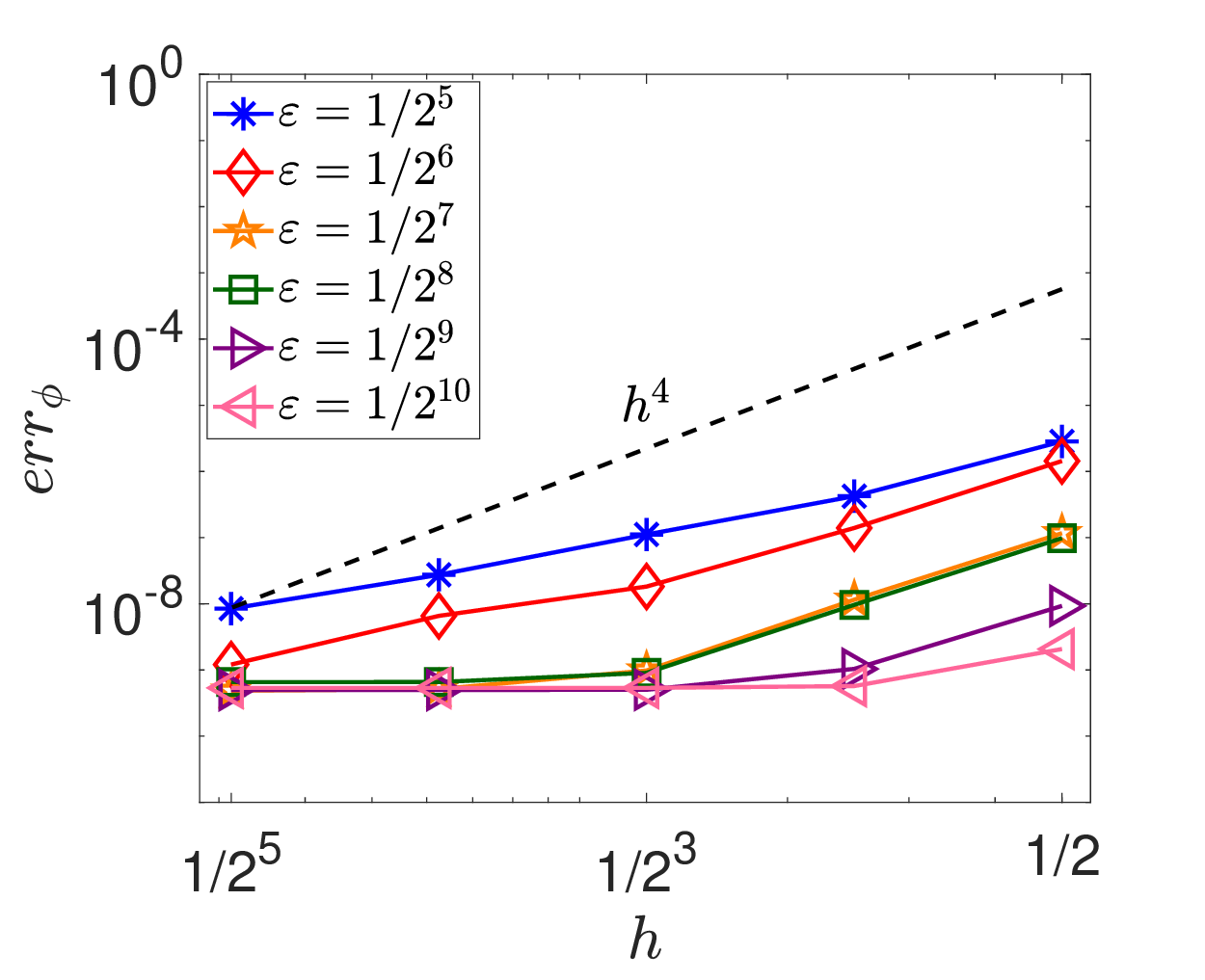}
        \includegraphics[width=0.38\linewidth,height=3.6cm]{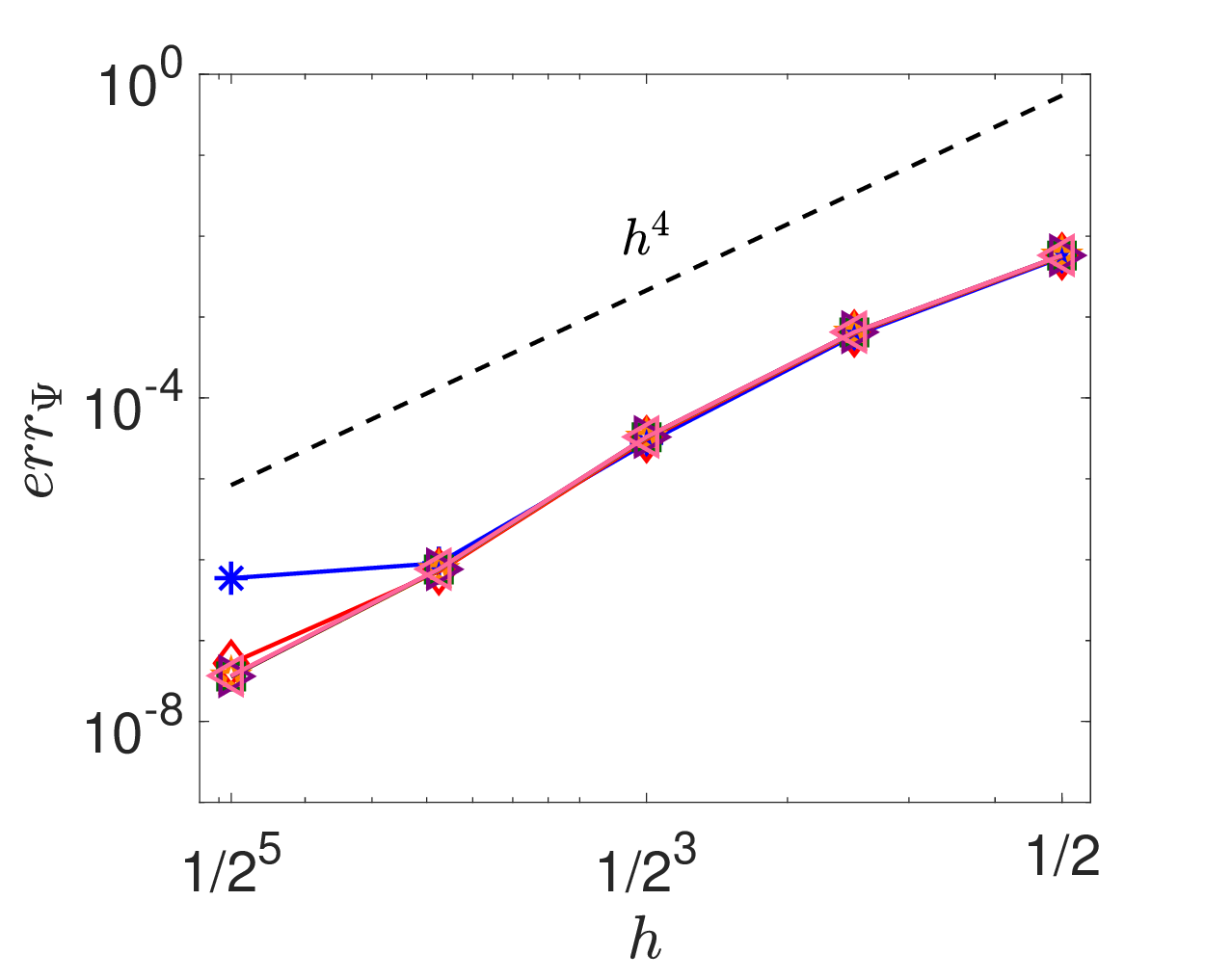}\\
        \includegraphics[width=0.38\linewidth,height=3.6cm]{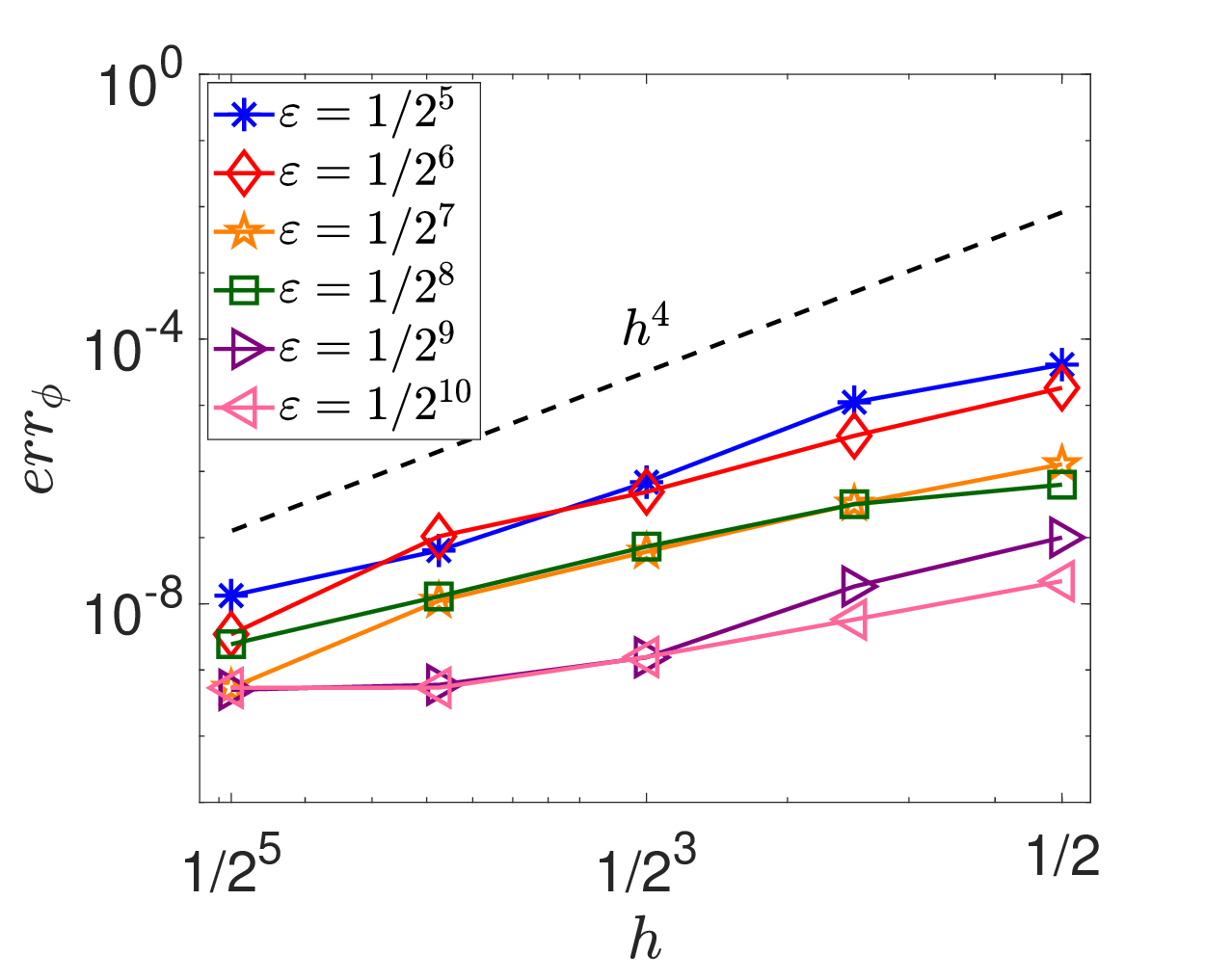}
        \includegraphics[width=0.38\linewidth,height=3.6cm]{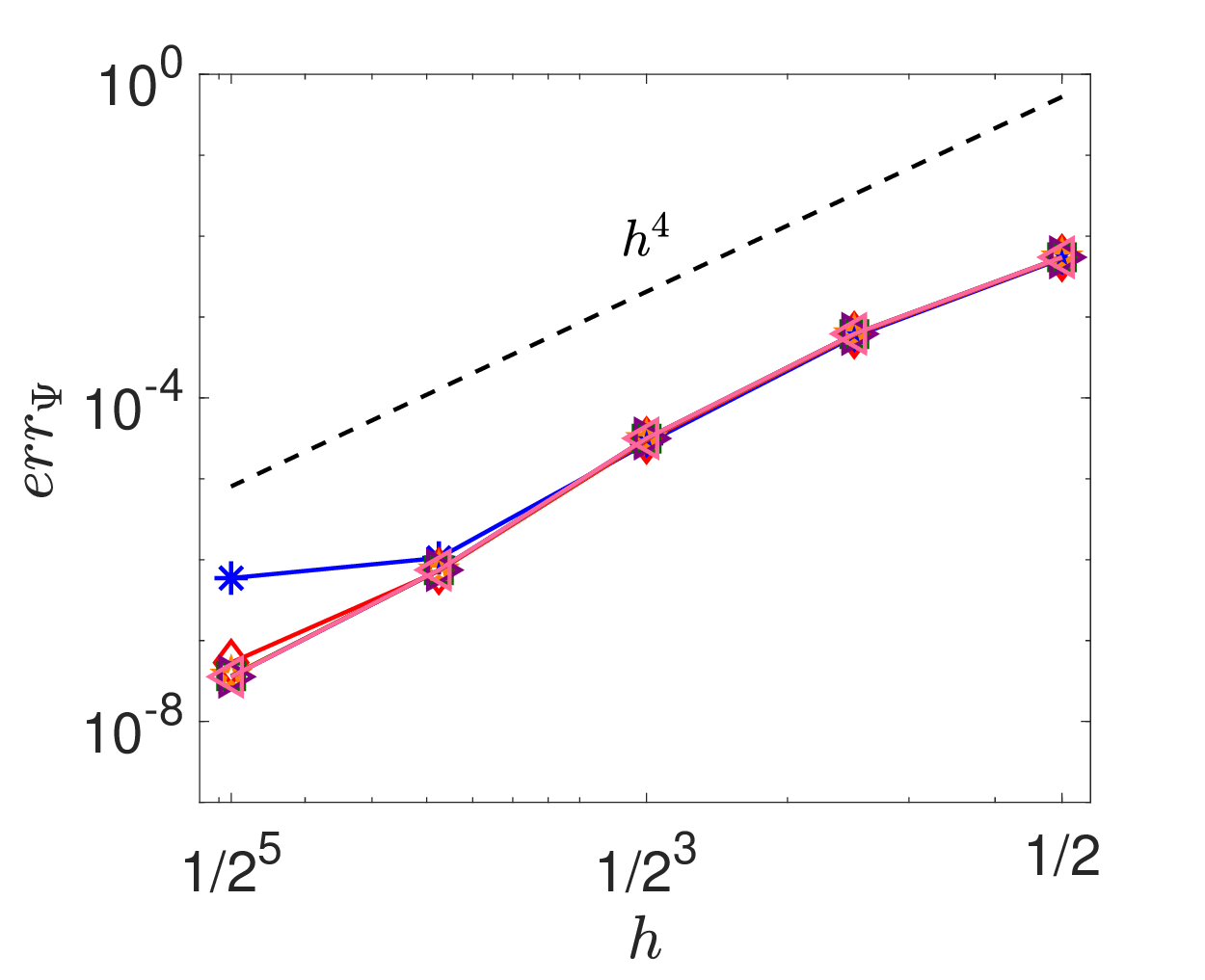}
     \caption{Temporal error of KGD in 1D at $T=1$ under different $\eps$ for SUA-EI4 scheme (first row) and EUA-EI4 scheme (second row).}
    \label{fig-5-1-2}
\end{figure}

\begin{figure}[t!]
    \centering
        \includegraphics[width=0.38\linewidth,height=3.6cm]{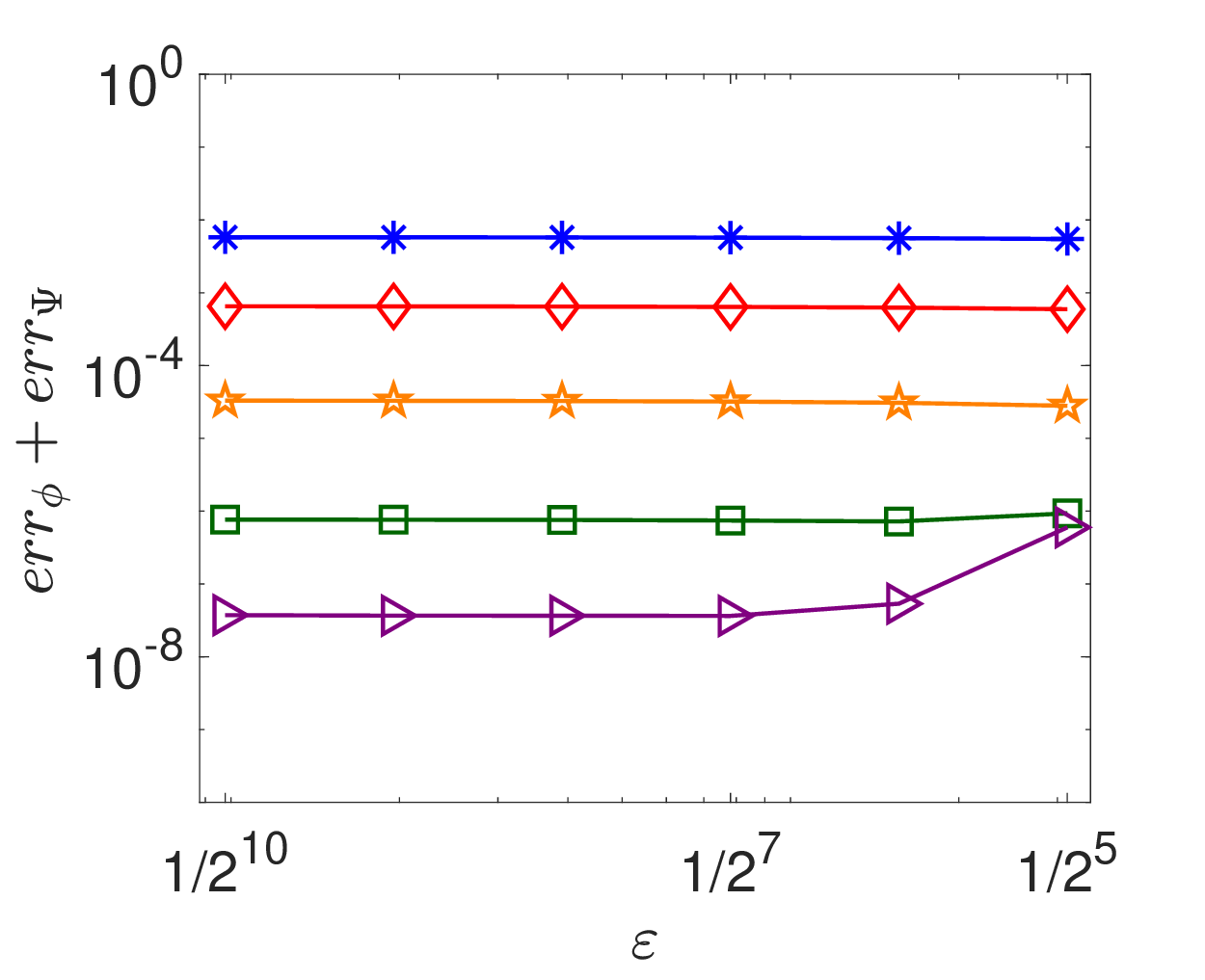}
        \includegraphics[width=0.38\linewidth,height=3.6cm]{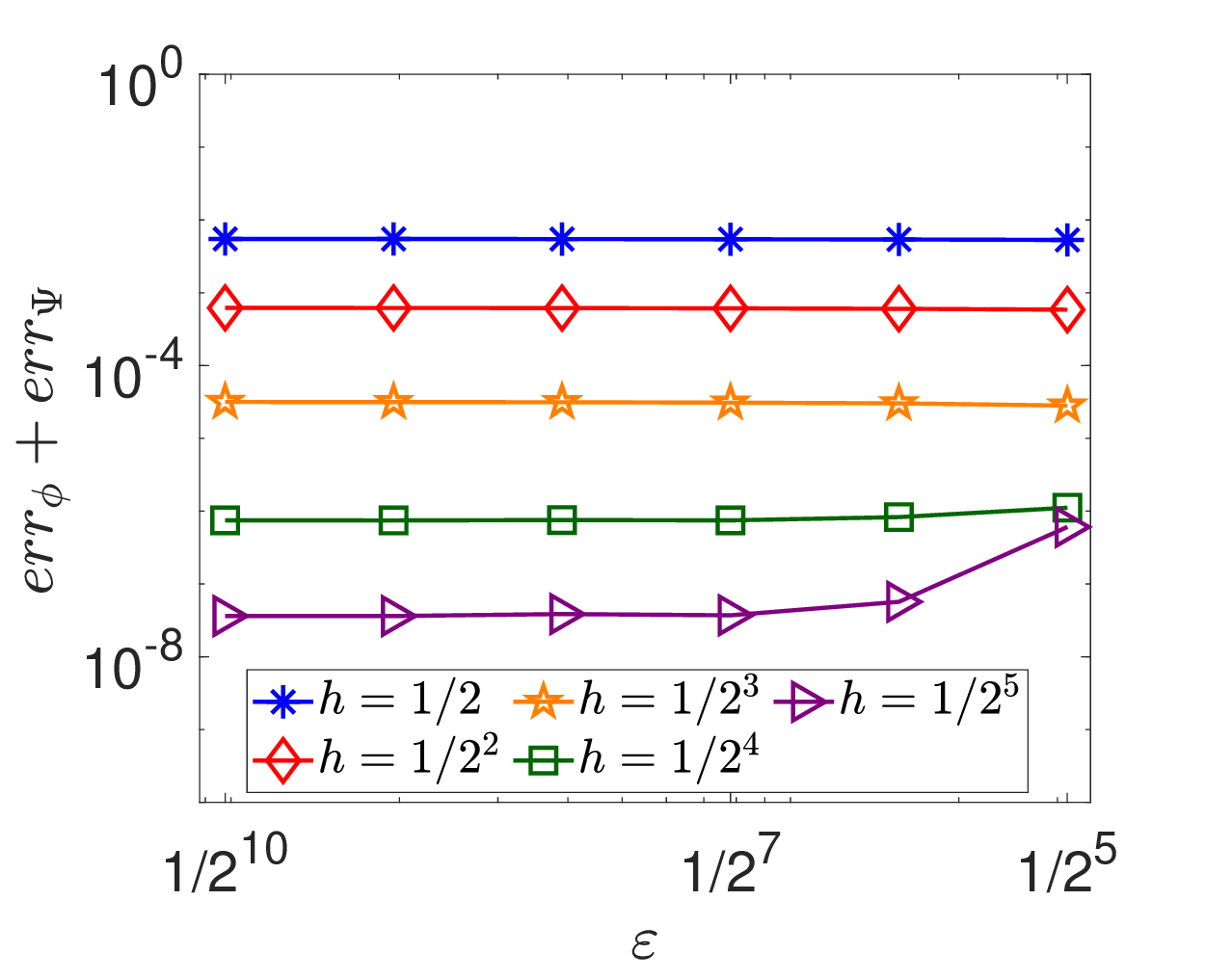}
    \caption{Temporal error of KGD in 1D at $T=1$ under different $h$ for SUA-EI4 scheme (left) and EUA-EI4 scheme (right).}
    \label{fig-5-1-3}
\end{figure}

\begin{figure}[t!]
    \centering
        \includegraphics[width=0.32\linewidth,height=3.1cm]{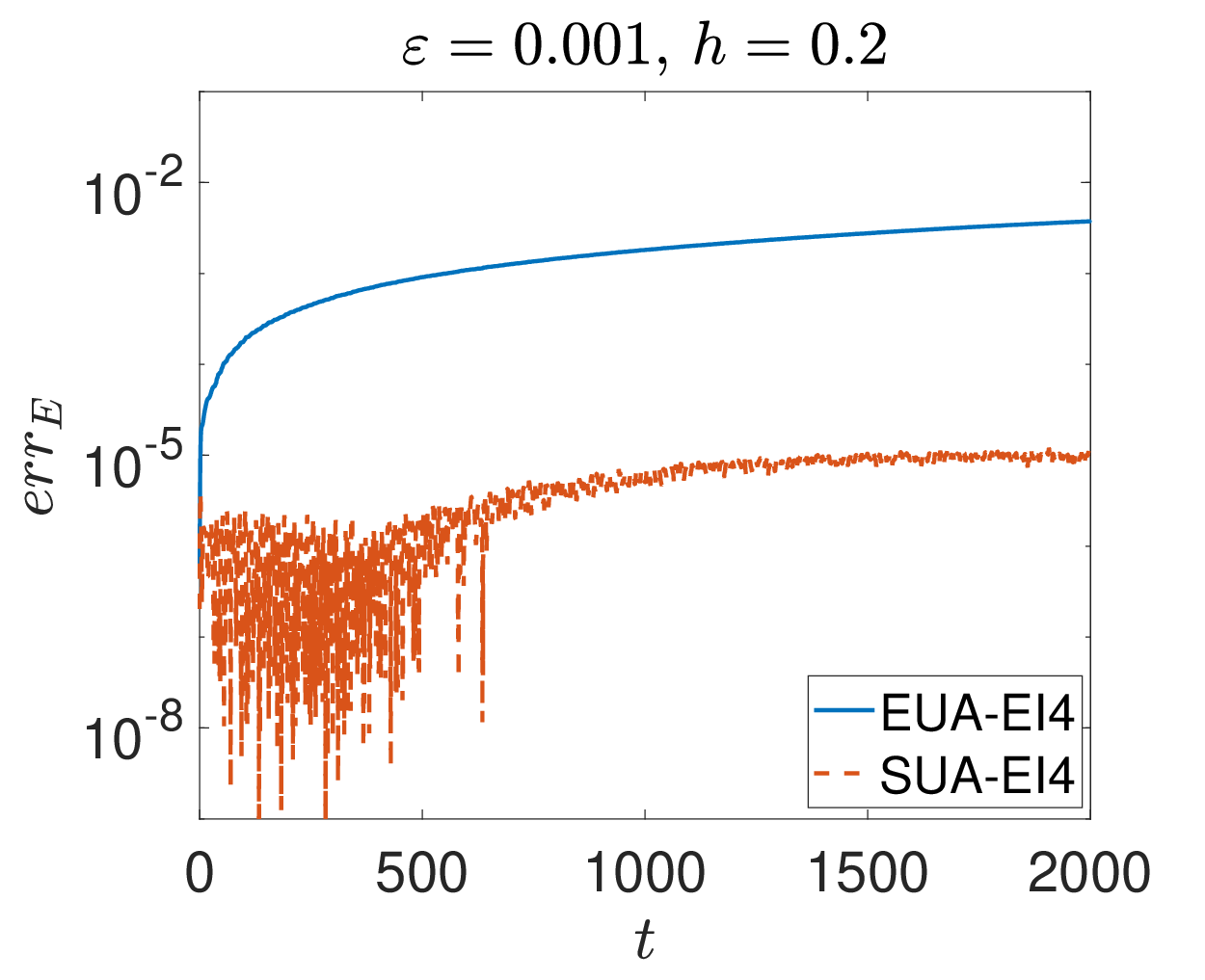}
        \includegraphics[width=0.32\linewidth,height=3.1cm]{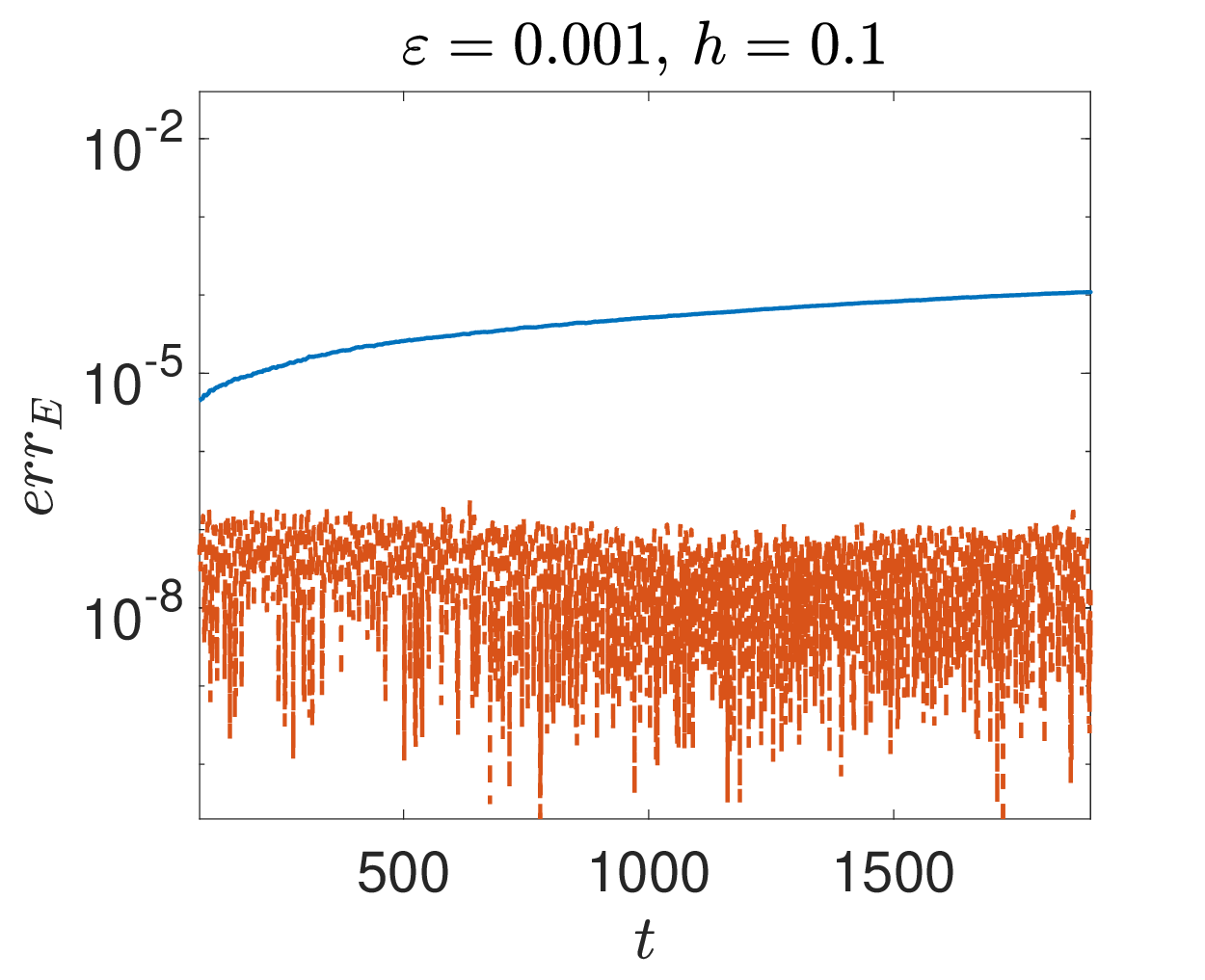}
        \includegraphics[width=0.32\linewidth,height=3.1cm]{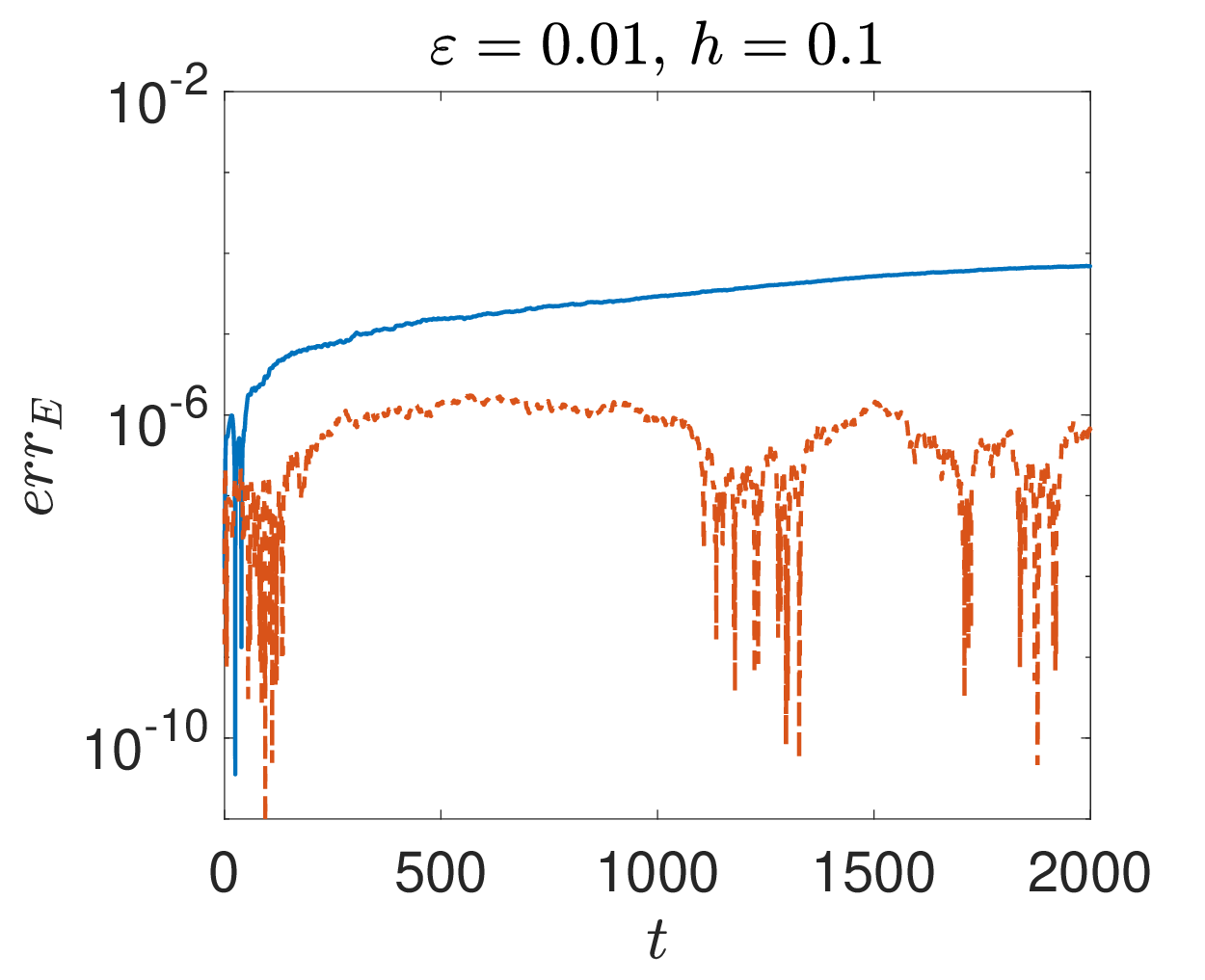}\\
        \includegraphics[width=0.32\linewidth,height=3.1cm]{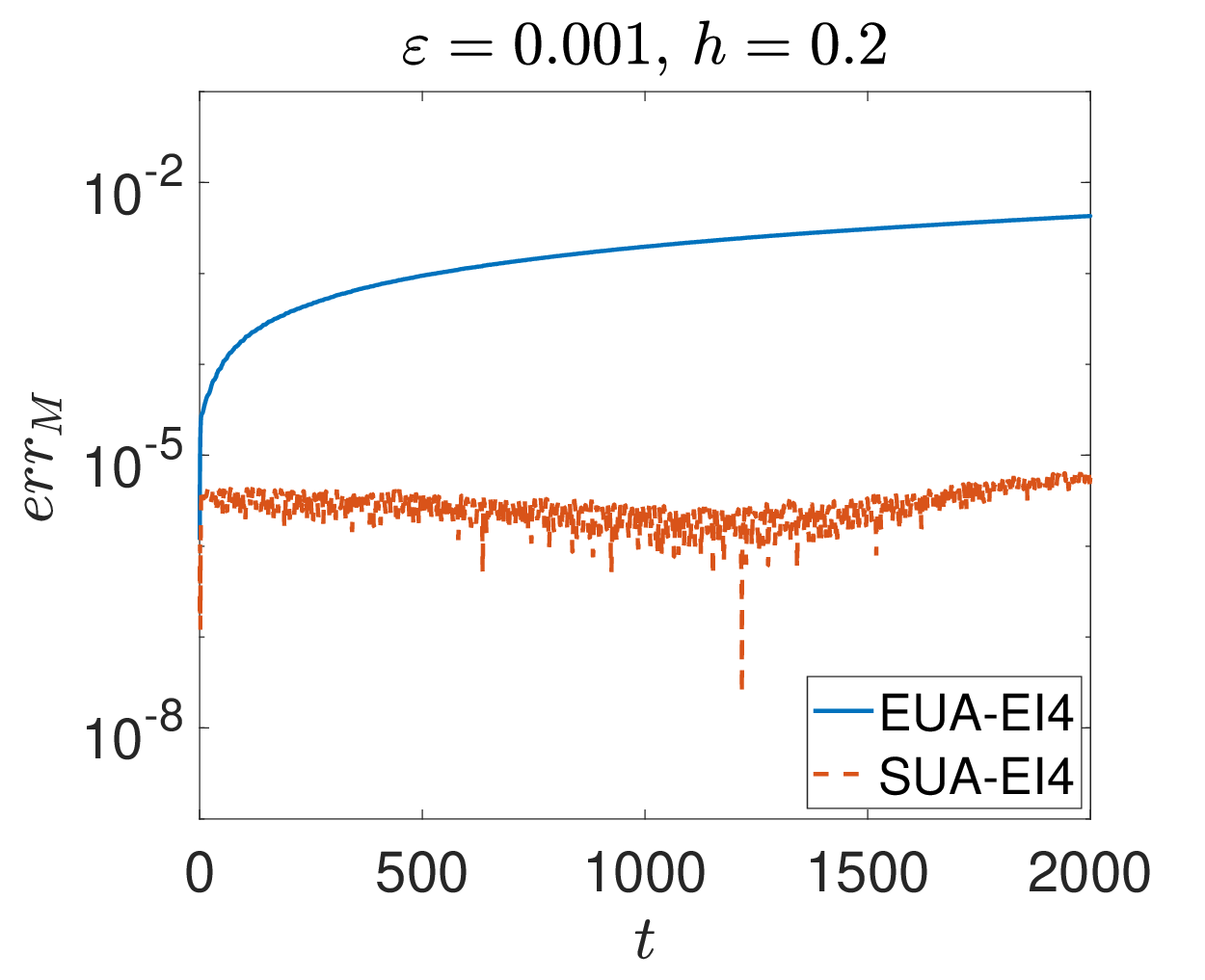}
        \includegraphics[width=0.32\linewidth,height=3.1cm]{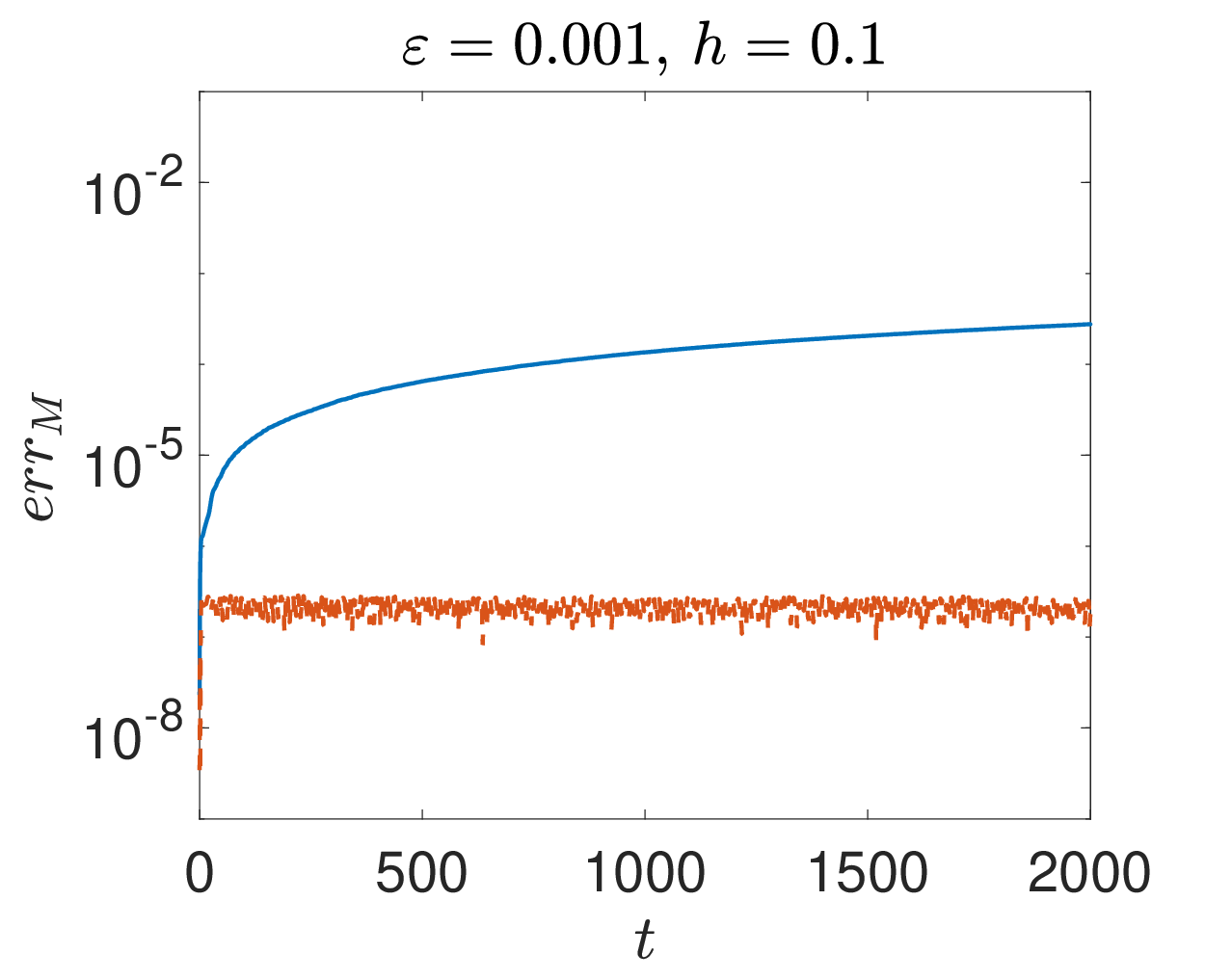}
        \includegraphics[width=0.32\linewidth,height=3.1cm]{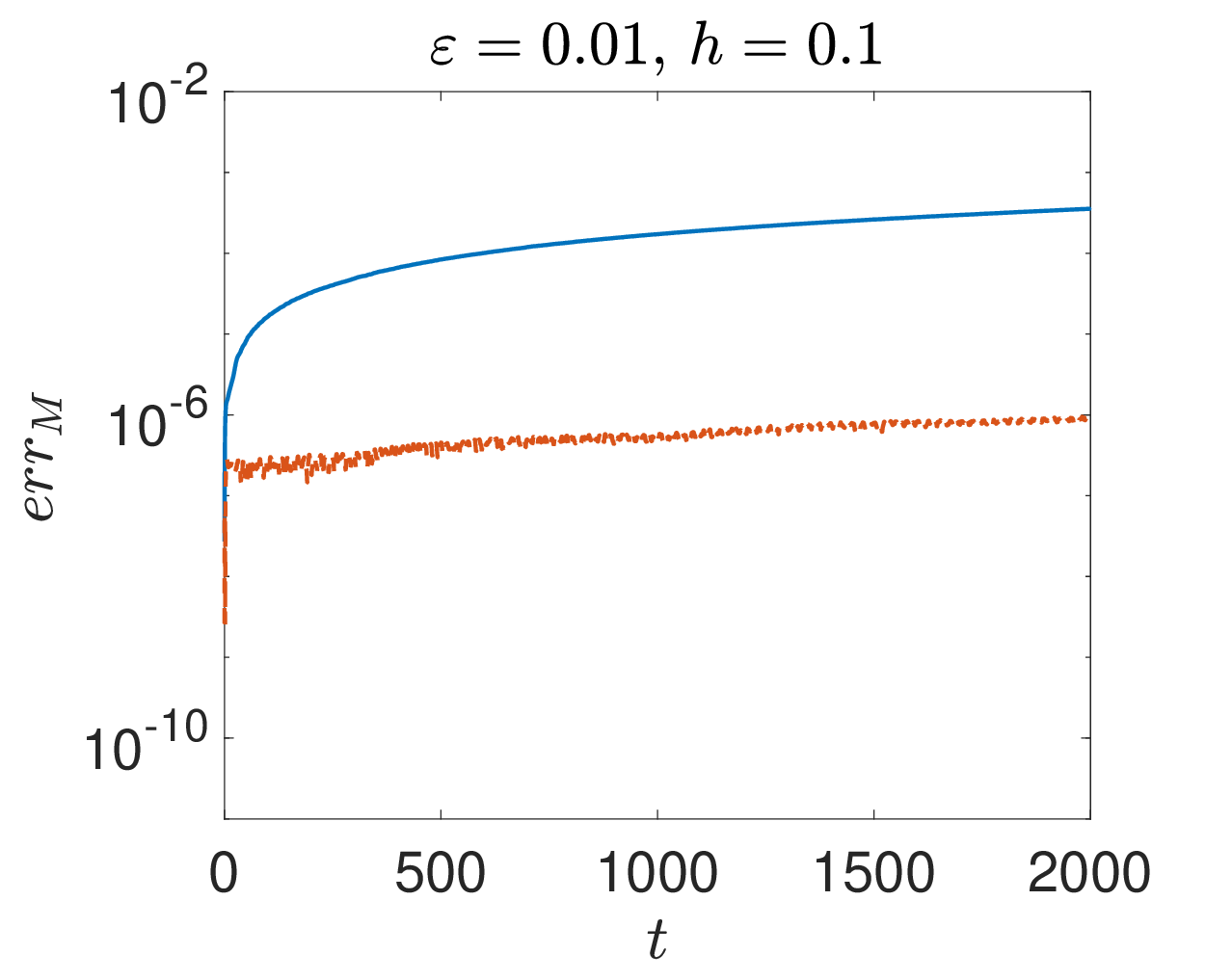}
    \caption{Energy error (top) and mass error (bottom) of KGD in 1D under different $\eps$ and $h$.}
    \label{fig-5-1-4}
\end{figure}

\subsubsection{Two-dimensional KGD equation}
We proceed to the 2D KGD equation with coefficients $g=1$ and $\omega=1$,
taking $\Omega=(-10,10)^{2}$, $\triangle x=1/16$, and $N_{\tau}=2^{4}$, together
with the honeycomb lattice potential \cite{YRS}
\begin{equation*}
V(\bx)=\cos\Big(\frac{4\pi}{\sqrt{3}}\mathbf{e}_{1}\cdot \bx\Big)+\cos\Big(\frac{4\pi}{\sqrt{3}}\mathbf{e}_{2}\cdot \bx\Big)+\cos\Big(\frac{4\pi}{\sqrt{3}}\mathbf{e}_{3}\cdot \bx\Big), \quad
A_{1}(\bx)=A_{2}(\bx)=0, \quad \bx=(x,y)^{\intercal}\in\mathbb{R}^{2}
\end{equation*}
with $\mathbf{e}_{1}=(-1,0)^{\intercal}$,
$\mathbf{e}_{2}=(1/2,\sqrt{3}/2)^{\intercal}$, and
$\mathbf{e}_{3}=(1/2,-\sqrt{3}/2)^{\intercal}$. The initial data are taken as
\[
\Psi_{0}(\bx)=\big(e^{-(x^2+y^2)/2},\,e^{-((x-1)^2+y^2)/2}\big)^{\intercal}, \quad
\phi_{0}(\bx)=e^{-(x^2+y^2)/2}, \quad
\phi_{1}(\bx)=\frac{3}{2}e^{-(x^2+y^2)/2}.
\]
Table \ref{tab-5-2-2} and Figures \ref{fig-5-2-1} and \ref{fig-5-2-3} report the
temporal errors and the long-time behavior of the 2D KGD equation \eqref{KGD},
respectively, with the reference solution supplied by the 2D TSFP method
\cite{YRS} using $h=10^{-6}$, $\triangle x=1/16$, and $\triangle\tau=1/16$. Both
fourth-order uniform accuracy and favorable long-time behavior are clearly
observed.

We next investigate the dynamics of the 2D KGD equation \eqref{KGD}, obtained by
SUA-EI4 with $h=0.01$. Figures \ref{fig-5-2-4} and \ref{fig-5-2-5} depict the
densities $\rho_{1}(t,\bx)=\abs{\phi(t,\bx)}^2$ and
$\rho_{2}(t,\bx)=\abs{\psi_{2}(t,\bx)}^2$ for $\eps=0.5$ and $\eps=0.01$ under
different interaction coefficients $g$ and $\omega$, respectively. The results
reveal that the KGD dynamics depend strongly on both the singular perturbation
parameter $\eps$ and the interaction strength. At $\eps=0.5$, the densities
fluctuate randomly, exhibiting the quantum Zitterbewegung phenomenon
\cite{Dirac}; as $\eps$ decreases, the relativistic effects vanish and the
densities spread ever more smoothly.

\begin{figure}[t!]
    \centering
        \includegraphics[width=0.38\linewidth,height=3.6cm]{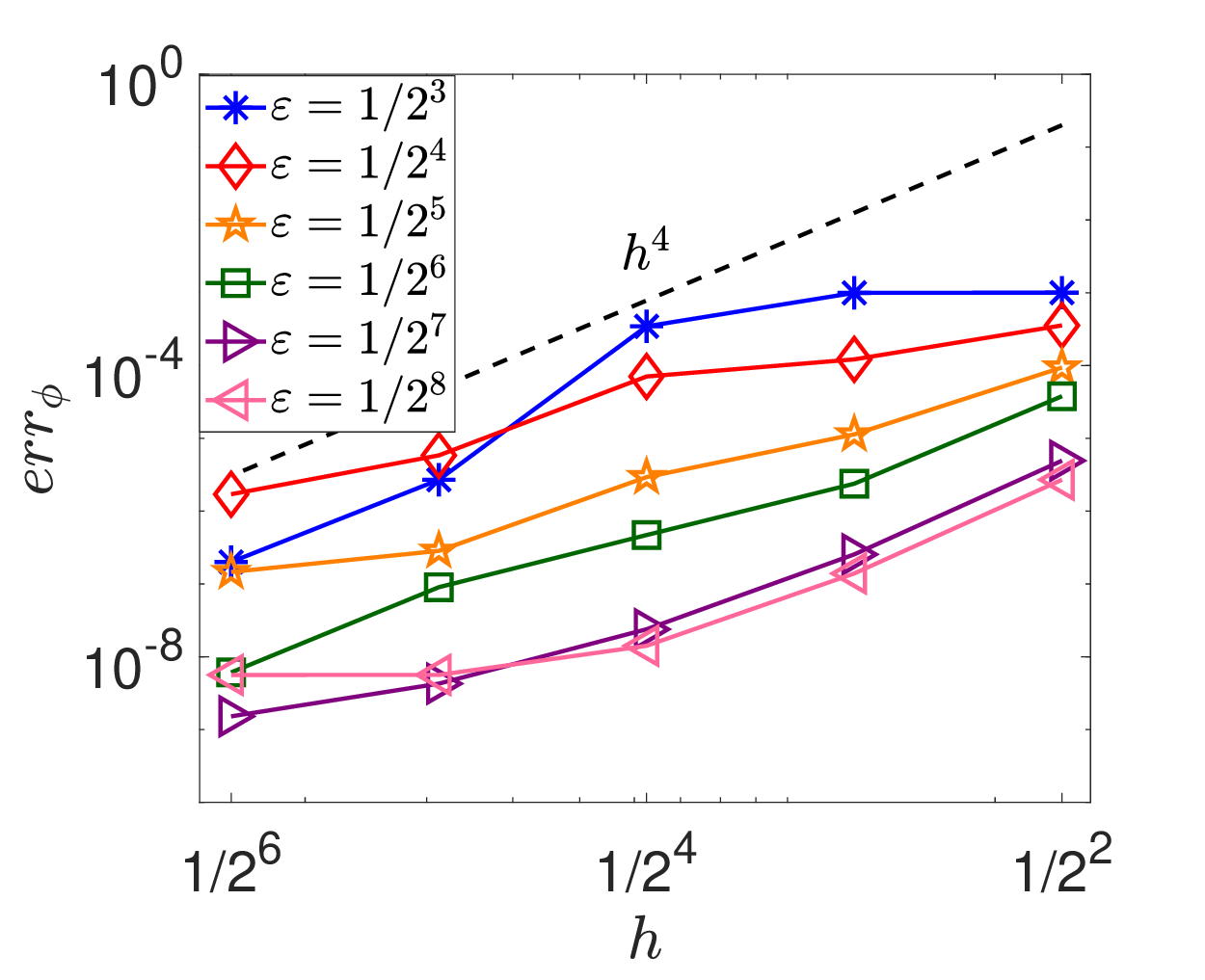}
        \includegraphics[width=0.38\linewidth,height=3.6cm]{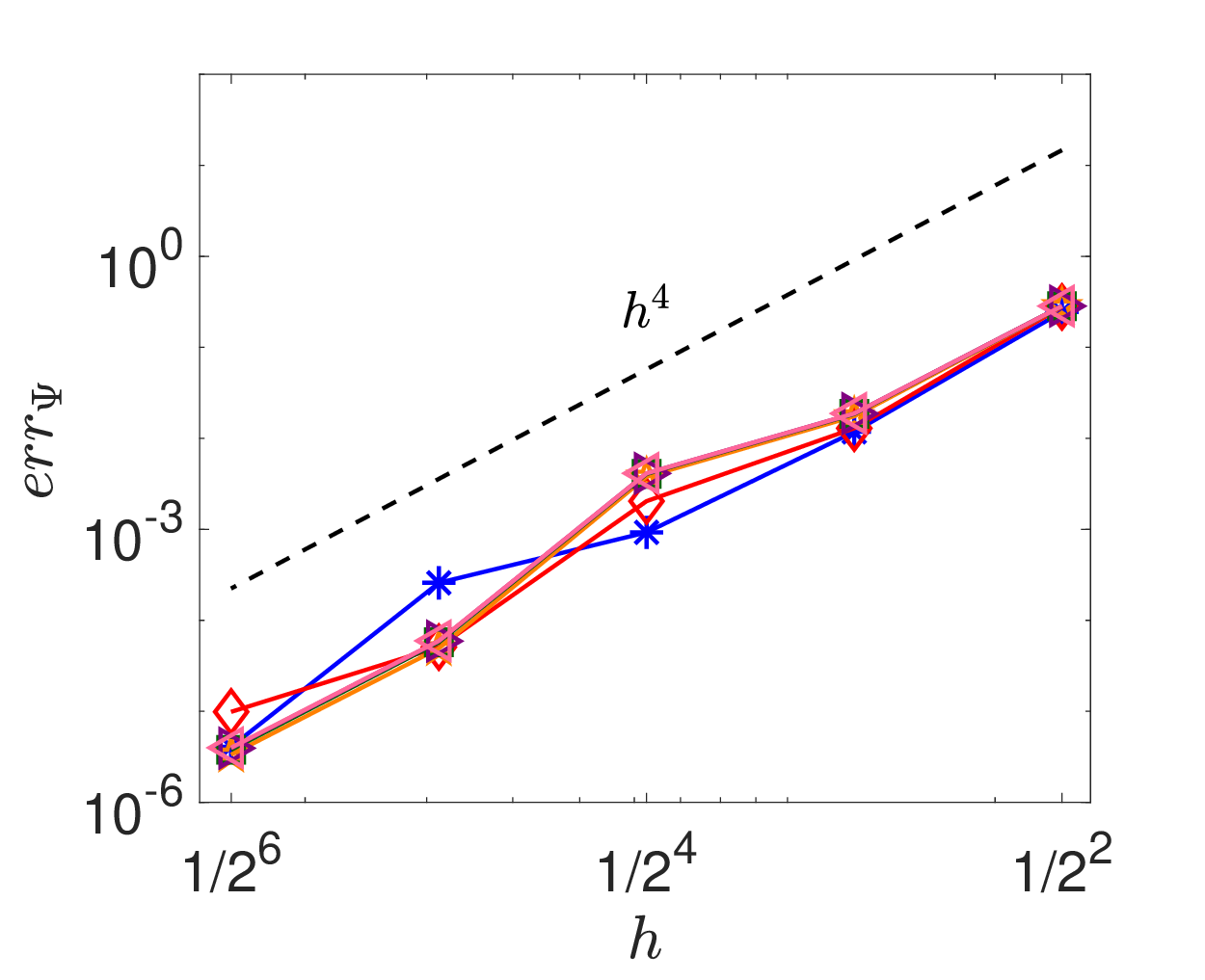}
    \caption{Temporal error of KGD in 2D at $T=1$ under different $\eps$ for SUA-EI4 scheme.}
    \label{fig-5-2-1}
\end{figure}

\begin{table}[t!]
    \centering
    \caption{Temporal error of KGD in 2D at $T=1$ under different $\eps$ for the EUA-EI4 scheme.}
    \label{tab-5-2-2}

    \renewcommand{\arraystretch}{1.15}
    \setlength{\tabcolsep}{5pt}

    \resizebox{\textwidth}{!}{%
    \begin{tabular}{c ccccc ccccc}
        \toprule
        &
        \multicolumn{5}{c}{\ $err_{\Psi}$\ } &
        \multicolumn{5}{c}{\ $err_{\phi}$\ } \\
        \cmidrule(lr){2-6}
        \cmidrule(lr){7-11}
        
        $\eps$ & $h_{0}=1/2^2$ & $h_{0}/2$ & $h_{0}/4$ & $h_{0}/8$ & $h_{0}/16$ &  $h_{0}$ & $h_{0}/2$ & $h_{0}/4$ & $h_{0}/8$ & $h_{0}/16$ \\
        \midrule
        $1/2^3$ & 2.26E-3 & 1.17E-3 & 3.04E-4 & 4.55E-6 & 2.13E-7 &  2.60E-1 & 1.23E-2 & 1.03E-3 & 2.74E-4 & 4.12E-6 \\
        order & -- & 0.94  & 1.95 & 6.06 & 4.41 & -- & 4.40  & 3.58 &  1.91  & 6.05    \\
        $1/2^4$ & 9.59E-4 & 2.50E-4 & 9.19E-5 & 7.50E-6 & 1.88E-6 & 2.95E-1 & 1.56E-2 & 1.94E-3 & 6.37E-5 & 1.04E-5    \\
        order & -- & 1.94 & 1.44 & 3.62 & 2.00 &  -- & 4.23 &  3.01  & 4.93 &  2.62         \\
        $1/2^5$ &  1.85E-4 & 2.53E-4 & 9.80E-6 & 1.12E-6 & 1.80E-7 & 3.01E-1 & 2.06E-2 & 3.65E-3 & 6.73E-5 & 4.42E-6  \\
        order & -- & -0.45 & 4.69 & 3.13 & 2.64 & -- &  3.86 & 2.50 & 5.76  &  3.93  \\
        $1/2^6$ & 6.38E-5 & 1.24E-4 & 2.87E-5 & 5.94E-7 & 2.21E-8 &  3.02E-1 & 2.18E-2 & 3.94E-3 & 7.27E-5 & 4.84E-6    \\
        order & -- & -0.96 & 2.11 & 5.59 & 4.75 & -- & 3.79 & 2.47 & 5.76 & 3.91  \\
        $1/2^7$ & 9.87E-6 & 2.92E-6 & 4.65E-7 & 1.12E-7 & 2.20E-8 & 3.02E-1 & 2.20E-2 & 3.99E-3 & 7.43E-5 & 5.00E-6  \\
        order & -- & 1.76 & 2.65 & 2.06 & 2.35 & -- & 3.78  & 2.46 & 5.75  &  3.89  \\
        $1/2^8$ & 5.28E-6 & 1.65E-6 & 2.89E-7 & 4.10E-8 & 8.23E-9 & 3.03E-1 & 2.20E-2 & 4.00E-3 & 7.47E-5 & 5.06E-6 \\
        order &  -- & 1.68 & 2.51 & 2.82 & 2.32 &  -- & 3.78  & 2.46 & 5.75  & 3.88 \\
       \bottomrule
    \end{tabular}%
    }
\end{table}

\begin{figure}[t!]
    \centering
        \includegraphics[width=0.32\linewidth,height=3.1cm]{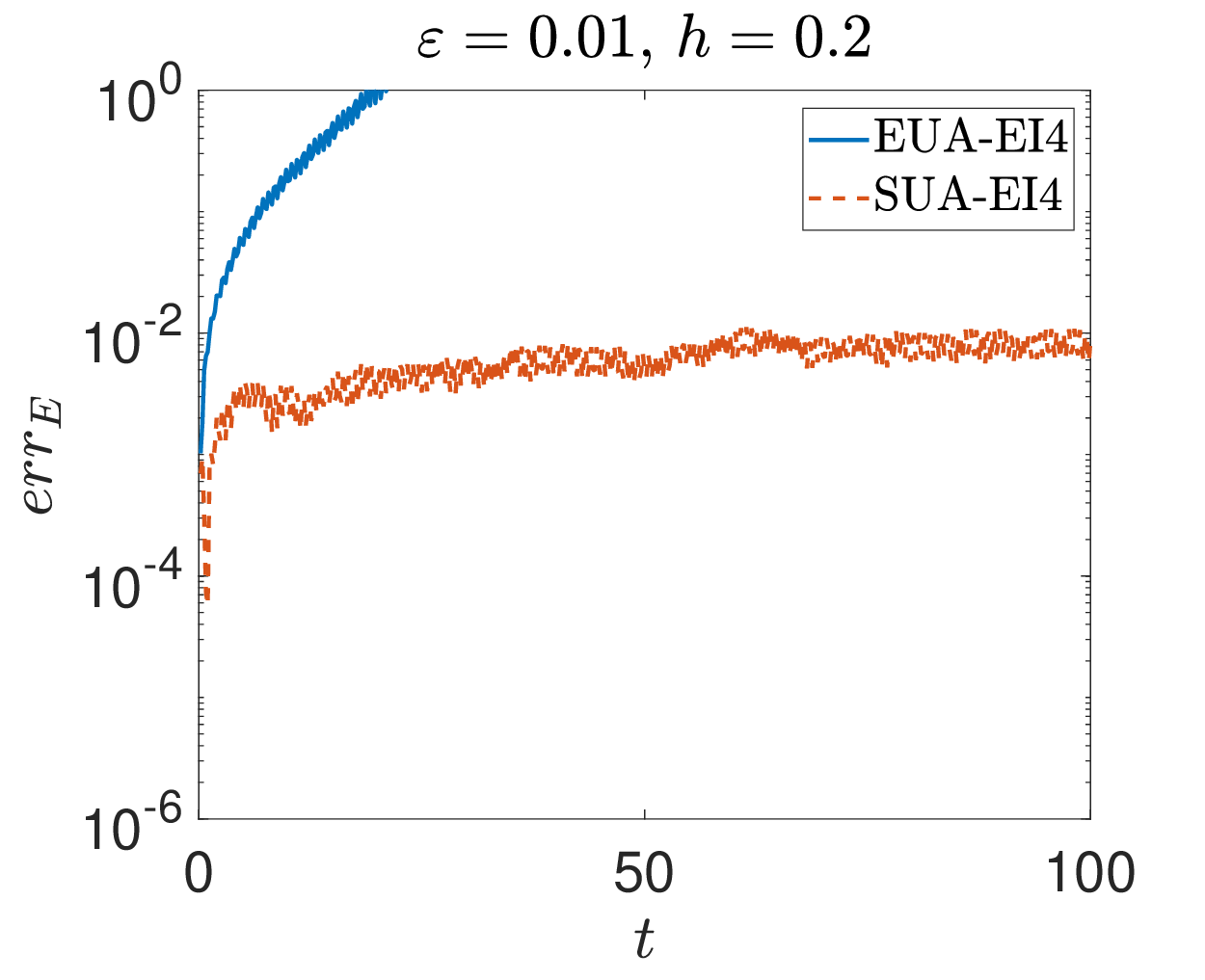}
        \includegraphics[width=0.32\linewidth,height=3.1cm]{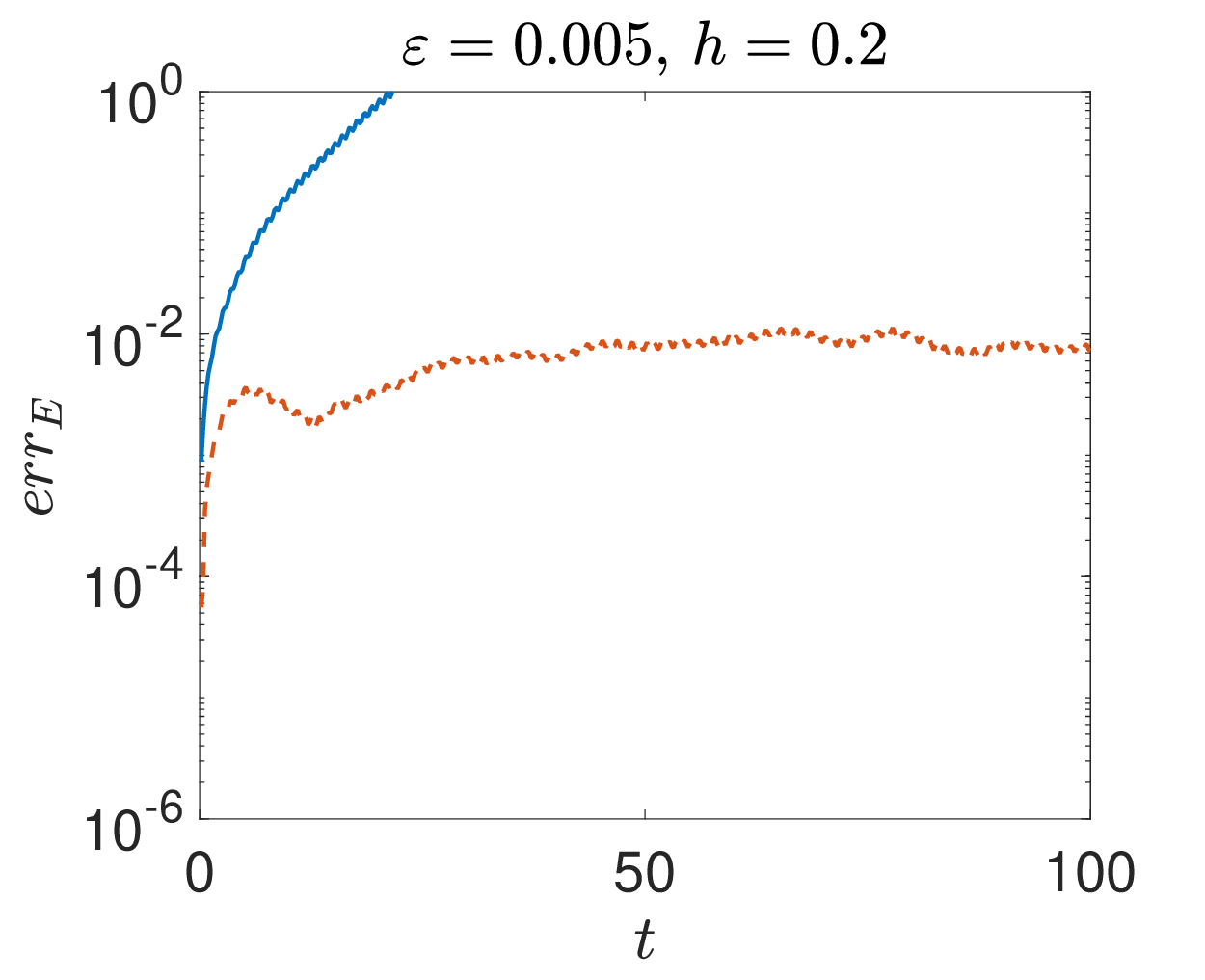}
        \includegraphics[width=0.32\linewidth,height=3.1cm]{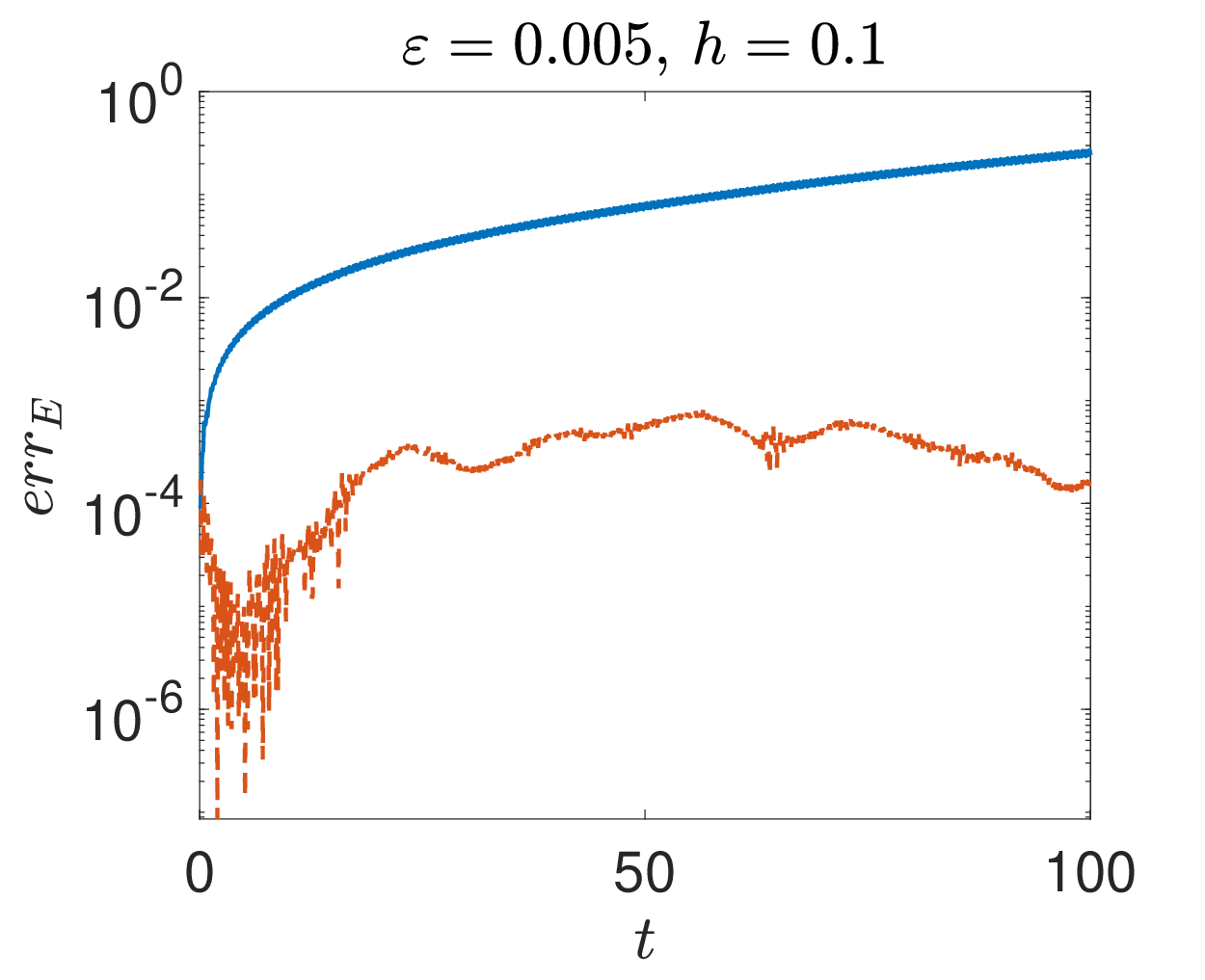}\\
        \includegraphics[width=0.32\linewidth,height=3.1cm]{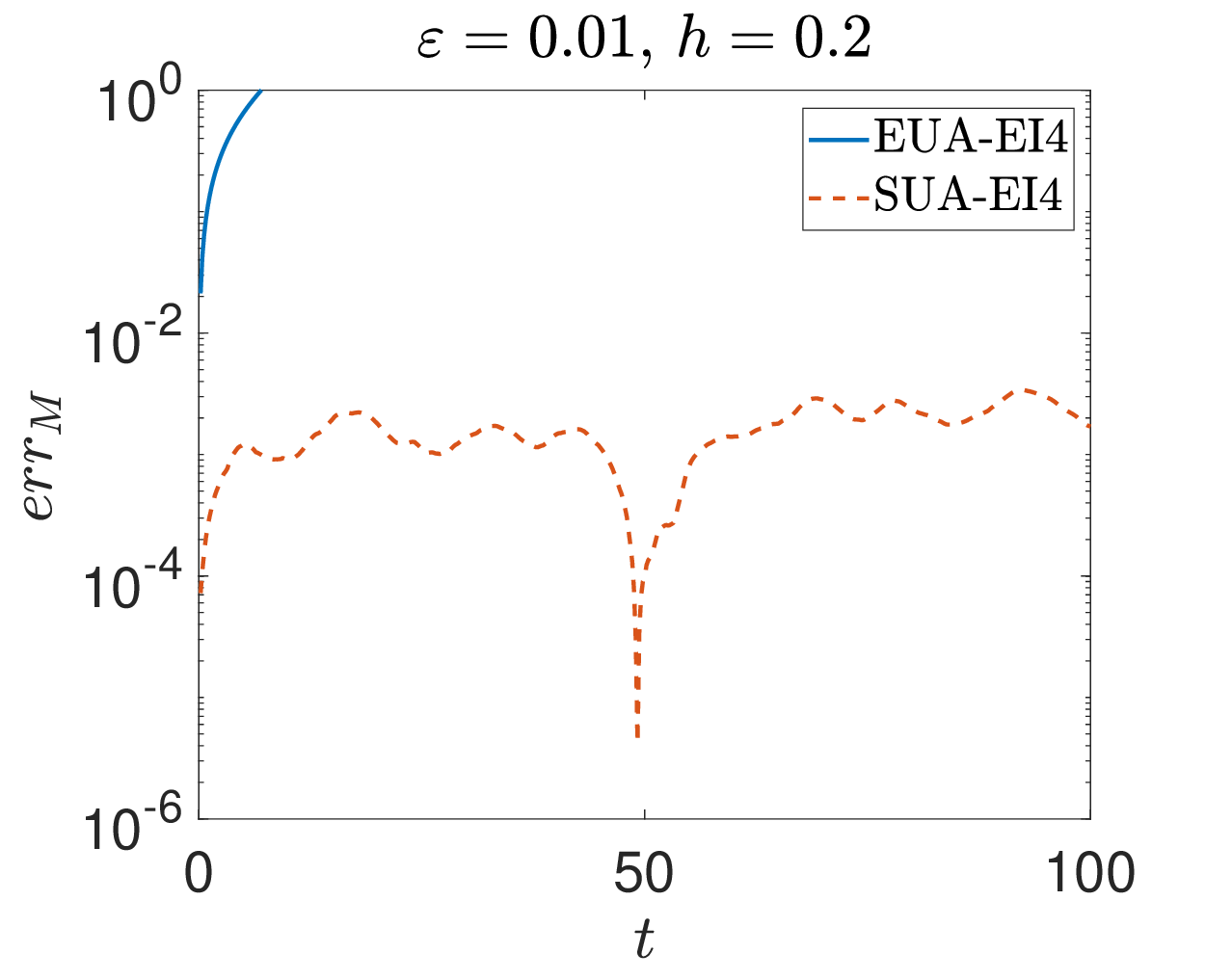}
        \includegraphics[width=0.32\linewidth,height=3.1cm]{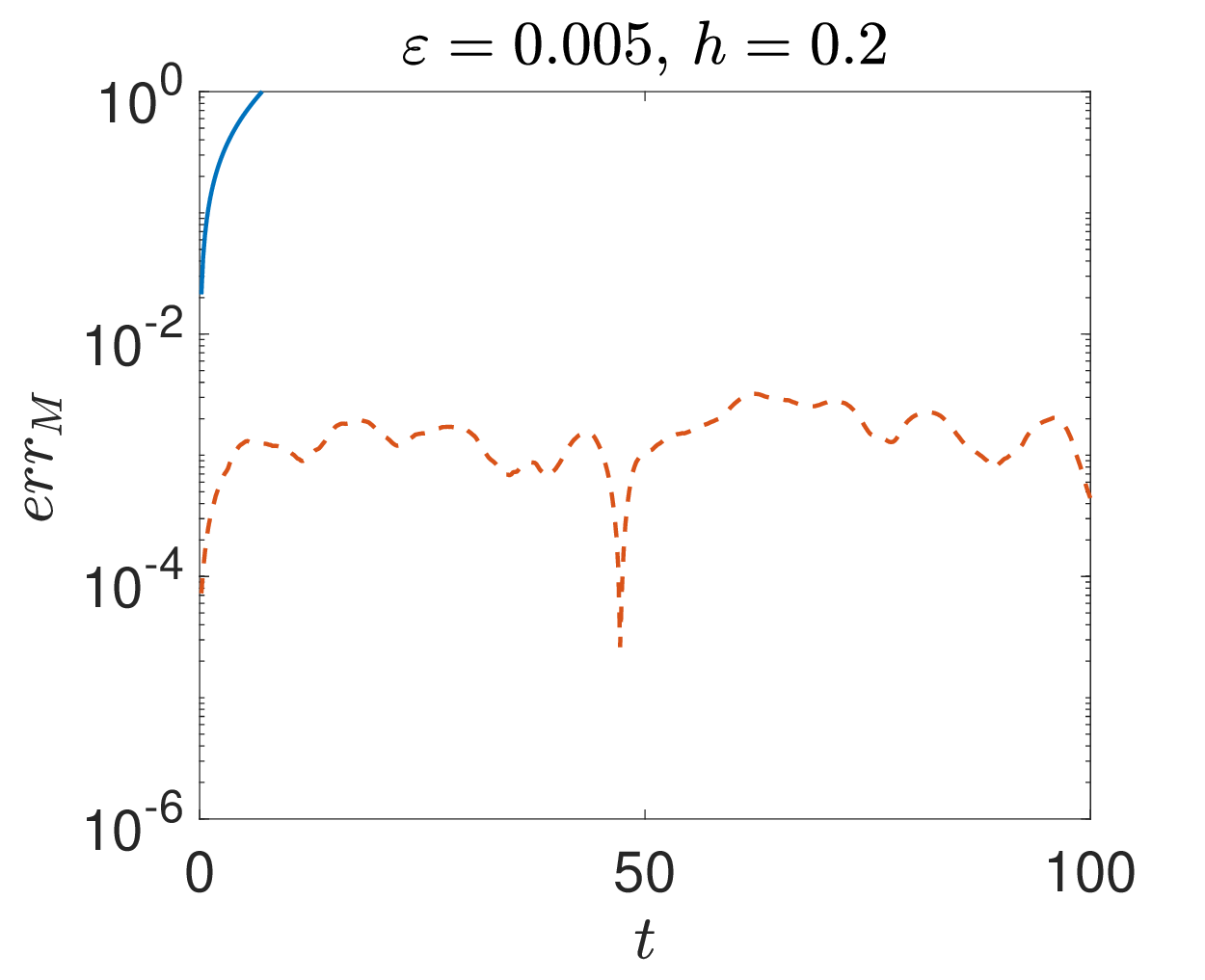}
        \includegraphics[width=0.32\linewidth,height=3.1cm]{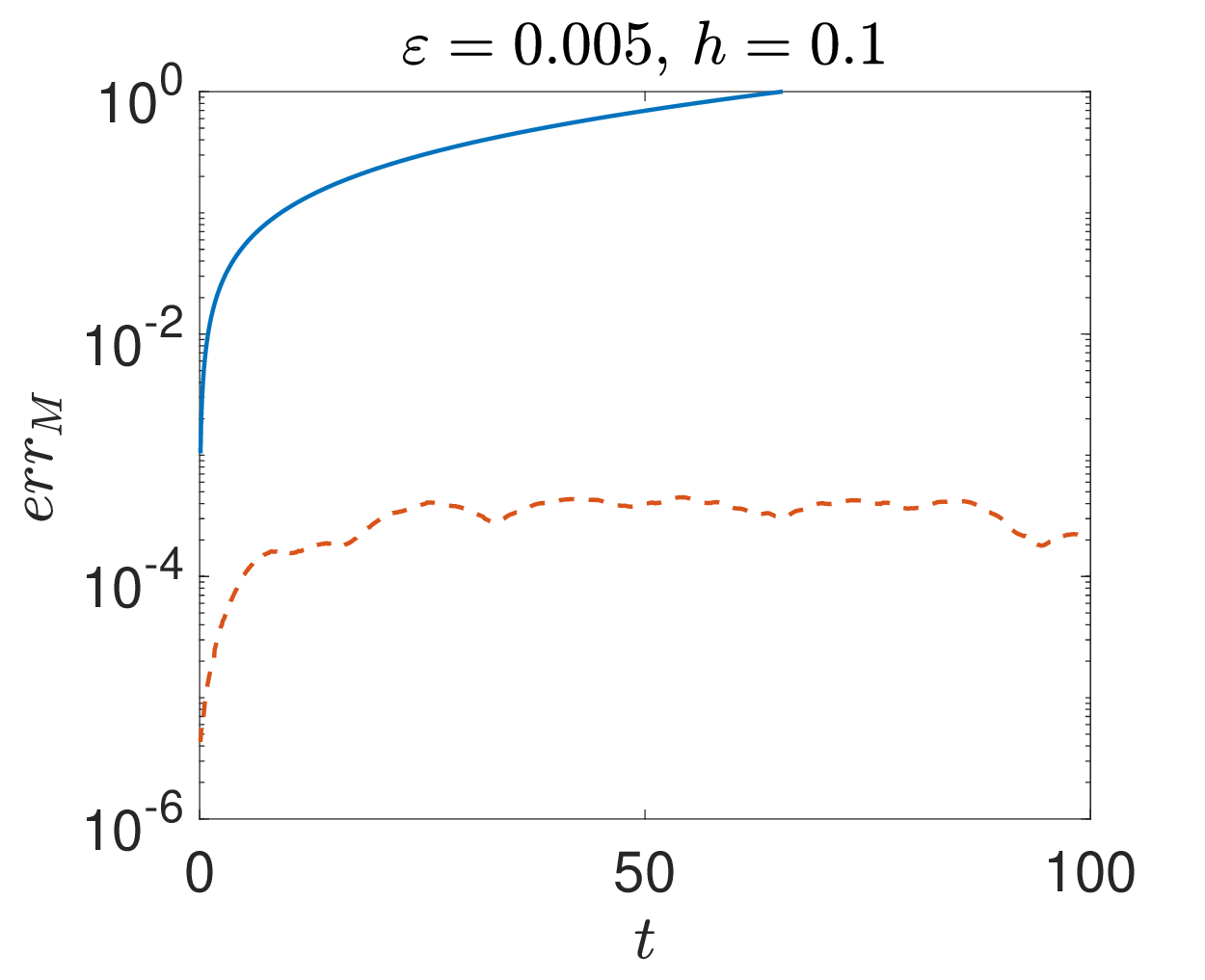}
    \caption{Energy error (top) and mass error (bottom) of KGD in 2D under different $\eps$ and $h$.}
    \label{fig-5-2-3}
\end{figure}

\begin{figure}[t!]
    \centering
        \includegraphics[width=0.22\linewidth,height=3.1cm]{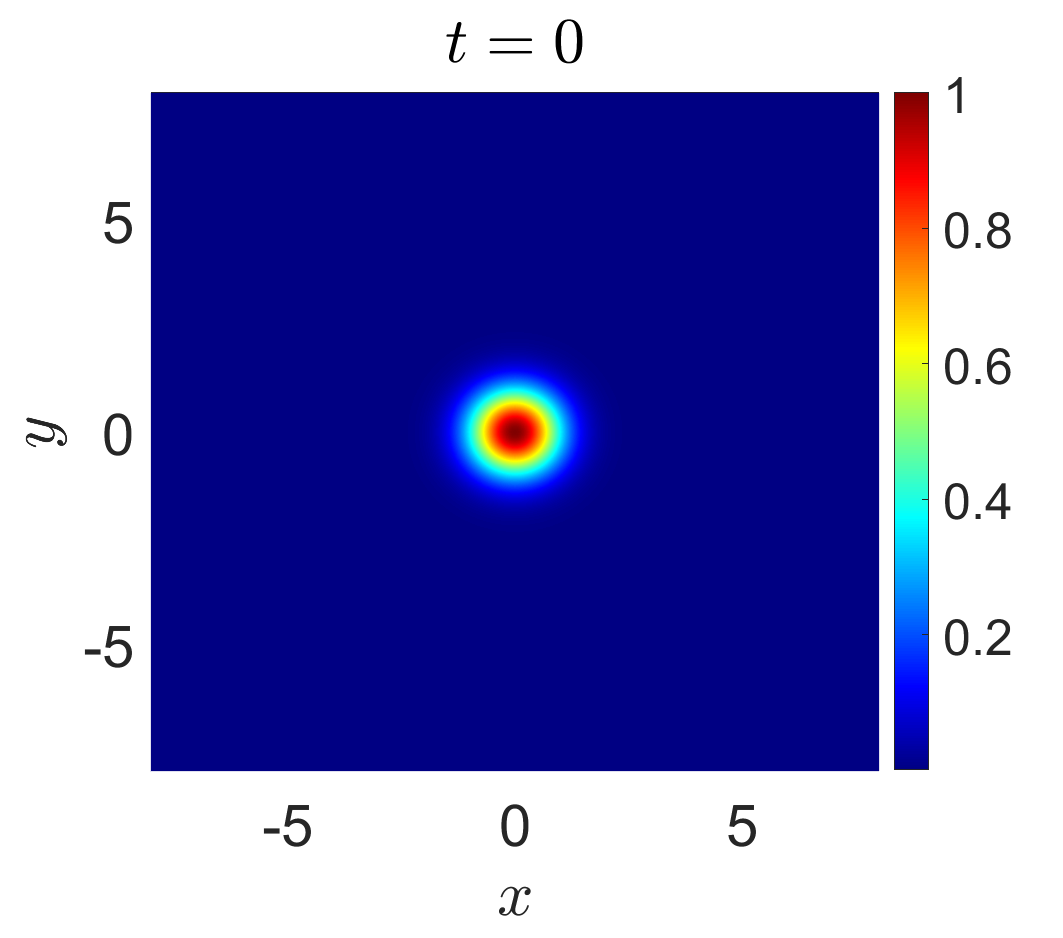}
        \includegraphics[width=0.22\linewidth,height=3.1cm]{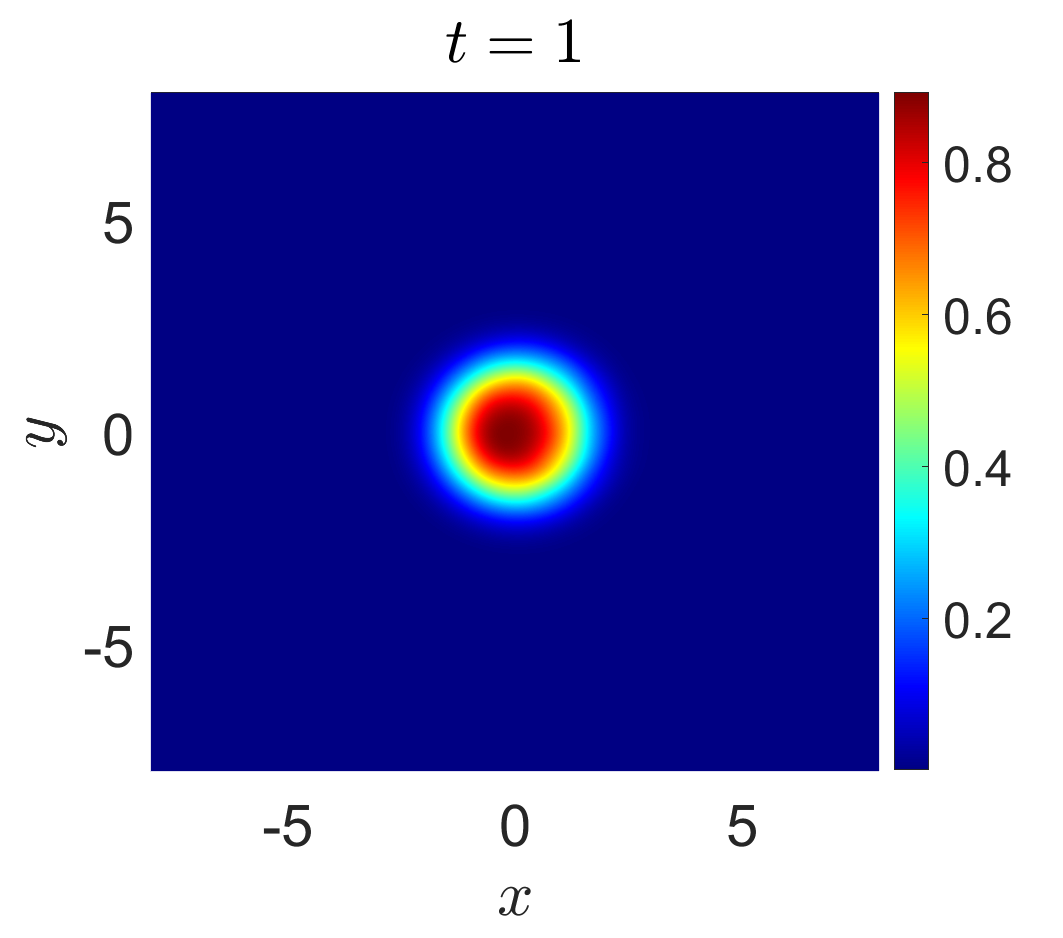}
        \includegraphics[width=0.22\linewidth,height=3.1cm]{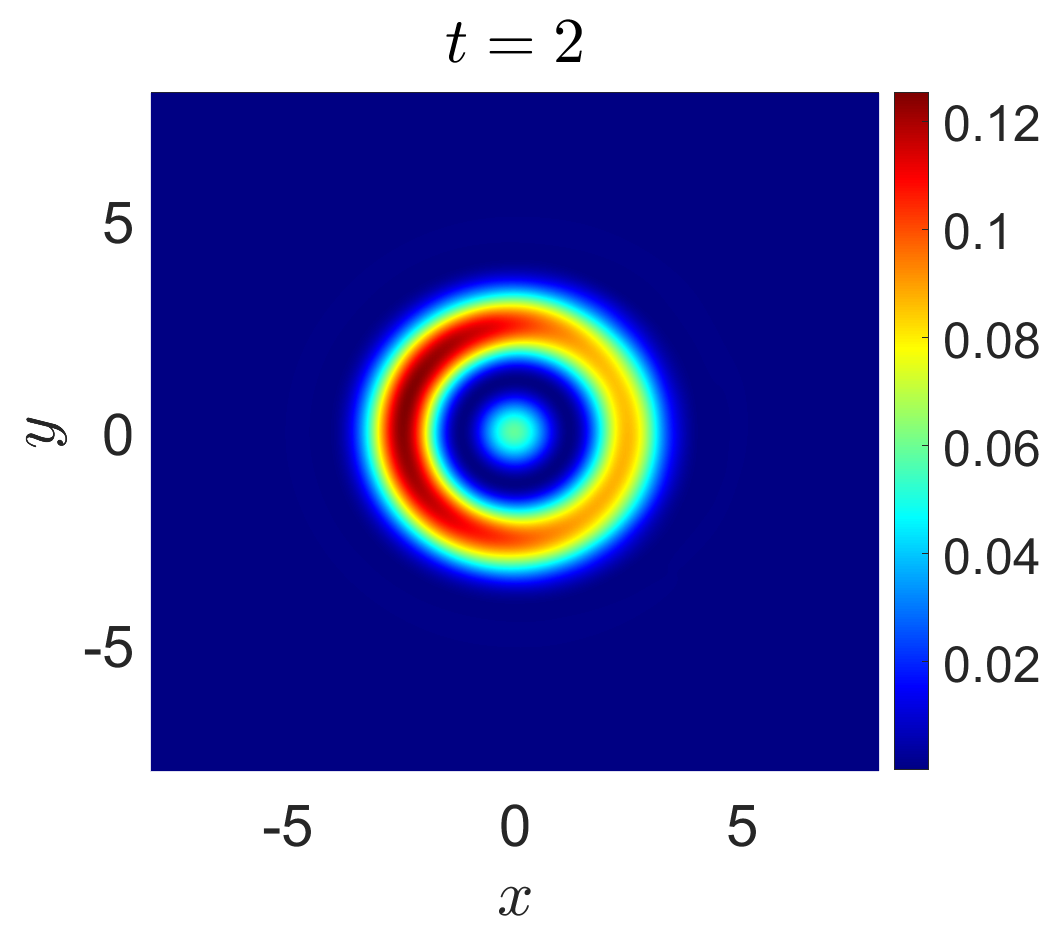}
        \includegraphics[width=0.22\linewidth,height=3.1cm]{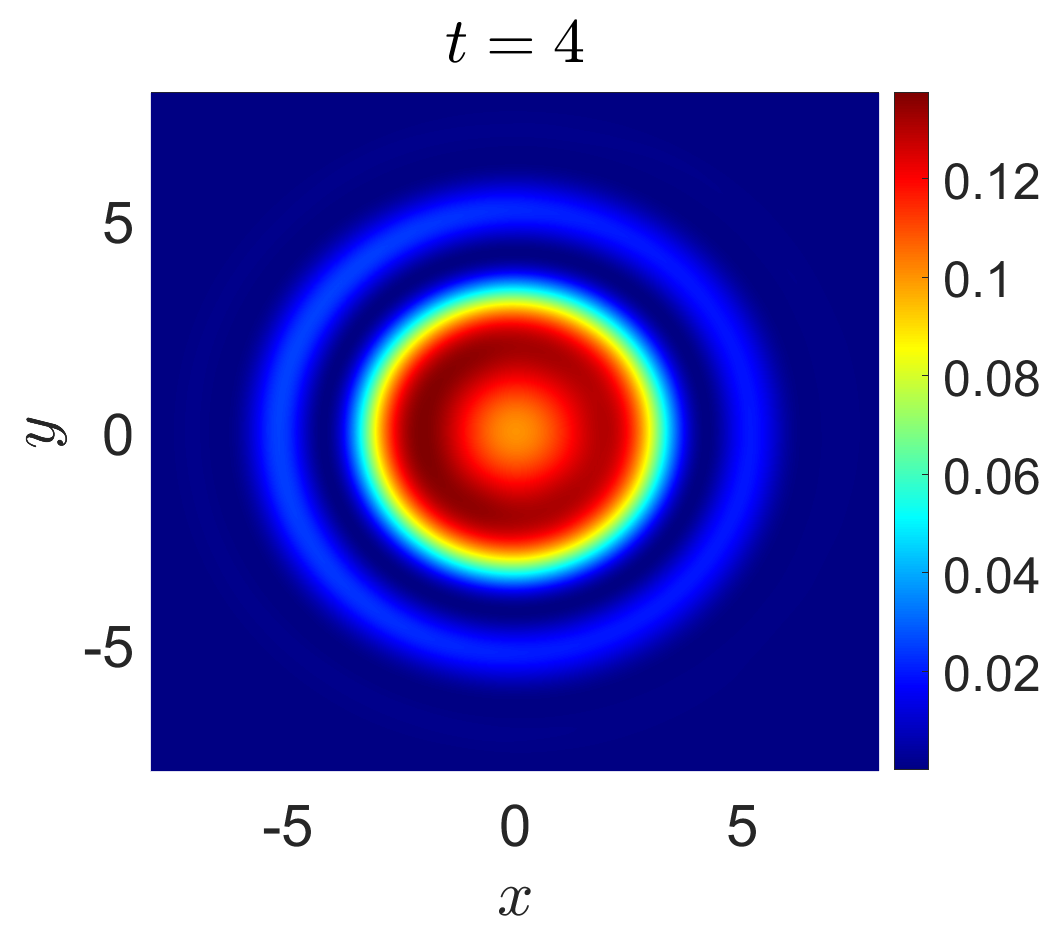}\\
        \includegraphics[width=0.22\linewidth,height=3.1cm]{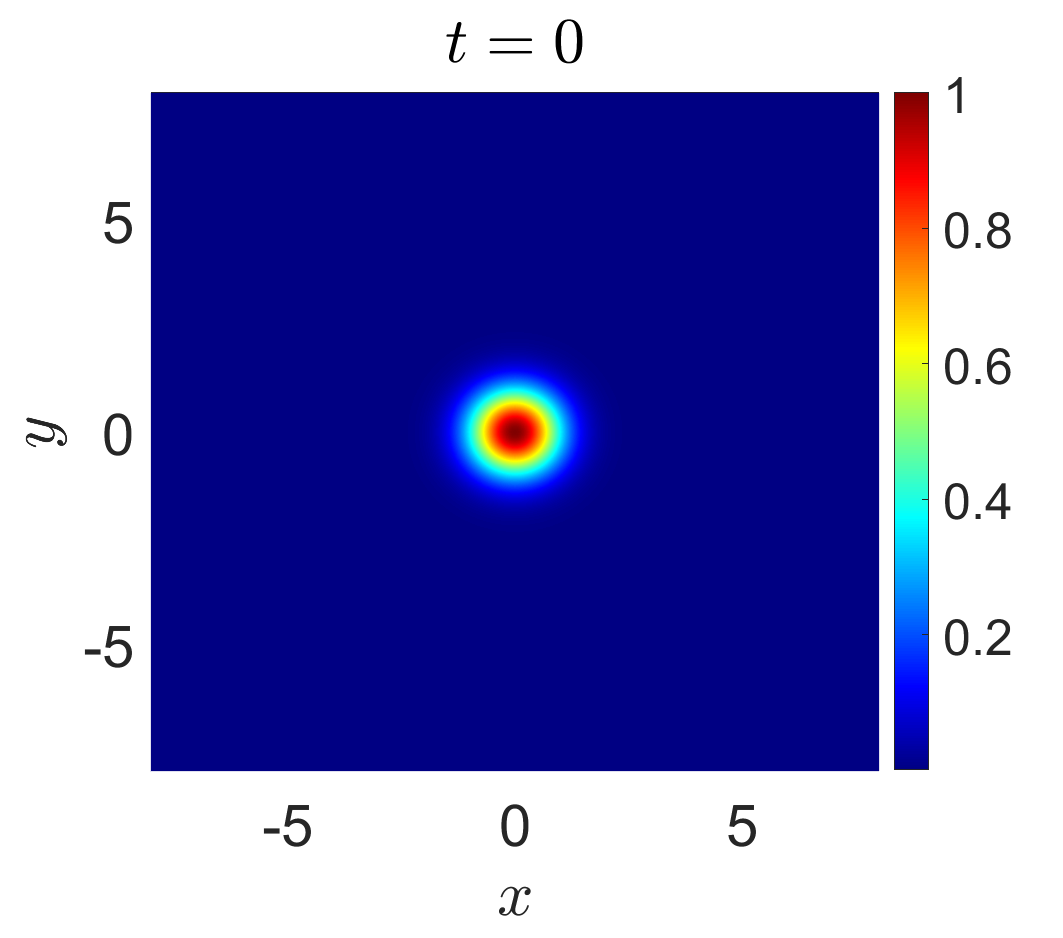}
        \includegraphics[width=0.22\linewidth,height=3.1cm]{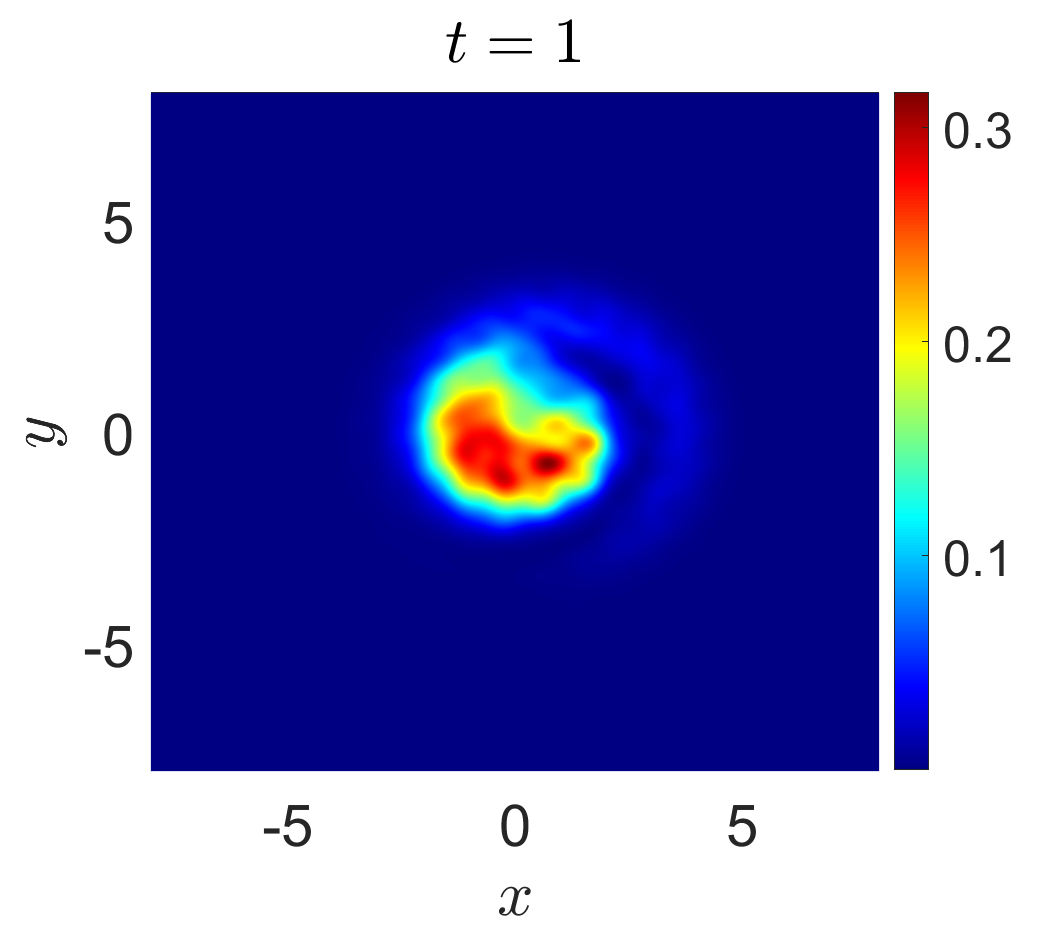}
        \includegraphics[width=0.22\linewidth,height=3.1cm]{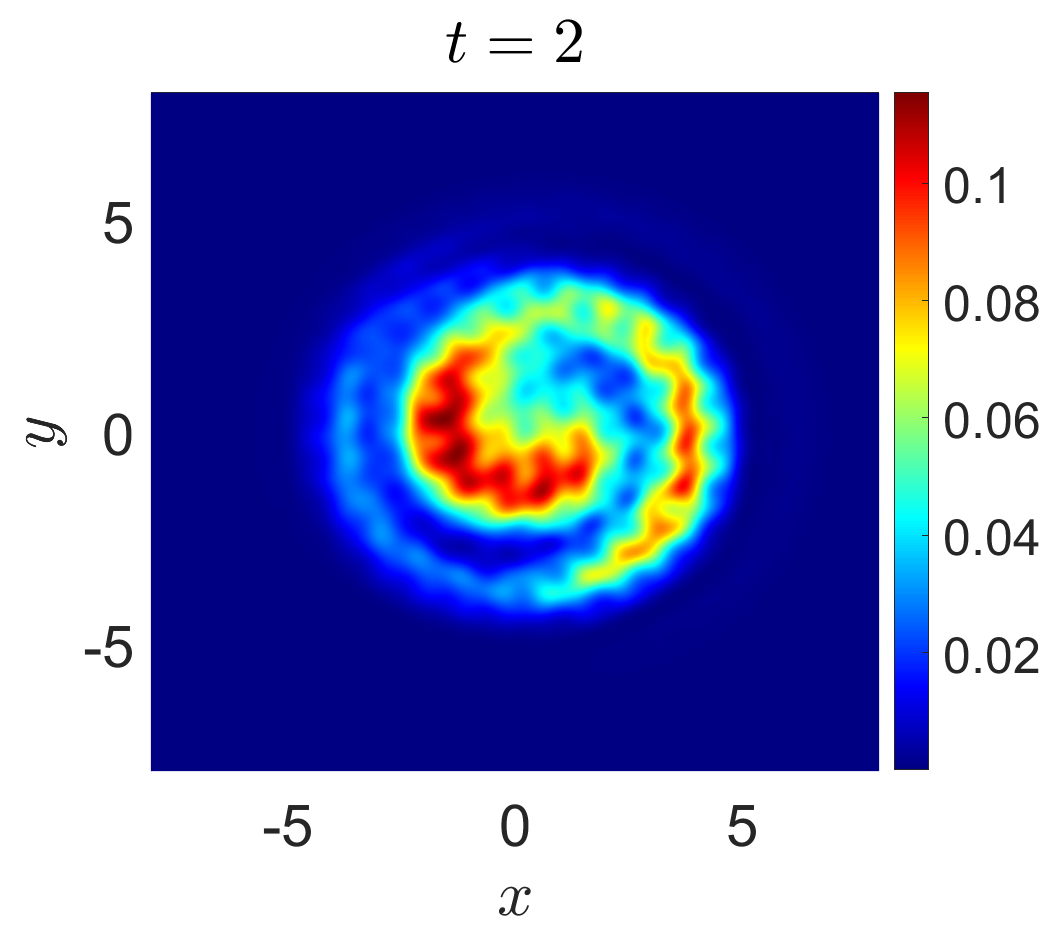}
        \includegraphics[width=0.22\linewidth,height=3.1cm]{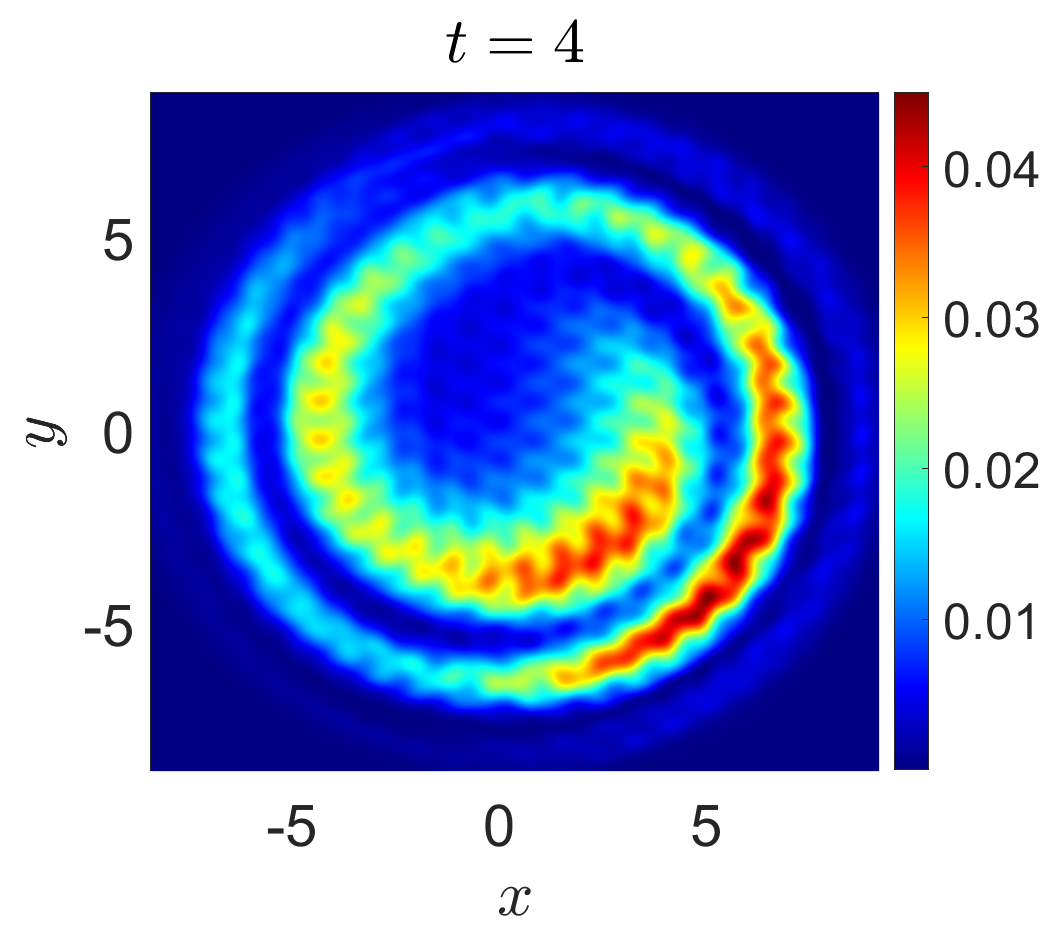}\\
    \caption{Dynamics of the densities $\rho_{1}(t,\bx)=\abs{\phi(t,\bx)}^{2}$ (first row), $\rho_{2}(t,\bx)=\abs{\psi_{1}(t,\bx)}^{2}$ (second row) of the KGD equation in 2D with a honeycomb potential with $\eps=0.5$, $g=1$ and $\omega=\frac{1}{2}$.}
    \label{fig-5-2-4}
\end{figure}

\begin{figure}[t!]
    \centering
        \includegraphics[width=0.22\linewidth,height=3.1cm]{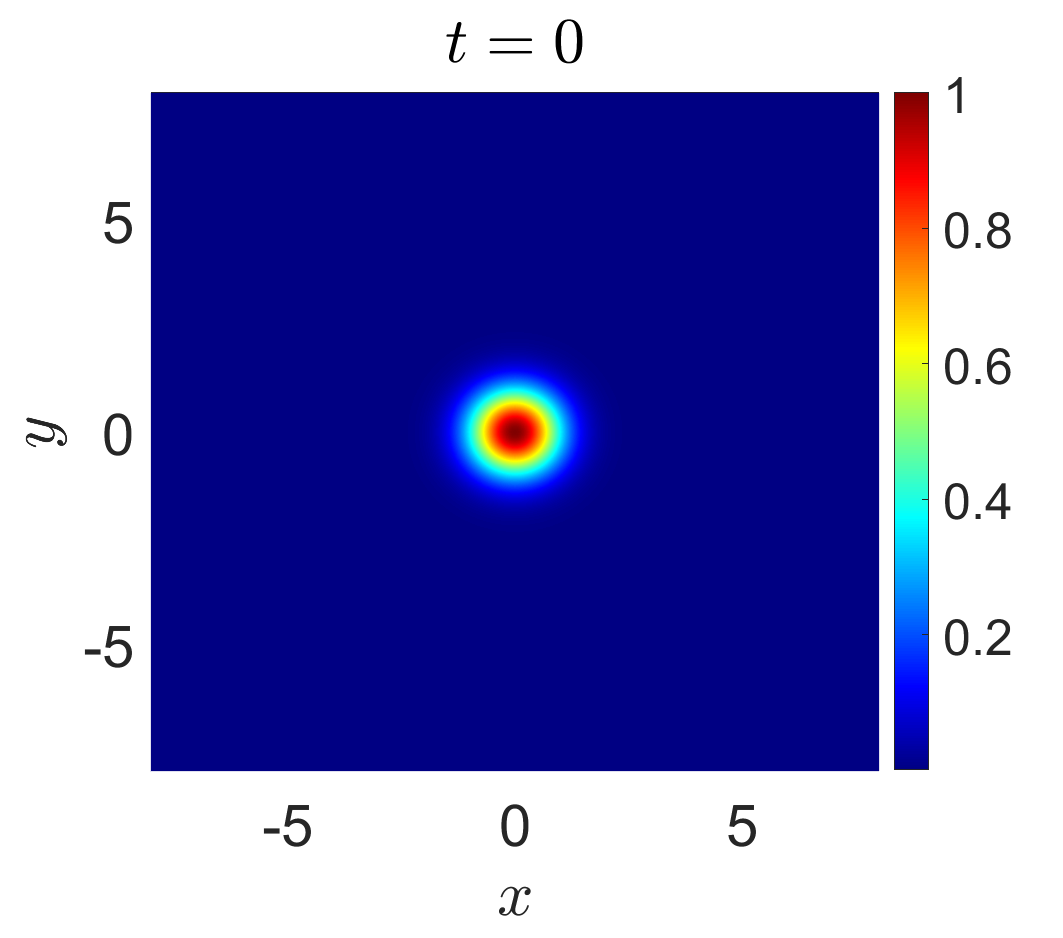}
        \includegraphics[width=0.22\linewidth,height=3.1cm]{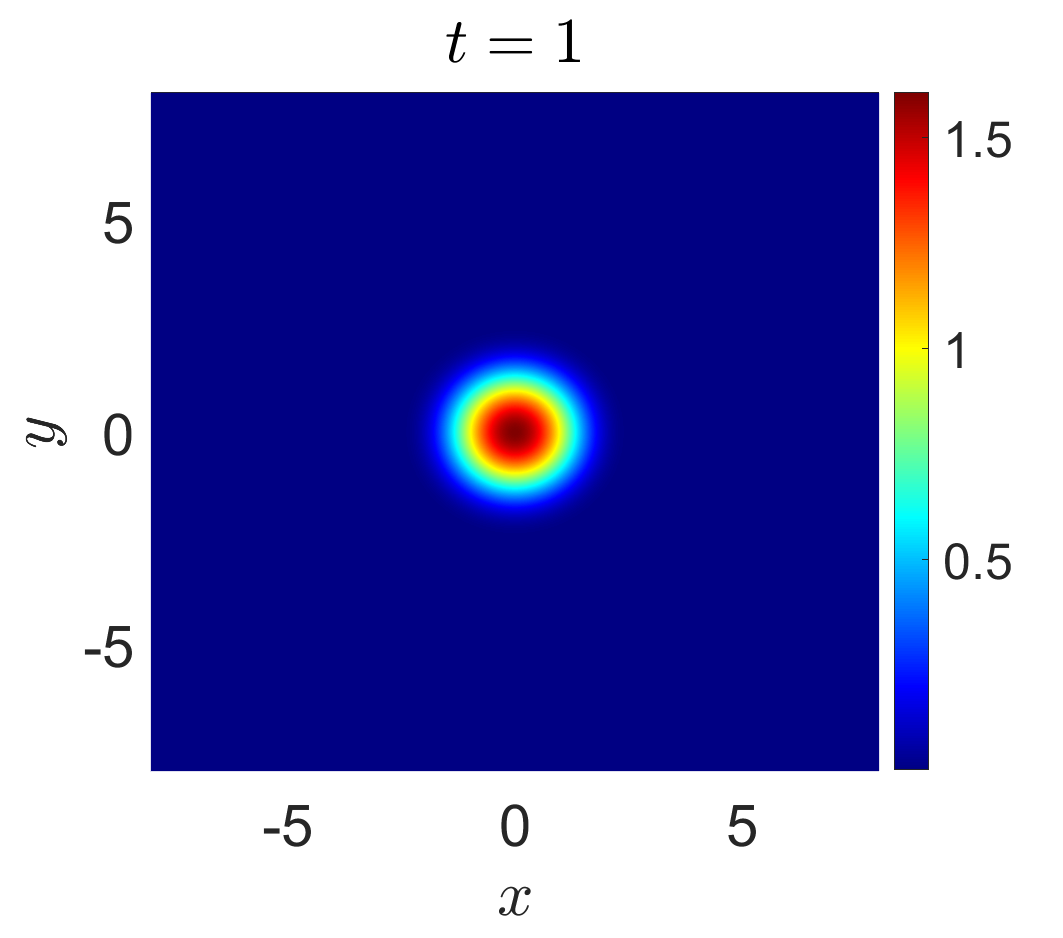}
        \includegraphics[width=0.22\linewidth,height=3.1cm]{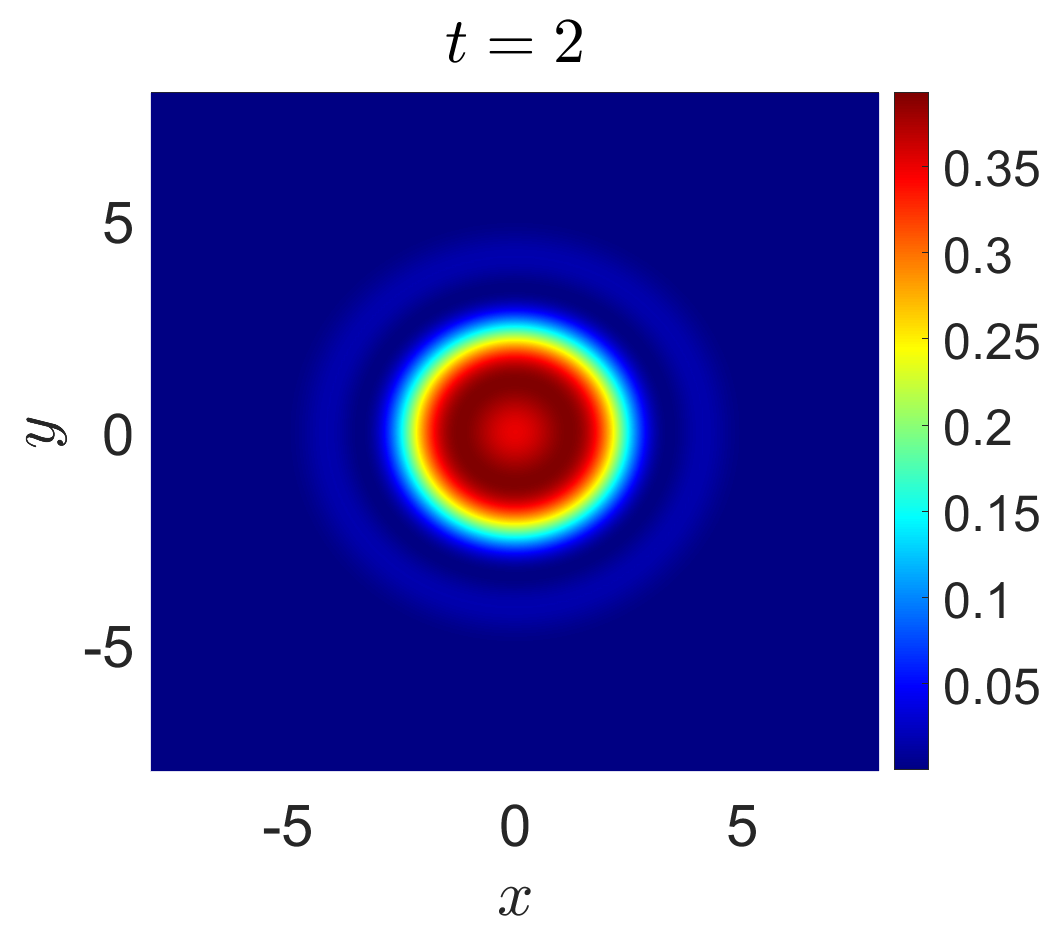}
        \includegraphics[width=0.22\linewidth,height=3.1cm]{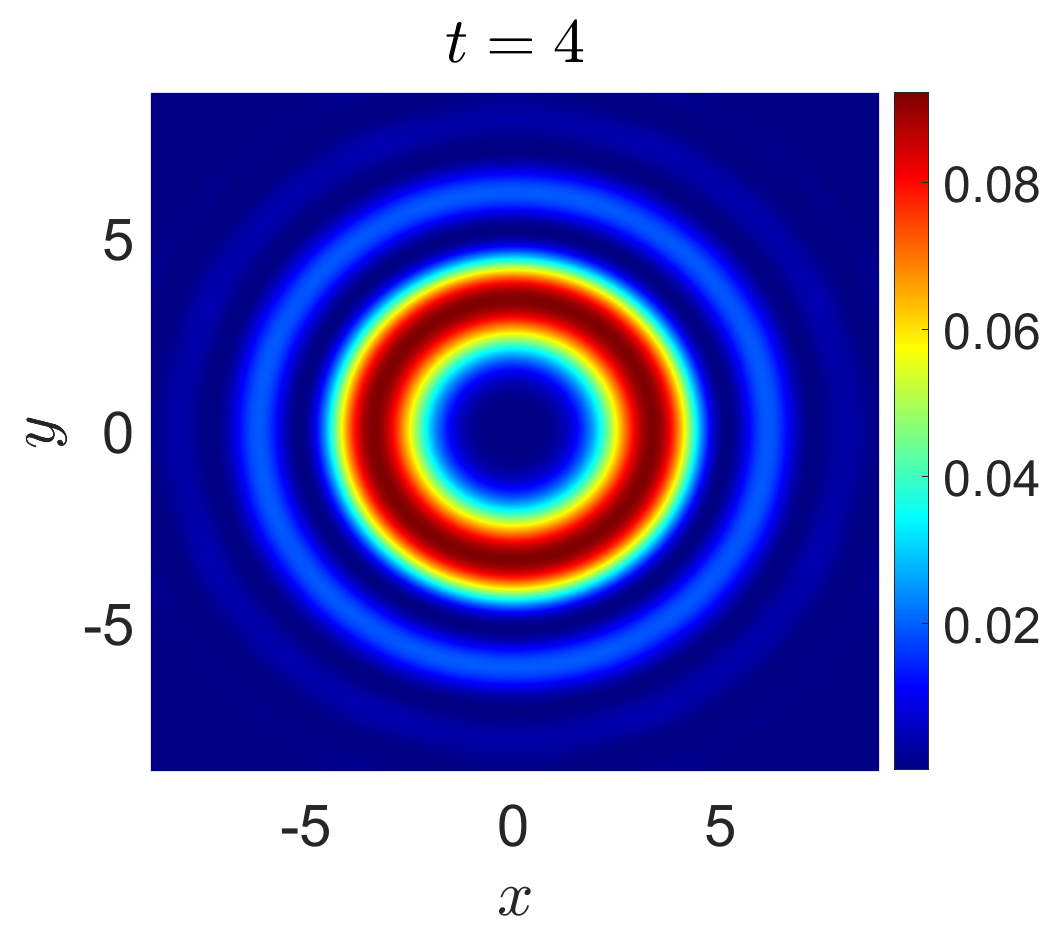}\\
        \includegraphics[width=0.22\linewidth,height=3.1cm]{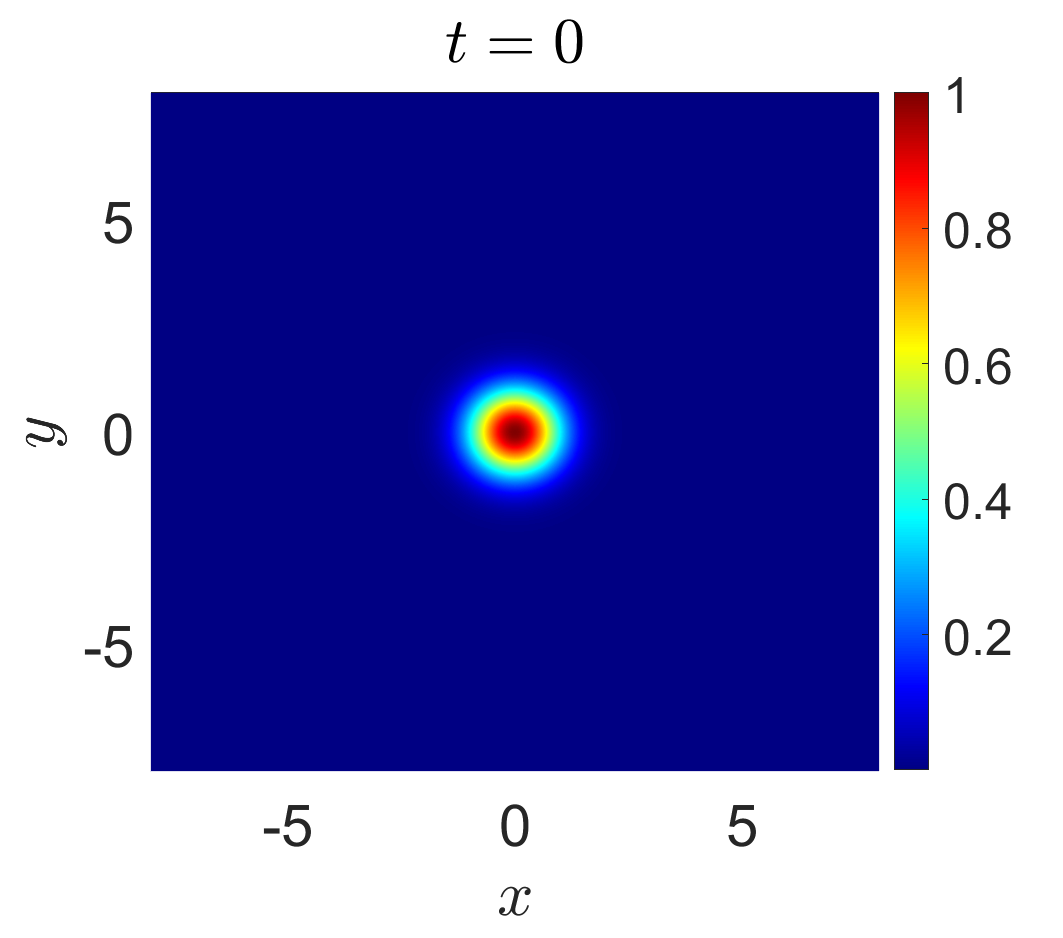}
        \includegraphics[width=0.22\linewidth,height=3.1cm]{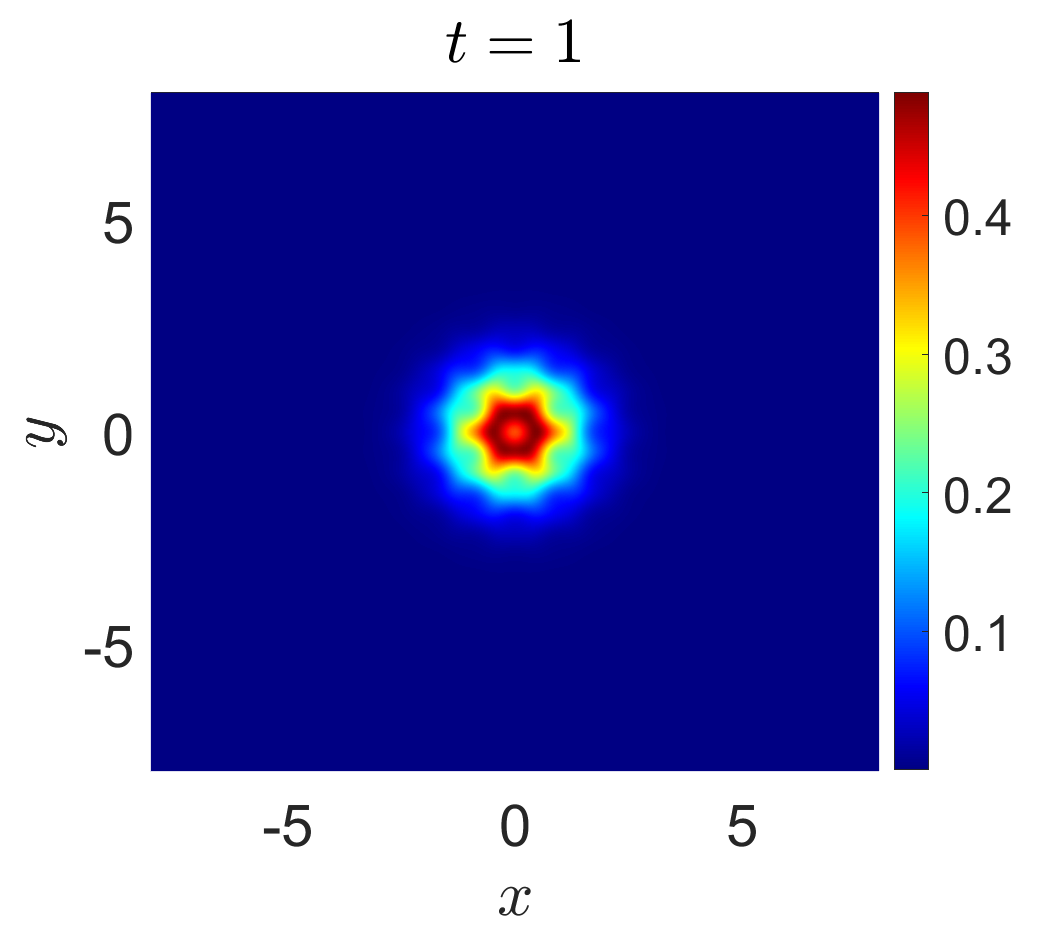}
        \includegraphics[width=0.22\linewidth,height=3.1cm]{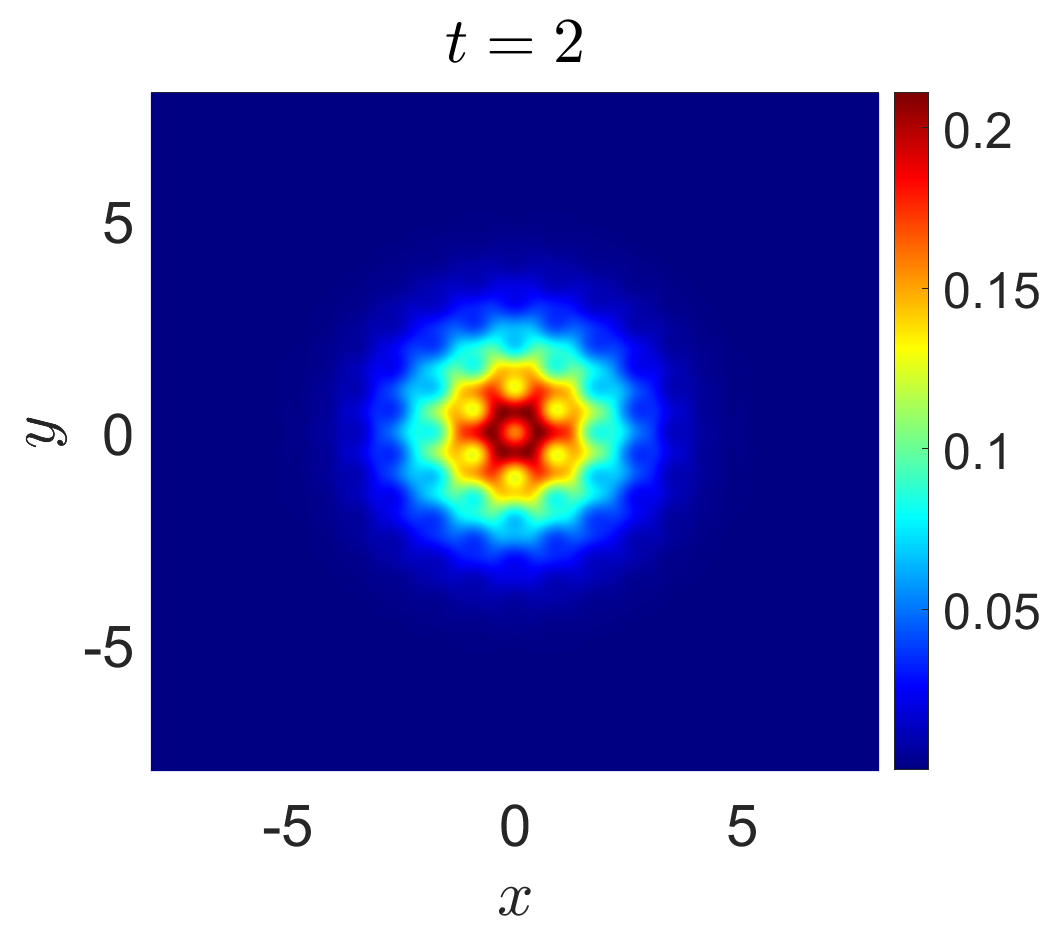}
        \includegraphics[width=0.22\linewidth,height=3.1cm]{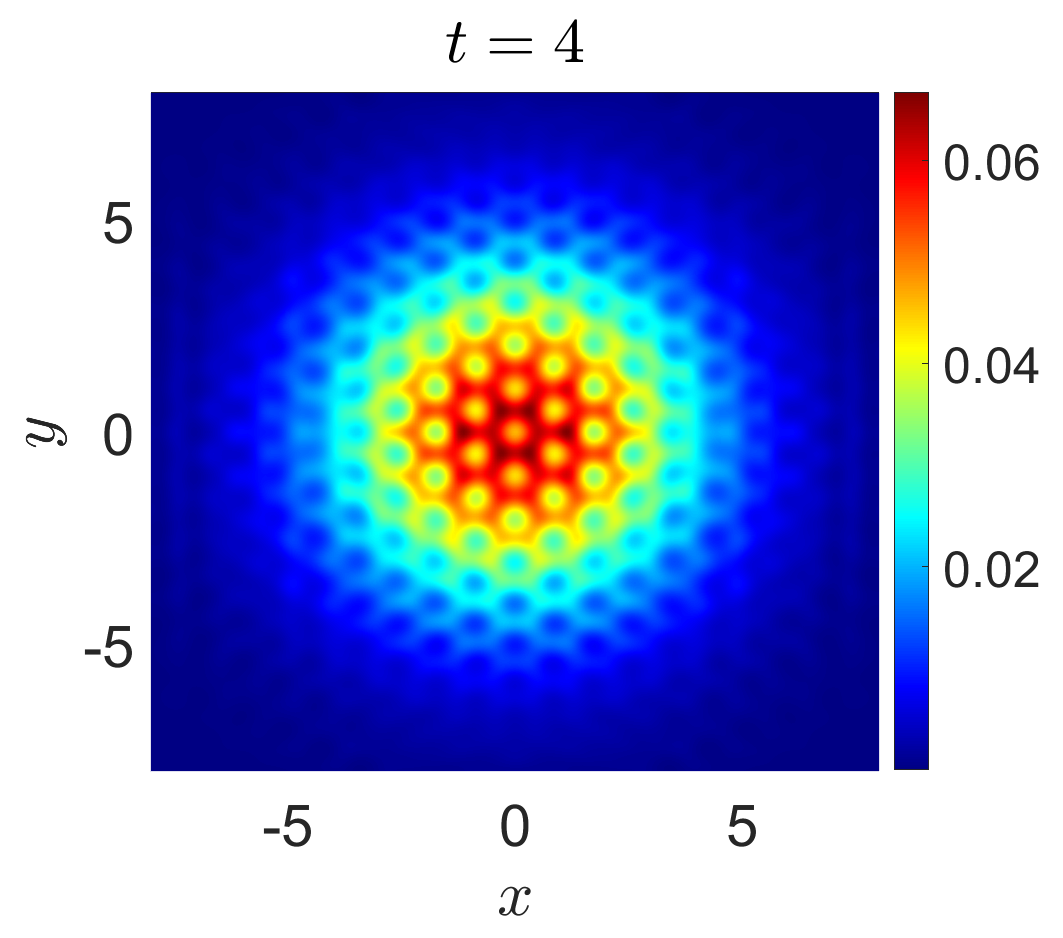}\\
    \caption{Dynamics of the densities $\rho_{1}(t,\bx)=\abs{\phi(t,\bx)}^{2}$ (first row), $\rho_{2}(t,\bx)=\abs{\psi_{1}(t,\bx)}^{2}$ (second row) of the KGD equation in 2D with a honeycomb potential with $\eps=0.01$, $g=10$ and $\omega=1$.}
    \label{fig-5-2-5}
\end{figure}

\subsubsection{Three-dimensional KGD equation}
We now consider the dynamics of the KGD system \eqref{KGD} in 3D with
coefficients $g=1$, $\omega=\frac{1}{2}$, and $\eps=0.1$. We take
$\Omega=(-5,5)^{3}$, $N_{x}=N_{y}=N_{z}=50$, and $N_{\tau}=2^{4}$, with
electromagnetic potentials $A_{j}=0$ for $j=1,2,3$ and
$V(\bx)=\frac{x+y+z-1}{x^2+y^2+z^2+1}$, $\bx=(x,y,z)^{\intercal}\in\mathbb{R}^{3}$.
The initial conditions are chosen as in \cite{GMYL} with $\tilde{\alpha}=0.01$:
\begin{align*}
&\phi_{0}=e^{-x^2-y^2-(z-2)^2}+\sum\limits_{k=0}^{1}e^{-(x+(-1)^{k}\sqrt{3})^2-y^2-(z+1)^2}, \quad
\phi_{1}=e^{-x^2-y^2-z^2}, \\
&\psi_{0,1}=\psi_{0,2}=\frac{1}{\sqrt{2}}\sum\limits_{k=0}^{1}e^{-(x+(-1)^{k}2)^2-y^2-z^2}e^{\ii\tilde{\alpha} (x+y+z)}, \quad
\psi_{0,3}=\psi_{0,4}=\frac{1}{\sqrt{2}}\psi_{0,1}.
\end{align*}

We first examine the dynamics of $\abs{\psi_{1}}$ using the EUA-EI4 scheme.
Initially, the isosurface consists of two ellipsoidal components
(Figure \ref{fig-5-3-1}(a)). These components radiate and come into contact at
$t=0.22$ (Figure \ref{fig-5-3-1}(b)), after which a smaller ellipsoid emerges at
the domain center and gradually expands (Figures \ref{fig-5-3-1}(c)--(d)). The
subsequent slices show continued radiation and the formation of increasingly
complex structures (Figures \ref{fig-5-3-1}(e)--(h)).
Figures \ref{fig-5-3-2}(a)--(d) show the evolution of the isosurface
$\phi=0.1$, while Figures \ref{fig-5-3-2}(e)--(h) present the corresponding
$y=0$ slices. At $t=0$, the isosurface comprises three distinct components,
which contract into a compact central structure by $t=0.15$, separate again at
$t=0.3$, and expand and merge into a connected structure by $t=0.5$. The slices
exhibit the same sequence of three localized peaks, a compact central structure,
three separated peaks, and an expanded connected pattern. These observations
reflect the radiation and complex interaction of the three localized structures.

Figure \ref{fig-5-3-3} illustrates the temporal evolution of the mass and energy
for the 3D KGD system. Under the symmetric scheme, both quantities are
approximately conserved, which demonstrates the effectiveness of the proposed
numerical method for solving the KGD system.

\begin{figure}[t!]
    \centering
    \begin{subfigure}[t]{0.245\textwidth}
        \centering
        \includegraphics[width=\linewidth, height=3.2cm]{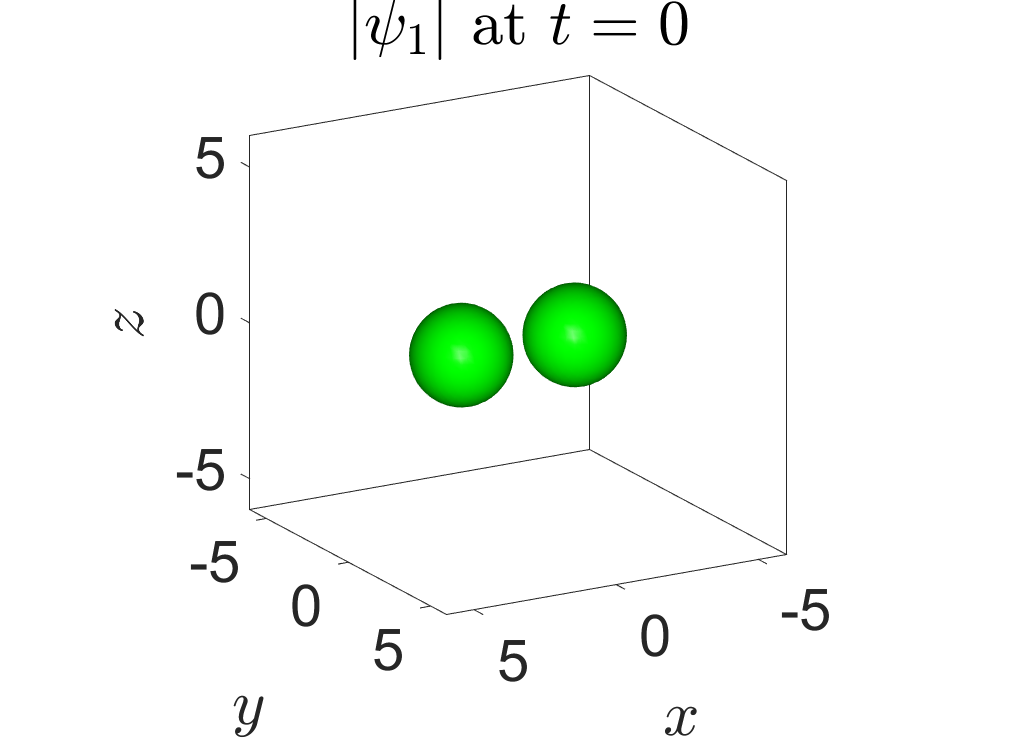}
        \caption{}
    \end{subfigure}
    \hfill
    \begin{subfigure}[t]{0.245\textwidth}
        \centering
        \includegraphics[width=\linewidth, height=3.2cm]{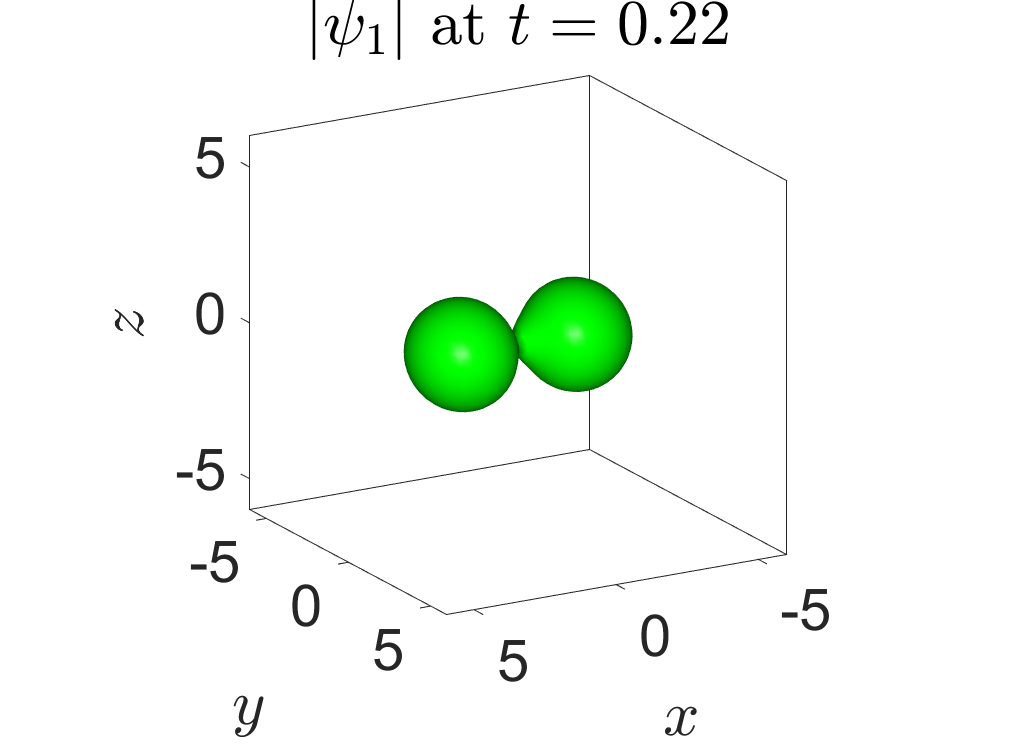}
        \caption{}
    \end{subfigure}
    \hfill
    \begin{subfigure}[t]{0.245\textwidth}
        \centering
        \includegraphics[width=\linewidth, height=3.2cm]{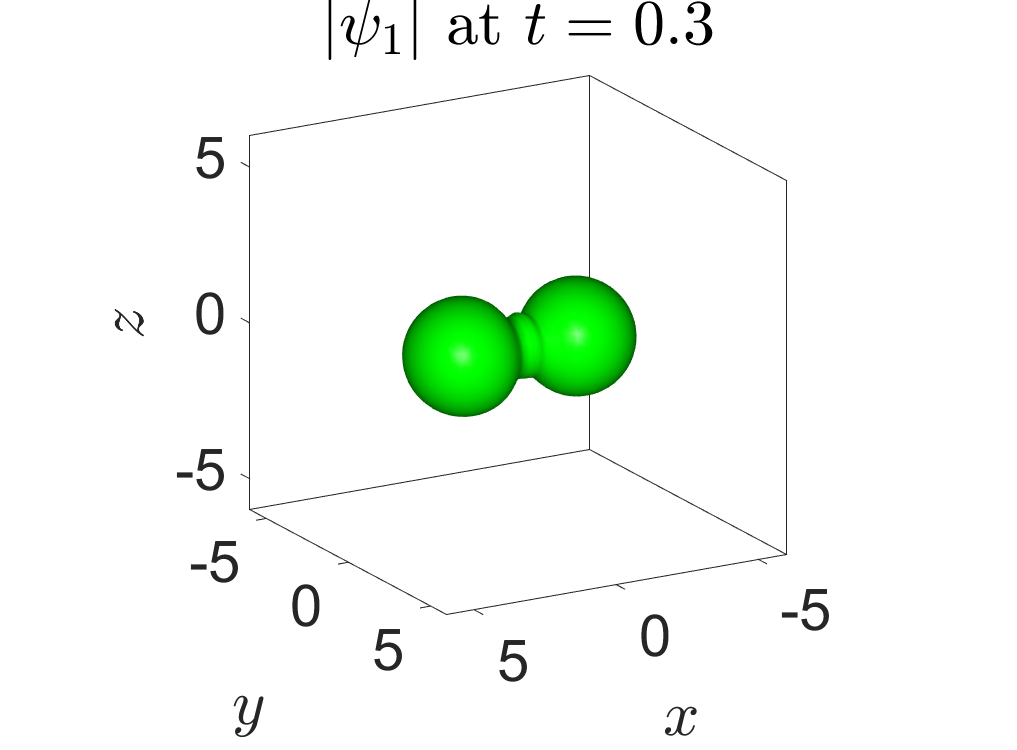}
        \caption{}
    \end{subfigure}
    \hfill
    \begin{subfigure}[t]{0.245\textwidth}
        \centering
        \includegraphics[width=\linewidth, height=3.2cm]{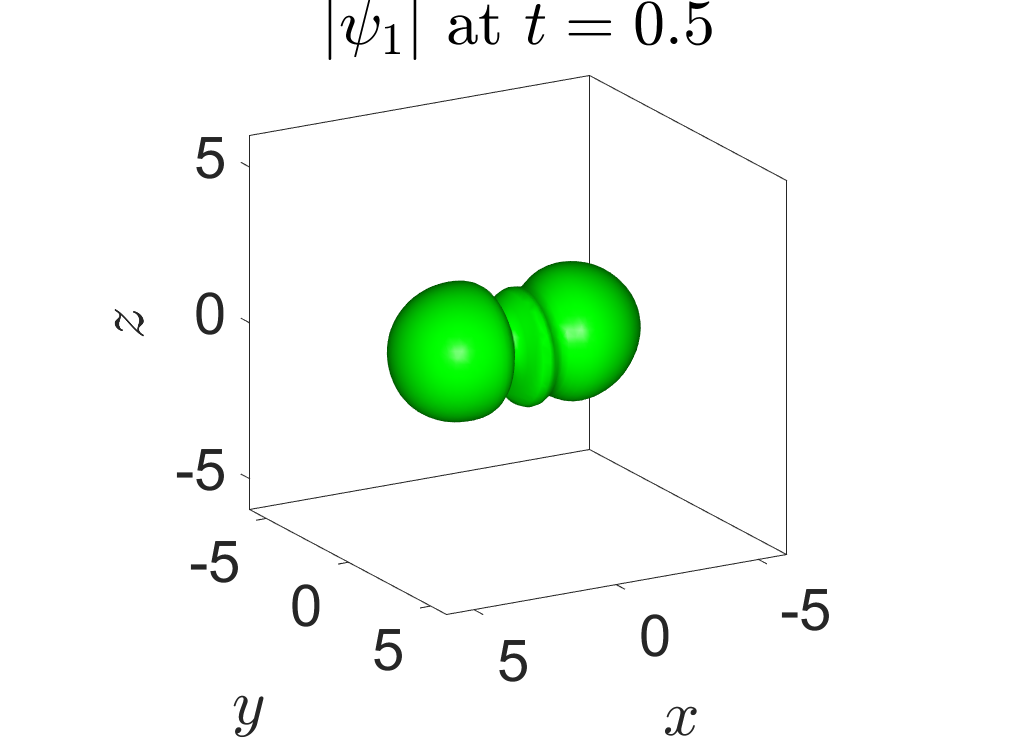}
        \caption{}
    \end{subfigure}
    \par
    \begin{subfigure}[t]{0.23\textwidth}
        \centering
        \includegraphics[width=\linewidth, height=3.2cm]{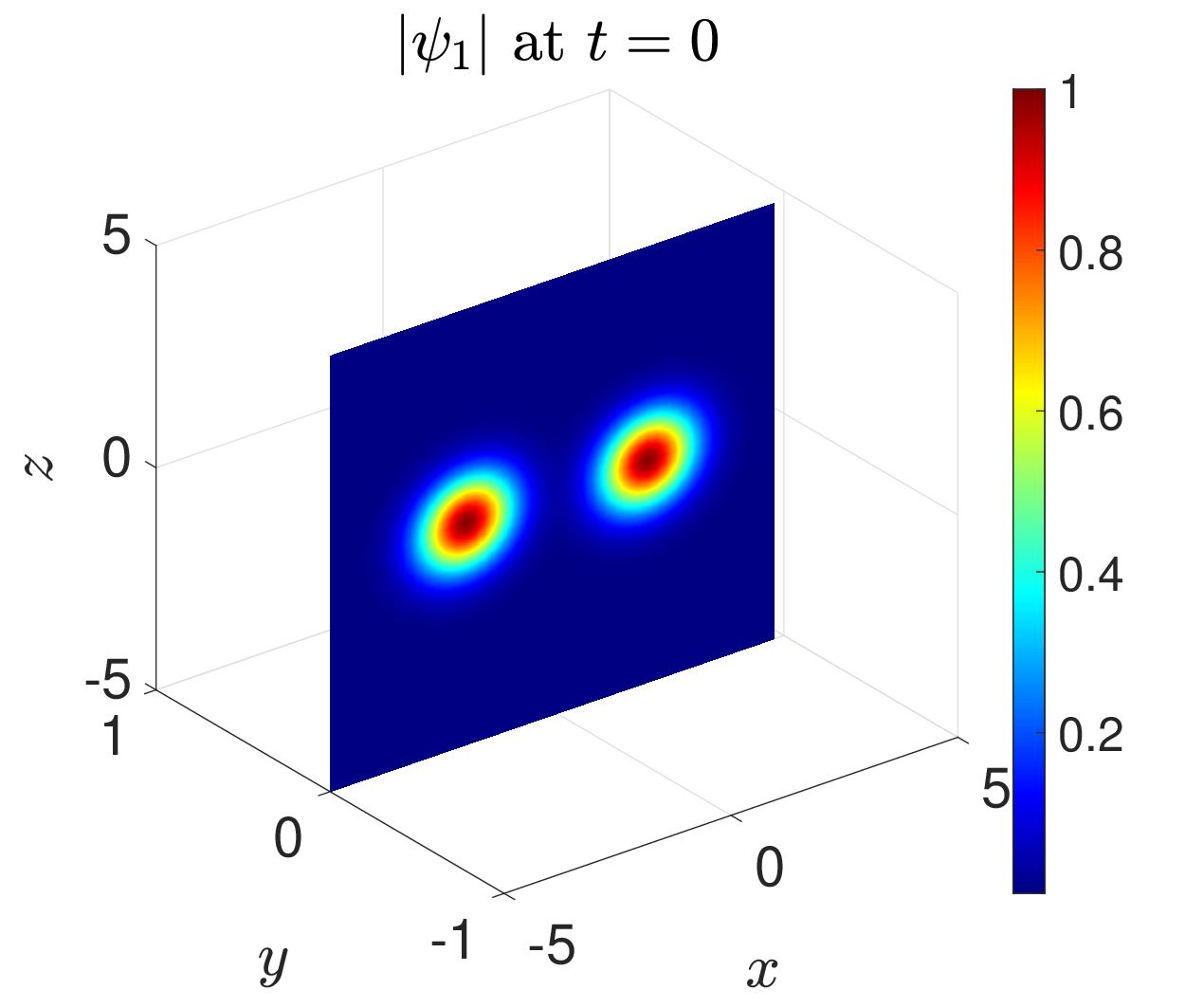}
        \caption{}
    \end{subfigure}
    \hfill
    \begin{subfigure}[t]{0.23\textwidth}
        \centering
        \includegraphics[width=\linewidth, height=3.2cm]{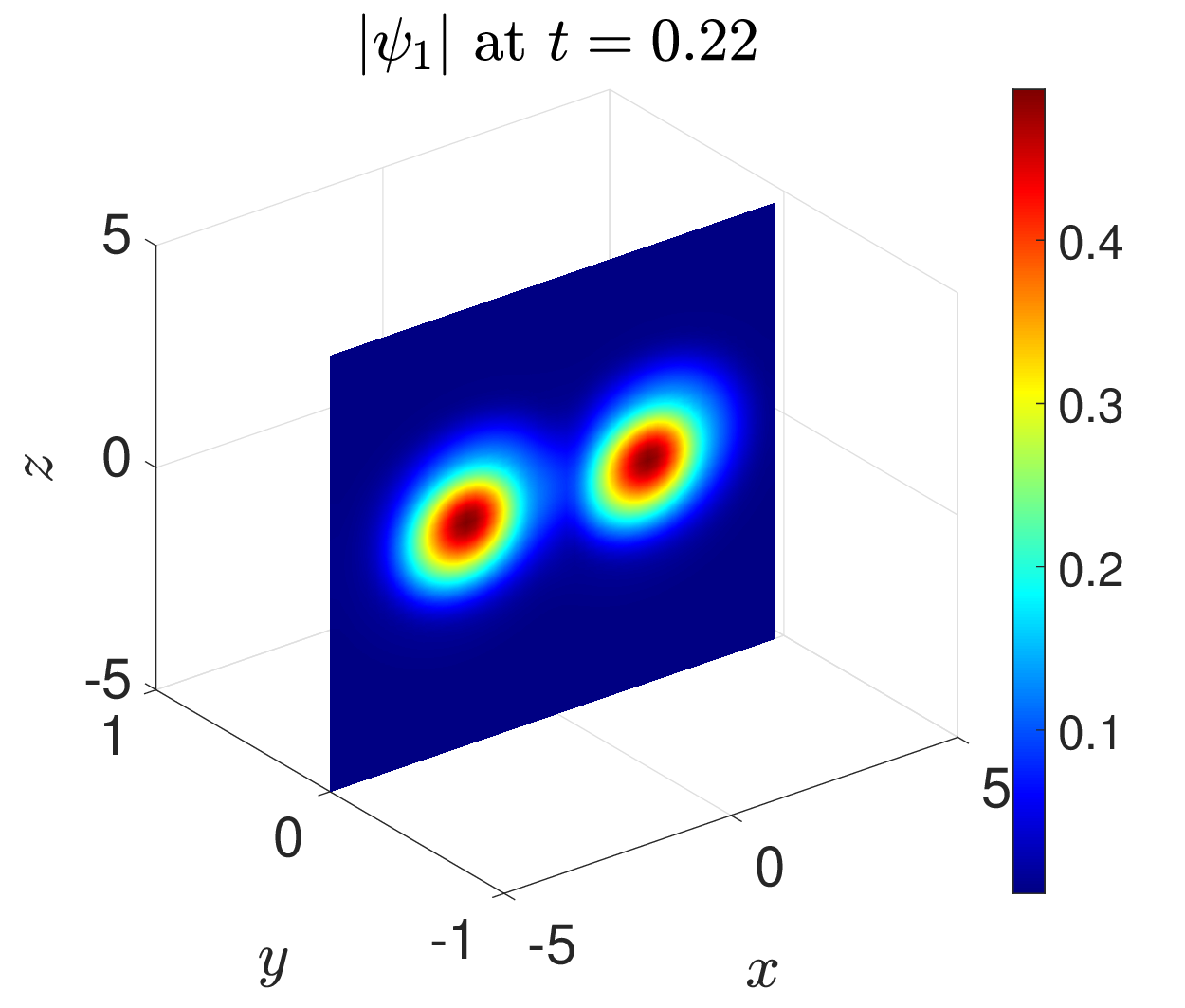}
        \caption{}
    \end{subfigure}
    \hfill
    \begin{subfigure}[t]{0.23\textwidth}
        \centering
        \includegraphics[width=\linewidth, height=3.2cm]{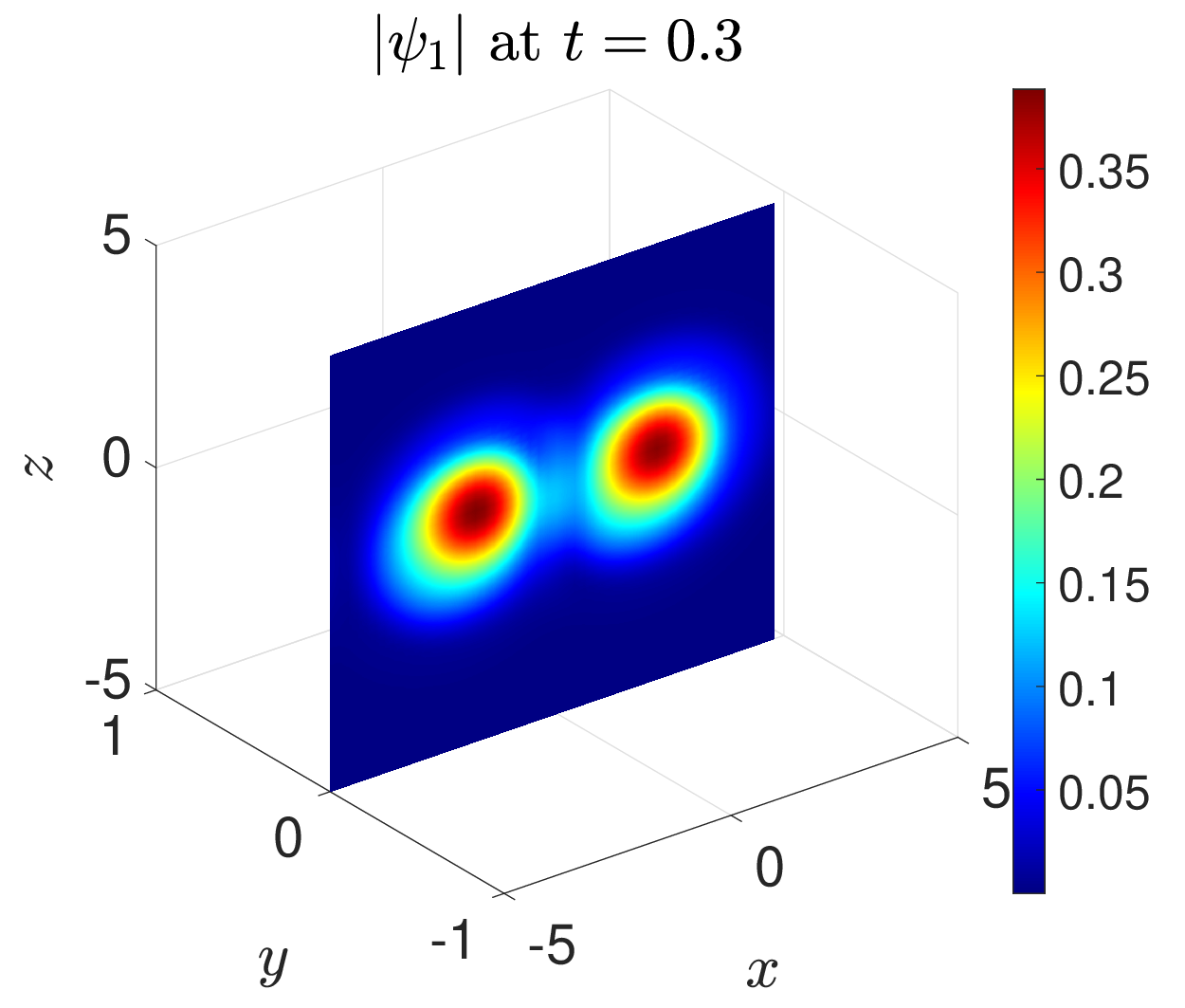}
        \caption{}
    \end{subfigure}
    \hfill
    \begin{subfigure}[t]{0.23\textwidth}
        \centering
        \includegraphics[width=\linewidth, height=3.2cm]{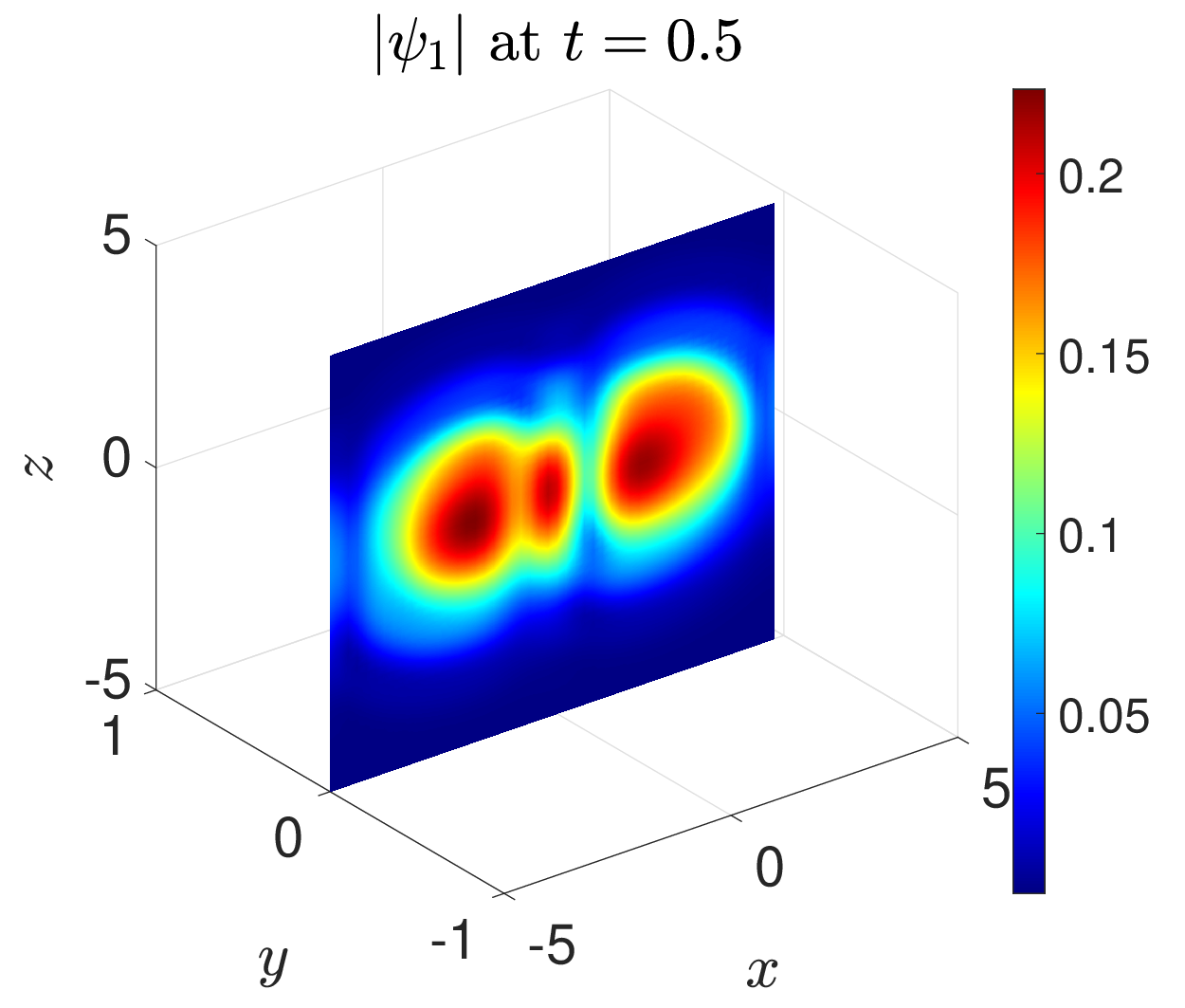}
        \caption{}
    \end{subfigure}
    \caption{The isosurfaces $\abs{\psi_{1}}=0.08$ (first row) and slices of $\abs{\psi_{1}}$ at $y=0$ (secong row) for KGD equation in 3D with $g=1$, $\omega=\frac{1}{2}$ and $\eps=0.1$.}
    \label{fig-5-3-1}
\end{figure}

\begin{figure}[t!]
    \centering
    \begin{subfigure}[t]{0.245\textwidth}
        \centering
        \includegraphics[width=\linewidth, height=3.2cm]{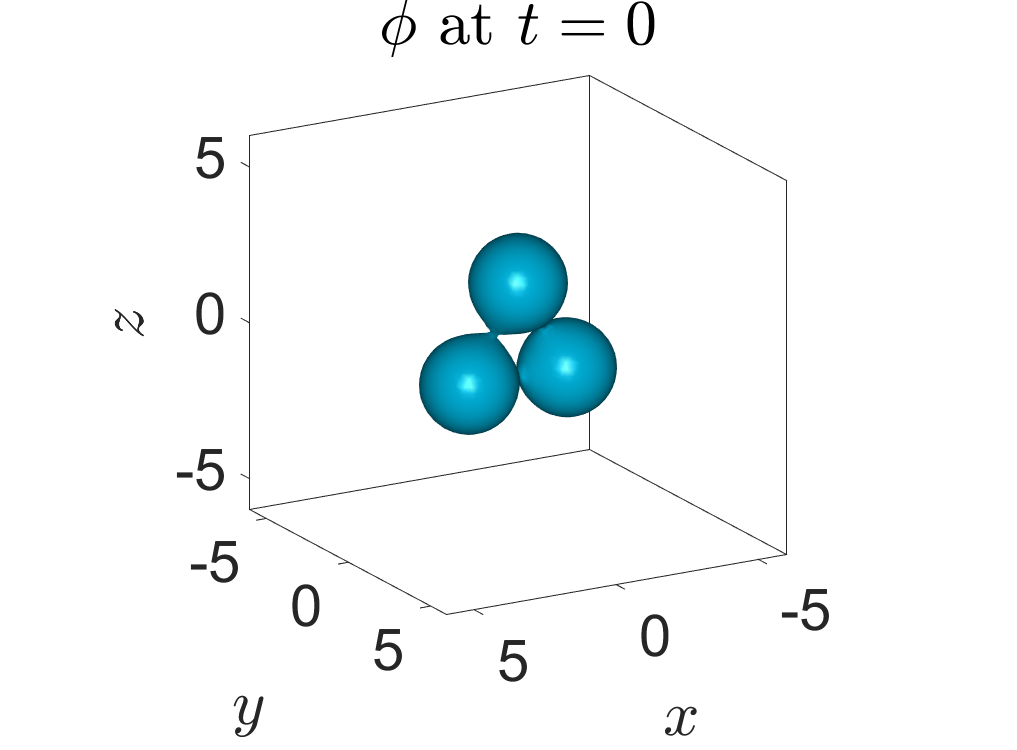}
        \caption{}
    \end{subfigure}
    \hfill
    \begin{subfigure}[t]{0.245\textwidth}
        \centering
        \includegraphics[width=\linewidth, height=3.2cm]{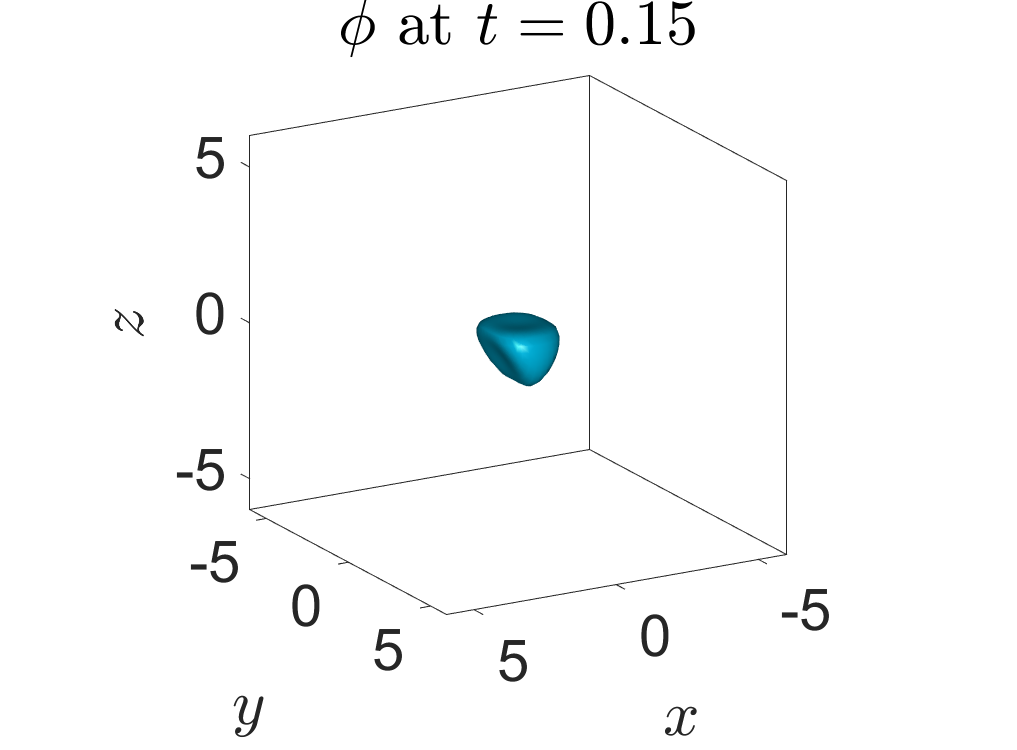}
        \caption{}
    \end{subfigure}
    \hfill
    \begin{subfigure}[t]{0.245\textwidth}
        \centering
        \includegraphics[width=\linewidth, height=3.2cm]{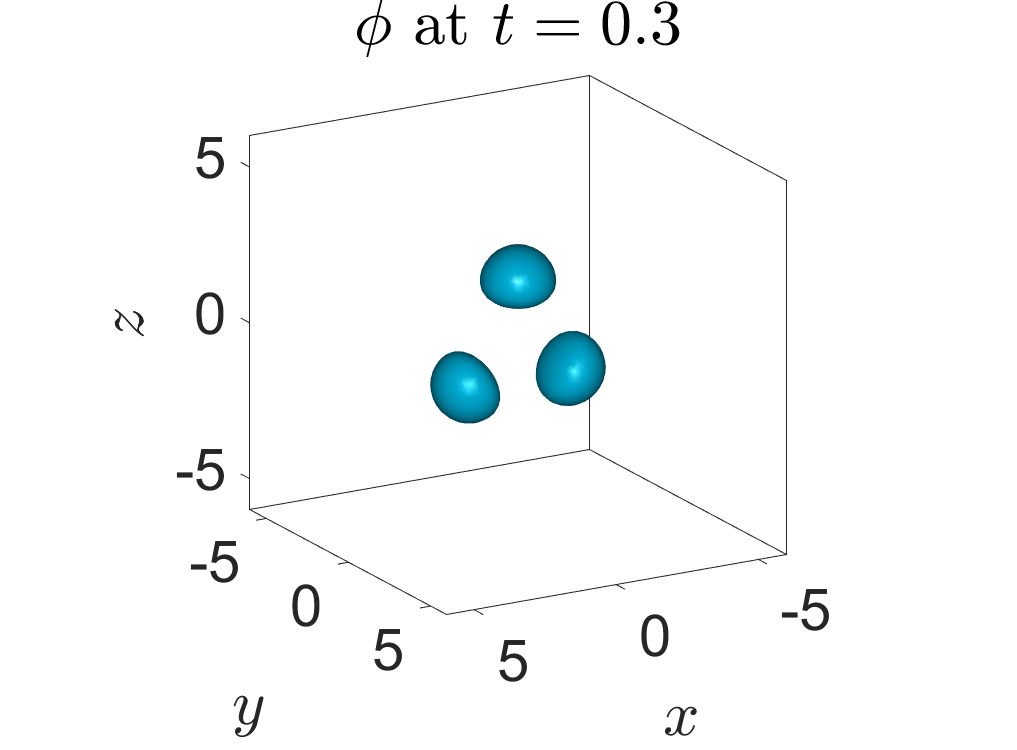}
        \caption{}
    \end{subfigure}
    \hfill
    \begin{subfigure}[t]{0.245\textwidth}
        \centering
        \includegraphics[width=\linewidth, height=3.2cm]{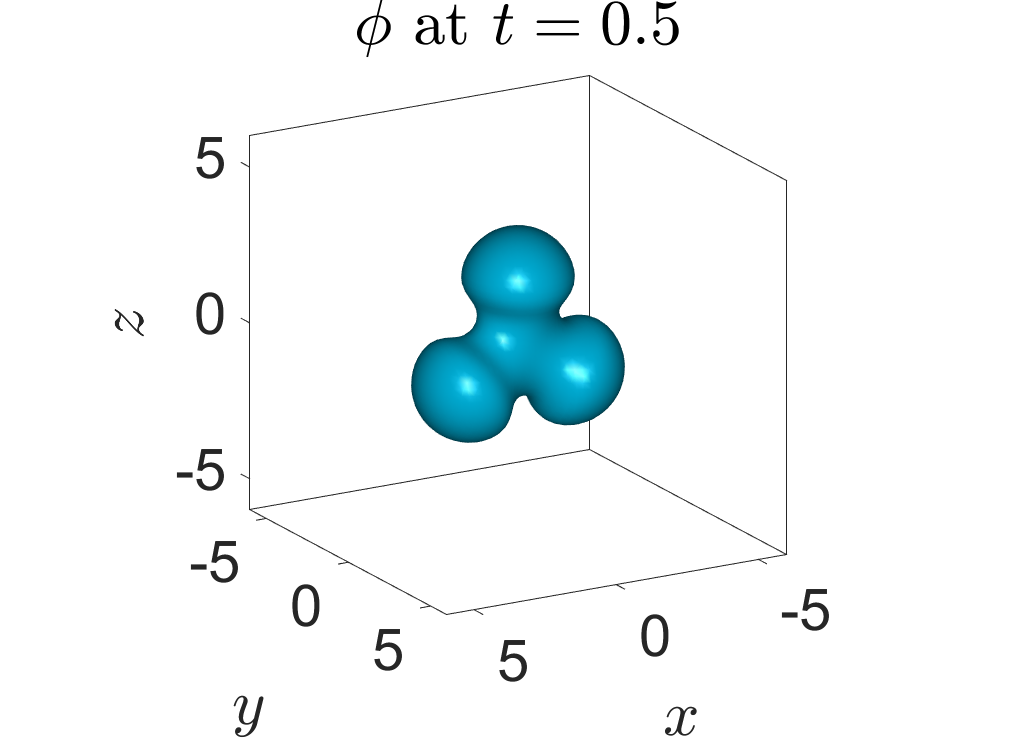}
        \caption{}
    \end{subfigure}
    \par
    \begin{subfigure}[t]{0.23\textwidth}
        \centering
        \includegraphics[width=\linewidth, height=3.2cm]{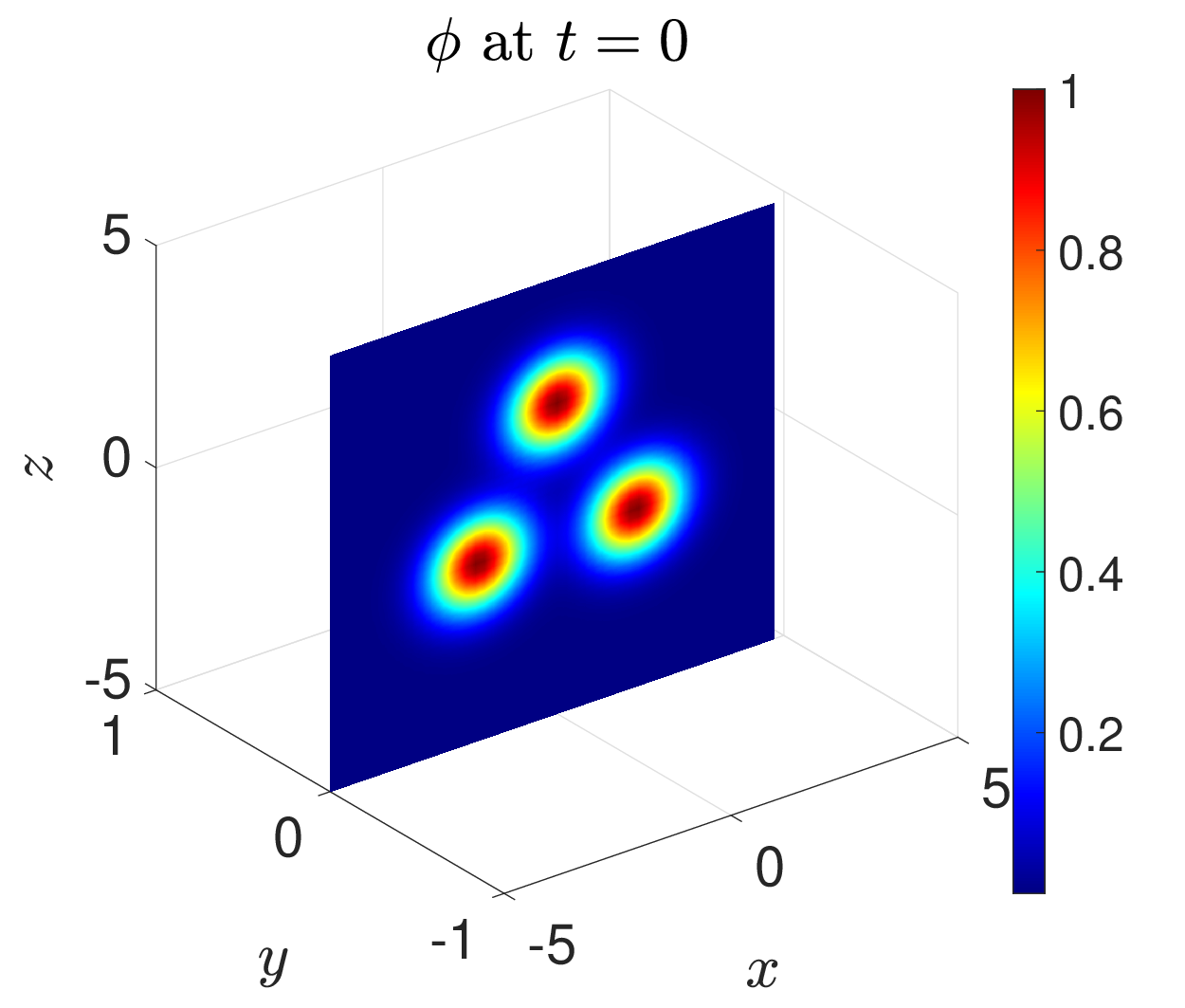}
        \caption{}
    \end{subfigure}
    \hfill
    \begin{subfigure}[t]{0.23\textwidth}
        \centering
        \includegraphics[width=\linewidth, height=3.2cm]{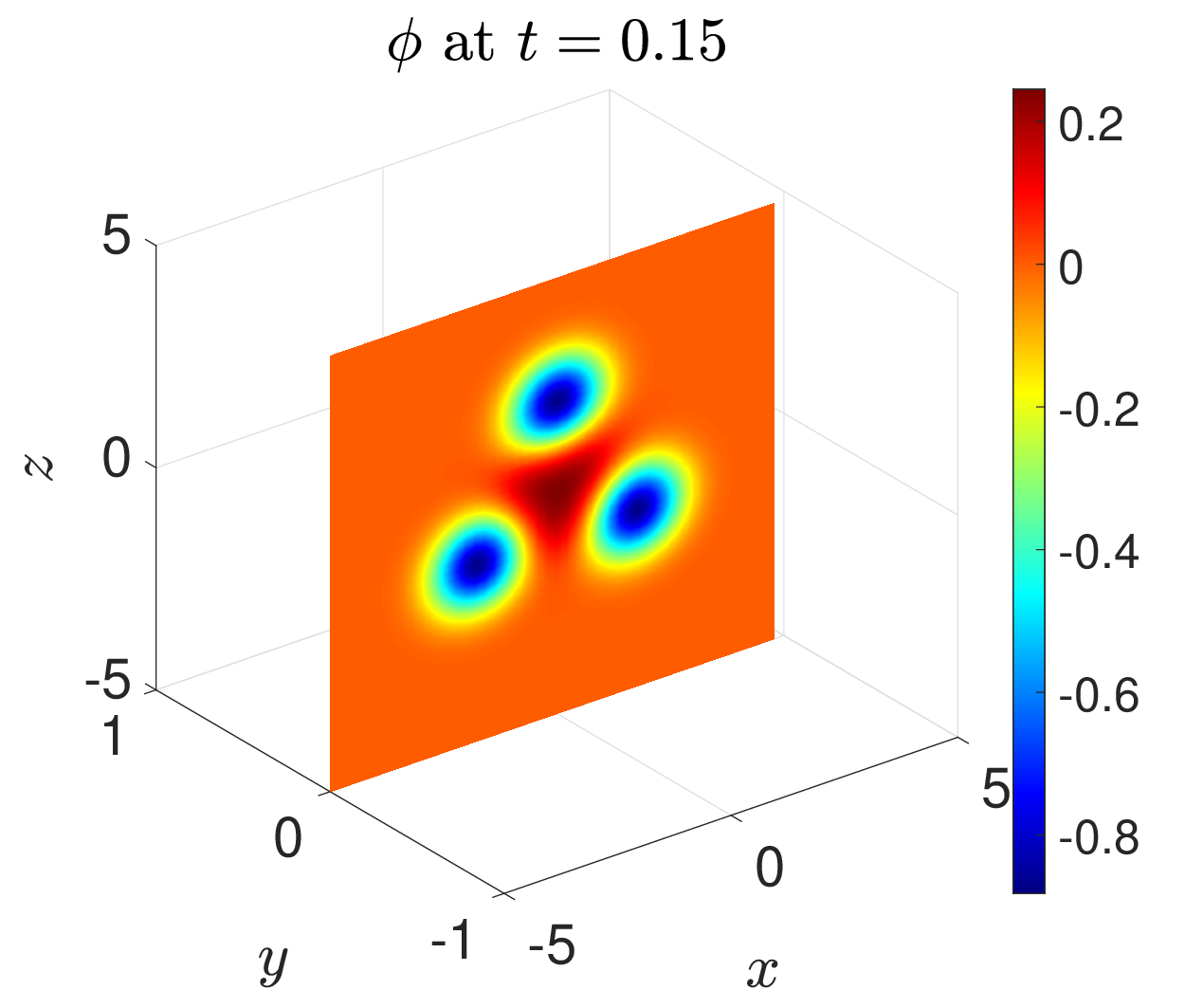}
        \caption{}
    \end{subfigure}
    \hfill
    \begin{subfigure}[t]{0.23\textwidth}
        \centering
        \includegraphics[width=\linewidth, height=3.2cm]{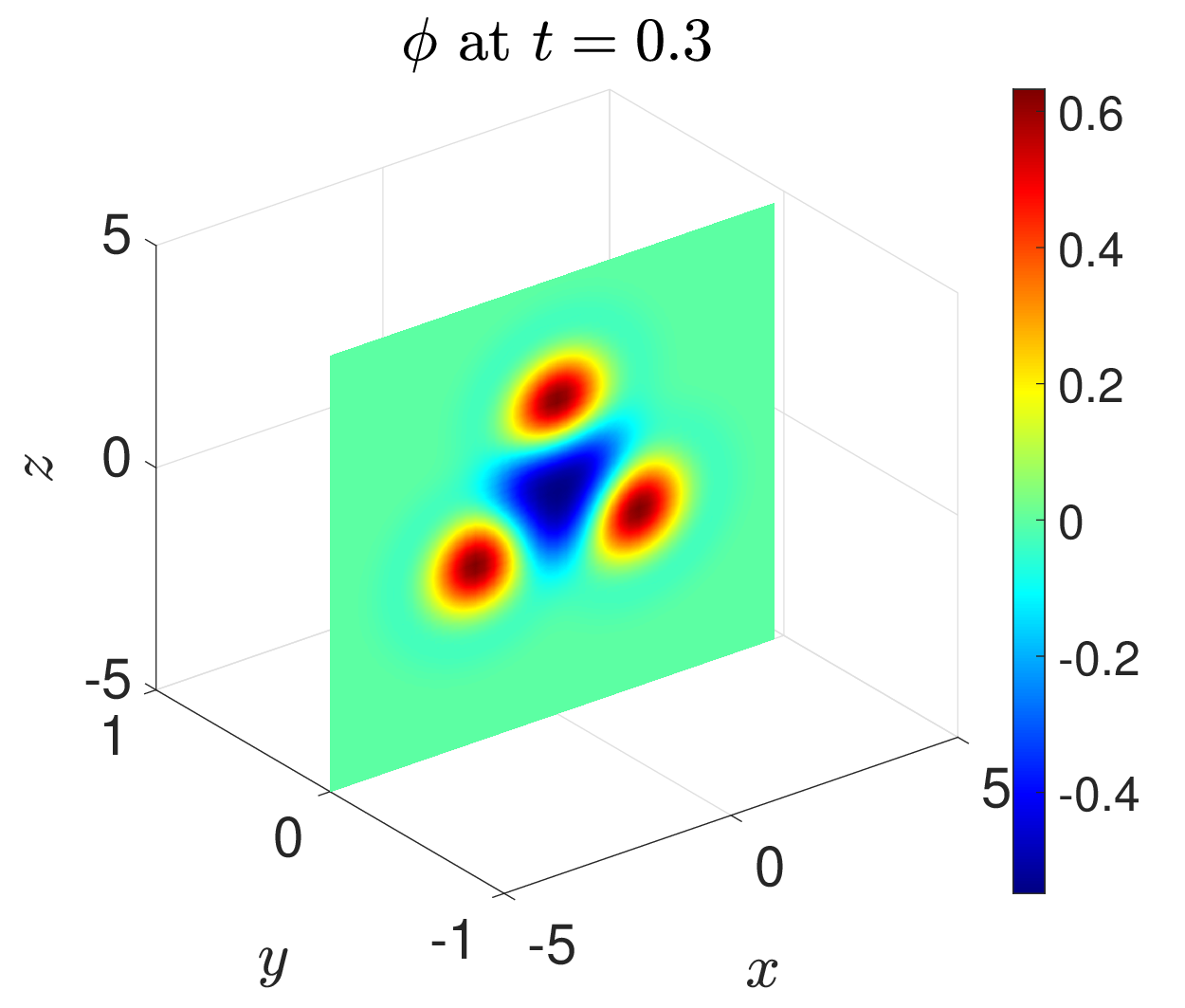}
        \caption{}
    \end{subfigure}
    \hfill
    \begin{subfigure}[t]{0.23\textwidth}
        \centering
        \includegraphics[width=\linewidth, height=3.2cm]{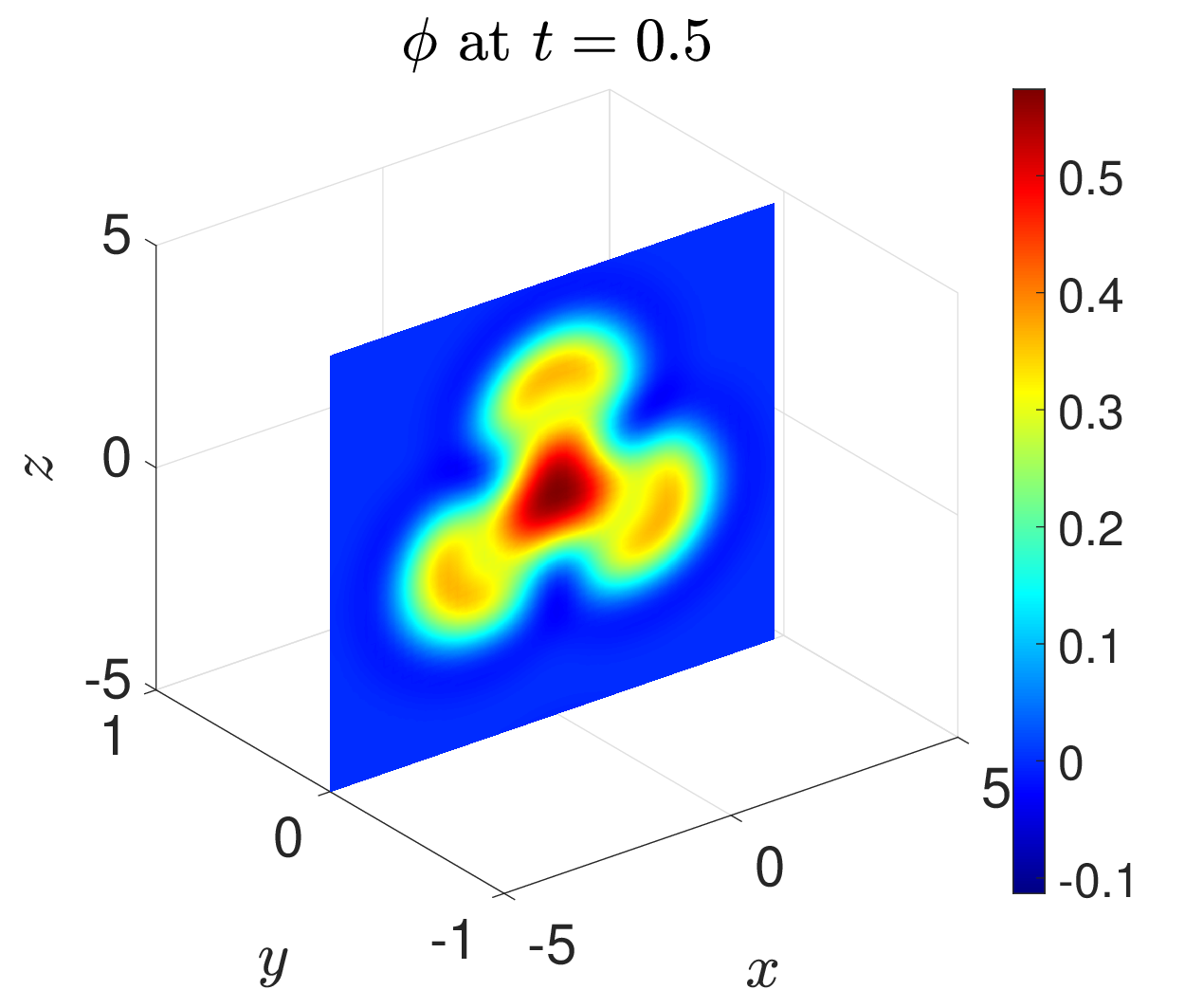}
        \caption{}
    \end{subfigure}
    \caption{The isosurfaces $\phi=0.1$ (first row) and slices of $\phi$ at $y=0$ (secong row) for KGD equation in 3D with $g=1$, $\omega=\frac{1}{2}$ and $\eps=0.1$.}
    \label{fig-5-3-2}
\end{figure}

\begin{figure}[t!]
    \centering
        \includegraphics[width=0.32\linewidth,height=3.1cm]{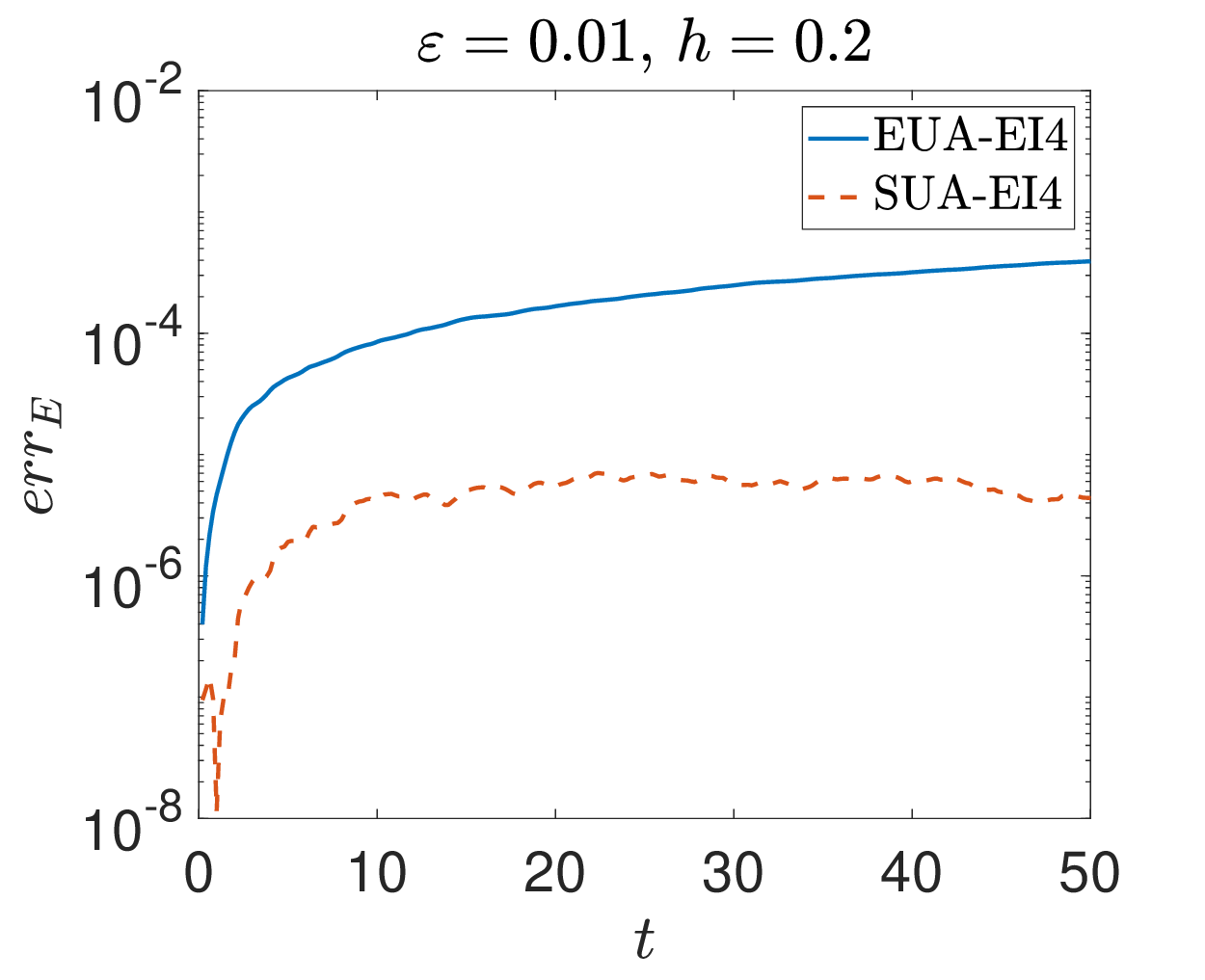}
        \includegraphics[width=0.32\linewidth,height=3.1cm]{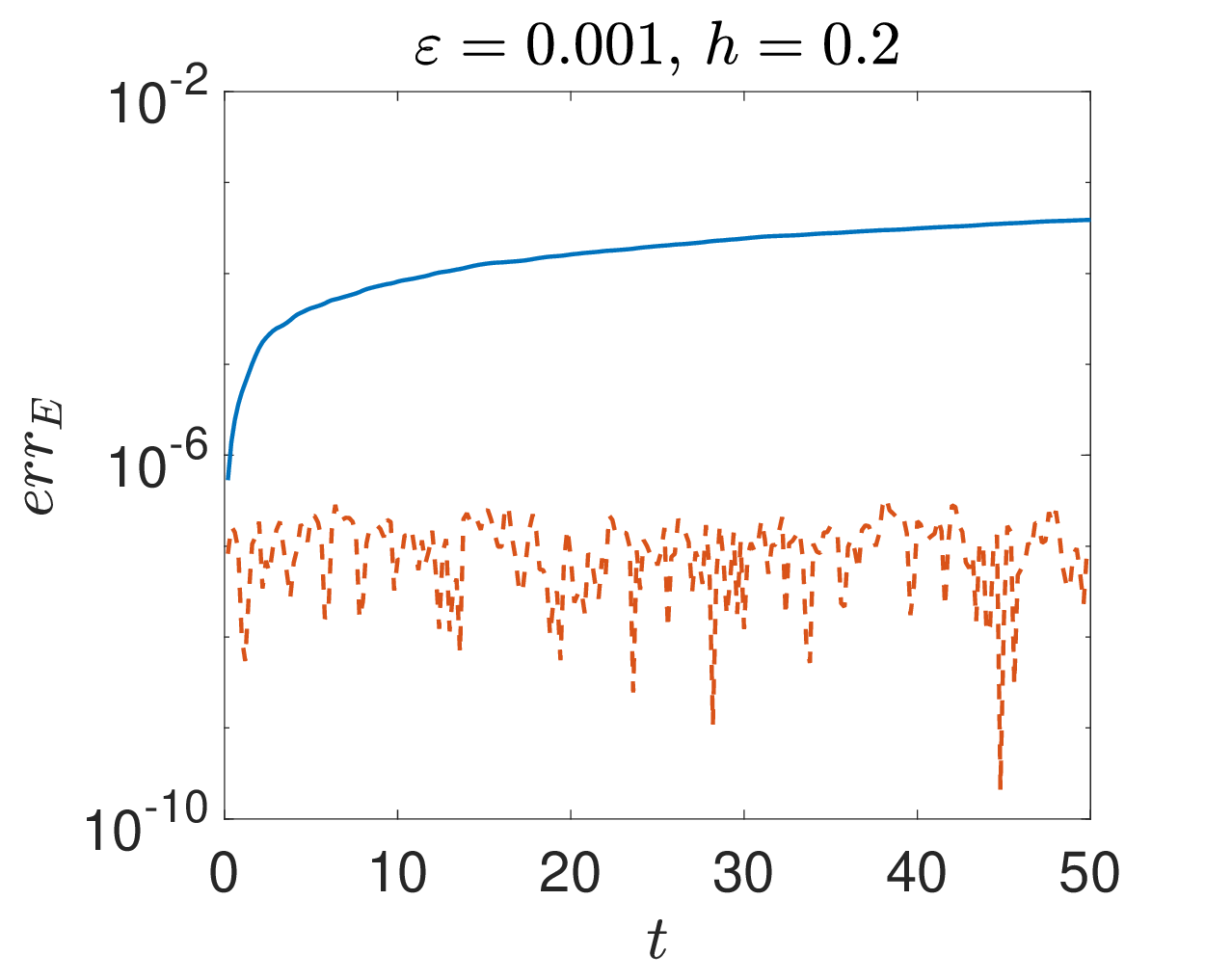}
        \includegraphics[width=0.32\linewidth,height=3.1cm]{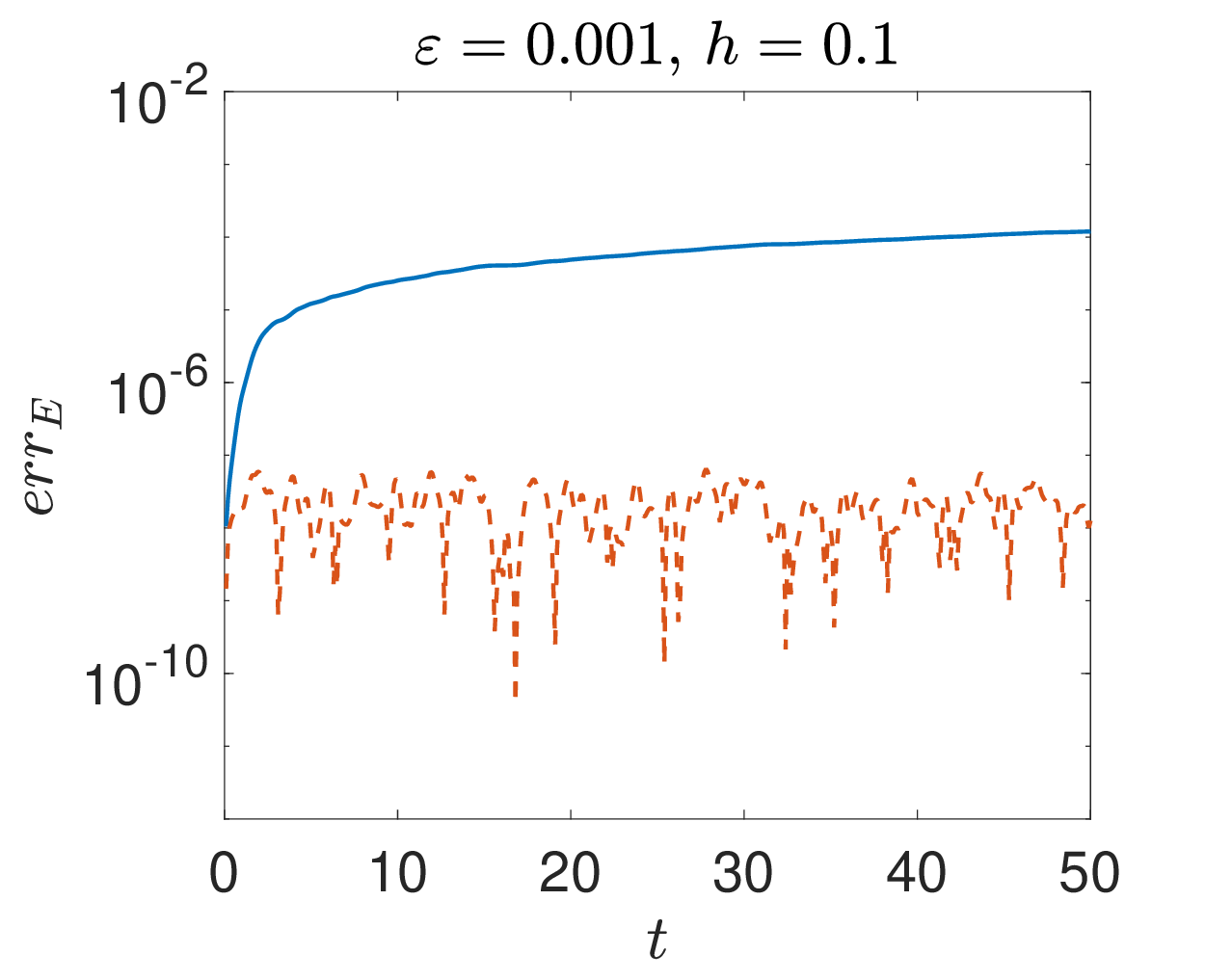}\\
        \includegraphics[width=0.32\linewidth,height=3.1cm]{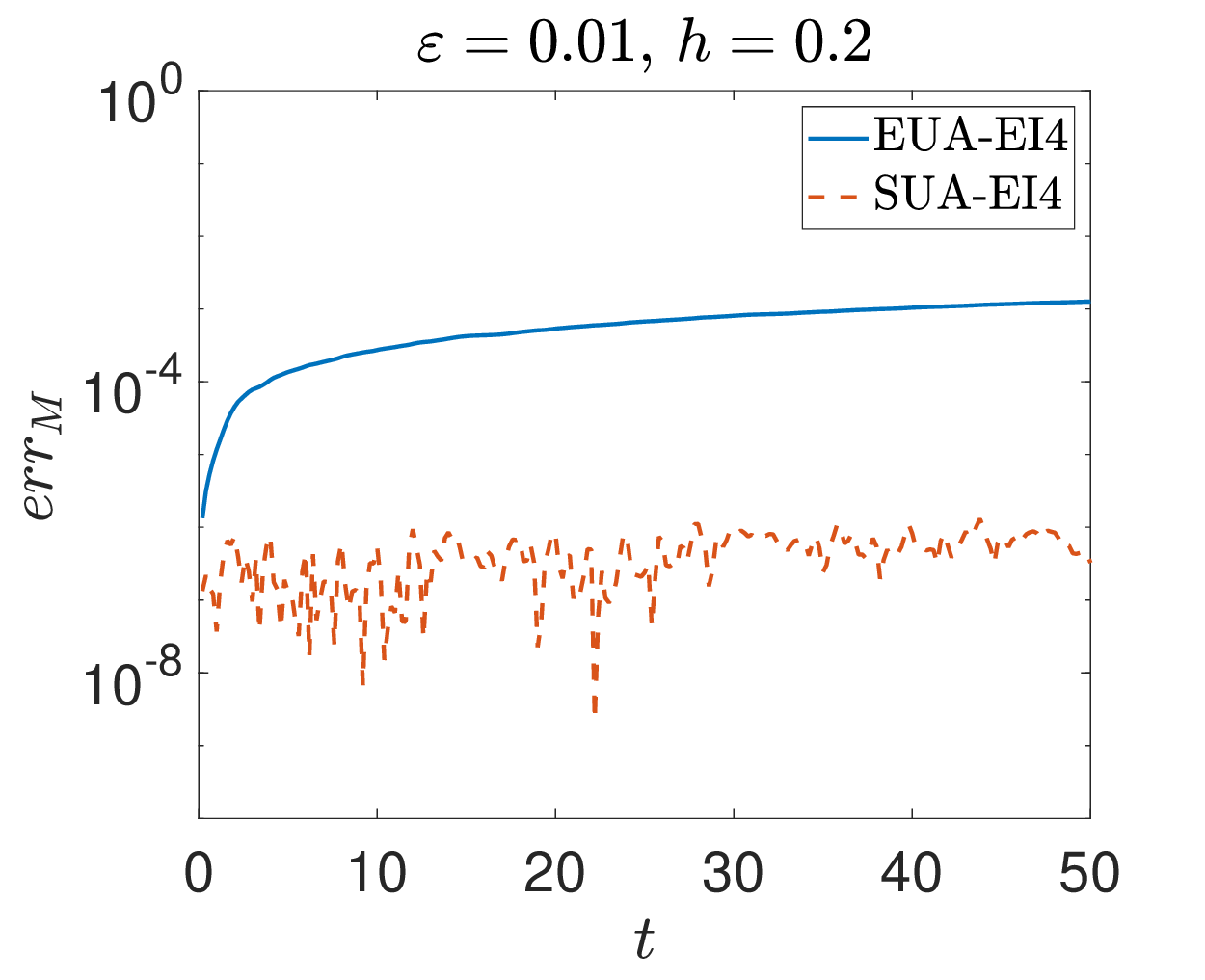}
        \includegraphics[width=0.32\linewidth,height=3.1cm]{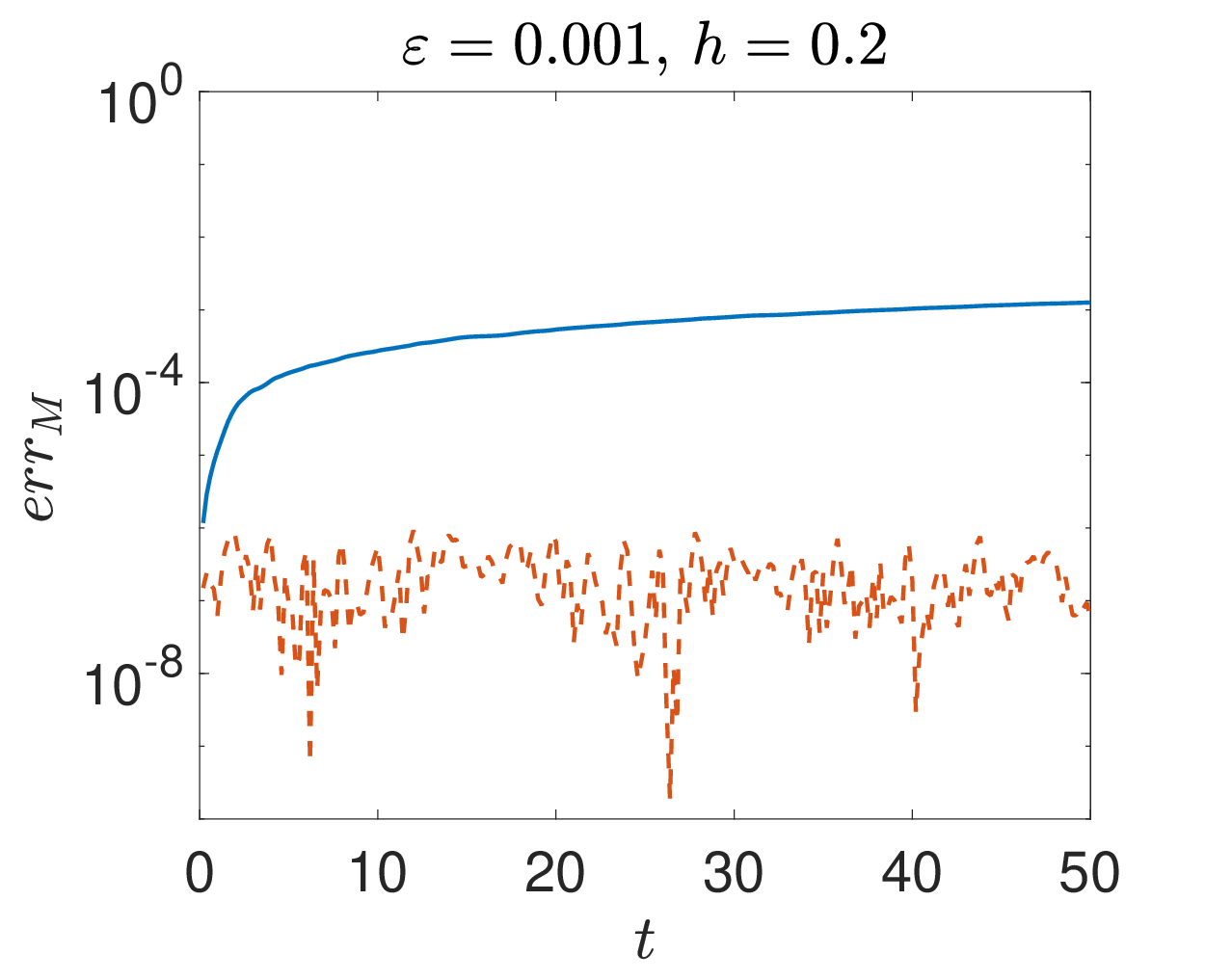}
        \includegraphics[width=0.32\linewidth,height=3.1cm]{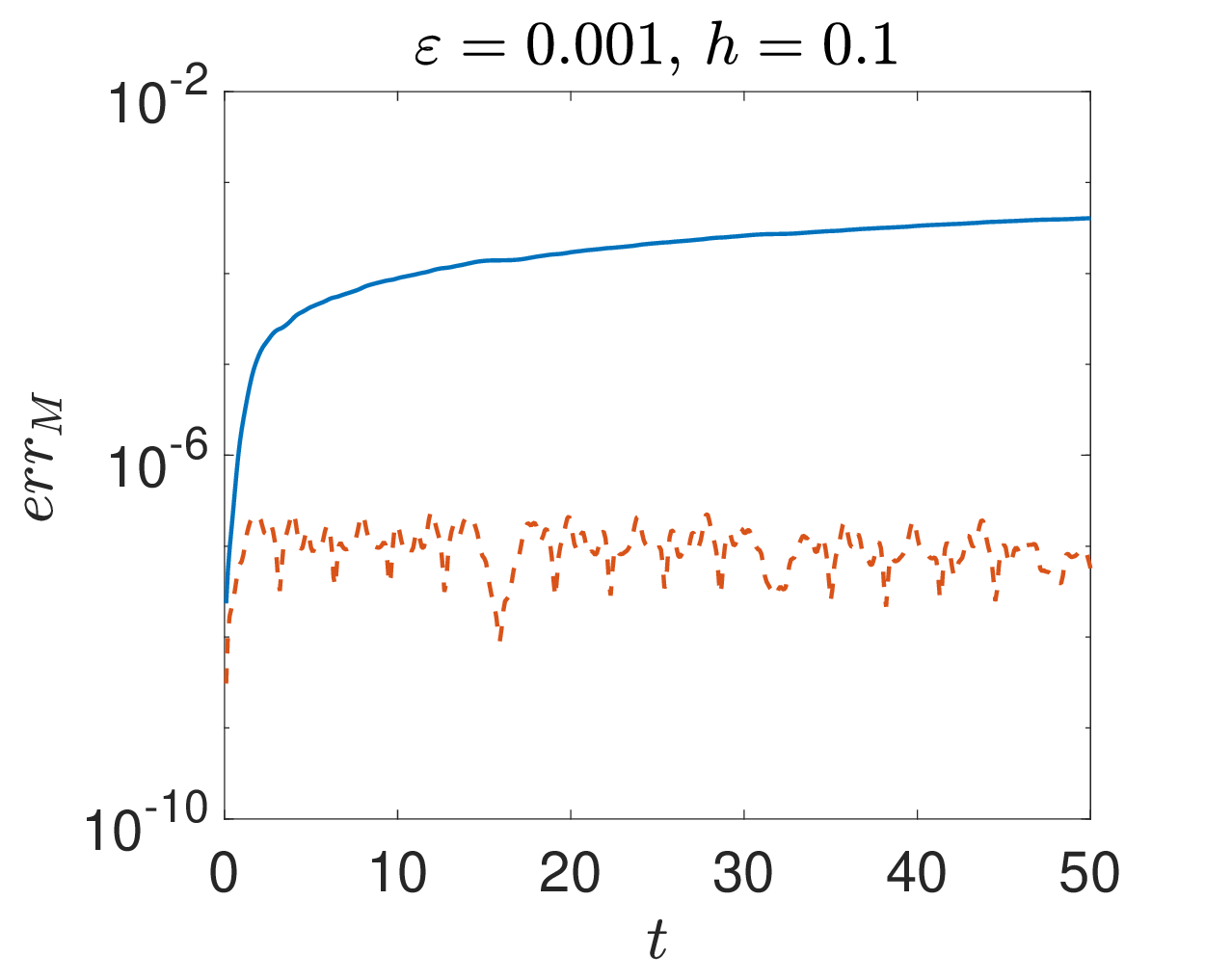}
    \caption{Energy error (top) and mass error (bottom) of KGD in 3D under different $\eps$ and $h$.}
    \label{fig-5-3-3}
\end{figure}

\subsection{KGS equation}\label{sec-4.2}
In this subsection, we present numerical results demonstrating the uniform
convergence and long-time near-conservation of the proposed UA-EI4 schemes for
$\eps\in(0,1]$. Likewise, the relative errors in accuracy $err_{\psi}$ and
$err_{\phi}$, mass $err_{M}$, and energy $err_{E}$ are employed as metrics.

\subsubsection{Two-dimensional KGS equation}
We first study the convergence and long-time behavior of the 2D KGS system with
$\mu=\lambda=1$, taking the bounded domain $\Omega=(-10,10)^{2}$ and the initial
data from \cite{BZ}:
\begin{align*}
\psi_{0}(\bx)=(1+i)e^{-(x^2+y^2)}, \quad \phi_{0}(\bx)=\operatorname{sech}(x^2+y^2), \quad \phi_{1}(\bx)=\sin(x+y)e^{-2(x^2+y^2)}, \quad \bx=(x,y)^{\intercal}\in\mathbb{R}^{2}.
\end{align*}
The spatial mesh sizes are set as $\triangle x=\triangle y=1/8$ and
$\triangle\tau=1/16$, with the reference solution computed by the 2D MTI-FP
method using a fine time step $h=10^{-6}$. Table \ref{tab-5-5-2} and Figures
\ref{fig-5-5-1} and \ref{fig-5-5-3} display the convergence and long-time
behavior, clearly revealing fourth-order uniform accuracy and near-conservation
of mass and energy.

Next, we evaluate the dynamics of the KGS equation in 2D with $\mu=\lambda=1$
and $\eps=0.01$, with the initial condition taken from \cite{BY,ZS}:
\begin{align*}
\psi_{0}(\bx)=\frac{2}{e^{x^2+2y^2}+e^{-(x^2+2y^2)}}e^{i5/\cosh(\sqrt{4x^2+y^2})}, \quad \phi_{0}(\bx)=e^{-(x^2+y^2)}, \quad \phi_{1}(\bx)=\frac{1}{2}e^{-(x^2+y^2)}, \quad \bx\in\mathbb{R}^{2}.
\end{align*}
This example is solved on the square $(-64,64)^{2}$ with mesh sizes
$\triangle x=\triangle y=1/4$ and $\triangle\tau=1/8$, and time step $h=0.01$,
using the SUA-EI4 scheme. Figures \ref{fig-5-5-4} and \ref{fig-5-5-5} present
the surface plots and snapshots of $\abs{\psi}^{2}$ and $\phi$ at different
times for $\eps=0.01$. It is observed that the nucleon density
$\abs{\psi}^{2}$ gradually splits from one wave into two waves propagating in
opposite directions along the $x$-axis, while the meson field $\phi$ becomes
flattened with a downward-concave crest.

\begin{figure}[t!]
    \centering
        \includegraphics[width=0.38\linewidth,height=3.6cm]{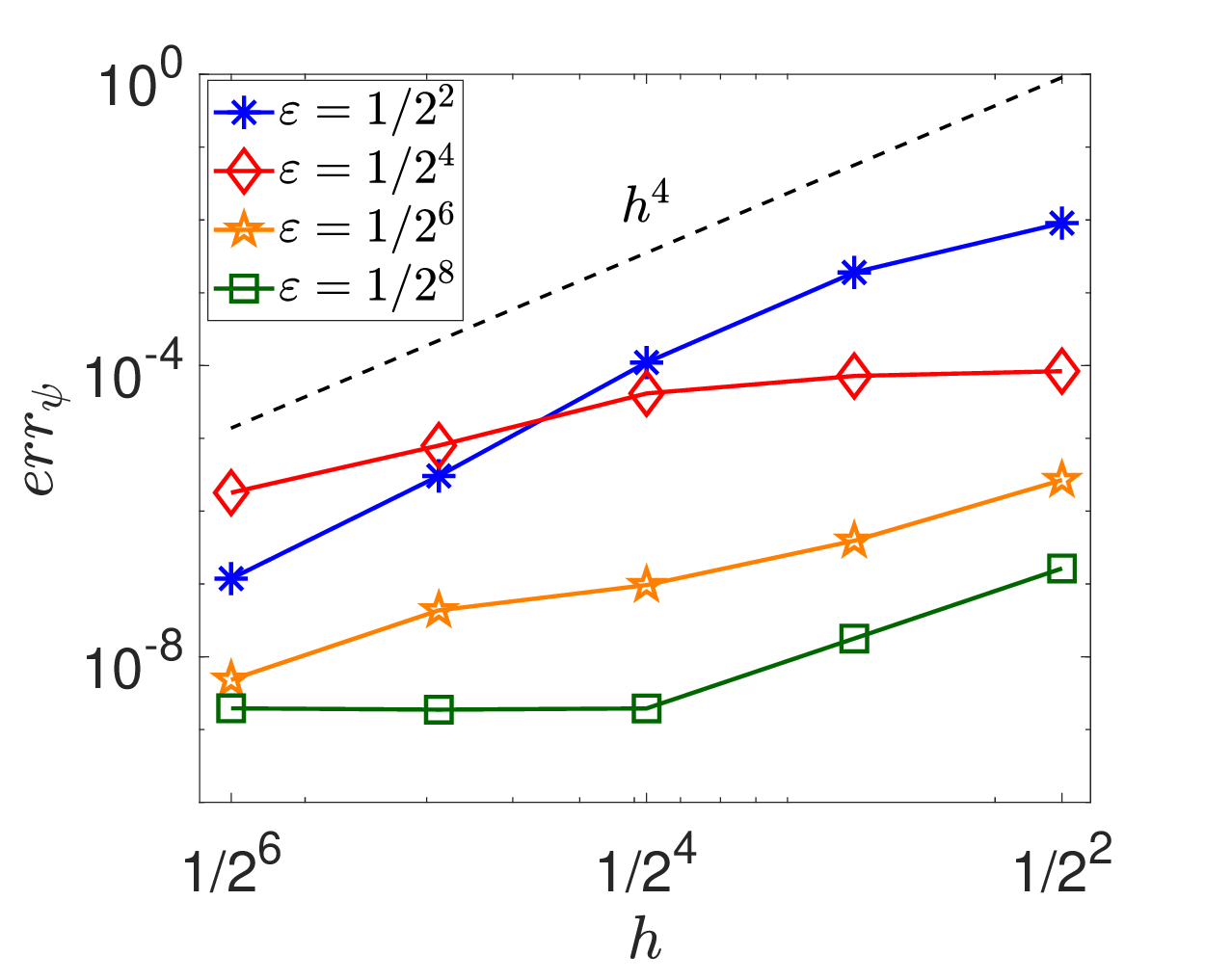}
        \includegraphics[width=0.38\linewidth,height=3.6cm]{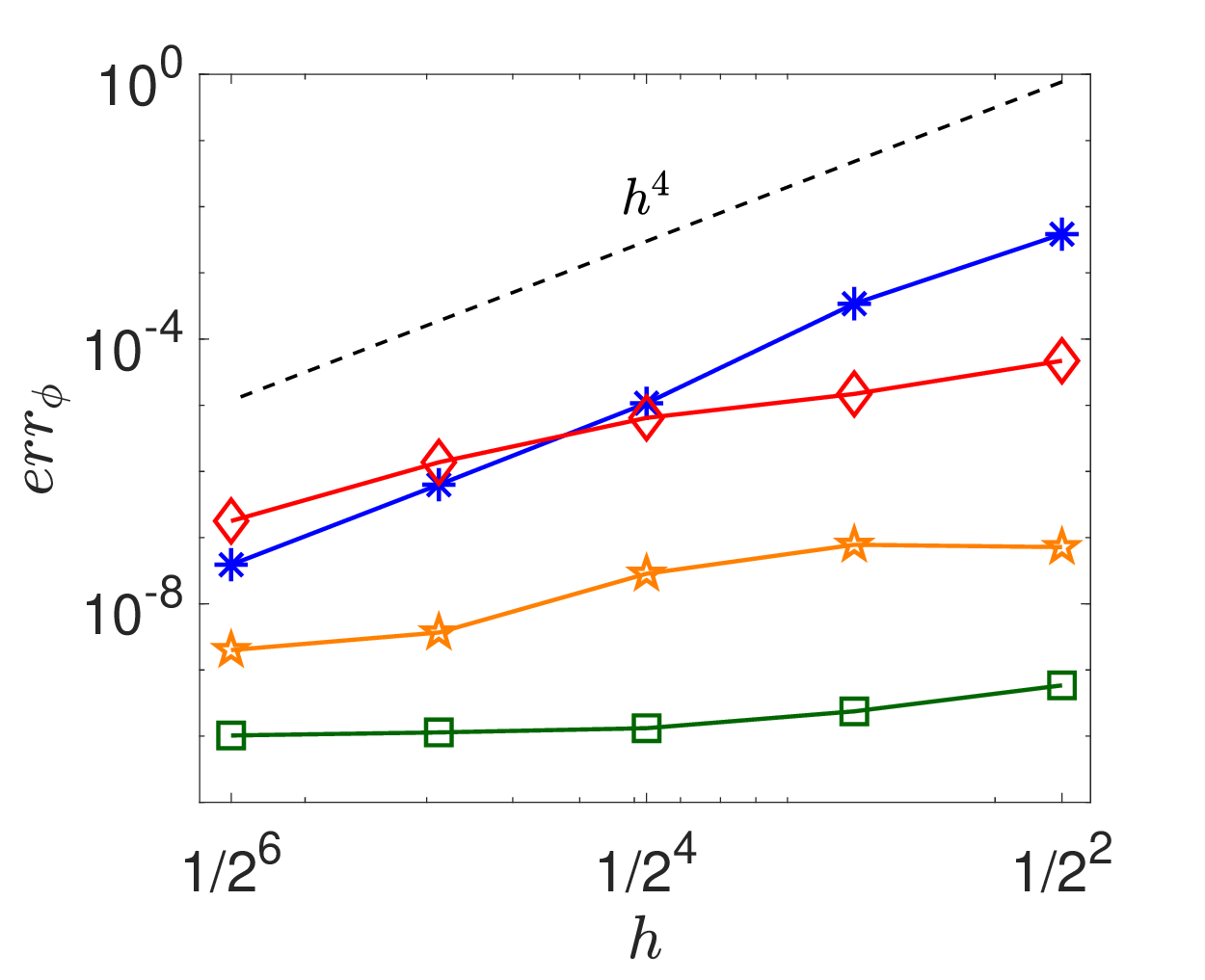}
    \caption{Temporal error of KGS in 2D at $T=1$ under different $\eps$ for SUA-EI4 scheme.}
    \label{fig-5-5-1}
\end{figure}

\begin{table}[t!]
    \centering
    \caption{Temporal error of KGS in 2D at $T=1$ under different $\eps$ for the EUA-EI4 scheme.}
    \label{tab-5-5-2}

    \renewcommand{\arraystretch}{1.15}
    \setlength{\tabcolsep}{5pt}

    \resizebox{\textwidth}{!}{%
    \begin{tabular}{c ccccc ccccc}
        \toprule
        &
        \multicolumn{5}{c}{\ $err_{\psi}$\ } &
        \multicolumn{5}{c}{\ $err_{\phi}$\ } \\
        \cmidrule(lr){2-6}
        \cmidrule(lr){7-11}
        
       $\eps$ & $h_{0}=1/2^2$ & $h_{0}/2$ & $h_{0}/4$ & $h_{0}/8$ & $h_{0}/16$ &  $h_{0}$ & $h_{0}/2$ & $h_{0}/4$ & $h_{0}/8$ & $h_{0}/16$ \\
        \midrule
        $1/2^2$ & 9.06E-3 & 1.90E-3 & 1.10E-4 & 3.04E-6 & 1.16E-7 &  4.13E-3 & 3.89E-4 & 1.13E-5 & 8.95E-7 & 6.21E-8      \\
        order & -- & 2.26 & 4.11 & 5.18 & 4.71 &  -- & 3.41 & 5.11 & 3.66 & 3.85            \\
        $1/2^4$ &  1.09E-4 & 7.18E-5 & 4.16E-5 & 7.96E-6 & 1.79E-6 &  9.07E-5  & 4.48E-5 & 1.66E-5 & 2.43E-6 & 2.78E-7      \\
        order &  -- & 0.60  &  0.79 & 2.38 & 2.15 &  -- & 1.02 & 1.43 & 2.78 & 3.13     \\
        $1/2^6$ &  4.03E-6 & 7.65E-7 & 1.44E-7 & 4.49E-8 & 5.00E-9 & 2.94E-6 & 1.78E-6 & 5.34E-7 & 9.13E-8 & 1.39E-8       \\
        order & -- & 2.40 & 2.41 & 1.68 & 3.17 &  -- & 0.73 & 1.74  & 2.55 & 2.71    \\
        $1/2^8$ & 2.46E-7 & 3.95E-8 & 5.55E-9 & 1.92E-9 & 1.94E-9 &  1.87E-7 & 1.14E-7 & 3.33E-8  & 5.76E-9 & 8.05E-10       \\
        order &  -- &  2.64  & 2.83 & 1.53 & -0.01 &  -- & 0.72  & 1.77 & 2.53 & 2.84    \\
       \bottomrule
    \end{tabular}%
    }
\end{table}

\begin{figure}[t!]
    \centering
        \includegraphics[width=0.32\linewidth,height=3.1cm]{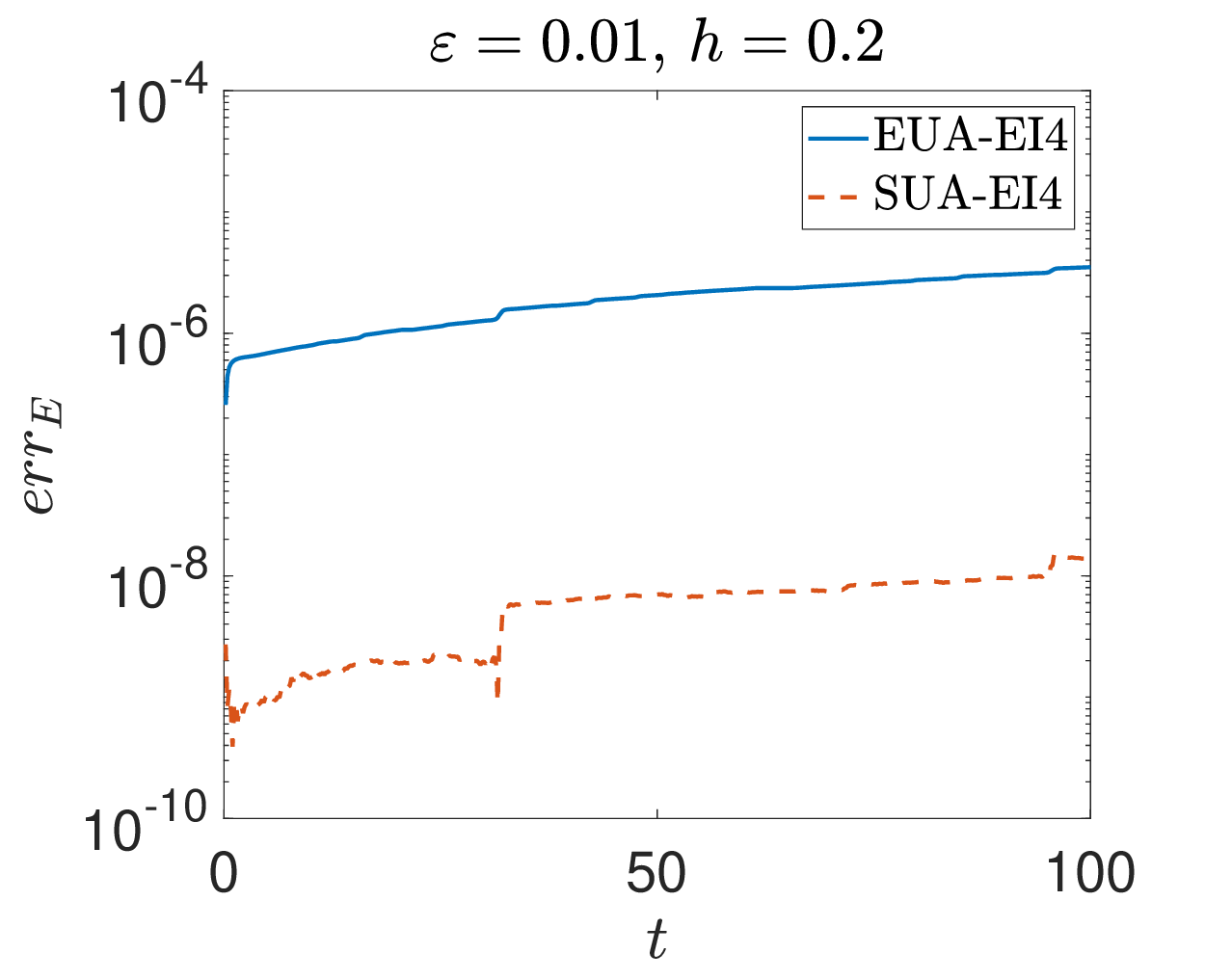}
        \includegraphics[width=0.32\linewidth,height=3.1cm]{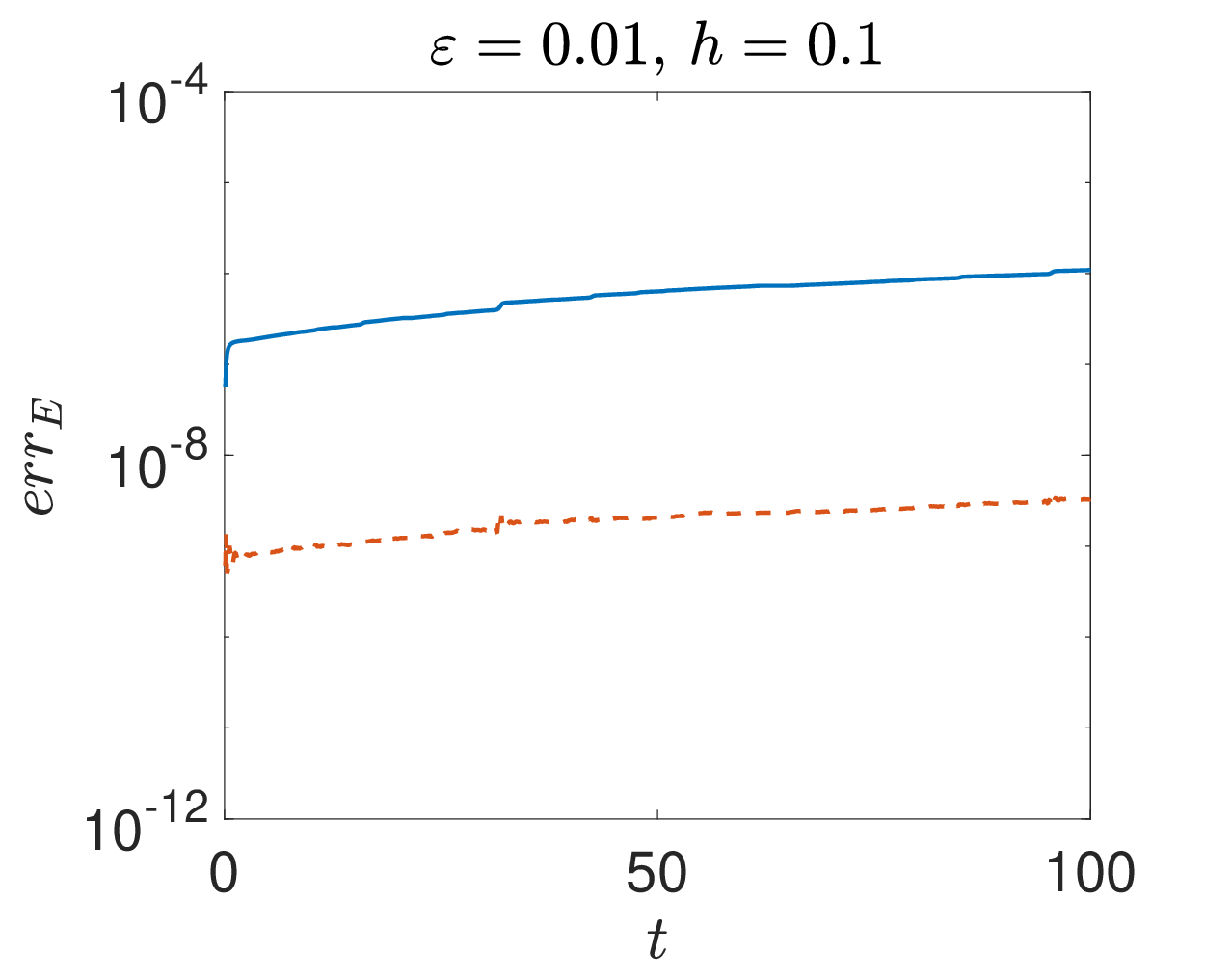}
        \includegraphics[width=0.32\linewidth,height=3.1cm]{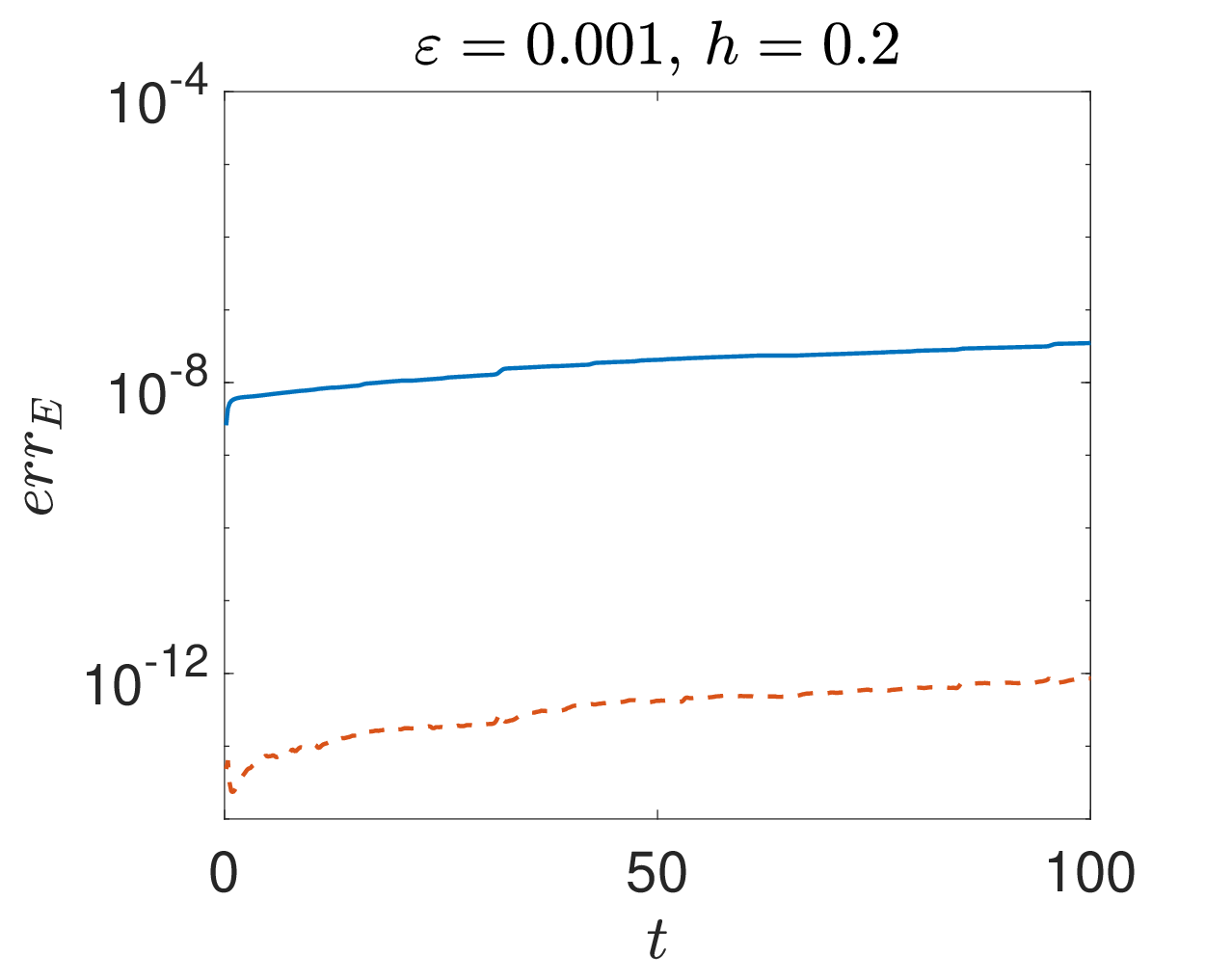}\\
        \includegraphics[width=0.32\linewidth,height=3.1cm]{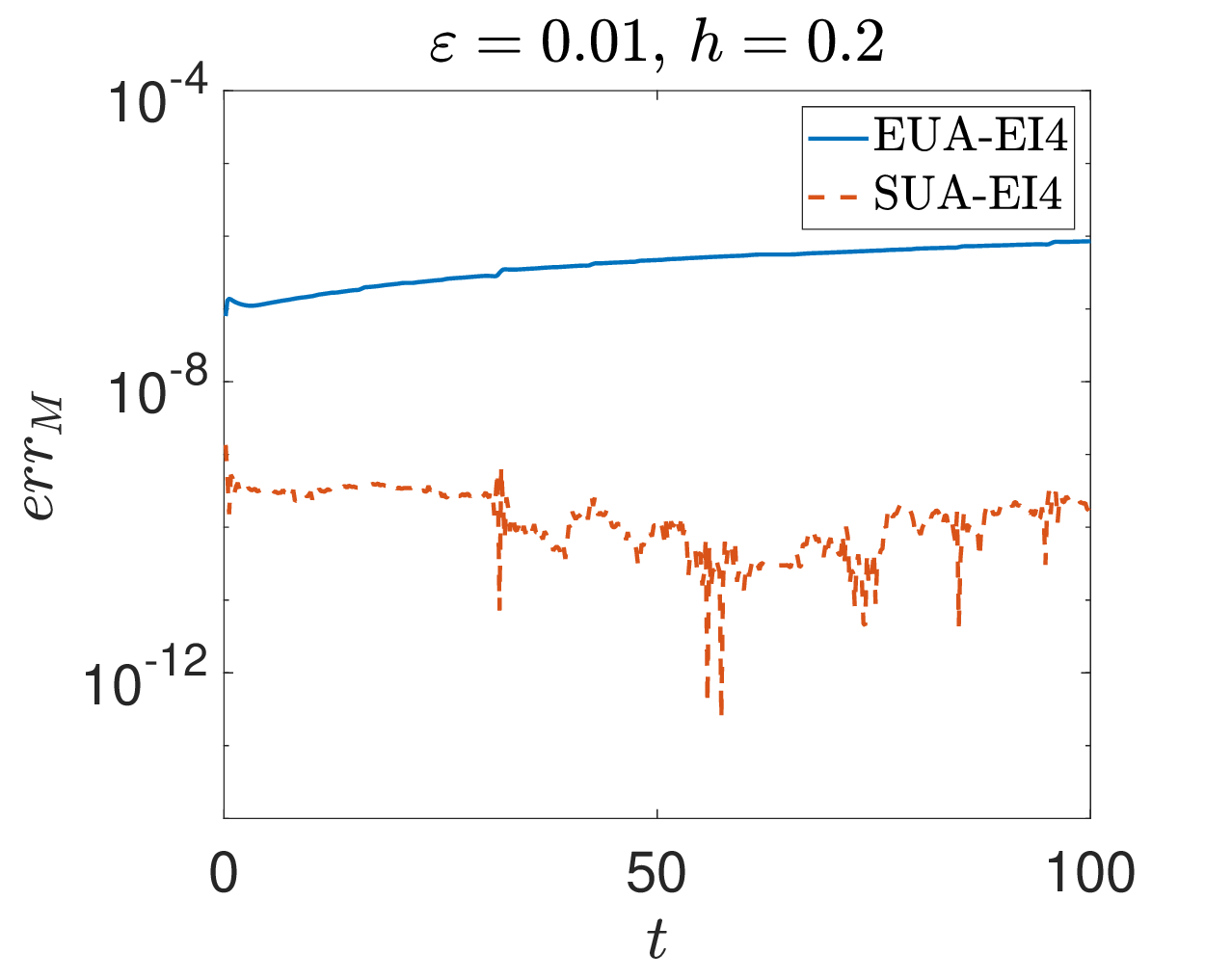}
        \includegraphics[width=0.32\linewidth,height=3.1cm]{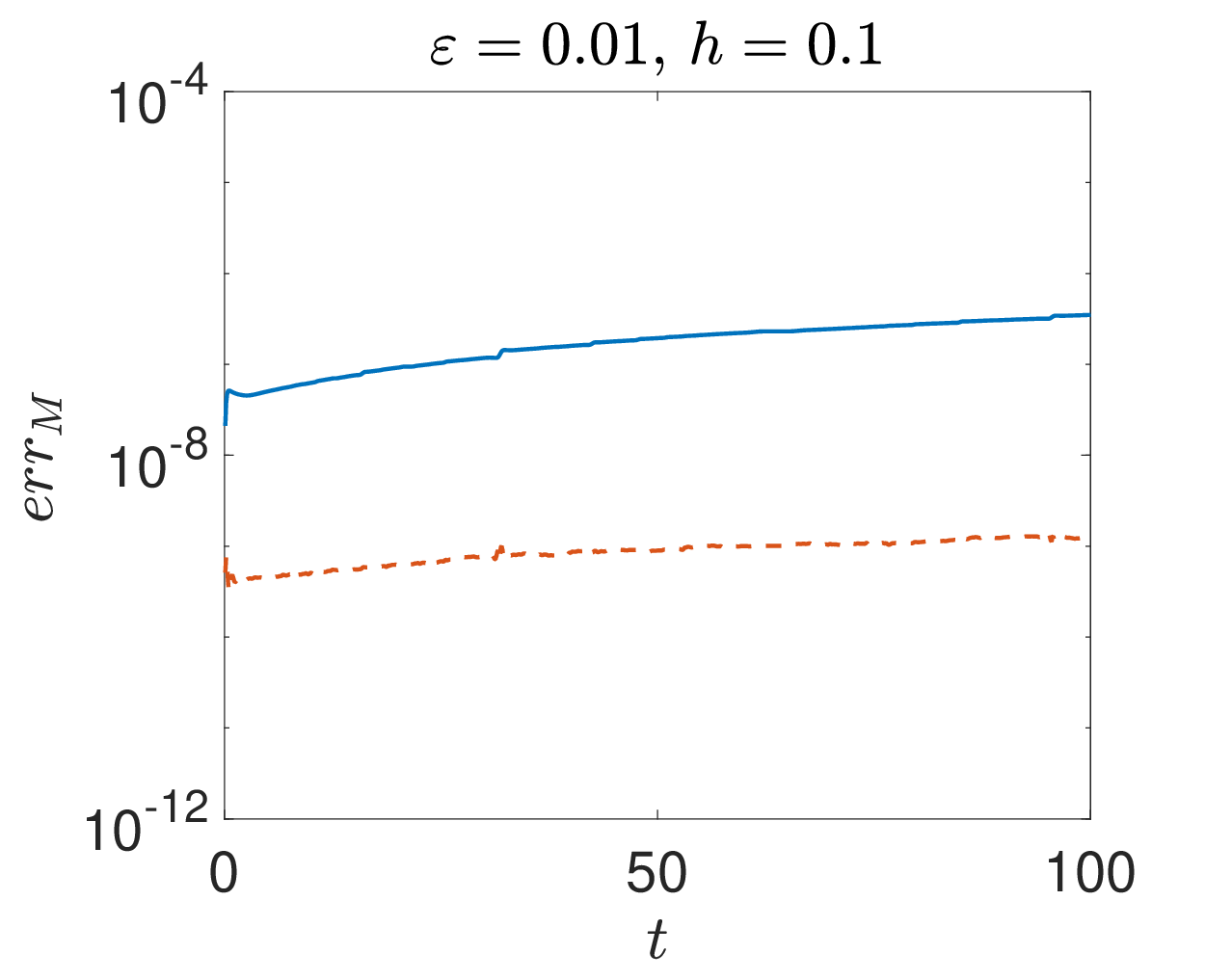}
        \includegraphics[width=0.32\linewidth,height=3.1cm]{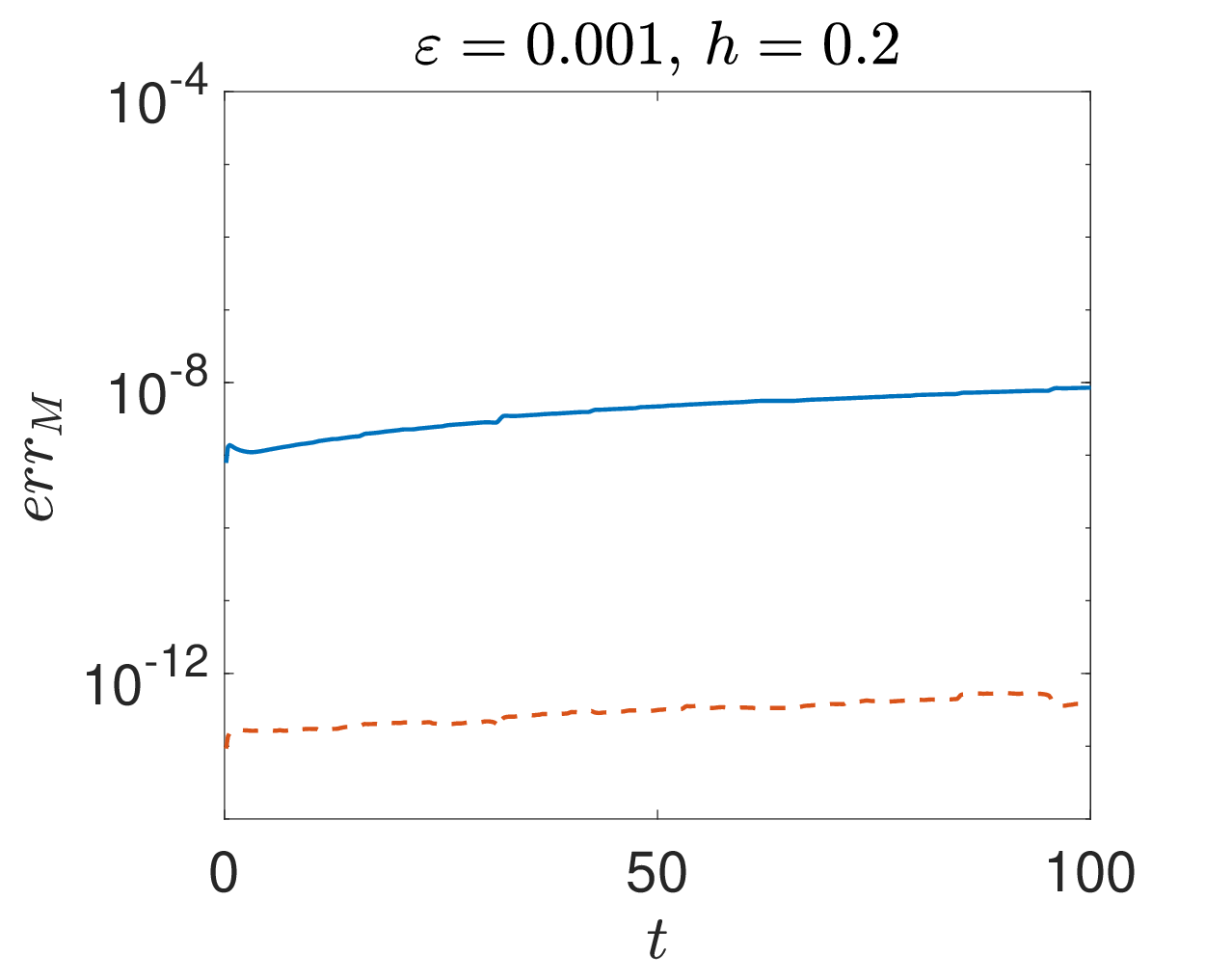}
    \caption{Energy error (top) and mass error (bottom) of KGS in 2D under different $\eps$ and $h$.}
    \label{fig-5-5-3}
\end{figure}

\begin{figure}[t!]
    \centering
        \includegraphics[width=0.24\linewidth,height=3.2cm]{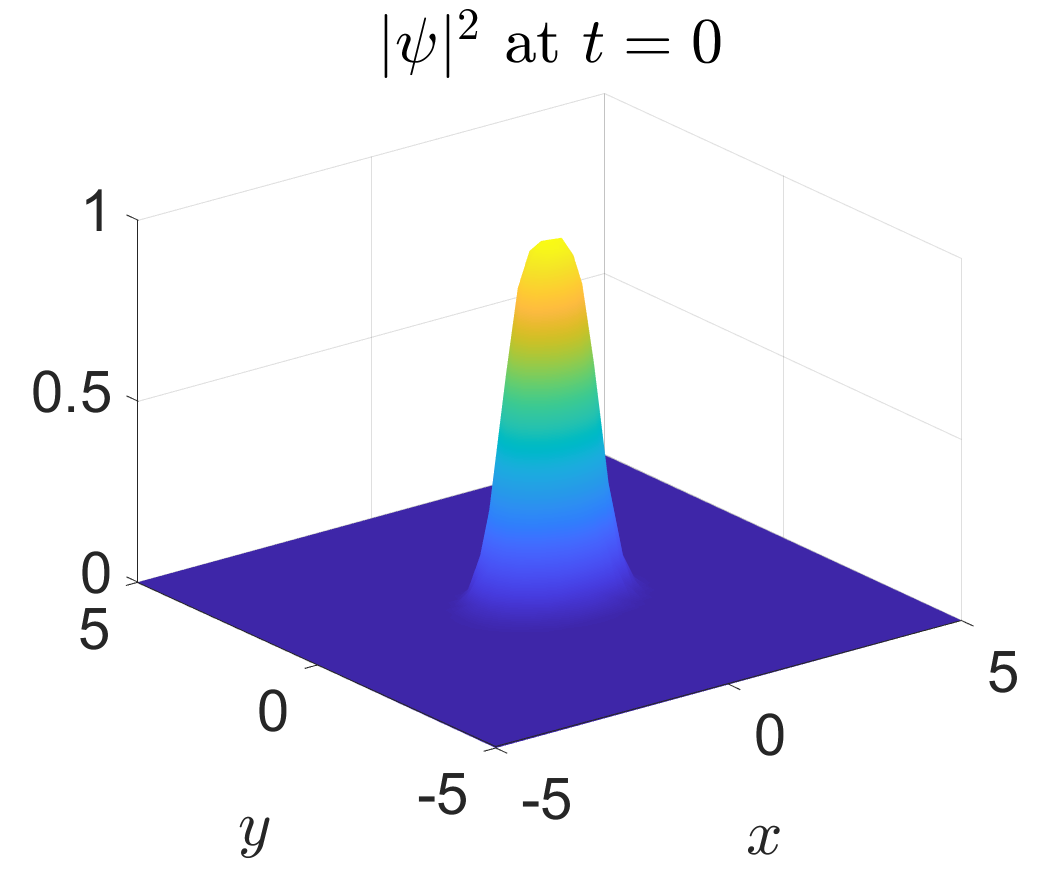}
        \includegraphics[width=0.24\linewidth,height=3.2cm]{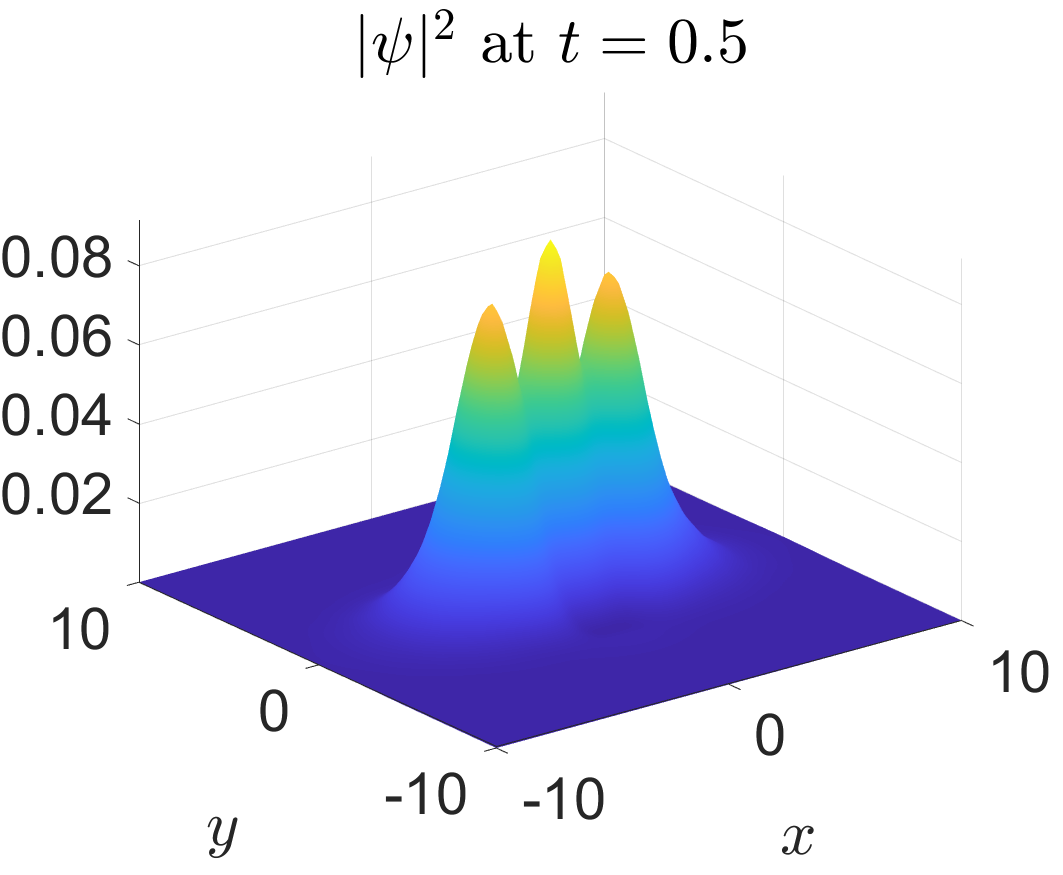}
        \includegraphics[width=0.24\linewidth,height=3.2cm]{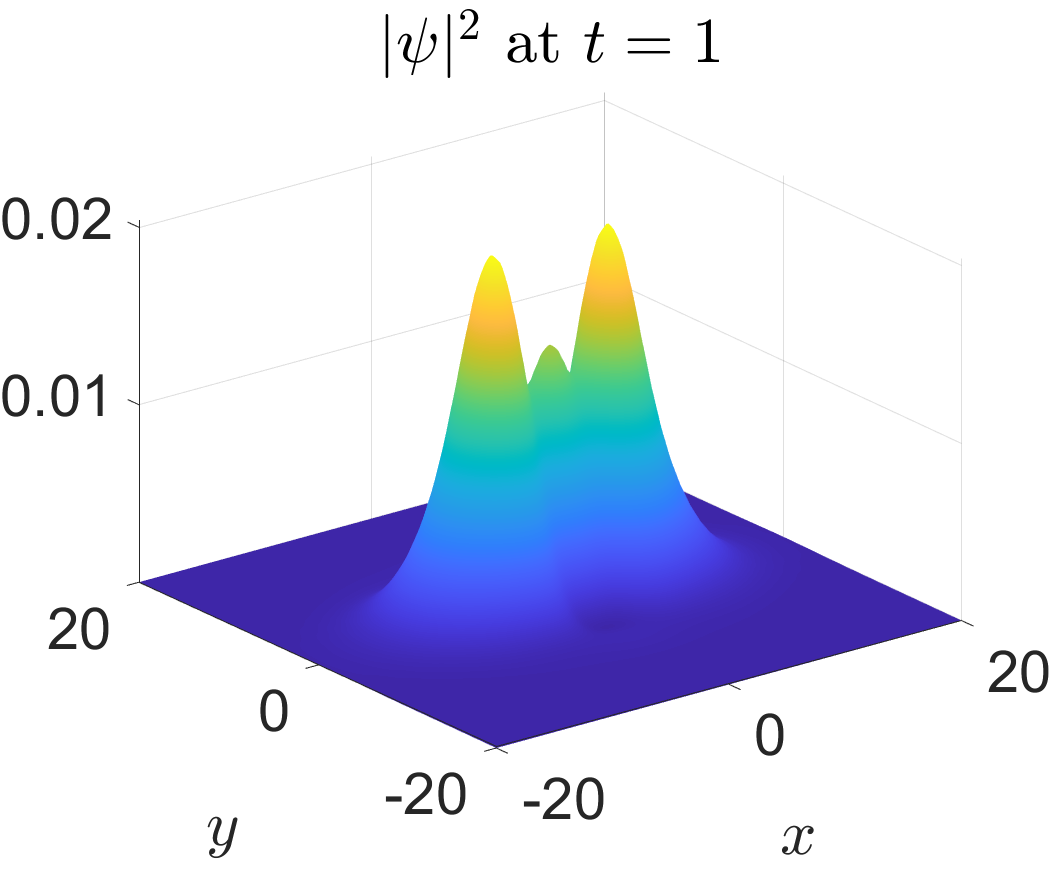}
        \includegraphics[width=0.24\linewidth,height=3.2cm]{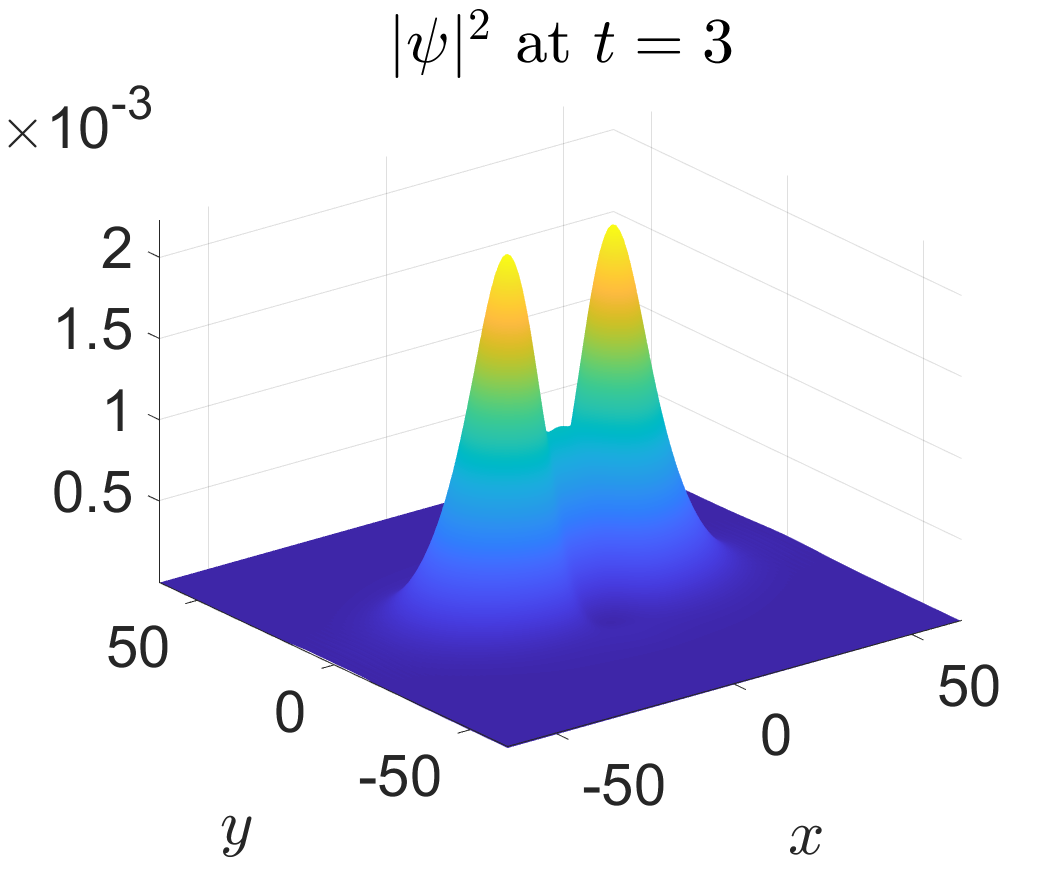}\\
        \includegraphics[width=0.24\linewidth,height=3.2cm]{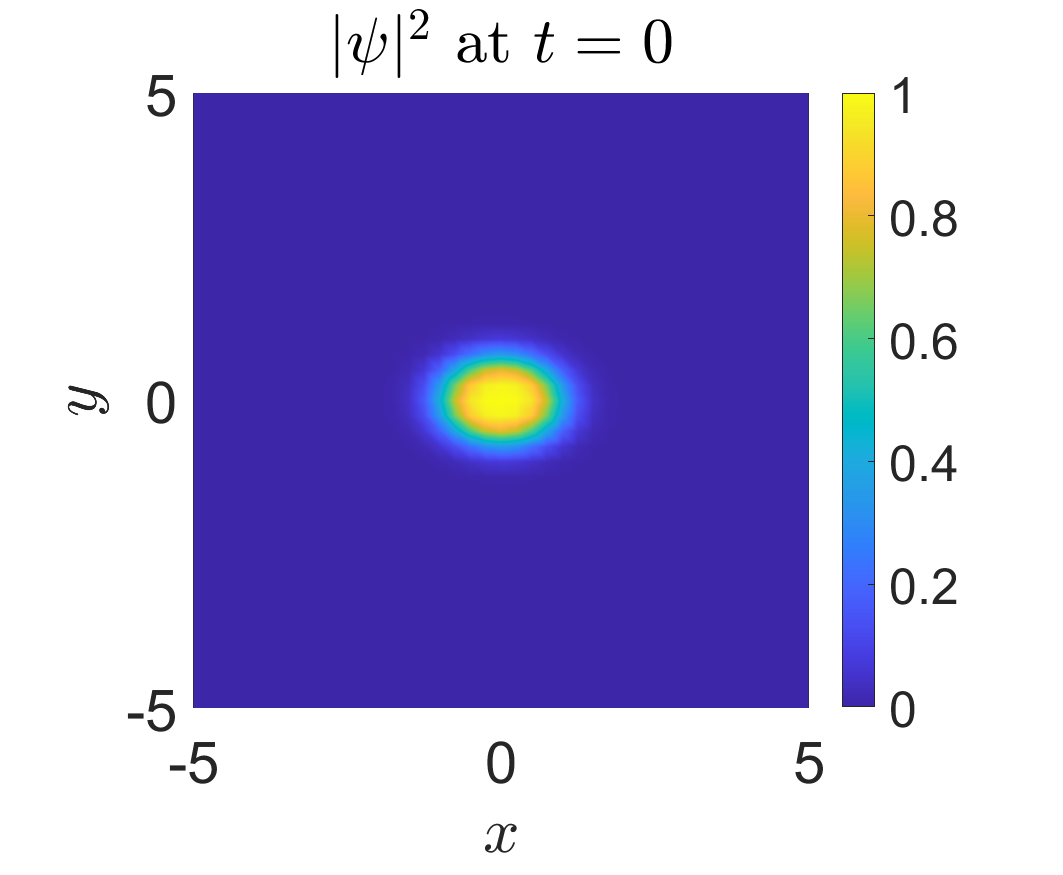}
        \includegraphics[width=0.24\linewidth,height=3.2cm]{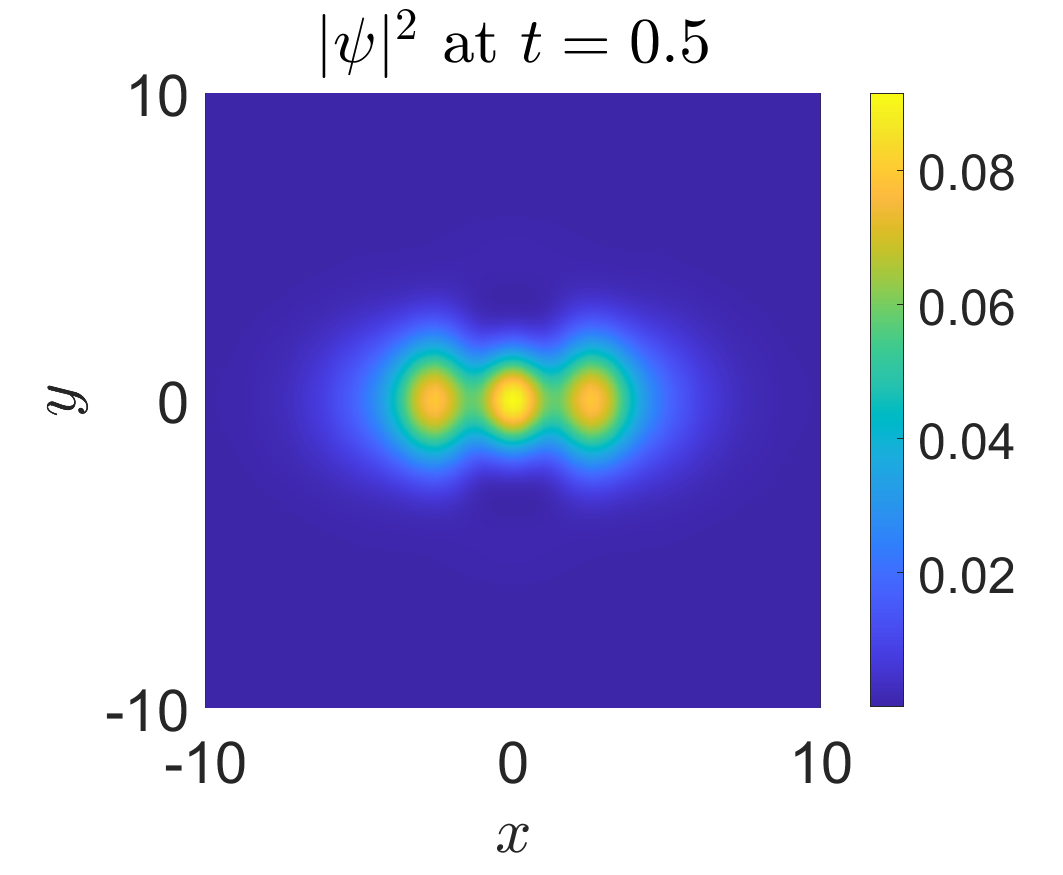}
        \includegraphics[width=0.24\linewidth,height=3.2cm]{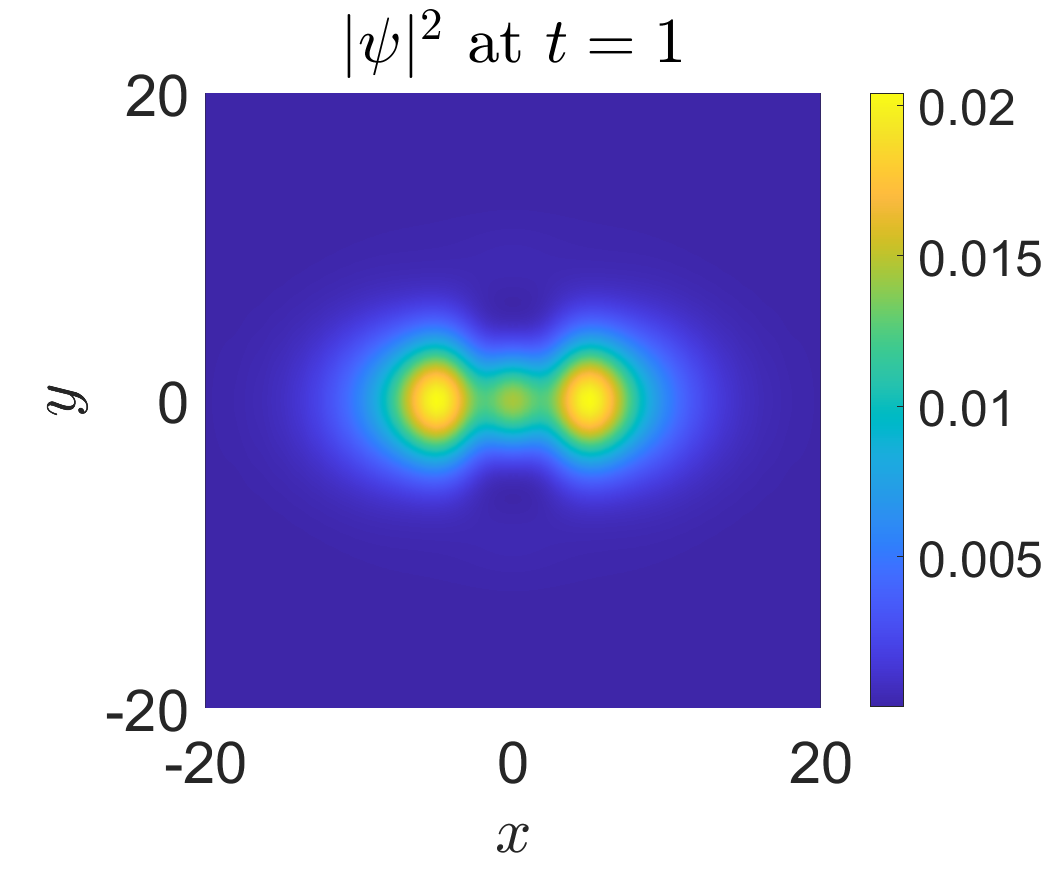}
        \includegraphics[width=0.24\linewidth,height=3.2cm]{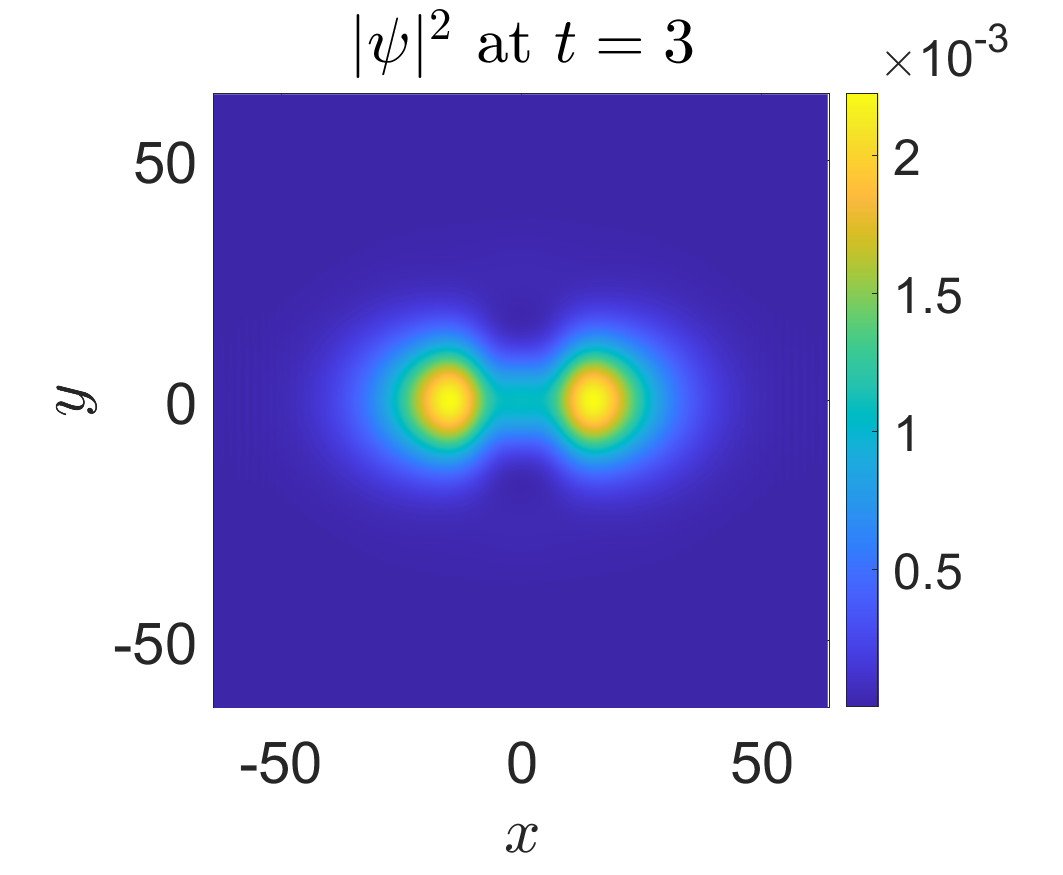}
    \caption{Surface plots (top) and snapshots (bottom) of the nucleon density $\abs{\psi}^{2}$ for the KGS in 2D at different times with $\eps=0.01$.}
    \label{fig-5-5-4}
\end{figure}

\begin{figure}[t!]
    \centering
        \includegraphics[width=0.24\linewidth,height=3.2cm]{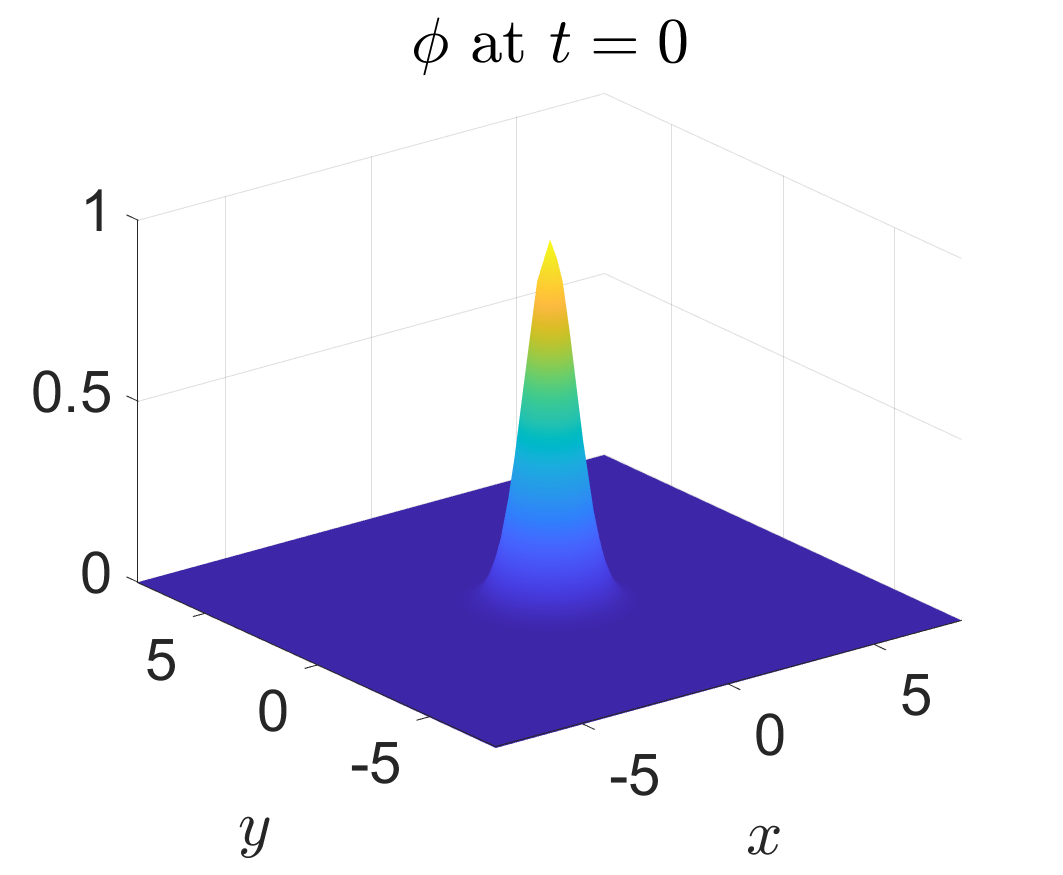}
        \includegraphics[width=0.24\linewidth,height=3.2cm]{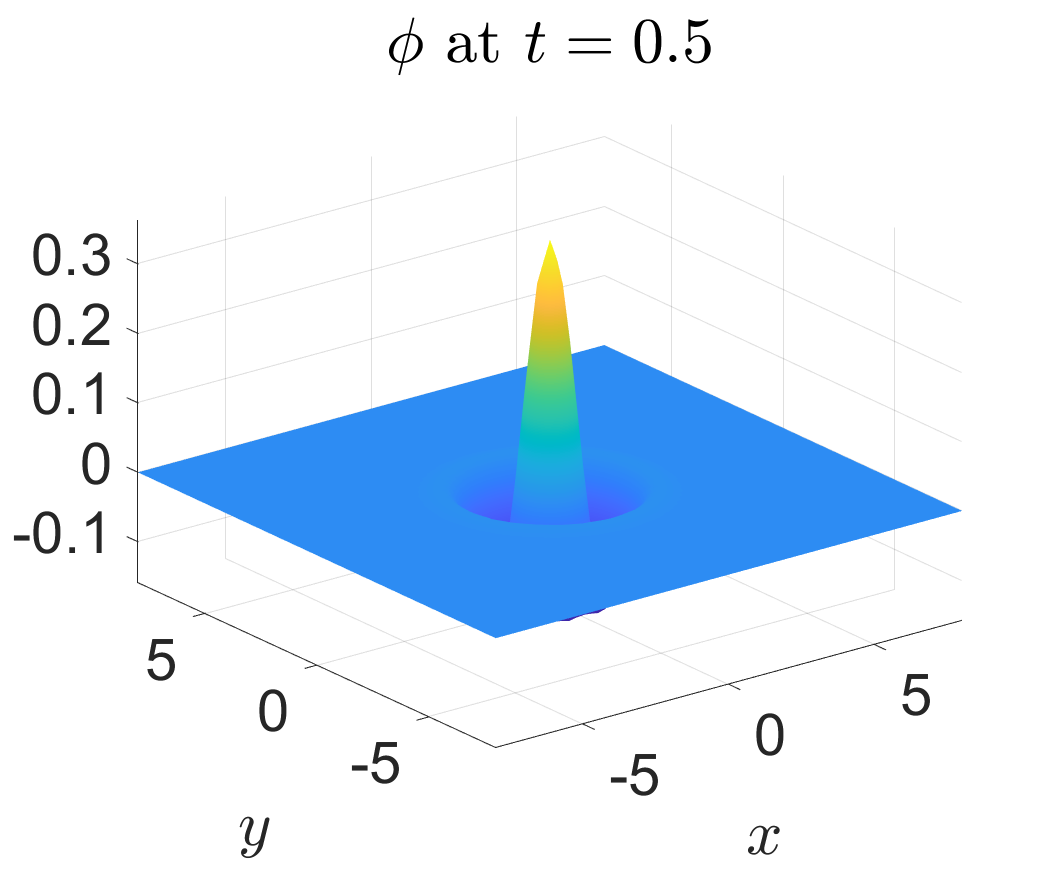}
        \includegraphics[width=0.24\linewidth,height=3.2cm]{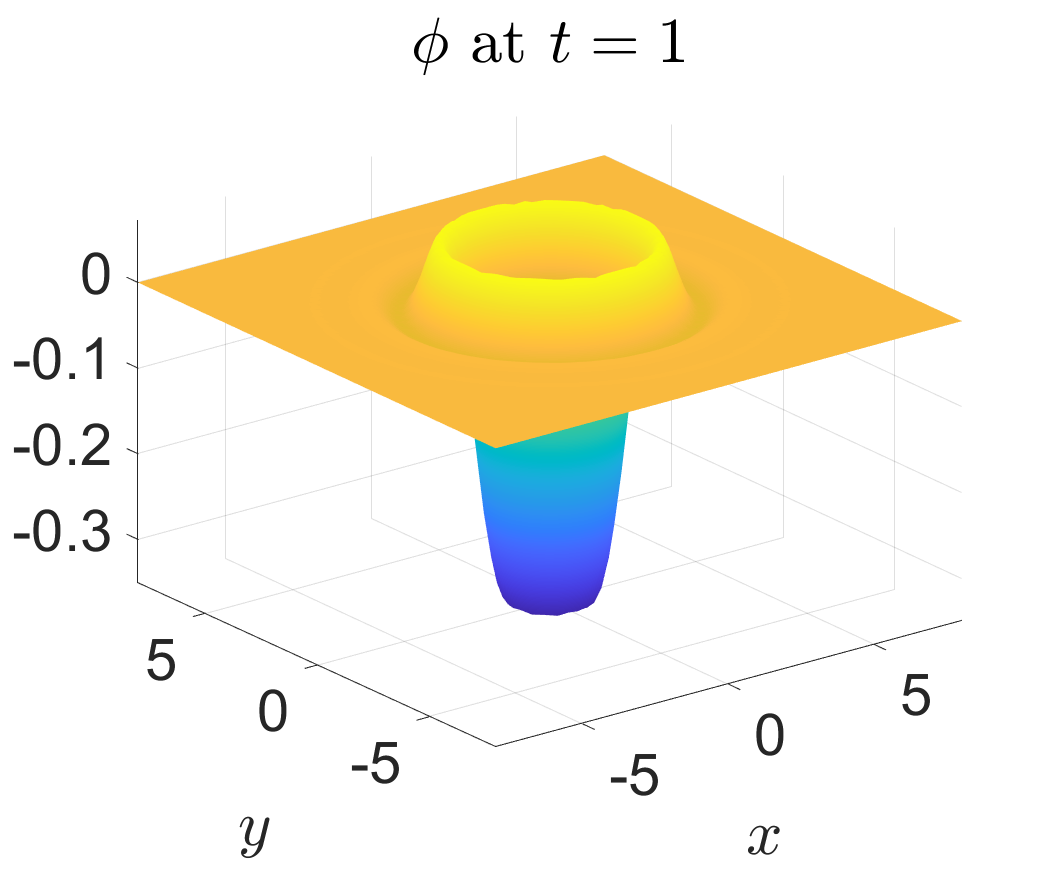}
        \includegraphics[width=0.24\linewidth,height=3.2cm]{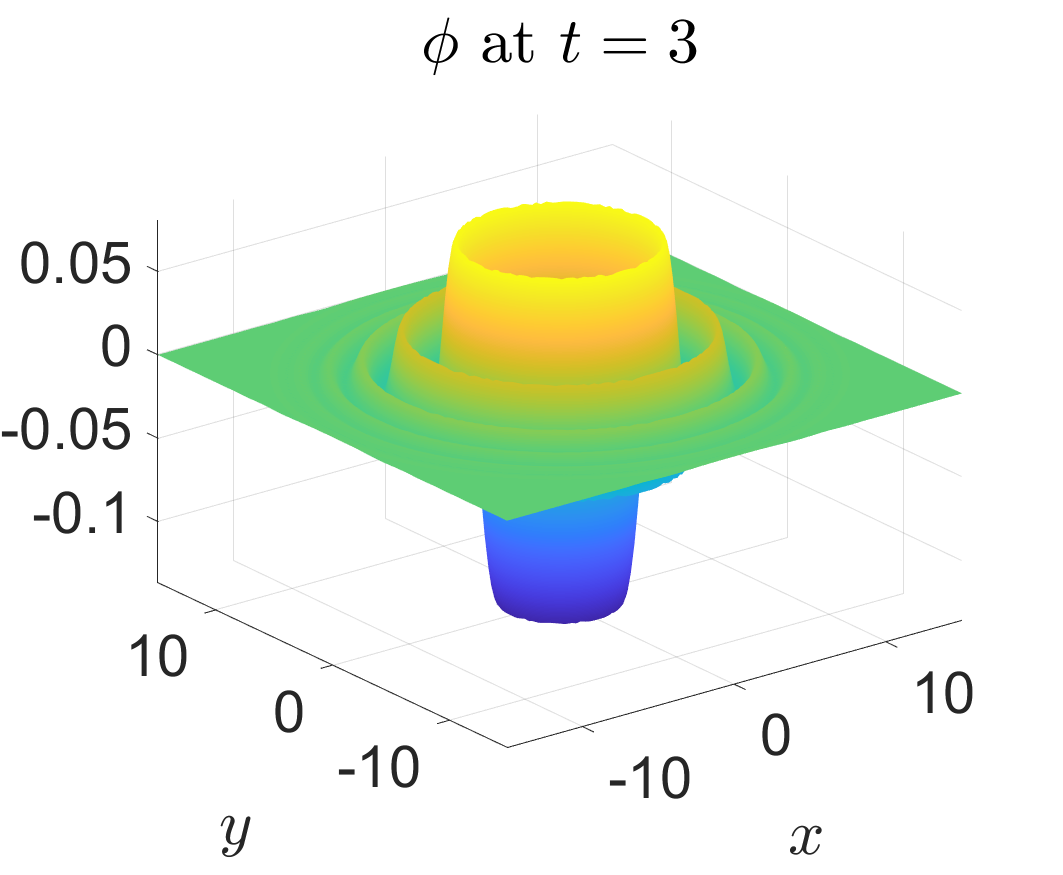}\\
        \includegraphics[width=0.24\linewidth,height=3.2cm]{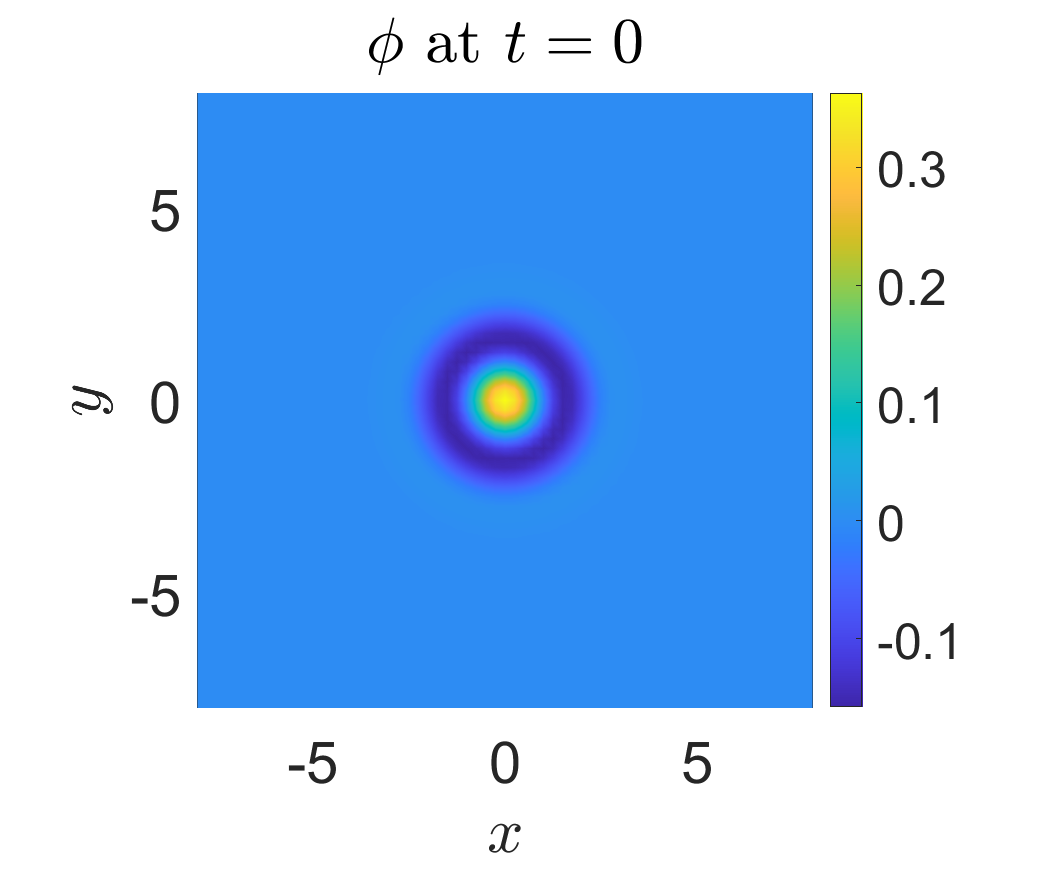}
        \includegraphics[width=0.24\linewidth,height=3.2cm]{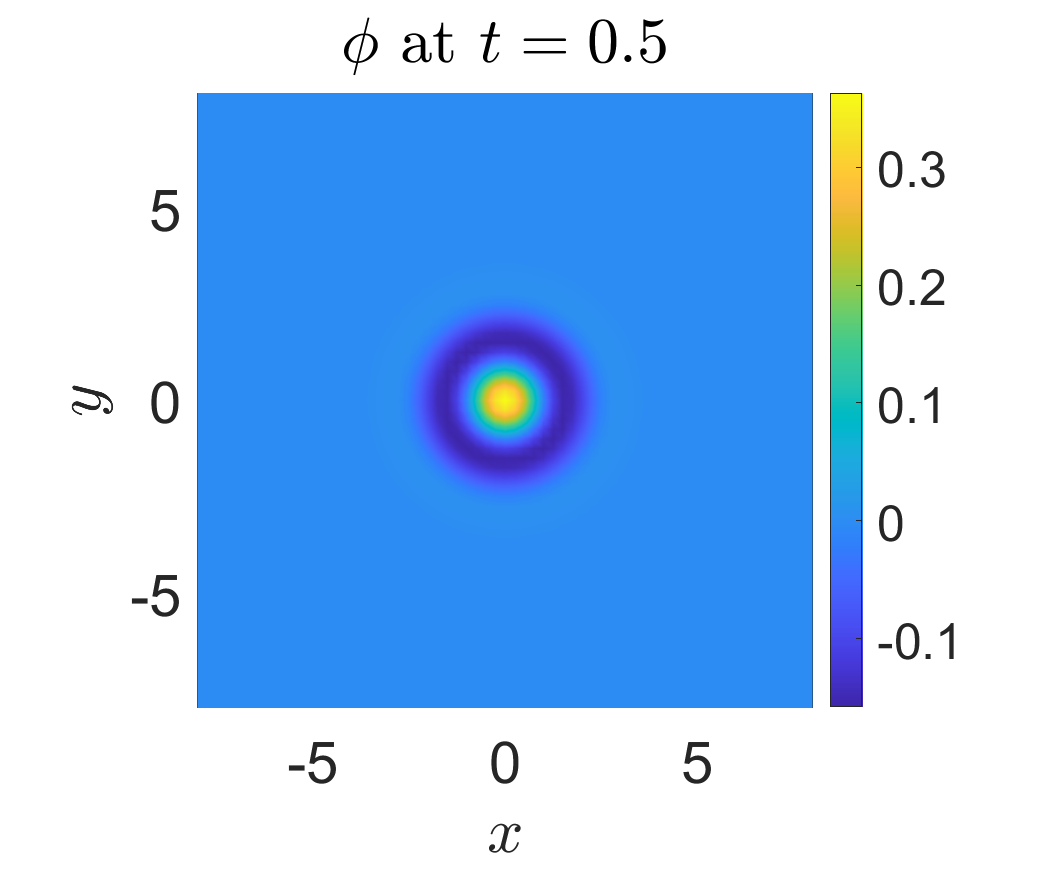}
        \includegraphics[width=0.24\linewidth,height=3.2cm]{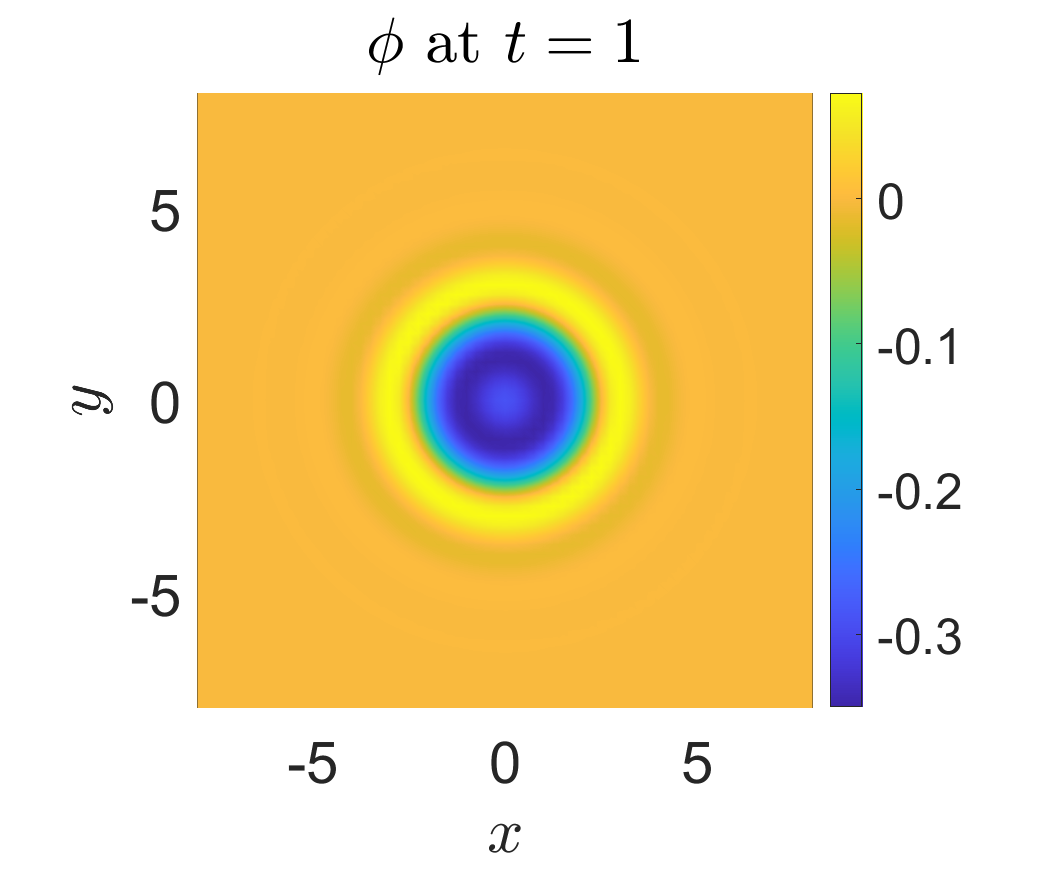}
        \includegraphics[width=0.24\linewidth,height=3.2cm]{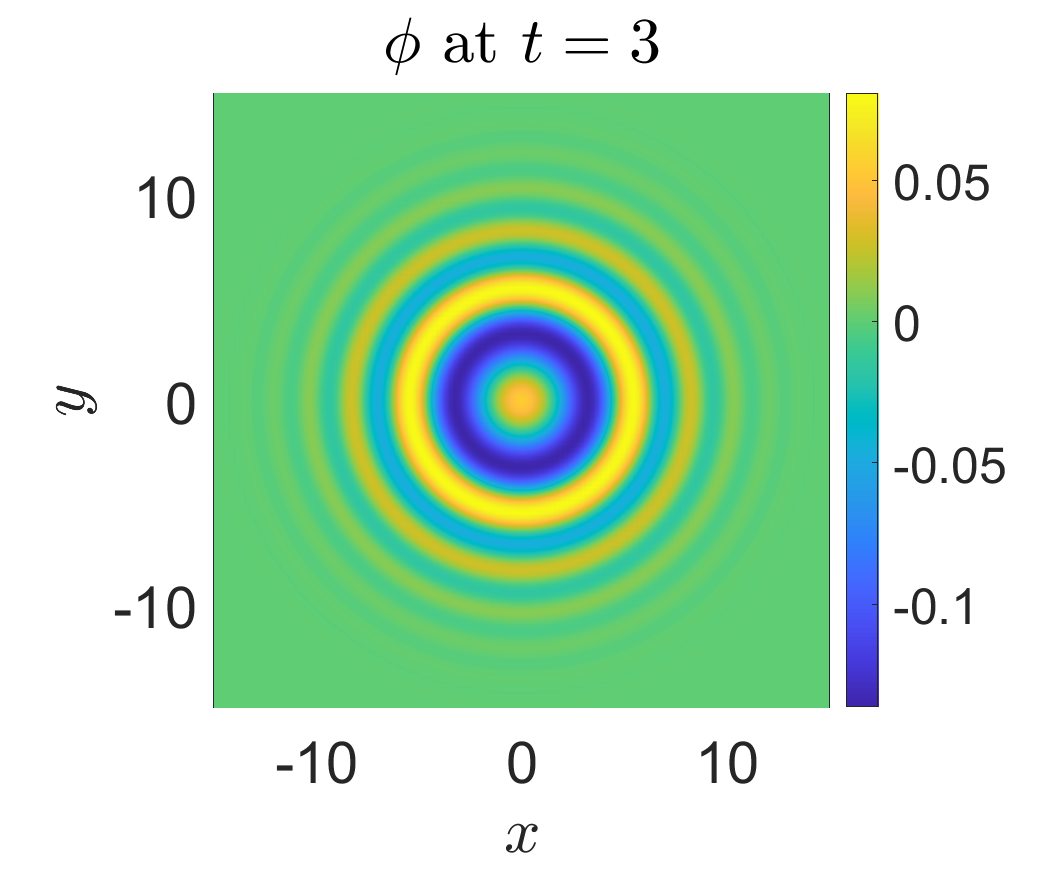}
    \caption{Surface plots (top) and snapshots (bottom) of the meson field $\phi$ for the KGS in 2D at different times with $\eps=0.01$.}
    \label{fig-5-5-5}
\end{figure}

\subsubsection{Three-dimensional KGS equation}
Finally, we evaluate the collisions of elliptical vector solitons in the 3D KGS
equation. We take the bounded domain $\Omega=(-5,5)^{3}$ with mesh sizes
$\triangle x=\triangle y=\triangle z=\frac{1}{5}$ and $\triangle\tau=\frac{1}{16}$,
and time step $h=0.01$, using the EUA-EI4 scheme. The initial data are taken
from \cite{GLLM}:
\begin{equation*}
\begin{aligned}
& \psi_{0}=\frac{(1+\ii)}{10}\sum\limits_{k=0}^{1}e^{-(x-(-1)^{k}1.8)^{2}-(2y)^{2}-(2z)^{2}}, \quad \phi_{0}=\frac{1}{10}\sum\limits_{k=0}^{1}\sum\limits_{j=0}^{1}e^{1-\sqrt{\frac{(4(x-(-1)^{k}1.8))^2}{2}+(4y)^{2}+(4(z-(-1)^{k}1.8))^2}}, \\
&\phi_{1}=\frac{1}{10}\sum\limits_{k=0}^{1}\sum\limits_{j=0}^{1}\arctan\left(e^{1-\sqrt{\frac{(4(x-(-1)^{k}1.8))^2}{2}+(4y)^{2}+(4(z-(-1)^{k}1.8))^2}}\right).
\end{aligned}
\end{equation*}

We study the interactions of the solitons in the component $\abs{\psi}$ in
Figure \ref{fig-5-6-1}, where the first and second rows show the isosurfaces of
$\abs{\psi}=0.01$ and the corresponding slices at $y=0$, respectively. Figure
\ref{fig-5-6-1}(a) shows four separated ellipsoids at $t=0$. As time evolves,
these ellipsoids deform and begin to connect, as illustrated in
Figure \ref{fig-5-6-1}(b). At $t=0.25$, they form a connected, four-lobed
structure with a small opening near the center, as shown in
Figure \ref{fig-5-6-1}(c). By $t=0.4$, the isosurface has developed into a more
complex, multi-lobed structure, as shown in Figure \ref{fig-5-6-1}(d). The
slices in Figures \ref{fig-5-6-1}(e)--(h) show four initially separated
localized regions that gradually broaden and overlap, with additional structures
emerging between them. These observations reveal the complex spatial patterns
that develop during the interactions.

Figure \ref{fig-5-6-2} presents the interactions of the solitons in the
component $\phi$, with the isosurfaces $\phi=0.005$ and the corresponding slices
at $y=0$ displayed in the first and second rows, respectively. At $t=0$, the
isosurface $\phi=0.005$ consists of four separated ellipsoids, as shown in
Figure \ref{fig-5-6-2}(a). These ellipsoids collide and merge into a single
compact structure at $t=0.15$, as illustrated in Figure \ref{fig-5-6-2}(b).
Subsequently, the merged structure splits into five separate parts at $t=0.25$:
four surrounding ellipsoids and a smaller central one, as shown in
Figure \ref{fig-5-6-2}(c). By $t=0.4$, these structures collide and merge again
into a single irregular structure, as shown in Figure \ref{fig-5-6-2}(d). The
slices in Figures \ref{fig-5-6-2}(e)--(h) present the corresponding evolution of
the spatial profiles, with distinct changes in both the central region and the
four surrounding regions. These results reveal complex spatial patterns during
the successive processes of merger, separation, and renewed merger.

The temporal evolution of the mass and energy of the 3D KGS system is shown in
Figure \ref{fig-5-6-3}. The favorable long-time behavior of the mass and energy
confirms the effectiveness of the proposed UA-EI4 schemes in solving the KGS
equation \eqref{KGS}.

\begin{figure}[t!]
    \centering
    \begin{subfigure}[t]{0.245\textwidth}
        \centering
        \includegraphics[width=\linewidth, height=3.2cm]{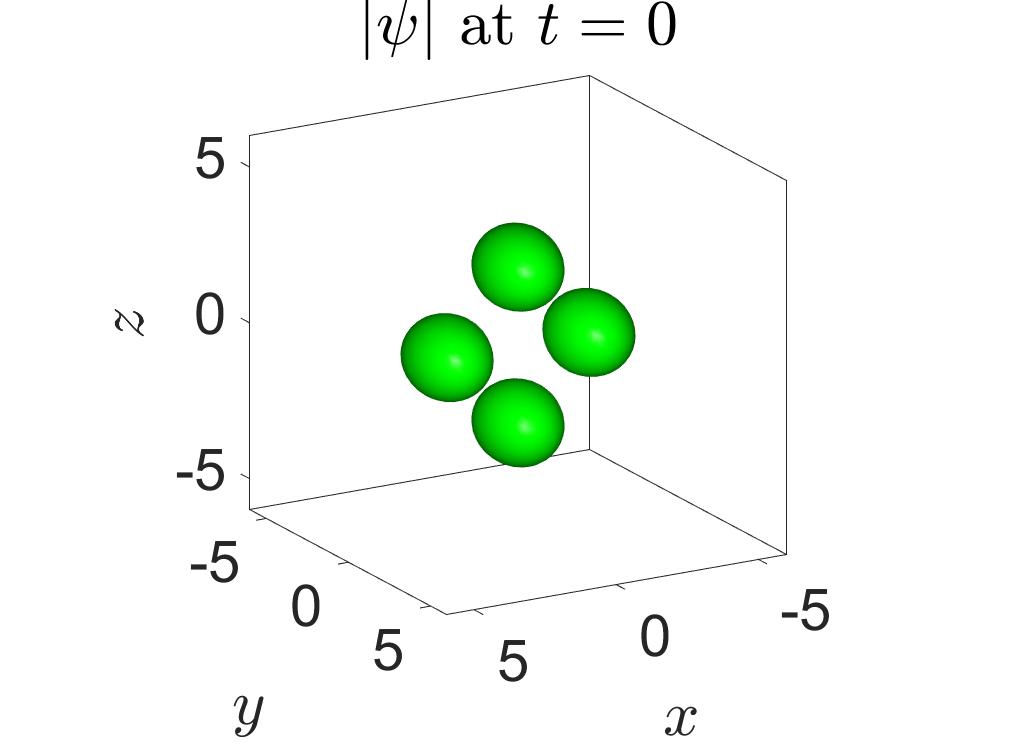}
        \caption{}
    \end{subfigure}
    \hfill
    \begin{subfigure}[t]{0.245\textwidth}
        \centering
        \includegraphics[width=\linewidth, height=3.2cm]{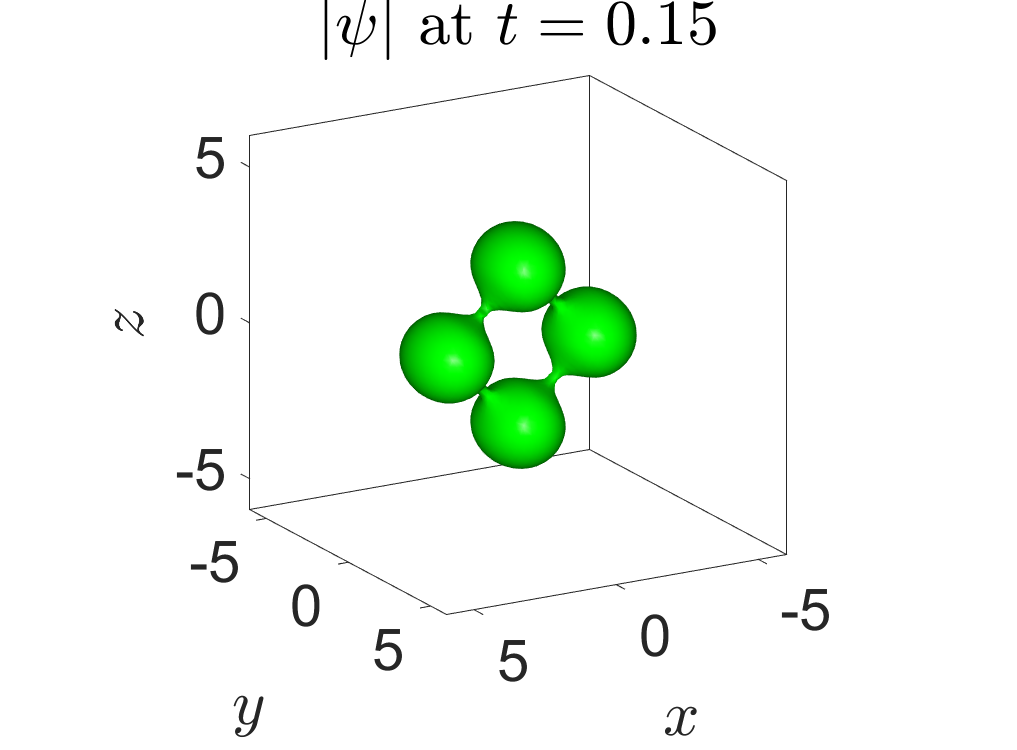}
        \caption{}
    \end{subfigure}
    \hfill
    \begin{subfigure}[t]{0.245\textwidth}
        \centering
        \includegraphics[width=\linewidth, height=3.2cm]{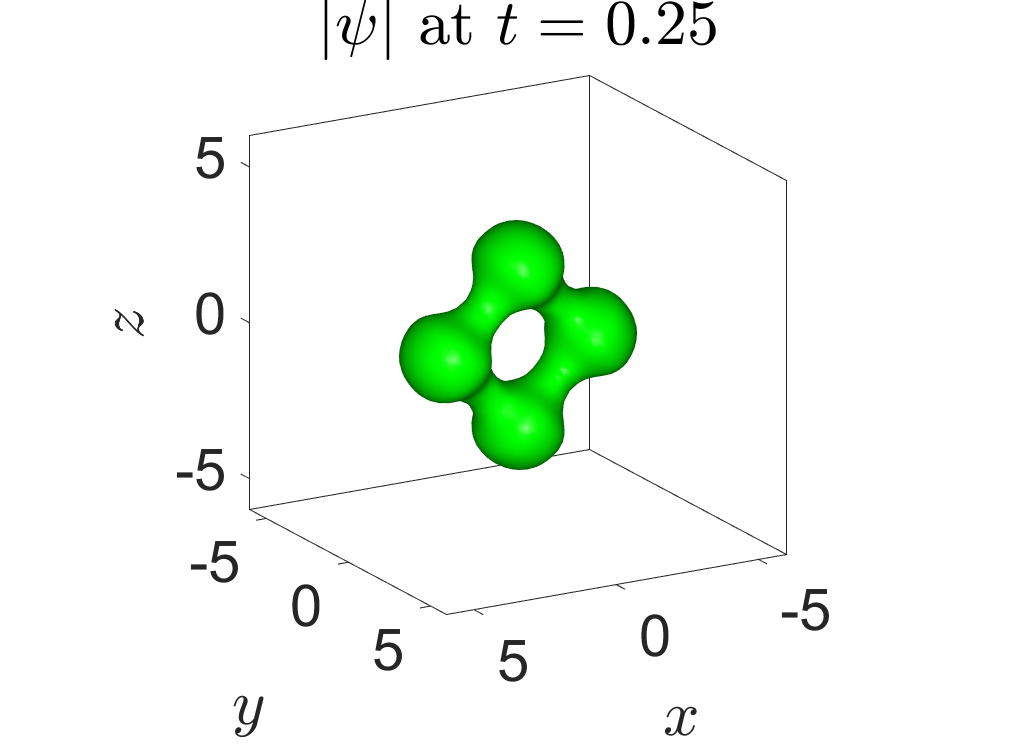}
        \caption{}
    \end{subfigure}
    \hfill
    \begin{subfigure}[t]{0.245\textwidth}
        \centering
        \includegraphics[width=\linewidth, height=3.2cm]{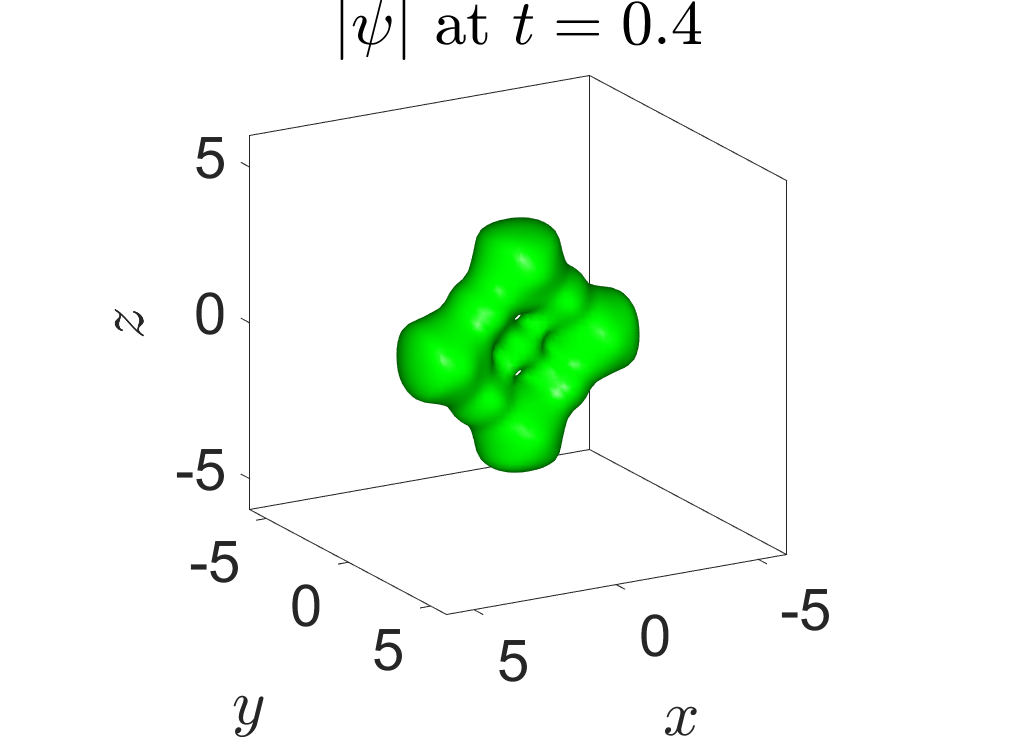}
        \caption{}
    \end{subfigure}
    \par
    \begin{subfigure}[t]{0.23\textwidth}
        \centering
        \includegraphics[width=\linewidth, height=3.2cm]{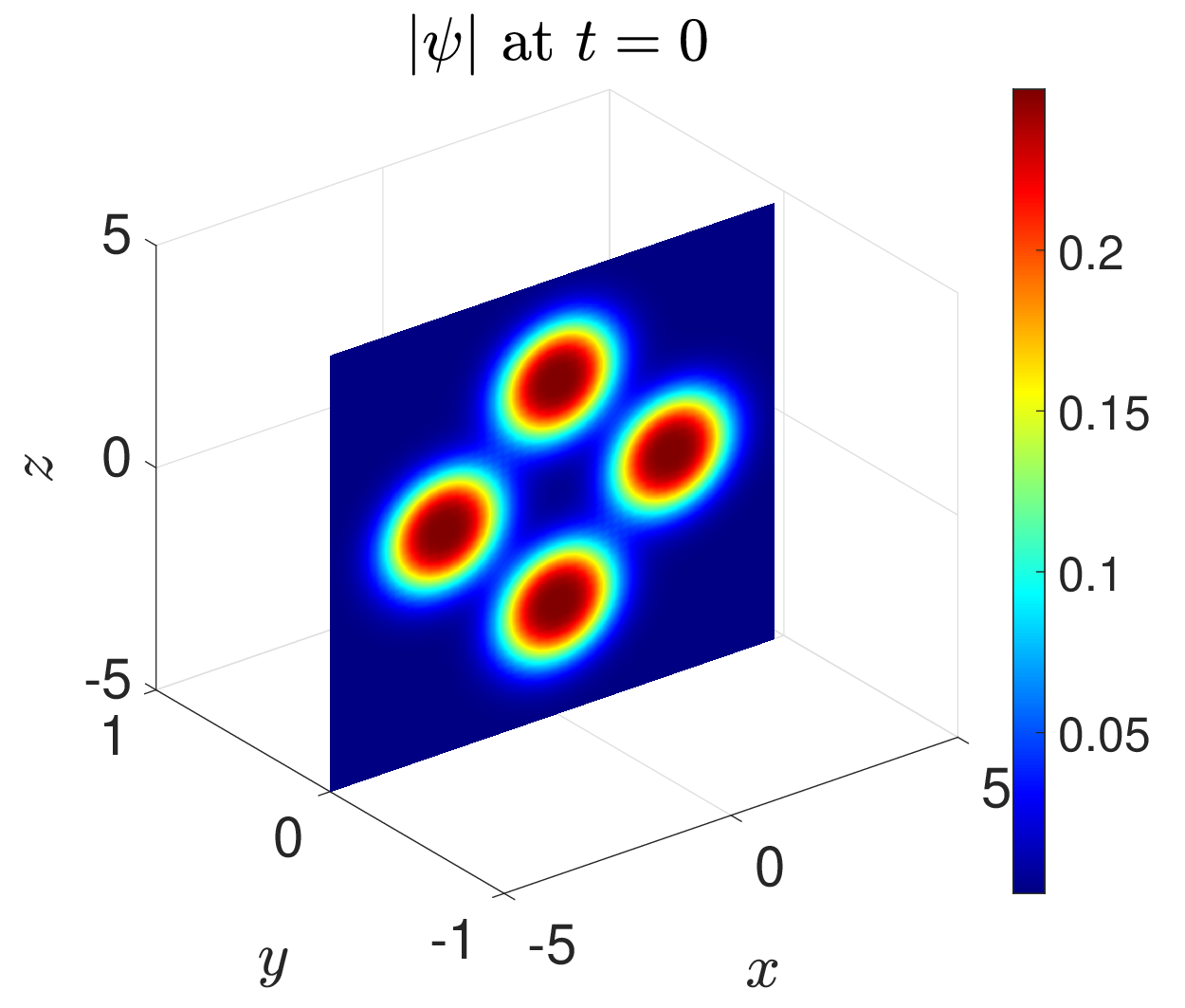}
        \caption{}
    \end{subfigure}
    \hfill
    \begin{subfigure}[t]{0.23\textwidth}
        \centering
        \includegraphics[width=\linewidth, height=3.2cm]{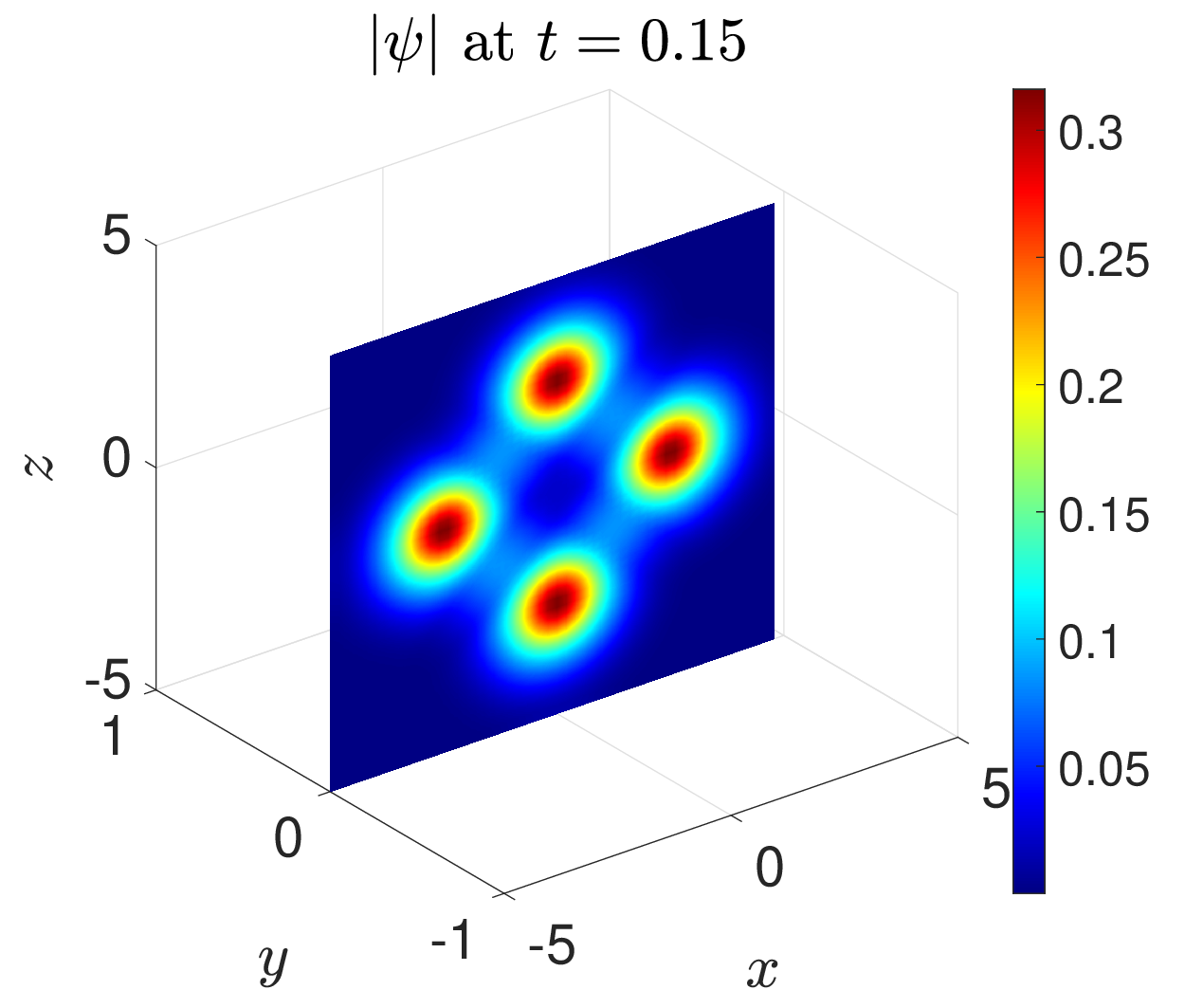}
        \caption{}
    \end{subfigure}
    \hfill
    \begin{subfigure}[t]{0.23\textwidth}
        \centering
        \includegraphics[width=\linewidth, height=3.2cm]{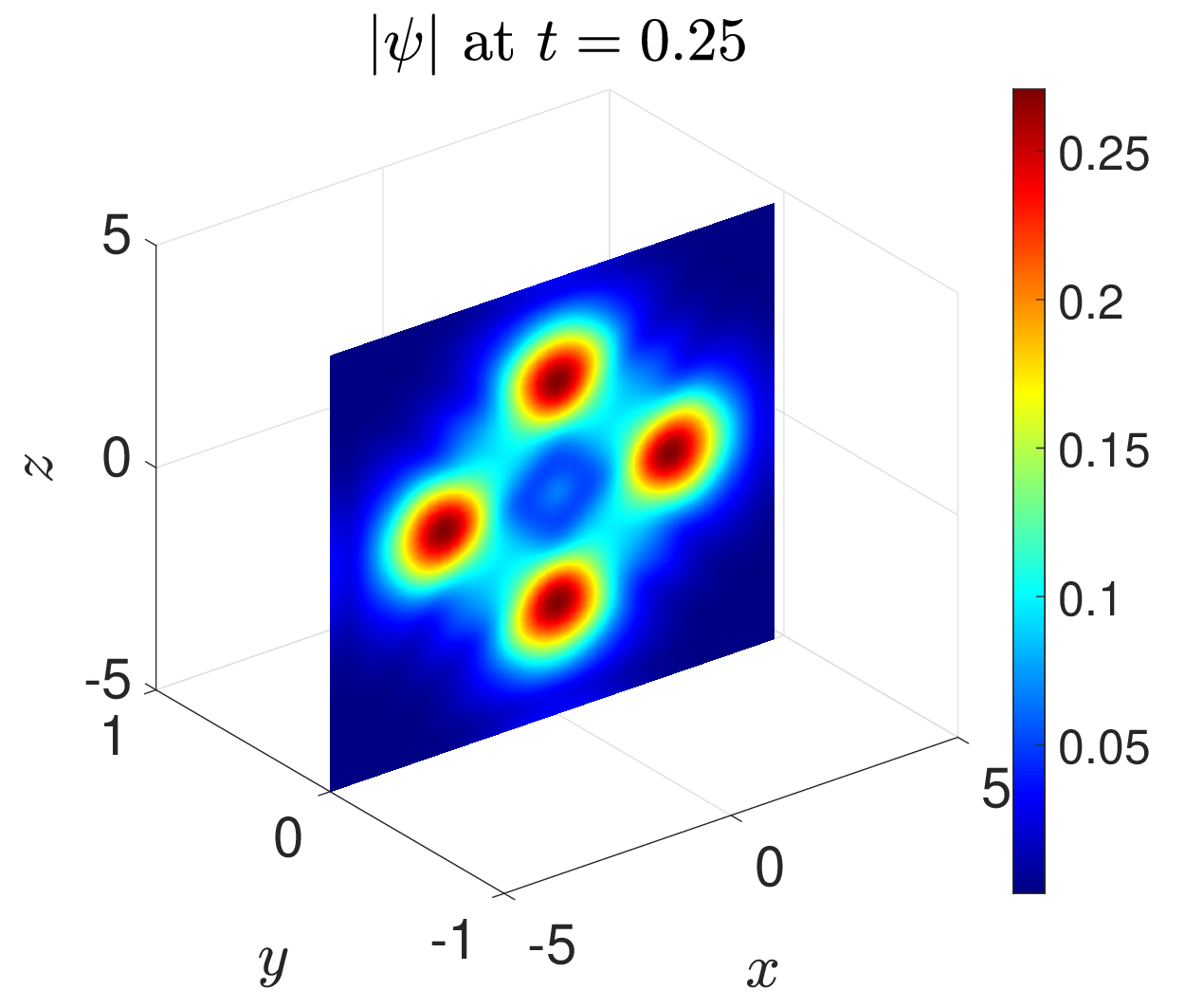}
        \caption{}
    \end{subfigure}
    \hfill
    \begin{subfigure}[t]{0.23\textwidth}
        \centering
        \includegraphics[width=\linewidth, height=3.2cm]{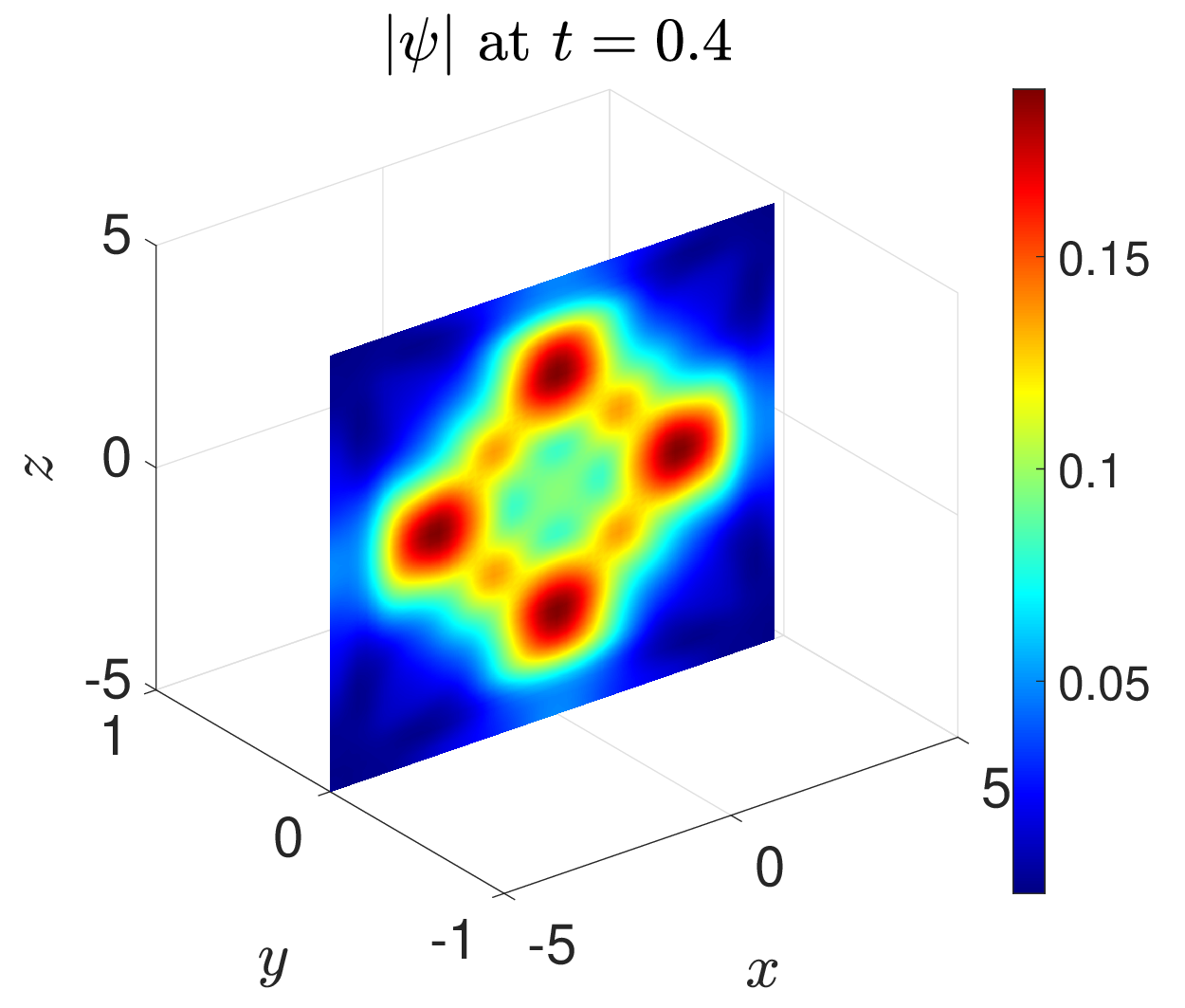}
        \caption{}
    \end{subfigure}
    \caption{The isosurfaces $\abs{\psi}=0.01$ (first row) and slices of $\abs{\psi}$ at $y=0$ (secong row) for KGS equation in 3D with $\mu=1$, $\lambda=0.8$ and $\eps=0.1$.}
    \label{fig-5-6-1}
\end{figure}

\begin{figure}[t!]
    \centering
    \begin{subfigure}[t]{0.245\textwidth}
        \centering
        \includegraphics[width=\linewidth, height=3.2cm]{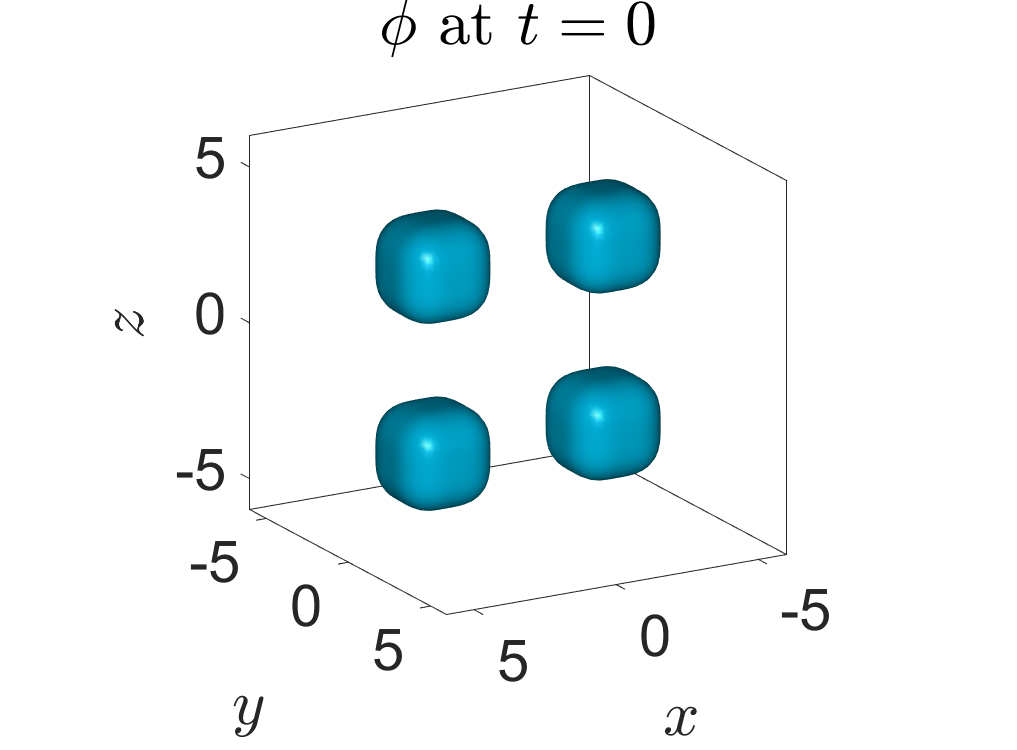}
        \caption{}
    \end{subfigure}
    \hfill
    \begin{subfigure}[t]{0.245\textwidth}
        \centering
        \includegraphics[width=\linewidth, height=3.2cm]{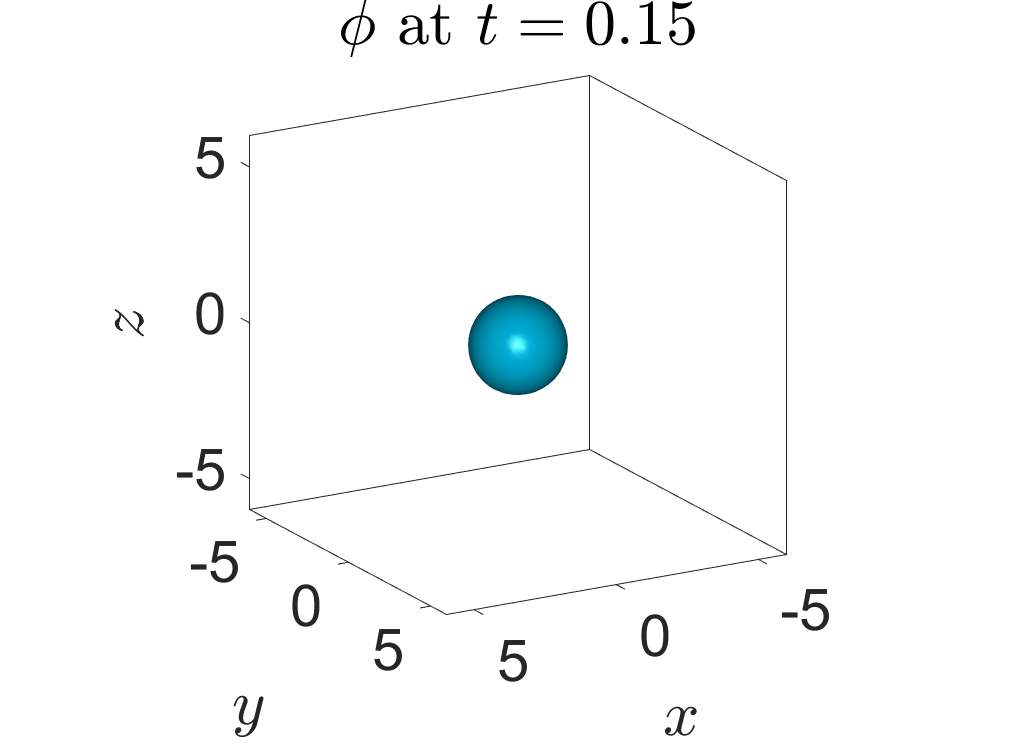}
        \caption{}
    \end{subfigure}
    \hfill
    \begin{subfigure}[t]{0.245\textwidth}
        \centering
        \includegraphics[width=\linewidth, height=3.2cm]{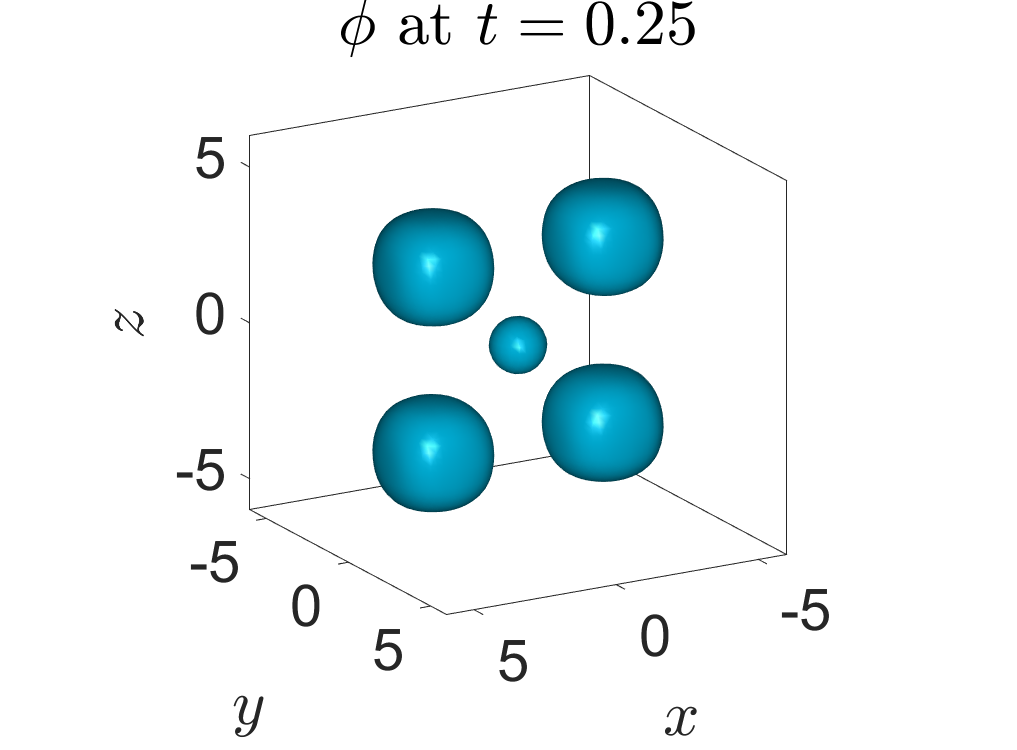}
        \caption{}
    \end{subfigure}
    \hfill
    \begin{subfigure}[t]{0.245\textwidth}
        \centering
        \includegraphics[width=\linewidth, height=3.2cm]{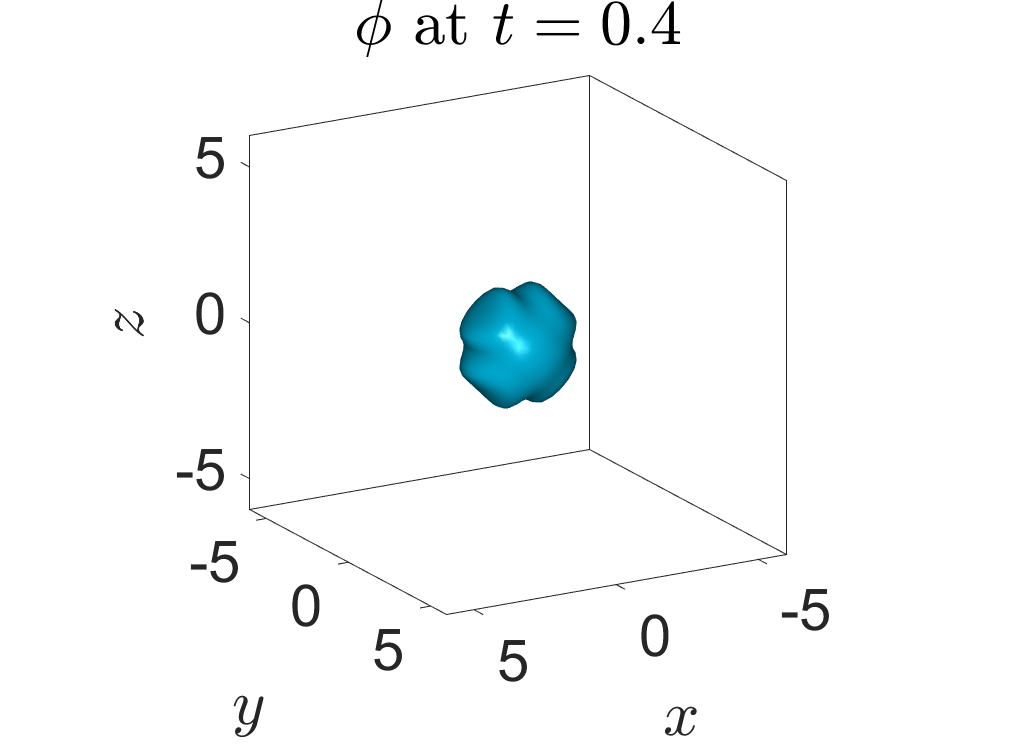}
        \caption{}
    \end{subfigure}
    \par
    \begin{subfigure}[t]{0.23\textwidth}
        \centering
        \includegraphics[width=\linewidth, height=3.2cm]{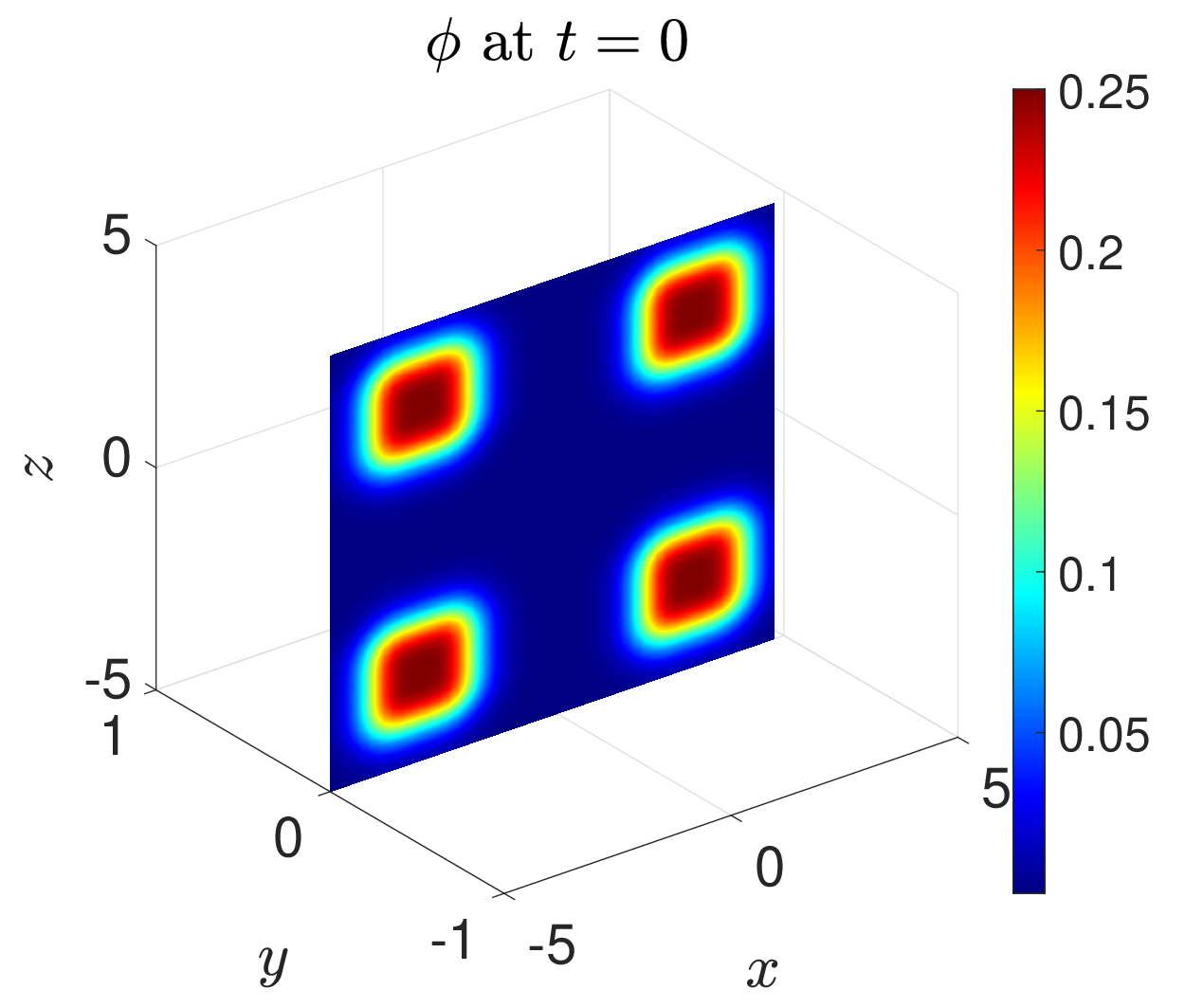}
        \caption{}
    \end{subfigure}
    \hfill
    \begin{subfigure}[t]{0.23\textwidth}
        \centering
        \includegraphics[width=\linewidth, height=3.2cm]{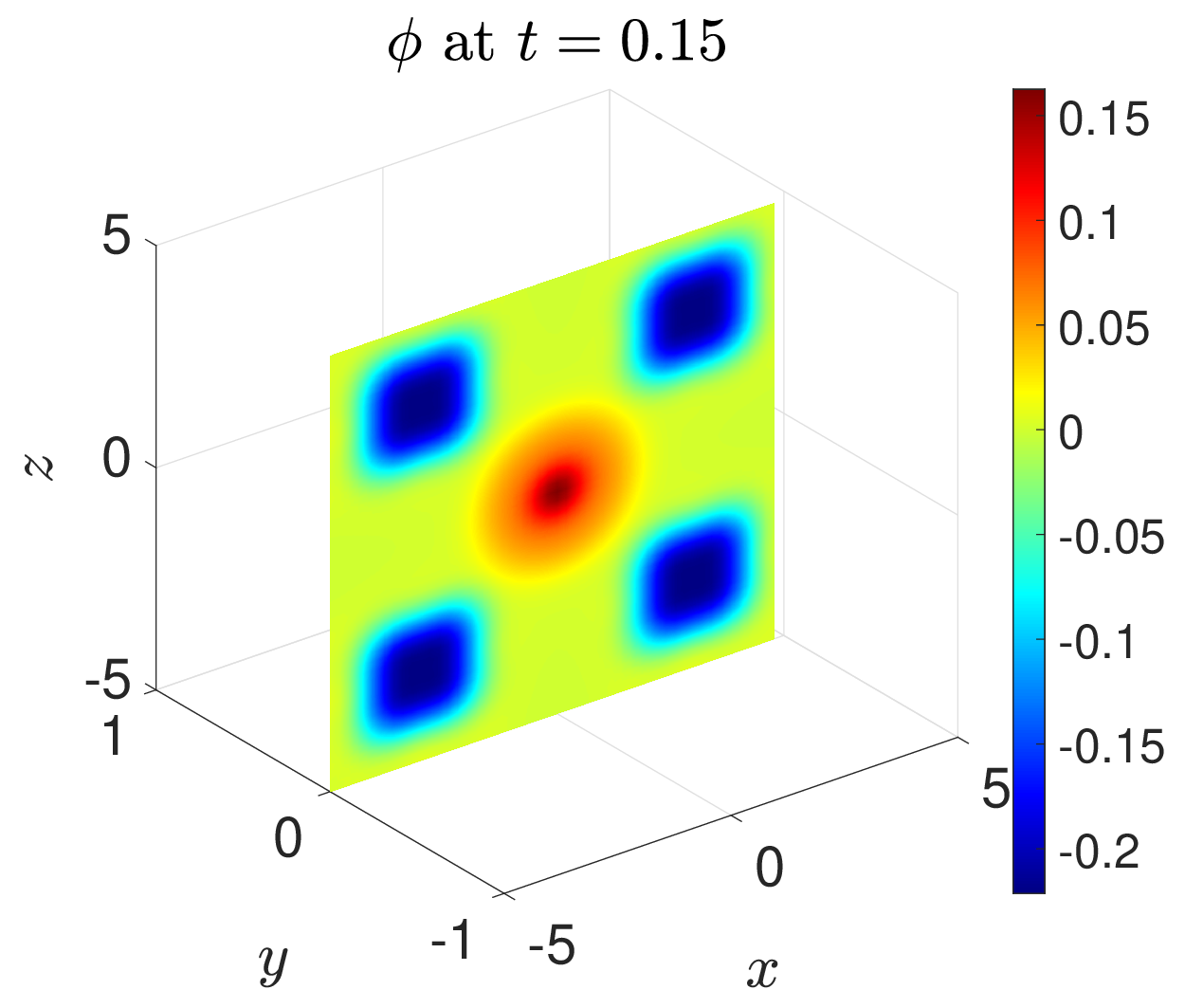}
        \caption{}
    \end{subfigure}
    \hfill
    \begin{subfigure}[t]{0.23\textwidth}
        \centering
        \includegraphics[width=\linewidth, height=3.2cm]{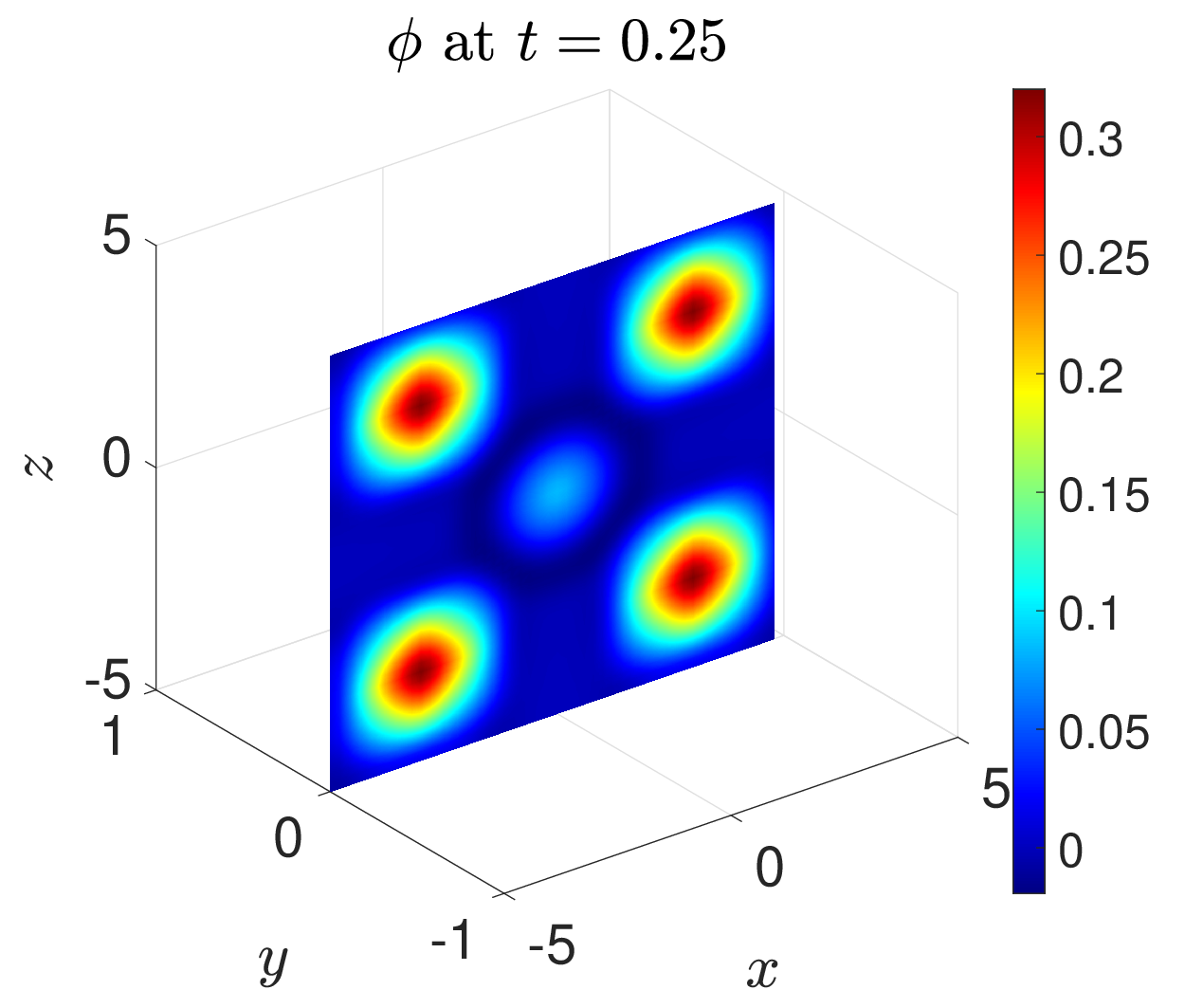}
        \caption{}
    \end{subfigure}
    \hfill
    \begin{subfigure}[t]{0.23\textwidth}
        \centering
        \includegraphics[width=\linewidth, height=3.2cm]{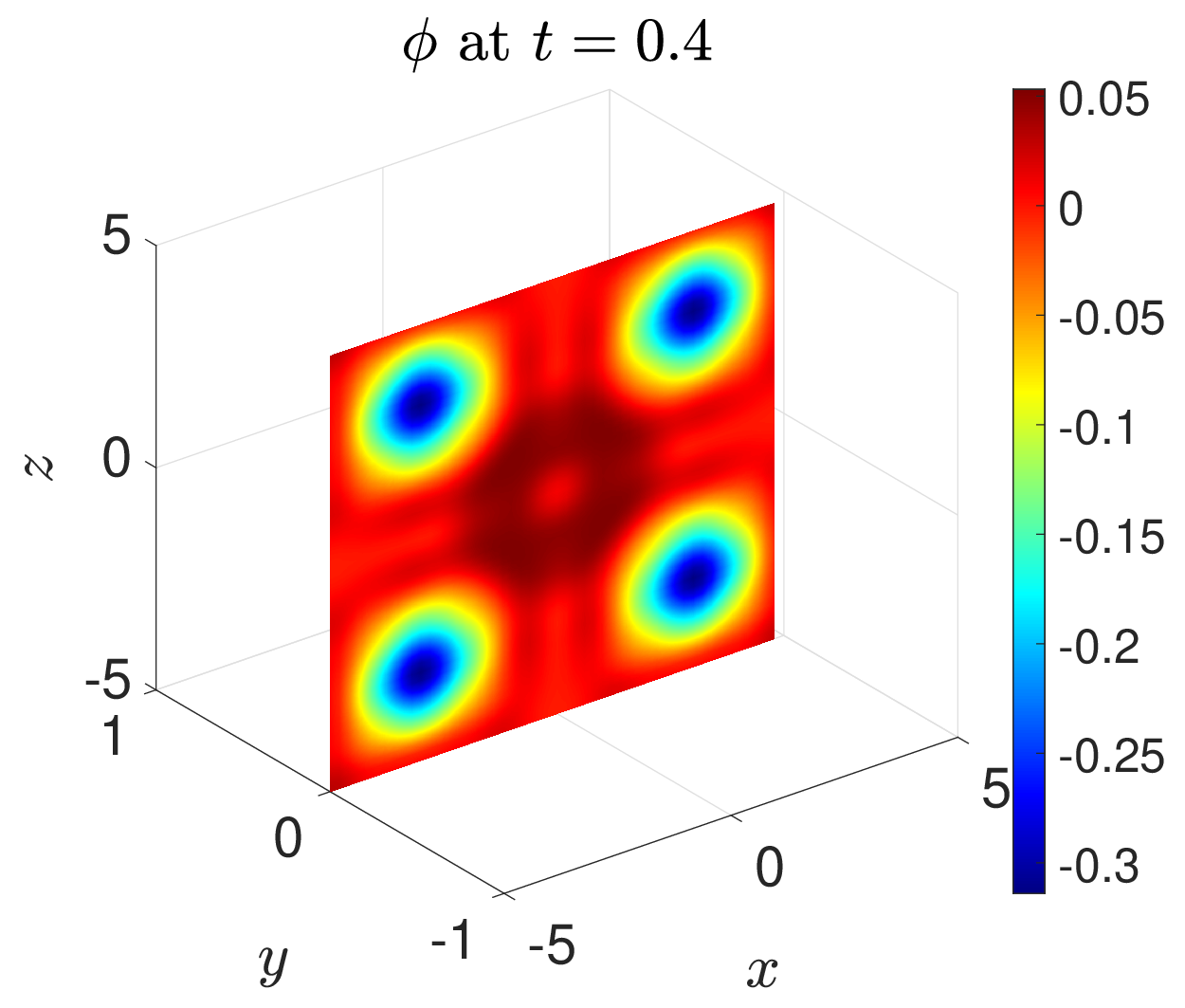}
        \caption{}
    \end{subfigure}
    \caption{The isosurfaces $\phi=0.005$ (first row) and slices of $\phi$ at $y=0$ (secong row) for KGS equation in 3D with $\mu=1$, $\lambda=0.8$ and $\eps=0.1$.}
    \label{fig-5-6-2}
\end{figure}

\begin{figure}[t!]
    \centering
        \includegraphics[width=0.32\linewidth,height=3.1cm]{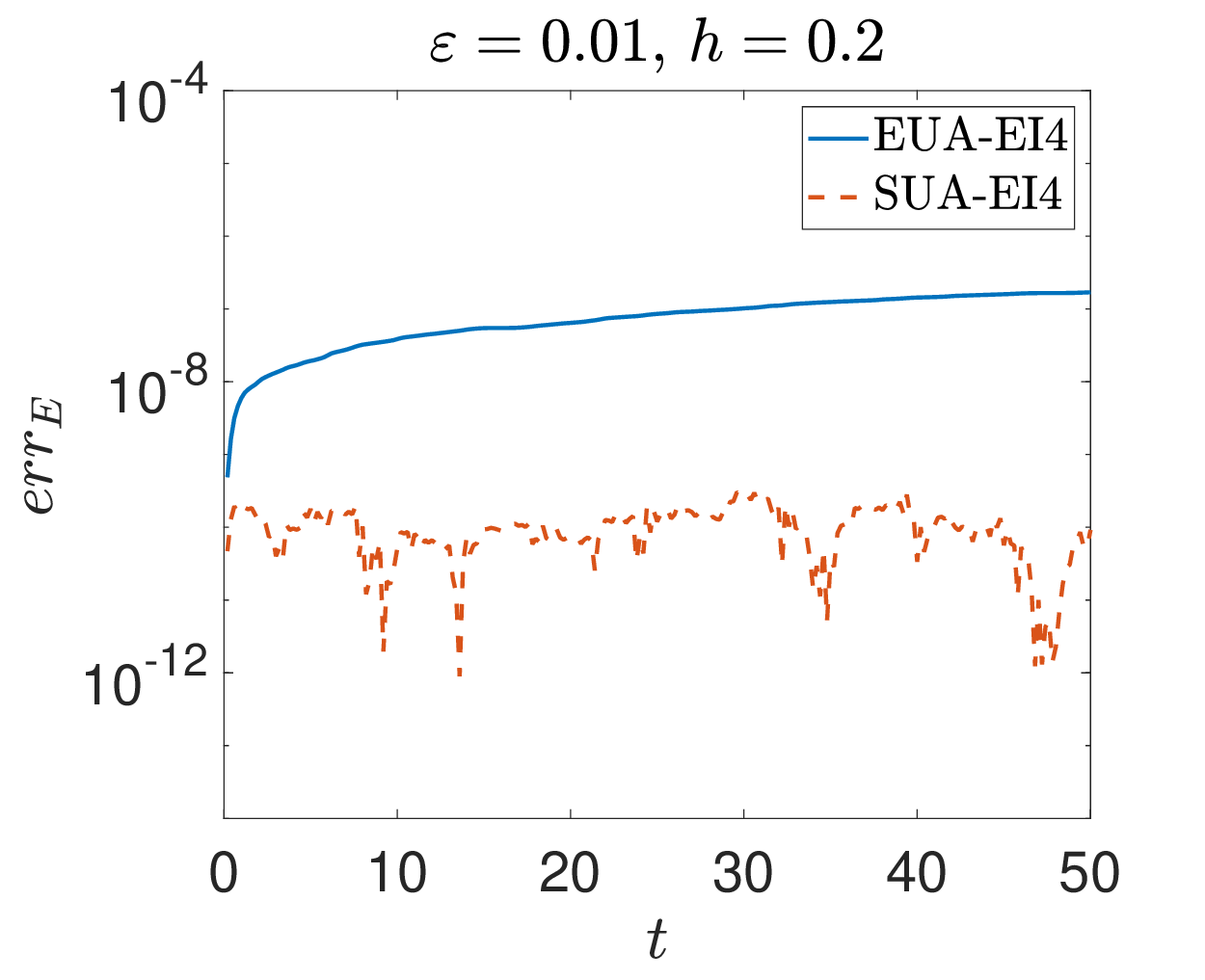}
        \includegraphics[width=0.32\linewidth,height=3.1cm]{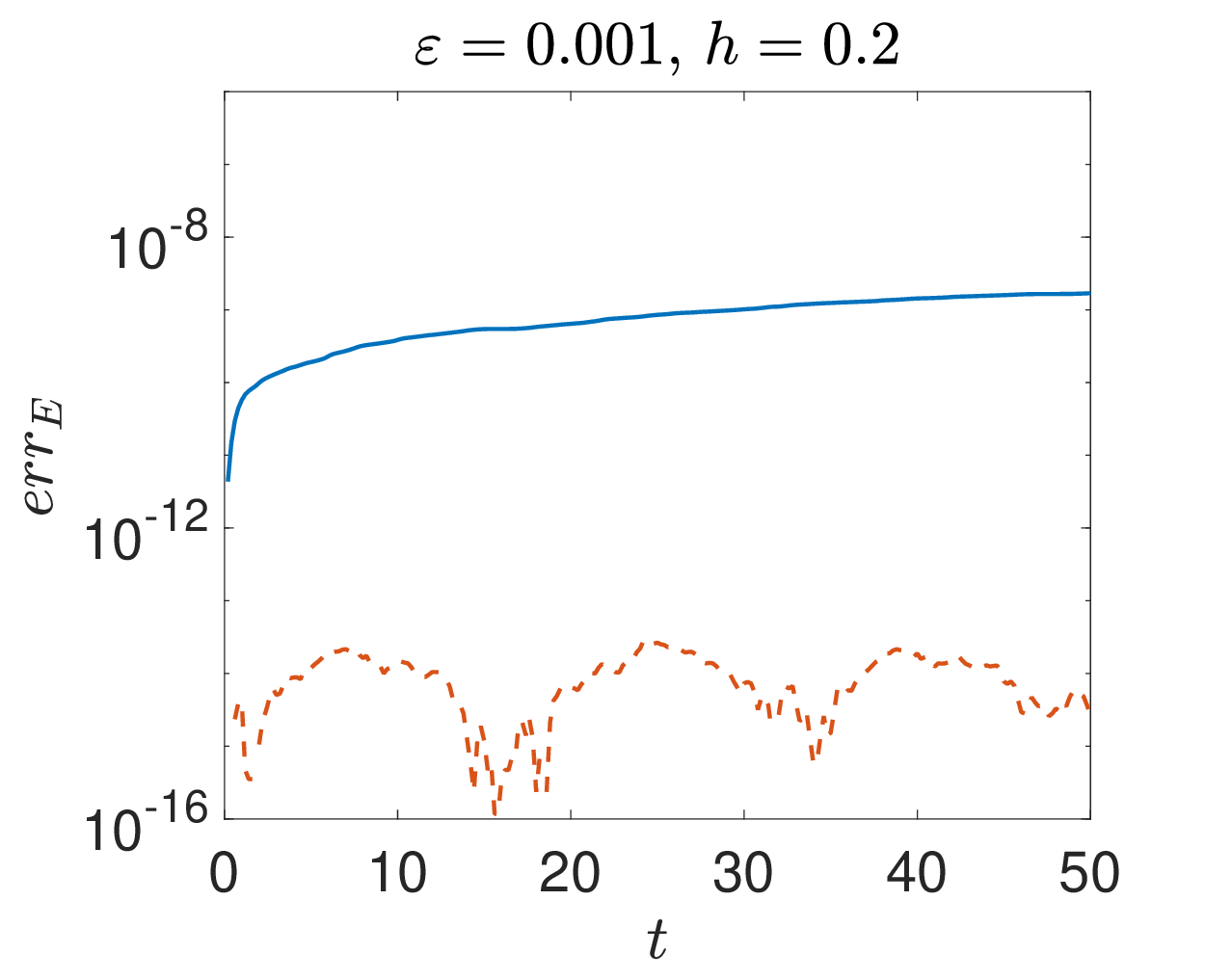}
        \includegraphics[width=0.32\linewidth,height=3.1cm]{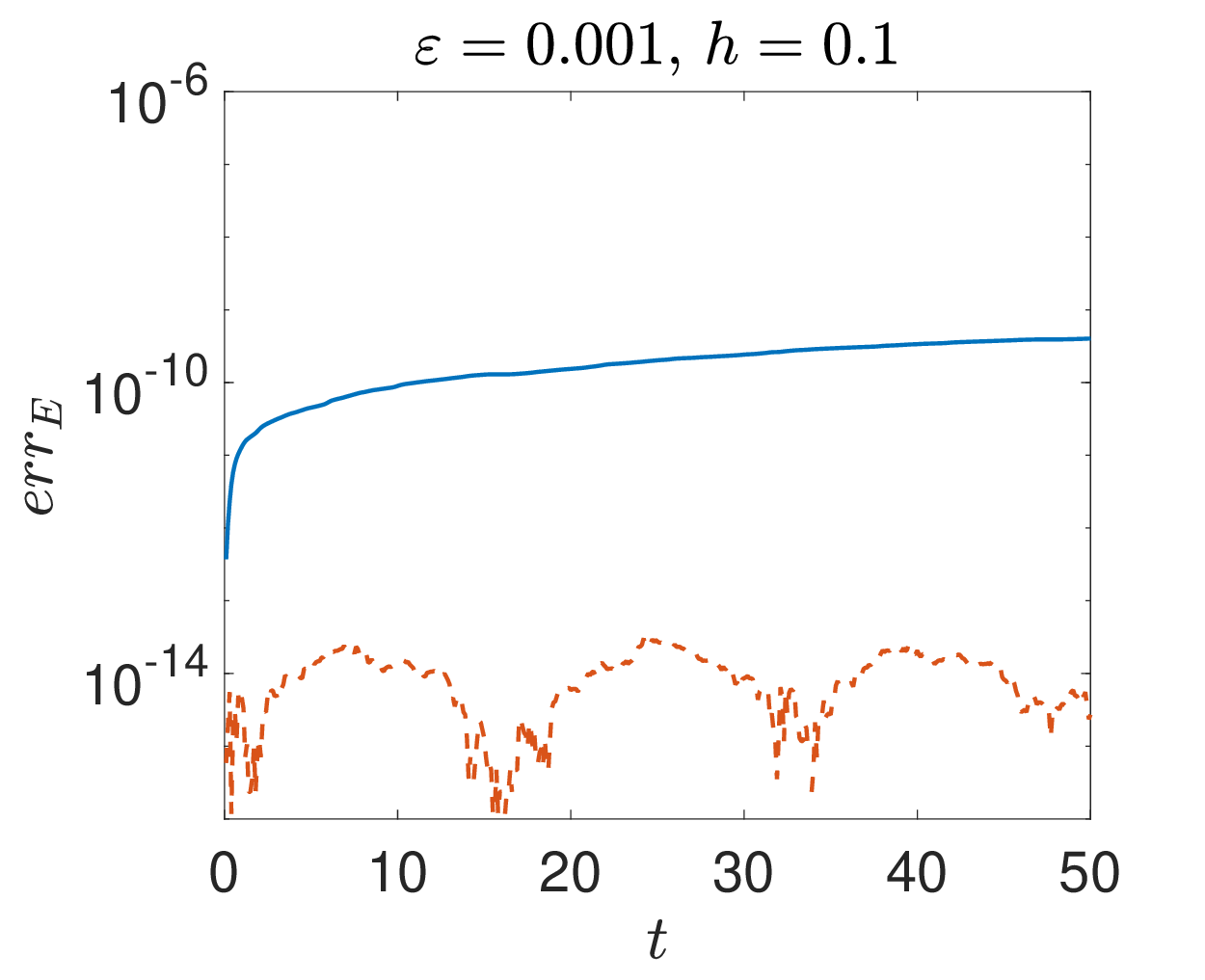}\\
        \includegraphics[width=0.32\linewidth,height=3.1cm]{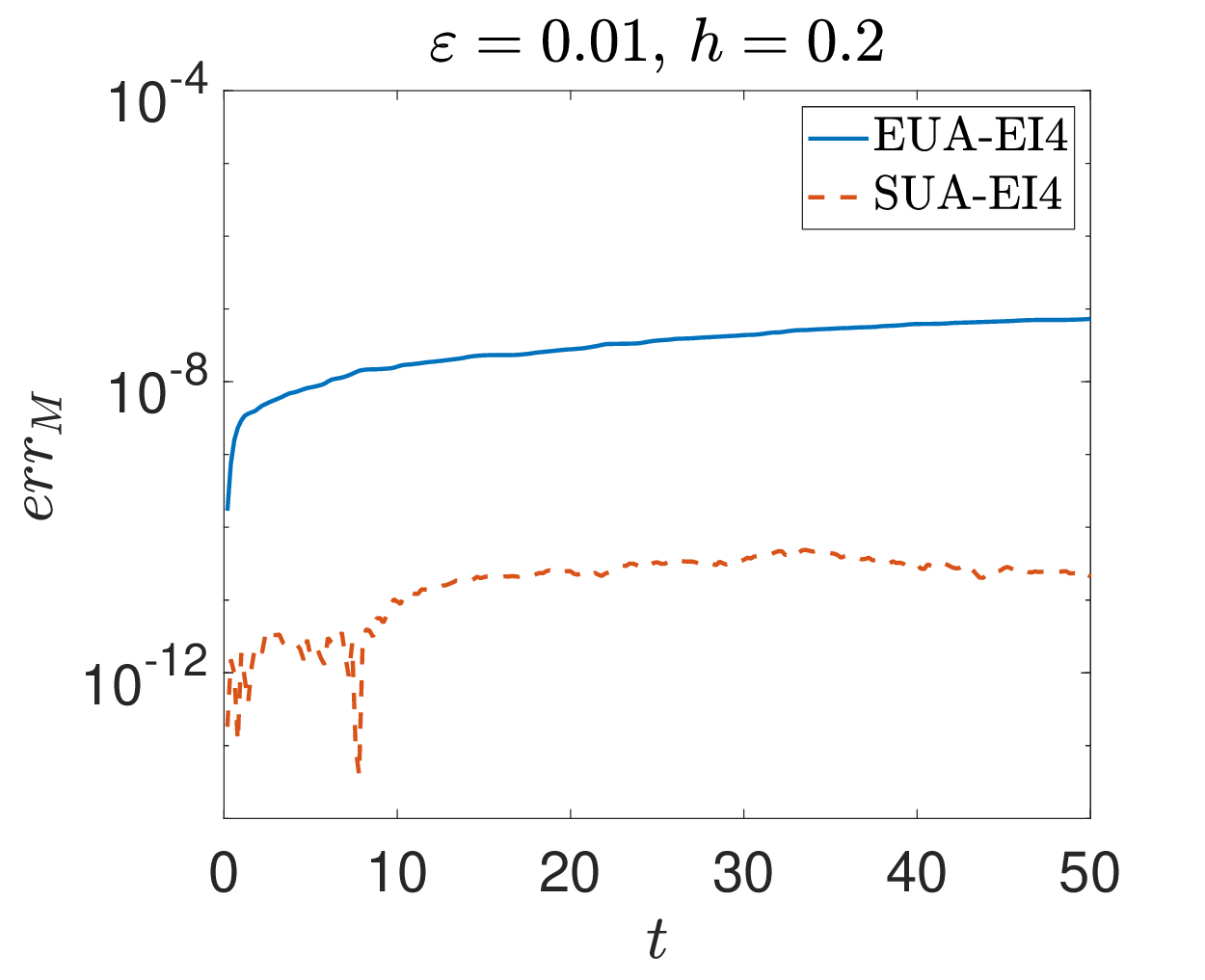}
        \includegraphics[width=0.32\linewidth,height=3.1cm]{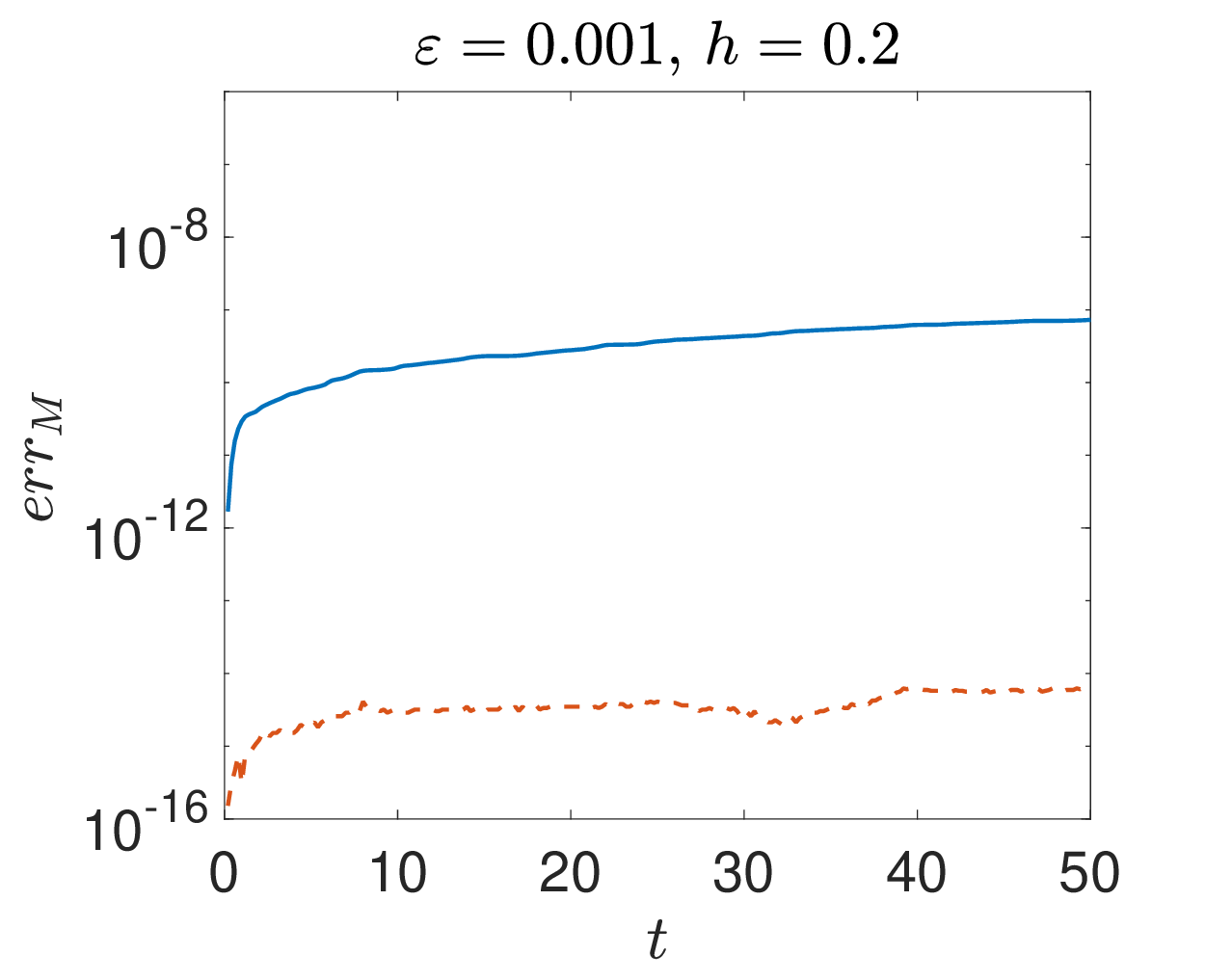}
        \includegraphics[width=0.32\linewidth,height=3.1cm]{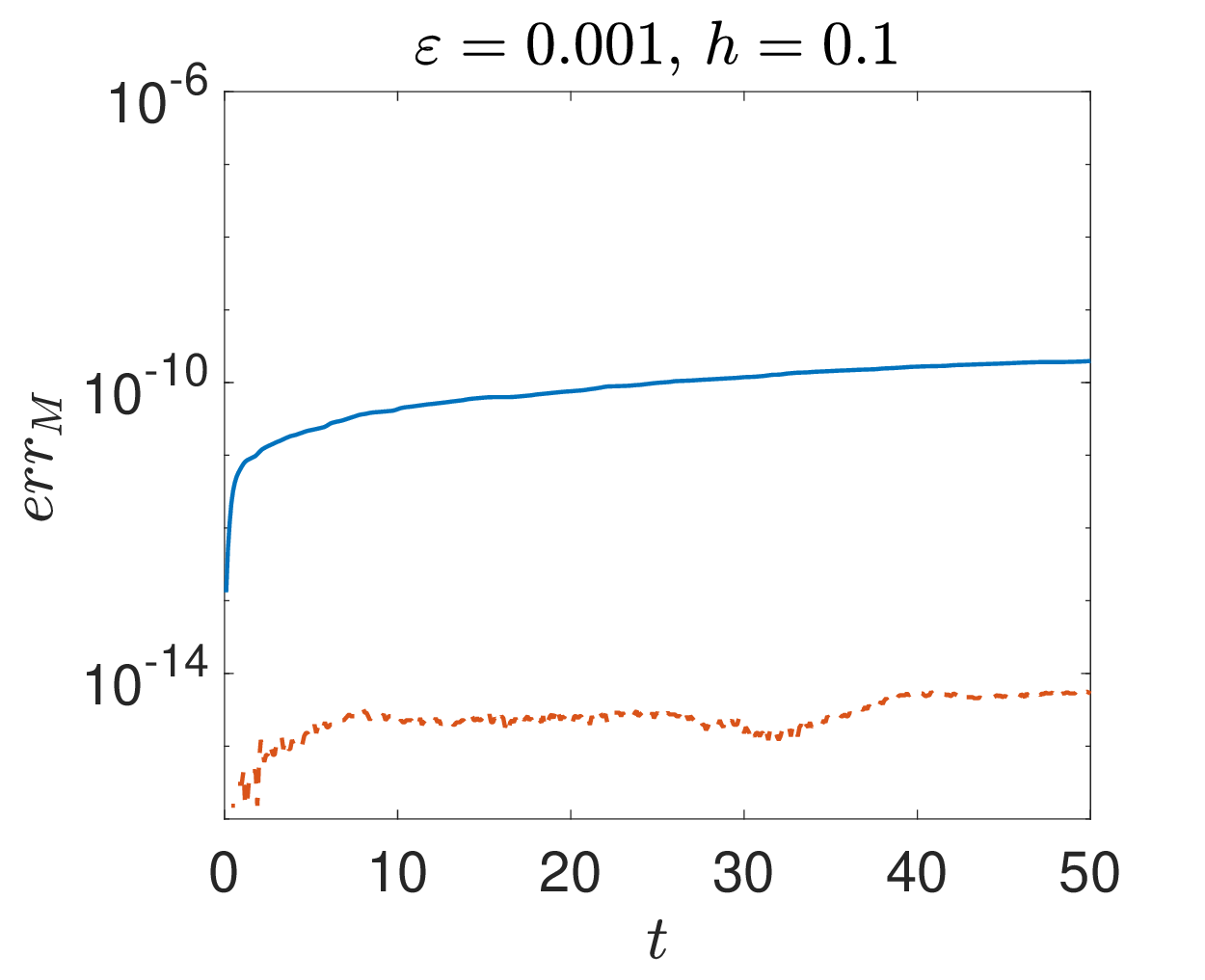}
    \caption{Energy error (top) and mass error (bottom) of KGS in 3D under different $\eps$ and $h$.}
    \label{fig-5-6-3}
\end{figure}

\section{Conclusion}
In this work, the two-scale framework has been applied for the first time to nonlinear dispersive coupled systems in the non-relativistic limit. On this
basis, a class of arbitrarily high-order UA fully discrete schemes has been
constructed, which combines the Fourier pseudospectral discretization in space
with exponential integrators in time. The explicit scheme enables efficient
computation, while the symmetric scheme achieves favorable long-time
near-conservation of mass and energy for both the KGD and KGS systems, owing to
its time-symmetric structure. Numerical experiments in two and three dimensions
reveal a wealth of dynamical behavior inherent in the two classes of systems.
Moreover, the proposed approach is applicable to the KGD equation with a
time-dependent electromagnetic potential; in this setting, energy conservation
no longer holds, whereas mass conservation can be assessed in an analogous
manner. Extending the two-scale framework to a broader class of highly
oscillatory nonlinear dispersive coupled systems, and constructing high-order UA
schemes that preserve such structures approximately over long times, will be
pursued in future work.

\bibliographystyle{model-num-names}

\begin{thebibliography}{00}

\bibitem{AA}
{\sc H. Added, S. Added}, Equations of Langmuir turbulence and nonlinear Schr\"{o}dinger equation: smoothness and approximation. J. Funct. Anal., 79 (1988), pp. 183-210.


\bibitem{BCJT}
{\sc W. Bao, Y. Cai, X. Jia, and Q. Tang}, A uniformly accurate (UA) multiscale time integrator pseudospectral method for the Dirac equation in the nonrelativistic limit regime. SIAM J. Numer. Anal., 52 (2015), pp. 1-41.

\bibitem{BCJY}
{\sc W. Bao, Y. Cai, X. Jia, and J. Yin}, Error estimates of numerical methods for the nonlinear Dirac equation in the nonrelativistic limit regime. Sci. China Math., 59 (2016), pp. 1-34.

\bibitem{BCZ}
{\sc W. Bao, Y. Cai, and X. Zhao}, A uniformly accurate multiscale time integrator pseudospectral method for the Klein-Gordon equation in the nonrelativistic limit regime, SIAM J. Numer. Anal., 52 (2014), pp. 2488-2511.

\bibitem{BS}
{\sc W. Bao and C. Su}, Uniform error estimates of a finite difference method for the Klein-Gordon-Schr\"{o}dinger system in the nonrelativistic and massless limit regimes, Kinet. Relat. Mod., 11 (2018), pp. 1037-1062.

\bibitem{BY}
{\sc W. Bao and L. Yang}, Efficient and accurate numerical methods for the Klein-Gordon-Schr\"{o}dinger equations, J. Comput. Phys., 225 (2007), pp. 1863-1893.

\bibitem{BZ}
{\sc W. Bao and X. Zhao}, A uniformly accurate (UA) multiscale time integrator Fourier pseudospectral method for the Klein-Gordon-Schr\"{o}dinger equations
in the nonrelativistic limit regime, Numer. Math., 135 (2017), pp. 833-873.

\bibitem{BFS}
{\sc S. Baumstark, E. Faou, K. Schratz}, Uniformly accurate exponential type integrators for Klein-Gordon equations with asymptotic convergence to the classical NLS splitting,  Math. Comp., 87 (2018), pp. 1227-1254.

\bibitem{BKS}
{\sc S. Baumstark, G. Kokkala, K. Schratz}, Asymptotic consistent exponential-type integrators for Klein-Gordon-Schr\"{o}dinger systems from relativistic to non-relativistic regimes, Electron. T. Numer. Ana., 48 (2018), pp. 63-80.

\bibitem{BMP}
{\sc P. Bechouche, N.J. Mauser, F. Poupaud}, (Semi)-Nonrelativistic limits of the Dirac equation with external time-dependent electromagnetric field, Commun. Math. Phys., 197 (1998), pp. 405-425.

\bibitem{Bil}
{\sc P. Biler}, Attractors for the system of Schrödinger and Klein-Gordon equations with Yukawa coupling, SIAM J. Math. Anal., 21 (1990), pp. 1190-1212.

\bibitem{Bou1}
{\sc N. Bournaveas}, Local existence of energy class solutions for the Dirac-Klein-Gordon equations, Commun. Partial Differ. Equ., 24 (1999), pp. 1167-1193.

\bibitem{Bou2}
{\sc N. Bournaveas},  Low regularity solutions of the Dirac-Klein-Gordon equations in two space dimensions, Commun. Partial Differ. Equ.,  26 (2021), pp. 1345-1366.

\bibitem{CY1}
{\sc Y. Cai and W. Yi}, A uniformly accurate method for the Klein-Gordon-Dirac system in the nonrelativistic regime. J. Comput. Phys., 486 (2023), 112105.

\bibitem{CY2}
{\sc Y. Cai and W. Yi}, Error estimates of finite difference time domain methods for the Klein-Gordon-Dirac system in the nonrelativistic limit regime. Commun. Math. Sci., 16 (2018), pp. 1325-1346.

\bibitem{CY3}
{\sc Y. Cai and W. Yi}, Uniformly accurate nested Picard iterative integrators for the Klein-Gordon-Dirac System in the nonrelativistic limit regime, J. Sci. Comput., 109 (2026), 2.

\bibitem{CZ}
{\sc Y. Cai and X. Zhou}, Uniformly accurate nested Picard iterative integrators for the Klein-Gordon-Schr\"{o}dinger equation in the nonrelativistic regime, Numer. Algorithms., 94 (2023), pp. 371-396.

\bibitem{Cal}
{\sc M.C. Calvo}, Uniformly accurate integrators for Klein-Gordon-Schr\"{o}dinger systems from the classical to non-relativistic limit regime, J. Comput. Appl. Math., 420 (2023), 114756.


\bibitem{CCLM}
{\sc Ph. Chartier, N. Crouseilles, M. Lemou, and F. M\'{e}hats}, Uniformly accurate numerical schemes for highly oscillatory Klein-Gordon and nonlinear Schr\"{o}dinger equations, Numer. Math., 129 (2015), pp. 211-250.

\bibitem{CCLMZ}
{\sc Ph. Chartier, N. Crouseilles, M. Lemou, F. M\'{e}hats and X. Zhao}, Uniformly accurate methods for three dimensional Vlasov equations under strong magnetic field with varying direction, SIAM J. Sci. Comput., 42 (2020), pp. B520-B547.

\bibitem{Zhao}
{\sc Ph. Chartier, M. Lemou, F. M\'{e}hats, and X. Zhao},
Derivative-free high-order uniformly accurate schemes for highly-oscillatory systems, IMA J. Numer. Anal.,  42 (2022), pp. 1623-1644.

\bibitem{CZ}
{\sc G. Chen and Y. Zheng}, Solitary waves for the Klein-Gordon-Dirac model. J. Differ. Equ., 7 (2012), pp. 2263-2284.

\bibitem{Dirac}
{\sc P.A.M. Dirac}, Principles of Quantum Mechanics. Oxford University Press, London (1958).

\bibitem{FS}
{\sc E. Faou, K. Schratz}, Asymptotic preserving schemes for the Klein-Gordon equation in the non-relativistic limit regime, Numer. Math., 126 (2014), pp. 441-469. 

\bibitem{FL}
{\sc Y. Feng and C. Liu}, A Uniformly Accurate Multiscale Time Integrator for the Klein-Gordon-Schr\"{o}dinger Equations in the Nonrelativistic Regime via Simplified Transmission Conditions, arXiv:2605.12936, 2026.

\bibitem{FXY}
{\sc Y. Feng, Z. Xu, and J. Yin},  Uniform error bounds of exponential wave integrator methods for the long-time dynamics of the Dirac equation with small
potentials, Appl. Numer. Math., 172 (2022), pp. 50-66.

\bibitem{FGWC}
{\sc Y. Fu, X. Gu, Y. Wang, and W. Cai}, Mass-, and Energy Preserving Schemes with Arbitrarily High Order for the Klein-Gordon-Schr\"{o}dinger Equations, J. Sci. Comput., 97 (2023), 75.

\bibitem{FT}
{\sc I. Fukuda, M. Tsutsumi}, On the Yukawa-coupled Klein-Gordon-Schr\"{o}dinger equations in three space dimensions. Proc. Jpn. Acad., 51 (1975), pp. 402-405.

\bibitem{Gre}
{\sc W. Greiner}, Relativistic Quantum Mechanics-Wave Equations, Springer, Berlin, Heidelberg, 1994.

\bibitem{GM}
{\sc B. Guo, C. Miao}, Asymptotic behavior of coupled Klein-Gordon-Schr\"{o}dinger equations. Sci. China Ser. A, 25 (1995), pp. 705-714.

\bibitem{GLLM}
{\sc S. Guo, C. Li, X. Li, and L. Mei}, Energy-conserving and time-stepping-varying ESAV-Hermite-Galerkin spectral scheme for nonlocal Klein-Gordon-Schr\"{o}dinger system with fractional Laplacian in unbounded domains, J. Comput. Phys., 458 (2022), 111096.

\bibitem{GMYL}
{\sc S. Guo, L. Mei, W. Yan, and Y. Yin}, Mass-, Energy-, and Momentum-Preserving spectral scheme for Klein-Gordon-Schr\"{o}dinger system on infinite domains, SIAM J. Sci. Comput., 45 (2023), pp. B200-B230.

\bibitem{HO06}
{\sc M. Hochbruck and A. Ostermann},  Explicit exponential Runge--Kutta methods for semilinear parabolic problems, SIAM J. Numer. Anal., 43 (2006), pp. 1069-1090.

\bibitem{HO10}
{\sc M. Hochbruck and A. Ostermann}, Exponential integrators, Acta Numer., 19 (2010), pp. 209-286.

\bibitem{Hol}
{\sc J.W.V. Holten}, On the electrodynamics of spinning particles, Nucl. Phys. B 356 (1991) 3-26.


\bibitem{Li}
{\sc J. Li}, Improved uniform error bounds of an exponential wave integrator method for the Klein-Gordon-Schr\"{o}dinger equation with the small coupling constant., Commun. Math. Sci., 22 (2024), pp. 583-612.

\bibitem{LJ}
{\sc J. Li and X. Jin}, Structure-preservmg exponential wave integrator methods and the long-time convergence analysis for the Klein-Gordon-Dirac equation with the small
coupling constant, Numer. Methods Partial Differ. Equ., 39 (2023), pp. 3375-3416.

\bibitem{LWCD}
{\sc J. Li, X. Wang, Q. Chen, and S. Deng}, Improved uniform error bounds of a Lawson-type exponential wave integrator method for the Klein-Gordon-Dirac equation. Appl. Math. Comput., 479 (2024), 128877.

\bibitem{LSZ}
{\sc X. Li, Z. Sheng, and L. Zhang}, High-order Lagrange multiplier method for the coupled Klein-Gordon-Schr\"{o}dinger system, J. Comput. Phys., 493 (2023), 112456.

\bibitem{LGW}
{\sc F. Liao, F. Geng, and T. Wang}, A mass and energy conservative fourth‑order compact  difference scheme for the Klein‑Gordon‑Dirac equations, Calcolo, 59 (2022), 9.

\bibitem{LWZ}
{\sc K. Liu, B. Wang, and X. Zhao}, Solving long-time nonlinear Schr\"{o}dinger equation by a class of oscillation-relaxation integrators, SIAM Multi. Model. Simul., 23 (2025), pp. 313-338.

\bibitem{MLJ}
{\sc L. Mei, X. Liu, and Y. Jiang}, Unconditionally stable explicit exponential methods for the Klein-Gordon-Schr\"{o}dinger equations, J. Comput. Phys., 533 (2025), 113993.

\bibitem{MRZ}
{\sc S.L. Musher, A.M. Rubenchik, and V.E. Zakharov}, Weak Langmuir turbulence, Phys. Rep. 252 (1995), pp. 77-274.

\bibitem{Ni}
{\sc D.R. Nicholson}, Topics in strong Langmuir turbulence, Phys. Scr. 27 (1983), pp. 77-82.

\bibitem{Ohl}
{\sc T. Ohlsson}, Relativistic Quantum Physics: From Advanced Quantum Mechanics to Introductory Quantum Field Theory, Cambridge University Press, 2011.

\bibitem{STW}
{\sc J. Shen, T. Tang, and L. Wang}, Spectral Methods: Algorithms, Analysis and Applications, Springer, 2011.

\bibitem{Shen}
{\sc Y. Shen}, The Principles of Nonlinear Optics, Wiley, New York, 1984.

\bibitem{SK}
{\sc J.J. Slawianowski, V. Kovalchuk}, Klein-Gordon-Dirac equation: physical justification and quantization attempts, Rep. Math. Phys. 49 (2002), pp. 249-257.

\bibitem{WZ23} 
{\sc B. Wang and X. Zhao}, Geometric two-scale integrators for highly oscillatory system: uniform accuracy and near conservations, SIAM J. Numer. Anal., 61 (2023),  pp. 1246-1277.


\bibitem{WWL}
{\sc L. Wang, B. Wang, and J. Li}, Fourth-order uniformly accurate integrators with long time near conservations for the nonlinear Dirac equation in the nonrelativistic regime, SIAM Multi. Model. Simul., 24 (2026), pp. 1-31.

\bibitem{WZ}
{\sc M. Wang and Y. Zhou}, The periodic wave solutions for the Klein-Gordon-Schr\"{o}dinger equations, Phys. Lett. A, 318 (2003), pp. 84-92.

\bibitem{WZJ}
{\sc T. Wang, X. Zhao, and J. Jiang},  Unconditional and optimal $H^{2}$-error estimates of two linear and conservative finite difference schemes for the klein-Gordon-Schr\"{o}dinger equation in high dimensions, Adv. Comput. Math., 44 (2018), pp. 477-503.

\bibitem{WZ}
{\sc Y. Wang and X. Zhao}, Symmetric high order Gautschi-type exponential wave integrators psedospectral method for the nonlinear Klein-Gordon equation in the nonrelativistic limit regime, J. Numer. Anal. Model., 15 (2018), pp. 405-427.

\bibitem{CY4}
{\sc W. Yi and Y. Cai}, Optimal error estimates of finite difference time domain methods for the Klein-Gordon-Dirac system. IMA J. Numer. Anal. 40 (2020), pp. 1266-1293.

\bibitem{YRS}
{\sc W. Yi, X. Ruan, and C. Su}, Optimal resolution methods for the Klein-Gordon-Dirac system in the nonrelativistic limit regime, J. Sci. Comput., 79 (2019), pp. 19907-1935.

\bibitem{ZS}
{\sc Y. Zhang and J. Shen}, Efficient Structure Preserving Schemes for the Klein-Gordon-Schr\"{o}dinger Equations, J. Sci. Comput., 89 (2021), 47.



\end{thebibliography}

\end{document}